\documentclass{amsart}

\newtheorem{theorem}{Theorem}[section]
\newtheorem{lemma}[theorem]{Lemma}
\newtheorem{cor}[theorem]{Corollary}
\theoremstyle{definition}
\theoremstyle{pro}
\newtheorem{definition}[theorem]{Definition}
\newtheorem{pro}[theorem]{Proposition}

\usepackage{tikz}
\usepackage[small,nohug,heads=vee]{diagrams}
\usepackage{mathrsfs}
\diagramstyle[labelstyle=\scriptstyle]
\theoremstyle{remark}
\newtheorem{remark}[theorem]{Remark}
\numberwithin{equation}{section}

\begin{document}

\title[The moduli space of $S^1$-type zero loci]{ The moduli space of $S^1$-type zero loci for $\mathbb{Z}/2$-harmonic spinors in dimension 3 }

\author{Ryosuke Takahashi}
\address{Department of mathematics, National Cheng Kung University, No.1, University Road, East District, Tainan, Taiwan}
\email{tryo@mail.ncku.edu.tw}



\date{}


\begin{abstract}
Let $M$ be a compact oriented 3-dimensional smooth manifold. In this paper, we construct a moduli space consisting of pairs $(\Sigma, \psi)$ where $\Sigma$ is a $C^1$-embedding simple closed curve in $M$, $\psi$ is a $\mathbb{Z}/2$-harmonic spinor vanishing only on $\Sigma$, and $\|\psi\|_{L^2_1}\neq 0$. We prove that when $\Sigma$ is $C^2$, a neighborhood of  $(\Sigma, \psi)$ in the moduli space can be parametrized by the space of Riemannian metrics on $M$ locally as the kernel of a Fredholm operator.
\end{abstract}

\maketitle
\setcounter{tocdepth}{1}
\tableofcontents

\section{Introduction}
\subsection{Main theorem and its background and motivations}
Let $M$ be a compact, oriented, three-dimensional smooth manifold without boundary.
\begin{definition} 
\begin{align*}
&\mathcal{X}=\{g\mbox{ }| \mbox{ }g\mbox{ is a Riemannian metric on } M\};\\
&\mathcal{A}=\{\Sigma \subset M\mbox{ }| \mbox{ }\Sigma\mbox{ is the image of a }C^1\mbox{ embedding }S^1\rightarrow M\}.
\end{align*}
\end{definition}
For each $\Sigma\in\mathcal{A}$, let $H_{\Sigma}\subset H^1(M\setminus\Sigma;\mathbb{Z}/2)$ consist of the elements with nonzero monodromy\footnote{Real line bundles on $M\setminus\Sigma$ can be totally determined by the transition functions $\{U_{\alpha\beta}, g_{\alpha\beta}\}$ with structure group $\mathbb{Z}/2$, which are one-one corresponding to the elements in the sheaf cohomology $H^1(M\setminus\Sigma;\mathbb{Z}/2)$. In other words, an element $e\in H^1(M\setminus\Sigma;\mathbb{Z}/2)$ has non-zero monodromy if and only if its corresponding real line bundle cannot be extended to the whole manifold $M$.} around $\Sigma$. As $\Sigma$ varies, the set $H_{\Sigma}$ varies continuously to define a finite-sheeted covering space of $\mathcal{A}$ with the fibres isomorphic to $H_{\Sigma}$.
\begin{definition}
\begin{align*}
&\mathcal{A}_H=\{(\Sigma,e)\mbox{ }|\mbox{ }\Sigma\in \mathcal{A}, e\in H_{\Sigma}\};\\
&\mathcal{Y}=\mathcal{X}\times \mathcal{A}_H.
\end{align*}
We also denote by $p_1:\mathcal{Y}\rightarrow\mathcal{X}$ the projection from $\mathcal{Y}$ to $\mathcal{X}$.
\end{definition}
Notice that once a topology on $\mathcal{A}$ is given, it will induce a topology on $\mathcal{A}_H$ ($H_{\Sigma}$ being equipped with the discrete topology). By the same reason, a topology defined on $\mathcal{X}\times\mathcal{A}$ will also induce a topology on $\mathcal{Y}$.

 Here we define the topology on $\mathcal{Y}$ for later use. Fix $(g,\Sigma,e)\in \mathcal{Y}$ and $r, C,C'\in \mathbb{R}^+$. Let $N_r$ be the tubular neighborhood of $\Sigma$ with radius $r$, 
\begin{align}
h:S^1\times \{z\in\mathbb{C}||z|<r\}\rightarrow N_r \label{AA_1}
\end{align} 
be a homeomorphism given by the exponential map with respect to the metric $g$. We define
\begin{align}
\mathcal{V}_{\Sigma,r,C}=\big\{\mbox{Im}(h\circ(id_{S^1},\eta))\mbox{ }|\mbox{ }\eta :S^1\rightarrow\mathbb{C}&,\|\eta\|_{C^1(S^1)}\leq C,\mbox{ }|\eta(t)|<r\big\};\label{AA_2}\\
\mathscr{V}_{g, r,C'}=\big\{\hat{g}\in \mathcal{X}\mbox{ }|\mbox{ }\|\hat{g}-g\|_{C^2(M)}< &C';\mbox{ }dist(\Sigma,supp(\hat{g}-g))> r\big\}.\label{AA_3}
\end{align}
It is direct to see that this defines a unique topology on $\mathcal{X}\times\mathcal{A}$ with a basis by the family of open sets 
\begin{align}
\{\mathscr{V}_{g,r,C'}\times\mathcal{V}_{\Sigma,r,C}\} \mbox{ }\mbox{ }\mbox{ }( r, C,C'\in \mathbb{R}^+\mbox{ and }\mbox{ }\Sigma\in\mathcal{A})\label{AA_4}
\end{align}
which induced the expected topology on $\mathcal{Y}$.

Each $(\Sigma,e)\in \mathcal{A}_H$ corresponds to a real line bundle $\mathcal{I}_{\Sigma,e}$ on $M\setminus\Sigma$, equipped with a Euclidean metric and a compatible flat connection. This real line bundle is unique up to bundle isomorphisms. Since $M$ is a compact, oriented, three-dimensional smooth manifold, it is a spin manifold. So each metric $g\in\mathcal{X}$ has at least one \footnote{The number of spinor bundles equals the number of elements in $H^1(M;\mathbb{Z}_2)$ (see \cite{E}), which is finite. So we may simply fix one element in $H^1(M;\mathbb{Z}_2)$ and consider the corresponding spinor bundle in this paper.} corresponding spinor bundle $\mathcal{S}_g \rightarrow M$. Denote by $\mathcal{S}_{g,\Sigma,e}$ the bundle $\mathcal{S}_g\otimes \mathcal{I}_{\Sigma,e}$. This is a spinor bundle over $M\setminus\Sigma$, which is called the $\mathbb{Z}/2$-{\bf spinor bundle}. For each $\mathcal{S}_{g,\Sigma,e}$, there exists a standard Dirac operator with respect to $g$ and the metric defined on $\mathcal{I}_{\Sigma,e}$, denoted by $D^{(g,\Sigma,e)}$. The precise definition of this Dirac operator will appear in Section 2.1.

We define the vector bundle $\mathcal{E} \rightarrow \mathcal{Y}$ of infinite rank as follows:
For $y = (g, \Sigma, e) \in \mathcal{Y}$, then the fiber $\mathcal{E}_y$ over $y$ is the infinite dimensional vector space of all $L^2_1$-sections over $M\setminus\Sigma$ of the $\mathbb{Z}/2$-spinor bundle $\mathcal{S}_{g,\Sigma,e}$. According to the topology we define above, one can obtain the local trivialization of $\mathcal{E} \rightarrow \mathcal{Y}$ as a vector bundle on
each connected component \footnote{Notice that we don't have the identification between fibers corresponding to non-isotopic knots.}.  Let $D^{(y)}=D^{(g,\Sigma,e)}$ denote the Dirac operator defined on $\mathcal{E}_y$, which induces a bounded, linear map from $\mathcal{E}_y$ to the space of square integrable sections of $\mathcal{S}_{g,\Sigma,e}$.

With the local trivialization property on $\mathcal{E}$ and the topology defined on $\mathcal{Y}$, one can obtain the topology defined on $\mathcal{E}$. 
\begin{definition}
With $\mathcal{E}$ understood, the space $\mathfrak{M}$ is defined as the set whose elements are pairs $(y , \psi)\in \mathcal{E}$ which obeys
\begin{align*}
&\bullet D^{(y)}\psi=0\mbox{ on }M\setminus\Sigma.\\
&\bullet |\psi| \mbox{ extends across } \Sigma \mbox{ as a }C^{0,\frac{1}{2}}\mbox{ function on } M \\
 &\mbox{ }\mbox{ }\mbox{ with its zero locus
containing } \Sigma.\\
&\bullet \liminf_{p\rightarrow x}\frac{|\psi|(p)}{\mbox{dist}(p,\Sigma)^{\frac{1}{2}}}>0\mbox{ for all }x\in\Sigma.
\end{align*}
\end{definition}

The set $\mathfrak{M}$ inherits a topology from $\mathcal{E}$. Our goal is to find the Kuranishi structure on $\mathfrak{M}$. To start with, we need a stronger regularity condition for $\Sigma$:
\begin{definition}
Let $(y = (g, \Sigma, e), \psi) \in \mathfrak{M}$ and $k\in\mathbb{N}$. We call it $C^k$-regular if and only if $\Sigma$ is a $C^k$-curve.
\end{definition}
When $\Sigma$ is a $C^1$-curve determined by a $\mathbb{Z}/2$-harmonic spinor, the higher regularity for $\Sigma$ is still an open question so far. Assuming all $\Sigma$ determined by $\mathbb{Z}/2-$harmonic spinors are $C^2$-regular. Then the following theorem proves the existence of the Kuranishi structure for $\mathfrak{M}$:
\begin{theorem}
Given a $C^2$-regular element $(y = (g, \Sigma, e), \psi)$ in $\mathfrak{M}$, there are finite
dimensional vector spaces $\mathbb{K}_0$ and $\mathbb{K}_1$, a ball $\mathbb{B}\subset \mathbb{K}_0$ centered at the origin, a neighborhood $\mathcal{N}\subset\mathcal{E}$ of $y$, and a map
\begin{align*}
f: p_1(\mathcal{N}) \times\mathbb{B}\rightarrow\mathbb{K}_1,
\end{align*}
such that $\mathcal{N}\cap\mathfrak{M}$ is homeomorphic to $f^{-1}(0)$. Here $p_1:\mathcal{E}\rightarrow\mathcal{X}$ is the projection defined in Definition 1.2. Moreover, $f$ is $C^1$ in the sense of Fr\'echet differentiation and the homeomorphism given by this theorem, $\Upsilon:f^{-1}(0)\rightarrow \mathcal{N}\cap\mathfrak{M}$, satisfies \begin{align*}
\Upsilon(f^{-1}(0)\cap\big(\mathbb{B}\times\{g\}\big))=\mathcal{N}\cap\mathfrak{M}\cap p_1^{-1}(g).
\end{align*}
\end{theorem}
We will denote by $\mathcal{B}$ the set $p_1(\mathcal{N})$ in this paper. In fact, $\mathcal{B}=\mathscr{V}_{g, r,C'}$ for some $r,C'$ as the one we defined in (\ref{AA_3}).\\

 The vector spaces $\mathbb{K}_0$ and $\mathbb{K}_1$ in Theorem 1.5 are the kernel and cokernel of a Fredholm operator respectively. This theorem shows us several facts. First, in order to stay in $\mathfrak{M}$, the component $\Sigma$ in $\mathfrak{M}$ can only be perturbed in finite dimensional directions. Second, when $dim(\mathbb{K}_1)=0$, $\mathfrak{M}$ is homeomorphic to $\mathbb{B} \times\mathcal{B}$ near $(y, \psi)$.\\
 
The operator that leads to $\mathbb{K}_1$ and $\mathbb{K}_0$ comes from a formal linearization of the
equations that are obtained by deforming the metric and the curve and the spinor so as to
stay in $\mathfrak{M}$. This operator is novel and the fact that it is Fredholm does not follow from the usual considerations. By the same token, the proof of Theorem
1.4 is not a standard application of the implicit function theorem as it requires a delicate
iteration to ``integrate'' the formal tangent space given by the kernel of $df$ at $(g, 0) \in \mathbb{B} \times \mathcal{B}$ to obtain the given parametrization of $\mathfrak{M}$.

The study of $\mathbb{Z}/2$-harmonic spinors started from the work of $PSL(2,\mathbb{C})$ compactness theorem proved by Clifford Taubes. In \cite{A}, he proves a generalized version of Uhlenbeck's compactness theorem \cite{B}. When $dim(M)=3$, Uhlenbeck's compactness theorem \cite{C} can be stated in the following way:
\begin{theorem} Let $P$ be a principal $G$ bundle over $M$ for some compact Lie group $G$ and $\{A_i\}$ be a sequence of connections on $P$ satisfying
\begin{align}
\|F(A_i)\|_{L^2}\leq C \label{AA_5}
\end{align}
for some constant $C$ which is independent of $i$. Then there exists a subsequence of $\{A_i\}$ converging (up to gauge transformations) weakly in $L^2_1$ to an $L^2_1$ connection.
\end{theorem}

 To state the theorem proved in \cite{A}, we need to introduce some notation. First, Taubes uses the fact that $\mathfrak{sl}(2,\mathbb{C})=\mathfrak{su}(2)\oplus i\mathfrak{su}(2)$ and $P$ can be regarded as one of its $SO(3)$-reductions associated with $PSL(2,\mathbb{C})$. Fix one such reduction $Q$ and set $P=Q\times_{SO(3)}PSL(2,\mathbb{C})$. Therefore, he can always decompose a connection $\mathbb{A}=A+ia$ where $A$ is the connection one form on the $SO(3)-$reduction of $P$ and $a$ is a $\mathfrak{su}(2)$-valued one form. Second, if we denote the group of gauge transformations (the automorphism group of $P$) by $\mathcal{G}$, then the Lie algebra $\mathfrak{sl}(2,\mathbb{C})$ does not have norms which are invariant under the action of $\mathcal{G}$. So we should refine the $L^2$ boundedness condition (\ref{AA_5}) as follows:
\begin{definition}
Let
\begin{align*}
\mathcal{F}(\mathbb{A})=\inf_{A+ia\in\mathcal{G}_{\mathbb{A}}}\int |F(A)-a\wedge a|^2+|d_Aa|^2+|d_A*a|^2
\end{align*}
where $\mathcal{G}_{\mathbb{A}}$ is the $\mathcal{G}$-orbit of $\mathbb{A}$. 
\end{definition}

 Now, the generalized Uhlenbeck's compactness theorem proved in \cite{A} can be stated as follows:
\begin{theorem}
  Let $\{\mathbb{A}_i=A_i+ia_i\}$ be a sequence of connections defined on $Q\times_{SO(3)}PSL(2,\mathbb{C})$ such that $\{\mathcal{F}(\mathbb{A}_i)\}$ is bounded. If $\|a_i\|_{L^2}\rightarrow \infty$, then there exists a closed subset $\Sigma$ of Hausdorff dimension at most one and a subsequence of $\{\mathbb{A}_i=A_i+ia_i\}$ such that
both $\{\frac{1}{\|a_i\|_{2}}a_i\}$ and $\{A_i\}$ converge weakly in the $L^2_{1,loc}$-sense on $M\setminus\Sigma$ up to automorphisms of $Q$.
\end{theorem}

Moreover, $\Sigma$ can be formulated as the zero locus of a $\mathbb{Z}/2$-harmonic spinor. In \cite{A}, Taubes shows the set $\Sigma$ will always have a corresponding $\mathbb{Z}/2$-
spinor $\psi$ satisfying the Dirac equation $D\psi=0$ such that $|\psi|$ can be extended H\"older
continuously to be zero on $\Sigma$.

The $PSL(2,\mathbb{C})$ compactness theorem suggests that data sets consisting of pairs
$(\Sigma,\psi)$ with $\Sigma$ a closed set of Hausdorff dimension 1 set and $\psi$ a $\mathbb{Z}/2$-harmonic spinor with norm zero on $\Sigma$ should play a role to play in three-dimensional differential topology. So a natural question we can ask is the following: Can we find a way to parametrize the data $(\Sigma, \psi)$?

 Meanwhile, in \cite{B}, Taubes shows that $\Sigma$ is a $C^1$-curve on an open dense subset of $\Sigma$. After this work, Zhang \cite{P} shows that $\Sigma$ is always a rectifiable curve. All these results indicate the conjecture that $\Sigma$ is a $C^1$ curve for the metric $g$ which is suitably generic. This conjecture is also mentioned in \cite{J}. So it is natural to take $\Sigma$ a 1-dimensional submanifold in this paper. 

\subsection{The outline of the proof and the structure of this paper}

In the first part of this paper, we shall study model solutions of Dirac equation with $\Sigma$ fixed. Let
\begin{align}
N_R=\{p\in M\mbox{ }|\mbox{ }dist(p,\Sigma)\leq R\}\label{AA_6}
\end{align}
be a tubular neighborhood of $\Sigma$. We parametrize $N_R$ by
\begin{align}
(t,z)\in [0,2\pi]\times\{z\in\mathbb{C}||z|<R\}.\label{AA_7}
\end{align}
Then, we show that any $\mathbb{Z}/2$-harmonic spinor $\psi$ which vanishes along $\Sigma$ is in\\ $\ker(D|_{L^2_1(\mathcal{S}_{g,\Sigma,e})})$ and vice versa. For any $\psi_0\in \ker(D|_{L^2_1(\mathcal{S}_{g,\Sigma,e})})$, it can be written locally as
\begin{align}
\psi_0=\left(\begin{array}{c}
a^+(t)\sqrt{z}\\
a^-(t)\sqrt{\bar{z}}
\end{array}\right)+\mbox{higher order term}\label{AA_8}
\end{align}
on $N_R$. Here the “higher order term" is a smooth section defined on $M\setminus\Sigma$ with order $O(|z|^{p})$ for some $p>\frac{1}{2}$ as $|z|\rightarrow 0$. In addition, by the Lichnerowicz-Weitzenb\"ock formula, we will see $dim(\ker(D|_{L^2_1(\mathcal{S}_{g,\Sigma,e})}))<\infty$. All these basic analysis results for $L^2_1-$harmonic spinors will be shown in Section 2 and 3. Also, the analysis of $L^2$-harmonic spinors will be derived in Section 2 and 3 for later use.\\

According to these observations, one can consider the linear perturbation for any given $p=(g_0,\Sigma_0,\psi_0)\in\mathfrak{M}$. (Note that the element $e\in H$ will be omitted in the rest of the paper because this discrete data will not change in any local perturbation.) This perturbation can be written as $(g_0+s\delta, \Sigma_s,\psi_s)$ for small $s\in \mathbb{R}$, where \\
\ \\[-2mm]
$\bullet$ $\delta\in\mathscr{V}:=\{\delta\mbox{ is a smooth symmetric}(2,0)\mbox{-tensor with }supp(\delta)\cap \Sigma_0=\emptyset\}$;\\
$\bullet$  $\Sigma_s=\{h(t,-s\eta(t))\}$ for some $\eta\in C^1(S^1;\mathbb{C})$(Recall the topology defined on $\mathcal{Y}$\\$\mbox{ }\mbox{ }\mbox{ }$in Section 1.1 and the definition of $h$ in (\ref{AA_1}));\\
$\bullet$ $\psi_s(t,z)=\psi_0(t,z-s\eta(t))+s\phi(t,z-s\eta(t))$ for some $\phi\in L^2_1(\mathcal{S}_{g_0,\Sigma_0})$.\\
\ \\
Let us denote by $D^{(s)}$ the Dirac operator with respect to $g_0+s\delta$ and define the operator
\begin{align}
\mathfrak{L}_p:\mathscr{V}\times C^1(S^1;\mathbb{C})\times L^2_1(\mathcal{S}_{g_0,\Sigma_0})&\rightarrow L^2(\mathcal{S}_{g_0,\Sigma_0});\nonumber\\
(\delta, \eta, 
\phi)&\mapsto\frac{d}{ds}(D^{(s)}\psi_s)|_{s=0}.\label{AA_9}
\end{align}
By (\ref{AA_9}) and some basic analysis results derived in Section 4, we prove that there exists a map
\begin{align}
\Phi:\mathscr{V}\rightarrow L^2(\mathcal{S}_{g_0,\Sigma_0})\label{AA_10}
\end{align}
such that 
\begin{align}
\mathfrak{L}_p(\delta, \eta, \phi)-\Phi(\delta)\in \text{range}(D|_{L^2(\mathcal{S}_{g_0,\Sigma_0})})\label{AA_11}
\end{align}
for any $\eta\in C^1(S^1;\mathbb{C})$ and $\phi\in L^2(\mathcal{S}_{g_0,\Sigma_0})$ (see Remark 6.4 for the definition of $\Phi$). In particular, we have $\Phi(0)=0$. Therefore, we define $\mathbb{K}_0$ and $\mathbb{K}_1$ to be the kernel and cokernel of $\mathfrak{L}_p|_{\delta=0}$. By (\ref{AA_11}) and the fact $\Phi(0)=0$, any element in $\ker(\mathfrak{L}_p|_{\delta=0})$ corresponds to an element in $\ker(D|_{L^2(\mathcal{S}_{g_0,\Sigma_0})})$. By (\ref{AA_8}), (\ref{AA_9}) and some straight-forward computation (which will be showed in the first two pages in Section 6.2), the corresponding element in   $\ker(D|_{L^2(\mathcal{S}_{g_0,\Sigma_0})})$ has the form
\begin{align}
\left(\begin{array}{c}
\frac{a^+(t)\eta(t)}{2\sqrt{z}}\\
\frac{a^-(t)\bar{\eta}(t)}{2\sqrt{\bar{z}}}
\end{array}
\right)+\mbox{higher order term}.\label{AA_12}
\end{align}
Here the ``higher order term'' is a smooth section defined on $M\setminus\Sigma$ with order $O(|z|^p)$ for some $p>-\frac{1}{2}$. Notice that the inequality $\liminf_{p\rightarrow x}\frac{|\psi_0|(p)}{\mbox{dist}(p,\Sigma_0)^{\frac{1}{2}}}> 0$ for all $x\in \Sigma$ in Definition 1.3 implies $|a^+|^2(t)+|a^-|^2(t) > 0$ for all $t\in [0,2\pi]$. The pair $(a^+,a^-)$ is called the leading term for $\psi_0$, which plays an important role in this paper.

Now, since $dim\big(\ker(D|_{L^2(\mathcal{S}_{g_0,\Sigma_0})})\big)=\infty$ in general, one cannot show directly that $\mathbb{K}_0$ is finite dimensional. To deal with this problem, we prove in Section 6.1 that there exists a dense subspace $\ker(D|_{L^2(\mathcal{S}_{g_0,\Sigma_0})})^0\subset \ker(D|_{L^2(\mathcal{S}_{g_0,\Sigma_0})})$ such that any $\mathfrak{u}\in \ker(D|_{L^2(\mathcal{S}_{g_0,\Sigma_0})})^0$ can be written as
\begin{align}
\mathfrak{u}=\left(\begin{array}{c}
\frac{u^+}{2\sqrt{z}}\\
\frac{u^-}{2\sqrt{\bar{z}}}
\end{array}
\right)+\mbox{higher order term}\label{AA_13}
\end{align}
with $\mathfrak{u}\rightarrow u^+$ being a Fredholm operator from $\ker(D|_{L^2(\mathcal{S}_{g_0,\Sigma_0})})^0$ to $L^2(S^1;\mathbb{C})$. We call $(u^+,u^-)\in L^2(S^1;\mathbb{C}^2)$ the leading term of $\mathfrak{u}$. So the vector space of leading terms determined by elements in $\ker(D|_{L^2(\mathcal{S}_{g_0,\Sigma_0})})^0$ will be isomorphic to a copy of $L^2(S^1;\mathbb{C})$ sitting in $L^2(S^1;\mathbb{C}^2)$, up to quotients of finite dimensional subspaces determined by the Fredholm operator. Meanwhile, by (\ref{AA_12}) and the fact $|a^+|^2+|a^-|^2>0$ everywhere, we expect that the vector space of leading coefficients of (\ref{AA_12}) is also isomorphic to another copy of $L^2(S^1;\mathbb{C})$ in $L^2(S^1;\mathbb{C}^2)$, up to finite dimensional quotients. We have to prove that these two images of isomorphisms intersect only on a finite dimensional subspace in $L^2(S^1;\mathbb{C}^2)$, which is $\mathbb{K}_0$. That will be the main result in Section 6. The computation of $\mathfrak{L}_p$ will be shown in Section 6.2. It is the most crucial part of this paper.\\

Finally, in Section 7 and Section 8, we will derive a particular kind of implicit function theorem to prove Theorem 1.5. Unfortunately, this part is very tedious because there is no standard notation for Kuranishi problems perturbing both the domain $M\setminus\Sigma$ and the section $\psi$ simultaneously.

\section{Basic setting and results}
\subsection{Functional spaces}
Let $(M,g)$ be a compact, 3-dimensional Riemannian manifold and $\Sigma\in\mathcal{A}$ be a $C^1$-embedded circle in $M$. Moreover, we suppose that $g$ is Euclidean near $\Sigma$ in first four sections of this paper. In Section 5, we will show that all these theorems and propositions established in the first four sections hold even the metric $g$ is not Euclidean near $\Sigma$.\\

 Under this setting, there exists $N_R$, a small tubular neighborhood of $\Sigma$ which is parametrized by coordinates $(t, r,\theta)\in [0,2\pi]\times [0,R] \times [0,2\pi]$, such that $g|_{N_R}=dt^2+dr^2+r^2d\theta^2$. We parametrize $\Sigma$ by $t\in [0,2\pi]$. Also, we use the following notation for cut-off functions: For any $a,b$ with $a<b\leq R$, we define a nonnegative smooth function
\begin{align}
\chi_{a,b}=
\left\{ \begin{array}{cc}
0 &\mbox{ on }N_{a}\\
1 &\mbox{ on } M\setminus N_{b}
\end{array} \right.\label{ctf}
\end{align}
with $|\nabla^k(\chi_{a,b})|\leq \frac{C}{(b-a)^k}$ for $k\leq 4$ and $C$ a universal constant. This notation will appear frequently in this paper.\\

Recall that (see Chapter 2 in \cite{D}) the spin structures on $M$ are one-to-one corresponding to the homology group $H^1(M,\mathbb{Z}_2)$, which is discrete. We fix a spin structure on $M$ in this paper. Let $\mathbf{\mathcal{S}}$ be the corresponding spinor bundle over $M$ with respect to $g$ (see Chapter 3 in \cite{D}). There is a corresponding Dirac operator (see Chapter 3 in \cite{D} and \cite{Q}) which can be written as
\begin{align}
D^g=e_1\cdot\nabla^\mathbf{\mathcal{S}}_{e_1}+e_2\cdot\nabla^\mathbf{\mathcal{S}}_{e_2}+e_3\cdot\nabla^\mathbf{\mathcal{S}}_{e_3}
\end{align}
locally, where $\{e_1,e_2,e_3\}$ is an orthonormal frame in $TM$, $\cdot$ is the Clifford multiplication equipped on $\mathbf{\mathcal{S}}$ and $\nabla^\mathbf{\mathcal{S}}$ is the Levi-Civita connection on $\mathbf{\mathcal{S}}$, see Section 3 in \cite{D}.\\

 Let $\mathcal{I}$ be a real line bundle defined on $M\setminus\Sigma$ . We suppose that $\mathcal{I}$ cannot be extended to the entire manifold $M$, which means $\mathcal{I}$ is a M\"obius band when restricted to a small circle linking $\Sigma$. It can be written formally as
\begin{align}
\mathcal{I}|_{ t=a, r=b}\simeq [0,2\pi]\times \mathbb{R}/\{ (0,x)\sim (2\pi,-x)\mbox{ for all } x\in \mathbb{R} \}\label{r2_1}
\end{align}
for all $a\in [0,2\pi]$ and $0<b<R$. We fix an inner product on $\mathcal{I}$ and define $|v\otimes w|=|v||w|$ for any $(v,w)\in \mathcal{S}\otimes \mathcal{I}$.\\

$\mathcal{S}$ itself is equipped with the connection $\nabla^{\mathcal{S}}$. Since we fix the inner product defined on $\mathcal{I}$, there exists a unique connection $\nabla^{\mathcal{I}}$ defined on $\mathcal{I}$ which is compatible with this inner product; i.e., $X\langle s_1,s_2\rangle=\langle \nabla_X^{\mathcal{I}}s_1,s_2\rangle+\langle s_1,\nabla_X^{\mathcal{I}}s_2\rangle$ for any vector field $X$ on $M$ and any smooth sections $s_1,s_2$ on $\mathcal{I}$. We define the connection 
\begin{align}
\nabla^{\mathcal{S}\otimes \mathcal{I}}=\nabla^{\mathcal{S}}\otimes id_{I}+id_{\mathcal{S}}\otimes \nabla^{\mathcal{I}}\label{BB_1}
\end{align}
on the bundle $\mathcal{S}\otimes \mathcal{\mathcal{I}}$. This connection induces a Dirac operator defined on $\mathcal{S}\otimes \mathcal{I}$, which can be written as
\begin{align}
D^{(g,\Sigma,e)}:=e_1\cdot\nabla^{\mathcal{S}\otimes \mathcal{\mathcal{I}}}_{e_1}+e_2\cdot\nabla^{\mathcal{S}\otimes \mathcal{\mathcal{I}}}_{e_2}+e_3\cdot\nabla^{\mathcal{S}\otimes \mathcal{\mathcal{I}}}_{e_3}\label{BB_2}
\end{align}
locally (Recall that we fix $e\in H_{\Sigma}$ in the rest of this paper).\\

With the norm and the connection defined, one can define the following functional spaces.

\begin{definition}
Let $\mathfrak{u}\in C^{\infty}(M\setminus\Sigma, \mathbf{\mathcal{S}} \otimes \mathcal{I} )$ be a smooth section of $\mathbf{\mathcal{S}} \otimes \mathcal{I}$. We define the following norms and corresponding spaces:\\[1mm]
$\bullet$ $\|\mathfrak{u}\|_{L^2}=(\int_{M\setminus\Sigma}|\mathfrak{u}|^2)^{\frac{1}{2}}$;\\[1mm]
$\bullet$ $\|\mathfrak{u}\|_{L^2_1}=(\int_{M\setminus\Sigma}|\mathfrak{u}|^2+|\nabla \mathfrak{u}|^2)^{\frac{1}{2}}$;\\[1mm]
$\bullet$ $\|\mathfrak{u}\|_{L^2_{-1}}=\sup\{\int_{M\setminus\Sigma}\langle  \mathfrak{v} ,\mathfrak{u}\rangle\mbox{ }|\mbox{ } \mathfrak{v}\in C^{\infty}(M\setminus\Sigma, \mathbf{\mathcal{S}} \otimes \mathcal{I} )\mbox{ and }\|\mathfrak{v}\|_{L^2_1}\leq 1 \}$.\\[1mm]
Moreover, the spaces of sections bounded with respect to these norms will be denoted by
\begin{align*}
L^2_i(M\setminus\Sigma; \mathbf{\mathcal{S}} \otimes \mathcal{I} )=\mbox{closure of }\{\mathfrak{u}\in C^{\infty}(M\setminus\Sigma, \mathbf{\mathcal{S}} \otimes \mathcal{I} )\mbox{ }|\mbox{ } \|\mathfrak{u}\|_{L^2_i}< \infty\}
\end{align*}
for $i=1,0,-1$. In the rest of this paper, we simply use the notation $L^2_i$ to denote $L^2_i(M\setminus\Sigma; \mathbf{\mathcal{S}} \otimes \mathcal{I} )$ and usually omit the subscript $i$ when it is zero.
\end{definition}
Similarly, we can define the space of compactly supported sections, $L^2_{i,cpt}$, by taking the closure of the set of smooth, compactly supported sections with respect to the norm $\|\cdot\|_{L^2_i}$.

\begin{remark}
We should always remember that the space $L^2_{-1}$ is the dual space of $L^2_{1}$ in our case. For a general open domain $\Omega$ on $\mathbb{R}^n$, the notation $L^2_{-1}(\Omega)$ usually denotes the dual space of $L^2_{1,cpt}(\Omega)$, the closure of smooth functions(sections) compactly supported in $\Omega$ in $L^2_1(\Omega)$. $L^2_1(\Omega)\neq L^2_{1,cpt}(\Omega)$ in general. However, in our case, we will see that $L^2_1(M\setminus\Sigma;\mathbf{\mathcal{S}}\otimes\mathcal{I}) = L^2_{1,cpt}(M\setminus\Sigma;\mathbf{\mathcal{S}}\otimes\mathcal{I})$ by Lemma 2.6 below (see Section 9.4 in Appendix for the proof). Therefore, our definition is consistent with the usual one.
\end{remark}

The space $L^2_{-1}$ has the following property. This is an analog version of Theorem 1 in Section 5.9 of \cite{G}.
\begin{pro}
Let $\mathfrak{f}\in L^2_{-1}$. Then there exists a pair 
\begin{align*}
(\mathfrak{f}_0, \mathfrak{f}_1)\in L^2(M\setminus\Sigma; \mathbf{\mathcal{S}}\otimes \mathcal{I}) \times L^2(M\setminus\Sigma; \mathbf{\mathcal{S}}\otimes \mathcal{I}\otimes T^*M)
\end{align*}
such that
\begin{align}
\int_{M\setminus\Sigma}\langle \mathfrak{v} ,\mathfrak{f} \rangle=\int_{M\setminus\Sigma}\langle\mathfrak{v}, \mathfrak{f}_0\rangle+\langle \nabla\mathfrak{v},\mathfrak{f}_1 \rangle\label{note1_1}
\end{align}
for all $\mathfrak{v}\in L^2_{1}$. Furthermore, we have
\begin{align*}
\|\mathfrak{f}\|_{L^2_{-1}}=\bigg(\int_{M\setminus\Sigma}|\mathfrak{f}_0|^2+|\mathfrak{f}_1|^2\bigg)^{\frac{1}{2}}.
\end{align*}
\end{pro} 

\begin{proof}
Let $T_{\mathfrak{f}}:L^2_{1}\rightarrow \mathbb{C}$ be a bounded functional sending each $\mathfrak{v}$ to $\int_{M\setminus\Sigma}\langle\mathfrak{v} , \mathfrak{f}\rangle$. By Riesz Representation Theorem, there exists $\mathfrak{u}\in L^2_{1}$ such that
\begin{align}
T_{\mathfrak{f}}(\mathfrak{v})=\int_{M\setminus\Sigma}\langle \mathfrak{v},\mathfrak{u}\rangle +\langle \nabla\mathfrak{v},\nabla \mathfrak{u}\rangle.\label{note1_2}
\end{align}
So we can simply take $\mathfrak{f}_0=\mathfrak{u}$ and $\mathfrak{f}_1=\nabla\mathfrak{u}$.\\

To prove the second part, by taking $\mathfrak{v}=\mathfrak{u}$ in ($\ref{note1_2}$), we have
\begin{align*}
\|\mathfrak{u}\|_{L^2_1}^2=T_{\mathfrak{f}}(\mathfrak{u})\leq \|\mathfrak{u}\|_{L^2_1}\|\mathfrak{f}\|_{L^2_{-1}}.
\end{align*}
This inequality implies that $(\int_{M\setminus\Sigma}|\mathfrak{f}_0|^2+|\mathfrak{f}_1|^2)^{\frac{1}{2}} =\|\mathfrak{u}\|_{L^2_1}\leq \|\mathfrak{f}\|_{L^2_{-1}}$.\\

Meanwhile, from ($\ref{note1_2}$) we have
\begin{align*}
|T_{\mathfrak{f}}(\mathfrak{v})|\leq \bigg(\int_{M\setminus\Sigma}|\mathfrak{f}_0|^2+|\mathfrak{f}_1|^2\bigg)^{\frac{1}{2}} 
\end{align*}
if $\|\mathfrak{v}\|_{L^2_1}\leq 1$. So by Definition 2.1, we have $\|\mathfrak{f}\|_{L^2_{-1}}\leq (\int_{M\setminus\Sigma}|\mathfrak{f}_0|^2+|\mathfrak{f}_1|^2)^{\frac{1}{2}} $.
\end{proof}

\subsection{Some analytical properties of Dirac operators on $M\setminus\Sigma$} We prove the following basic properties in this section. These are very similar to some well-known results (see Chapter 4 in \cite{D}). We simply denote $D^{(g,\Sigma,e)}$ by $D$ in the rest of this section because our $g,\Sigma$ and $e$ are fixed here.
\begin{pro}
Let $D|_{L_1^2}:L^2_1(M\setminus\Sigma;\mathbf{\mathcal{S}} \otimes \mathcal{I})\rightarrow L^2(M\setminus\Sigma;\mathbf{\mathcal{S}} \otimes \mathcal{I})$ be the Dirac operator. Then we have the following properties:\\[1mm]
{\bf a.} $\ker(D|_{L^2_1})$ is finite dimensional.\\[1mm]
{\bf b.} ${\rm range}(D|_{L^2_1})$ is closed.\\[1mm]
{\bf c.} $D|_{L^2_1}$ is essentially self-adjoint, i.e., the closure $D|_{L^2}:=\overline{D|_{L^2_1}}=D|_{L^2_1}^*$.
\begin{align*}
L^2={\rm range}(D|_{L^2_1})\oplus \ker(D|_{L^2}).
\end{align*}
\end{pro}
\begin{remark}
$\ker(D|_{L^2})$ is not finite dimensional in general.
\end{remark}
To prove this proposition, we need the following lemma, which is also very useful in the rest of this article. Here $N_r$ is the tubular neighborhood of radius $r\leq R$ as we defined in (\ref{AA_6}) and $R$ is a fixed number with $N_R$ being a tubular neighborhood of $\Sigma$, too.
\begin{lemma}
For any $\mathfrak{u}\in L^2_1(M\setminus\Sigma;\mathbf{\mathcal{S}}\otimes \mathcal{I})$, we have
\begin{align*}
\int_{N_r}|\mathfrak{u}|^2\leq 64\pi^2 r^2\int_{N_r}|\nabla\mathfrak{u}|^2
\end{align*}
for all $r\leq R$.
\end{lemma}
\begin{proof}
Let $\mathfrak{u}\in L^2_1$ and $\{\mathfrak{u}_n\}$ be a sequence of smooth sections such that
\begin{align*}
\mathfrak{u}_n\rightarrow \mathfrak{u}
\end{align*}
in the $L^2_1$ norm. We can write $\mathfrak{u}_n=(u_{n,1}+iu_{n,2}, u_{n,3}+iu_{n,4})$ locally for some real-valued functions $u_{n,i}$. Because $\mathcal{I}$ is nontrivial along the $\theta$ direction, for any $r,t$, there exists $\theta_i\in [0,2\pi]$ such that $u_{n,i}(r,\theta_i,t)=0$ for $i=1,2,3,4$. By the fundamental theorem of calculus, we have
\begin{align*}
|\mathfrak{u}_n(r,s,t)|&\leq 4\int_0^{2\pi}|\partial_{\theta}\mathfrak{u}_n(r,\theta,t)|d\theta\\
&\leq 4\int_0^{2\pi}|\nabla_{e_2}\mathfrak{u}_n(r,\theta,t)|rd\theta\\
&\leq 4\sqrt{2\pi} r^{\frac{1}{2}}\bigg(\int_0^{2\pi}|\nabla_{e_2}\mathfrak{u}_n(r,\theta,t)|^2rd\theta\bigg)^{\frac{1}{2}}
\end{align*}
for any $s,t\in [0,2\pi]$, $0<r\leq R$, where $e_2=\frac{1}{r}\partial_\theta$. So we have
\begin{align*}
\int_{N_r}|\mathfrak{u}_n|^2&= 16\int_0^r\int_0^{2\pi}\int_0^{2\pi}|\mathfrak{u}_n(r,s,t)|^2rdsdtdr\\
&\leq 64\pi^2 r^2\int_{N_r}|\nabla_{e_2}\mathfrak{u}_n|^2.
\end{align*}
By taking $n\rightarrow \infty$, we prove this lemma.
\end{proof}

\begin{proof}[Proof of Proposition 2.4] \ \\
First, the proof of essential self-adjoint for $D|_{L^2_1}$ in the third property of Proposition 2.4 follows from the standard argument which can be found in Section 4.1 in \cite{D}. So we omit the proof for this part.\\

Second, for any $\mathfrak{u}\in L^2_1$, one can write the Lichnerowicz-Weitzenb\"ock formula
\begin{align*}
D^2\mathfrak{u} =\nabla^*\nabla \mathfrak{u} +\frac{\mathscr{R}}{4}\mathfrak{u}
\end{align*}
in the following sense:
\begin{align}
\int\langle D\zeta, D\mathfrak{u}\rangle =\int\langle\nabla \zeta, \nabla \mathfrak{u}\rangle +\int\frac{\mathscr{R}}{4}\langle\zeta, \mathfrak{u}\rangle\label{1_1}
\end{align}
for all $\zeta \in L^2_{1,cpt}$. Here $\mathscr{R}$ is the scalar curvature of $M$. We will now prove that (\ref{1_1}) is true for all $\zeta\in L^2_1$.\\

By Lemma 2.6, we have
\begin{align}
\int_{N_r}|\zeta|^2\leq 64\pi^2 r^2\int_{N_r}|\nabla \zeta|^2\label{1_2_1}
\end{align}
for all $\zeta\in L^2_1$. Let us denote $(\int_{N_r}|\nabla \zeta|^2)^{\frac{1}{2}}=f_{\zeta}(r)$. We have $f_{\zeta}(r)\rightarrow 0$ as $r\rightarrow 0$.\\

We now take the family of cut-off functions $\chi_{\delta}:=\chi_{\frac{2}{3}\delta,\delta}$ with $|\nabla(\chi_{\delta})|\leq \frac{C}{\delta}$ for $\delta>0$ (Recall the definition (\ref{ctf})). So by (\ref{1_1}), we have
\begin{align}
\int\langle D(\chi_\delta\zeta), D\mathfrak{u}\rangle =\int\langle\nabla (\chi_\delta\zeta), \nabla \mathfrak{u}\rangle +\int\frac{\mathscr{R}}{4}\langle\chi_\delta\zeta, \mathfrak{u}\rangle\label{1_3}
\end{align}
for all $\zeta\in L^2_1$. Clearly the second terms on the right-hand side of (\ref{1_3}) converges to
$\int\frac{\mathscr{R}}{4}\langle\zeta, \mathfrak{u}\rangle\label{1_2}$ as $\delta\rightarrow 0$ by Cauchy's inequality.\\

For the left-hand side of (\ref{1_3}), we have
\begin{align*}
\int\langle D(\chi_\delta\zeta), D\mathfrak{u}\rangle=\int\chi_\delta\langle D\zeta, D\mathfrak{u}\rangle+\mathfrak{e}_{\delta}
\end{align*}
where
\begin{align*}
\mathfrak{e}_{\delta}=\int\langle(\nabla \chi_\delta)\zeta, \nabla \mathfrak{u}\rangle +\int\frac{\mathscr{R}}{4}\langle\chi_\delta\zeta, \mathfrak{u}\rangle.
\end{align*}
Because of the inequality (\ref{1_2_1}), $\mathfrak{e}$ can be bounded as follows.
\begin{align*}
|\mathfrak{e}_\delta|&\leq \frac{C}{\delta}\int_{N_\delta}| \zeta||D \mathfrak{u}|
\leq \frac{C}{\delta}\bigg(\int_{N_\delta}|\mathfrak{\zeta}|^2\bigg)^{\frac{1}{2}}\|D\mathfrak{u}\|_{L^2}\leq C f_{\zeta}(\delta)\|D\mathfrak{u}\|_{L^2}.
\end{align*}
So we have
\begin{align*}
\int\langle D(\chi_\delta\zeta), D\mathfrak{u}\rangle\rightarrow \int\langle D \zeta, D\mathfrak{u}\rangle
\end{align*}
as $\delta\rightarrow 0$.\\

Similarly, we have
\begin{align*}
\int\langle\nabla (\chi_\delta\zeta), \nabla \mathfrak{u}\rangle\rightarrow \int\langle\nabla \zeta, \nabla \mathfrak{u}\rangle
\end{align*}
as $\delta\rightarrow 0$. So
\begin{align}
\int\langle D\zeta, D\mathfrak{u}\rangle =\int\langle\nabla \zeta, \nabla \mathfrak{u}\rangle +\int\frac{\mathscr{R}}{4}\langle\zeta, \mathfrak{u}\rangle\label{1_4}
\end{align}
for all $\zeta \in L^2_{1}$.\\

Once we have (\ref{1_4}) for all $\zeta \in L^2_{1}$, the proof of Proposition 2.4 will be obtained from the standard argument. Readers can see p. 107 in \cite{D} for details.
\end{proof}

We have proved that $D|_{L^2_1}$ has closed range and finite dimensional kernel. However, the cokernel of $D|_{L^2_1}$, which is also the kernel of $D|_{L^2}:L^2\rightarrow L^2_{-1}$, is infinite dimensional in general. In Section 3, we will describe the elements of $\ker(D|_{L^2})$ explicitly in terms of Bessel functions on a tubular neighborhood of $\Sigma$.

\section{Harmonic sections defined on the tubular neighborhood with the Euclidean metric}
\subsection{$L^2$ and $L^2_1$ harmonic sections expressed by modified Bessel functions} In this section, we will write down the Fourier expansion for $L^2$- and $L^2_1$-harmonic spinors. This expansion gives us the growth rate of these spinors and explains (\ref{AA_8}) and (\ref{AA_13}) in Section 1.2. Then, we will define the notation for leading terms and leading coefficients for these spinors. As we will see in Section 6.3, the leading coefficient of $\psi_0$ is the key to define the Fredholm operator corresponding to the linearization of moduli space $\mathfrak{M}$.\\

Let us consider the space $N=\mathbf{S}^1 \times \mathbb{R}^2$, which can be regarded as a local model for the tubular neighborhood of $\Sigma$. The Dirac operator on $N$ can be written as
\begin{align}
D=e_1\cdot\frac{\partial}{\partial t}+e_2\cdot\frac{\partial}{\partial z}+e_3\cdot\frac{\partial}{\partial {\bar{z}}}\label{CC_1}
\end{align}
where
\begin{align*}
e_1=\left( \begin{array}{cc}
-i & 0\\
0 & i 
\end{array} \right), e_2=\left( \begin{array}{cc}
0 & 0\\
-1 & 0 
\end{array} \right), e_3=\left( \begin{array}{cc}
0 & 1\\
0 & 0 
\end{array} \right)
\end{align*}
and $z=x+iy$.\\

Using the cylindrical coordinates, $r:=\sqrt{x^2+y^2}$ and $\theta=\arctan (\frac{y}{x})$, we can write down the Fourier expansion of $\mathfrak{u}$ as follows:
\begin{align*}
\mathfrak{u}(t,r,\theta)=\sum_{l,k}e^{ilt}\left( \begin{array}{c}
e^{i(k-\frac{1}{2})\theta}U^+_{k,l}(r)\\
e^{i(k+\frac{1}{2})\theta}U^-_{k,l}(r)
\end{array} \right)
\end{align*}
for any $\mathfrak{u}\in C^{\infty}(M\setminus\Sigma;\mathcal{S}\otimes \mathcal{I})$. (The structure (\ref{r2_1}) gives us the $e^{i(k-\frac{1}{2})\theta}$ and $e^{i(k+\frac{1}{2})\theta}$ exponents along $\theta$-direction.) Here $k$ runs over all integers and $l$ can be either in $\mathbb{Z}$ or $\mathbb{Z}+\frac{1}{2}$, according to the spin structure we chose (see Chapter 2 in \cite{F}). The Dirac operator can be written in terms of $\theta$, $r$ by changing of coordinates:
\begin{align*}
\frac{\partial}{\partial z}=\frac{1}{2}\Big(\frac{\partial}{\partial x}-i\frac{\partial}{\partial y}\Big)=\frac{1}{2} e^{-i\theta}\Big(\frac{\partial}{\partial r}+\frac{i\partial}{r\partial \theta}\Big);\\
\frac{\partial}{\partial \bar{z}}=\frac{1}{2}\Big(\frac{\partial}{\partial x}+i\frac{\partial}{\partial y}\Big)=\frac{1}{2}
e^{i\theta}\Big(\frac{\partial}{\partial r}-\frac{i\partial}{r\partial \theta}\Big).
\end{align*}

Suppose $\mathfrak{u}$ is a harmonic section. Then we have
\begin{align}
D\mathfrak{u}=\sum_{l,k}e^{ilt}\left( \begin{array}{c}
e^{i(k-\frac{1}{2})\theta}(2lU^+_{k,l} +\frac{d}{dr}U^-_{k,l} +\frac{(k+\frac{1}{2})}{r}U^-_{k,l})\\
e^{i(k+\frac{1}{2})\theta}(-2lU^-_{k,l} -\frac{d}{dr}U^{+}_{k,l}+\frac{(k-\frac{1}{2})}{r}U^{+}_{k,l})
\end{array} \right)=0\label{CC_2}
\end{align}
which gives us the following system of equations:
\begin{align}
\frac{d}{dr}\left( \begin{array}{c}
U^{+}_{k,l}\\
U^{-}_{k,l}
\end{array} \right)
=
\left( \begin{array}{cc}
\frac{(k-\frac{1}{2})}{r}& -2l\\
-2l & -\frac{(k+\frac{1}{2})}{r}
\end{array} \right)
\left( \begin{array}{c}
U^{+}_{k,l}\\
U^-_{k,l}
\end{array} \right).\label{CC_3}
\end{align}
By the standard ODE theory, we can reduce this system of first order ODEs to the following modified Bessel equations:
\begin{align}
\frac{d^2}{dr^2}
U^+_{k,l}+\frac{1}{r}\frac{d}{dr}U^+_{k,l} -\Bigg(4l^2+\frac{(k-\frac{1}{2})^2}{r^2}\Bigg)U^+_{k,l}=0,\\
\frac{d^2}{dr^2}U^-_{k,l}+\frac{1}{r}\frac{d}{dr}U^-_{k,l} -\Bigg(4l^2+\frac{(k+\frac{1}{2})^2}{r^2}\Bigg)U_{k,l}^-=0.\label{CC_4}
\end{align}
There are standard solutions for this type of equations called {\bf modified Bessel functions} (see \cite{N}).
\begin{align}
I_p(r)=\sum_{m=0}^{\infty}\frac{1}{m!\Gamma(m+p+1)}\big(\frac{r}{2}\big)^{2m+p}\label{CC_5}
\end{align}
According to the study of modified Bessel functions, $U^+_{k,l}(r)$ will be a linear combination of $I_{k-\frac{1}{2}}(2lr)$ and $I_{-k+\frac{1}{2}}(2lr)$ and $U^-_{k,l}(r)$ will be a linear combination of $I_{k+\frac{1}{2}}(2lr)$ and $I_{-k-\frac{1}{2}}(2lr)$, when $l\neq 0$.\\

Therefore, by equation (\ref{CC_3}), we have
\begin{align*}
\left( \begin{array}{c}
U^{+}_{k,l}\\
U^{-}_{k,l}
\end{array} \right)
=
\left( \begin{array}{c}
u^{+}_{k,l}(2l)^{-k+\frac{1}{2}}I_{k-\frac{1}{2}}(2lr)-u^{-}_{k,l}(2l)^{k+\frac{1}{2}}I_{-k+\frac{1}{2}}(2lr)\\
-u_{k,l}^{+}(2l)^{-k+\frac{1}{2}}I_{k+\frac{1}{2}}(2lr)+u_{k,l}^{-}(2l)^{k+\frac{1}{2}}I_{-k-\frac{1}{2}}(2lr)
\end{array} \right)
\end{align*}
for some $u^+_{k,l},u^-_{k,l}\in\mathbb{C}$ when $l\neq 0$. From (\ref{CC_5}), we have $I_p(r)=O(r^p)$. To normalize the leading coefficient of $I_p(2lr)$, we define
\begin{align*}
\mathfrak{I}_{p,l}(r)=(2l)^{-p}I_p(2lr).
\end{align*}

 When $l=0$, we have
\begin{align*}
\left( \begin{array}{c}
U^{+}_{k,0}\\
U^{-}_{k,0}
\end{array} \right)
=
\left( \begin{array}{c}
u^{+}_{k,0}r^{k-\frac{1}{2}}\\
u_{k,0}^{-}r^{-k-\frac{1}{2}}
\end{array} \right).
\end{align*}

Now, we apply these results to sections of $\mathcal{S}\otimes \mathcal{I}$ over $N$. Fix $R>0$, we simply have $N_R=N\cap \{r<R\}$. Suppose  $\mathfrak{u}\in L^2(N;\mathbf{\mathcal{S}}\otimes\mathcal{I} )$ and $D\mathfrak{u}|_{N_R}=0$, then
\begin{align*}
\mathfrak{u}=&\sum_{k\geq 0;l\neq 0}u^+_{k,l}e^{ilt}\left( \begin{array}{c}
e^{i(k-\frac{1}{2})\theta}\mathfrak{I}_{k-\frac{1}{2},l}(r)\\
-e^{i(k+\frac{1}{2})\theta}l\mathfrak{I}_{k+\frac{1}{2},l}(r)
\end{array} \right)+\sum_{k\geq 0}\left( \begin{array}{c}
u^{+}_{k,0}e^{i(k-\frac{1}{2})\theta}r^{k-\frac{1}{2}}\\
0
\end{array} \right)\\
&+
\sum_{k\leq 0;l\neq 0}u^-_{k,l}e^{ilt}\left( \begin{array}{c}
-e^{i(k-\frac{1}{2})\theta}l\mathfrak{I}_{-k+\frac{1}{2},l}(r)\\
e^{i(k+\frac{1}{2})\theta}\mathfrak{I}_{-k-\frac{1}{2},l}(r)
\end{array} \right)+
\sum_{k\leq 0}\left( \begin{array}{c}
0\\
u^{-}_{k,0}e^{i(k+\frac{1}{2})\theta}r^{-k-\frac{1}{2}}
\end{array} \right)
\end{align*}
which has the leading term with order $r^{-\frac{1}{2}}$, i.e.
\begin{align*}
\mathfrak{u}=\sum_{l\neq 0}e^{ilt}\Bigg[u^{+}_{0,l}\left( \begin{array}{c}
e^{-i\frac{1}{2}\theta}\mathfrak{I}_{-\frac{1}{2},l}(r)\\
-le^{i\frac{1}{2}\theta}\mathfrak{I}_{\frac{1}{2},l}(r)
\end{array} \right)&+
u^{-}_{0,l}\left( \begin{array}{c}
-le^{-i\frac{1}{2}\theta}\mathfrak{I}_{\frac{1}{2},l}(r)\\
e^{i\frac{1}{2}\theta}\mathfrak{I}_{-\frac{1}{2},l}(r)
\end{array} \right)\Bigg]\\
&+\left( \begin{array}{c}
u^{+}_{0,0}e^{-i\frac{1}{2}\theta}r^{-\frac{1}{2}}\\
u^{-}_{0,0}e^{i\frac{1}{2}\theta}r^{-\frac{1}{2}}
\end{array} \right)
+\mbox{ higher order term}
\end{align*}
where the higher order term is $O(r^{p})$ for some $p > -\frac{1}{2}$.\\ 

The Bessel functions $I_{\frac{1}{2}}(x)$ and $I_{-\frac{1}{2}}(x)$ can be written explicitly as $\sqrt{\frac{2}{\pi x}} \sinh (x)$ and $\sqrt{\frac{2}{\pi x}}\cosh (x)$. So the leading term can be expressed in terms of  $\{\frac{e^{2lr}}{\sqrt{r}},{\frac{e^{-2lr}}{\sqrt{r}}}\}$. Let us use this expression, then
\begin{align}
\mathfrak{u}=
\sum_{l}e^{ilt}\Bigg[\hat{u}^{+}_{0,l}\left( \begin{array}{c}
\frac{e^{2|l|r}}{\sqrt{z}}\\
-\mbox{sign}(l)\frac{e^{2|l|r}}{\sqrt{\bar{z}}}
\end{array} \right)+
\hat{u}^{-}_{0,l}\left( \begin{array}{c}
\frac{e^{-2|l|r}}{\sqrt{z}}\\
\mbox{sign}(l)\frac{e^{-2|l|r}}{\sqrt{\bar{z}}}
\end{array} \right)\Bigg]\label{CC_6}\\[3mm]
+\mbox{ higher order term}\nonumber\\[-4mm]
\nonumber
\end{align}
where $\hat{u}^+_{0,l}=(u^+_{0,l}-\mbox{sign}(l)u^-_{0,l})$, $\hat{u}^-_{0,l}=(u^+_{0,l}+\mbox{sign}(l)u^-_{0,l})$ and the higher order term is $O(r^{p})$ for some $p>-\frac{1}{2}$. Notice that by (\ref{CC_6}), we the following inequality:
\begin{align}
\sum_{l\neq 0}|l|^{-2}\big(|\hat{u}_{0,l}^+|^2+|\hat{u}_{0,l}^-|^2\big)\leq \|\mathfrak{u}\|^2_{L^2(N;\mathcal{S}\otimes\mathcal{I})}.\label{CC_7}
\end{align}

We have seen from the equation (\ref{AA_12}) and (\ref{AA_13}), the leading terms of $\mathbb{Z}/2-$\\harmonic spinors in $L^2_1(M\setminus \Sigma;\mathbf{\mathcal{S}}\otimes\mathcal{I})$ and  $L^2(M\setminus \Sigma;\mathbf{\mathcal{S}}\otimes\mathcal{I})$ as $r\rightarrow 0$ involve in the argument of the lineariation of $\mathfrak{M}$. These leading terms play an important role in the proof of Fredholmness. So here we define the notation of the leading terms and lading coefficients of $\mathbb{Z}/2-$harmonic $L^2(M\setminus \Sigma;\mathbf{\mathcal{S}}\otimes\mathcal{I})$ and $L^2_1(M\setminus \Sigma;\mathbf{\mathcal{S}}\otimes\mathcal{I})$ sections in Definition 3.3 and Definition 3.5 for later use. Moreover, we also introduce a space called $\mathcal{K}_R$ in Definition 3.2. This space help us to obtain the regularity for the leading terms defined of $L^2(M\setminus \Sigma;\mathbf{\mathcal{S}}\otimes\mathcal{I})$ in Proposition 3.6.\\

To begin with , let us recall the notation for sequence spaces $\ell^2$ and $\ell^2_k$:
\begin{definition}
Let $k\in\mathbb{Z}$. We define
\begin{align}
\ell^2&:=\Big\{\{a_l\}_{l\in\mathbb{Z}}\in \mathbb{C}^{\mathbb{Z}}\Big|\sum_{l\in\mathbb{Z}}|a_l|^2< \infty\Big\};\label{CC_8}\\
\ell^2_k&:=\Big\{\{a_l\}_{l\in\mathbb{Z}}\in \mathbb{C}^{\mathbb{Z}}\Big|\sum_{l\in\mathbb{Z}}(1+|l|)^{2k}|a_l|^2< \infty\Big\}.\label{CC_9}
\end{align}
\end{definition}

\begin{definition} Given $R>0$,
let $\mathcal{K}_R$ be a subspace of $L^2(N_{R};\mathcal{S}\otimes \mathcal{I})$ defined by
\begin{align*}
\mathcal{K}_R=\Big\{\mathfrak{u}\in L^2(N_{R};\mathcal{S}\otimes \mathcal{I})\Big| D\mathfrak{u}=0 \mbox{ and }
 \hat{u}^-_{0,l}=0 \mbox{ for all } |l|>\frac{1}{2R}\Big\}.
\end{align*}
\end{definition}
By (\ref{CC_7}), we obtain the Fourier coefficients of $L^2$-harmonic spinors is an $\ell^2_{-1}$ sequence. We define the following notations.
\begin{definition}
Let $\mathfrak{u}$ be a harmonic section in $L^2(N_R;\mathcal{S}\otimes \mathcal{I})$ with the corresponding Fourier coefficients $\{u^{\pm}_{k,l}\}$. Then\\[1mm]
$\bullet$ We define
 $\big\{(\hat{u}^+_{0,l}+\hat{u}^-_{0,l},-\mbox{sign}(l)\hat{u}^+_{0,l}+\mbox{sign}(l)\hat{u}^-_{0,l})\big\}\in \ell^2_{-1}\times \ell^2_{-1}$ to be the sequence \\ $\mbox{ }\mbox{ }\mbox{ }$of leading coefficients of $\mathfrak{u}$.\\[1mm]
 $\bullet$ Define the space
\begin{align}
\ker(D|_{L^2(N_R;\mathcal{S}\otimes\mathcal{I})})^0:=\{\mathfrak{u}\in&\ker(D|_{L^2(N_R;\mathcal{S}\otimes\mathcal{I})})\label{CC_10}\\
&\mbox{ with its leading coefficients in }\ell^2\times\ell^2\}.\nonumber
\end{align} 
$\bullet$ We define $\Big(\sum_l(\hat{u}^+_{0,l}+\hat{u}^-_{0,l}) e^{ilt},\sum_l(-\mbox{sign}(l)\hat{u}^+_{0,l}+\mbox{sign}(l)\hat{u}^-_{0,l})e^{ilt}\Big)$ to be the leading \\
 $\mbox{ }\mbox{ }\mbox{ }$term of $\mathfrak{u}$ when $\mathfrak{u}\in\ker(D|_{L^2(N_R;\mathcal{S}\otimes\mathcal{I})})^0$.
\end{definition}

\begin{definition}
Let $\mathfrak{u}\in\ker(D|_{L^2(N_R;\mathcal{S}\otimes\mathcal{I})})^0$ and $\{u^{\pm}_{k,l}\}$ be its Fourier coefficients. We define the following terminologies:\\[1mm]
$\bullet$ Define $\{(u^+_l,u^-_l)\}\in \ell^2\times \ell^2$ to be
\begin{align*}
(u^{+}_l,u^-_{l})&=(\hat{u}^+_{0,l},-\mbox{sign}(l)\hat{u}^+_{0,l}) \mbox{ for } |l|>\frac{1}{2R}\\
(u^{+}_l,u^-_{l})&=(\hat{u}^+_{0,l},-\mbox{sign}(l)\hat{u}^+_{0,l})+(\hat{u}^-_{0,l},\mbox{sign}(l)\hat{u}^-_{0,l})\mbox{ for } |l|\leq \frac{1}{2R}.
\end{align*}
$\mbox{ }\mbox{ }\mbox{ }$We call $\big\{(u^+_l,u^-_l)\big\}$ the sequence of $\mathcal{K}_R$-leading coefficients of $\mathfrak{u}$.
\\[1mm]
$\bullet$ Define $u^+(t)=\sum_lu^+_le^{ilt}$ and $u^-(t)=\sum_lu^-_le^{ilt}$ where $\{u^{\pm}_l\}$ is the sequence of \\[1mm]
$\mbox{ }\mbox{ }\mbox{ }\mathcal{K}_R$-leading coefficients of $\mathfrak{u}$. We call $u^{\pm}(t)$ to be the $\mathcal{K}_R$-leading term of $\mathfrak{u}$.\\[1mm]
$\bullet$ Let $u^{\pm}(t)$ be the $\mathcal{K}_R$-leading term of $\mathfrak{u}$. We call $(u^+(t)\frac{1}{\sqrt{z}},u^-(t)\frac{1}{\sqrt{\bar{z}}})$ the $\mathcal{K}_R$-\\
$\mbox{ }\mbox{ }\mbox{ }$dominant term of $\mathfrak{u}$.\\
Moreover, we can see that if $\mathfrak{u}\in \mathcal{K}_R$, then the $\mathcal{K}_R$-leading term (coefficients) will be the leading term (coefficients) of $\mathfrak{u}$. The sequence of leading coefficients of $\mathfrak{u}$ can also be regarded as the sequence of $\mathcal{K}_0$-leading coefficients of $\mathfrak{u}.$
\end{definition}

 When we perturb $(g,\Sigma,\psi)\in\mathfrak{M}$, the leading term of $\psi$ plays a crucial rule in the linearization of $\mathfrak{M}$. So readers should be familiar with these definitions.\\

Now if we consider $\mathfrak{v}\in L^2_1(N_R;\mathbf{\mathcal{S}}\otimes \mathcal{I})$ satisfying $D\mathfrak{v}=0$, we will have
\begin{align*}
\mathfrak{v}=&
\sum_{k\geq 1;l\neq 0}v^+_{k,l}e^{ilt}\left( \begin{array}{c}
e^{i(k-\frac{1}{2})\theta}\mathfrak{I}_{k-\frac{1}{2},l}(r)\\
-e^{i(k+\frac{1}{2})\theta}l\mathfrak{I}_{k+\frac{1}{2},l}(r)
\end{array} \right)+\sum_{k\geq 1}\left( \begin{array}{c}
v^{+}_{k,0}e^{i(k-\frac{1}{2})\theta}r^{k-\frac{1}{2}}\\
0
\end{array} \right)
\\
&+\sum_{k\leq -1;l\neq 0}v^-_{k,l}e^{ilt}\left( \begin{array}{c}
-e^{i(k-\frac{1}{2})\theta}l\mathfrak{I}_{-k+\frac{1}{2},l}(r)\\
e^{i(k+\frac{1}{2})\theta}\mathfrak{I}_{-k-\frac{1}{2},l}(r)
\end{array} \right)
+
\sum_{k\leq -1}\left( \begin{array}{c}
0\\
v^{-}_{k,0}e^{i(k+\frac{1}{2})\theta}r^{-k-\frac{1}{2}}
\end{array} \right).
\end{align*}
So we can write
\begin{align}
\mathfrak{v}=
\left( \begin{array}{c}
v^{+}_{1,0}e^{i\frac{1}{2}\theta}r^{\frac{1}{2}}\\
v^{-}_{-1,0}e^{-i\frac{1}{2}\theta}r^{\frac{1}{2}}
\end{array} \right)
&+
\sum_{l\neq 0}e^{ilt}\left( \begin{array}{c}
v^{+}_{1,l}e^{i\frac{1}{2}\theta}\mathfrak{I}_{\frac{1}{2},l}(r)\\
v^{-}_{-1,l}e^{-i\frac{1}{2}\theta}\mathfrak{I}_{\frac{1}{2},l}(r)
\end{array} \right)\label{CC_11}\\[3mm]
&+\mbox{ higher order term}\nonumber\\[-4mm]
\nonumber
\end{align}
where the higher order term is $O(r^{p})$ for some $p>\frac{1}{2}$. By the definition of modified Bessel function (\ref{CC_5}), we can check that
\begin{align}
\{v^{+}_{1,l}\},\{v^-_{-1,l}\}\in \ell^2.\label{CC_12}
\end{align}

 Again, we define leading coefficients and the leading term for $\mathfrak{v}$.
 
\begin{definition}
Let $\mathfrak{v}$ be a harmonic section in $L^2_1(N_R;\mathcal{S}\otimes \mathcal{I})$.\\[1mm]
$\bullet$ We call the Fourier coefficients, $\{(v^+_{1,l}, v^-_{-1,l})\}$, denoted by $\{v^{\pm}_l\}\in \ell^2\times\ell^2$, to be\\ $\mbox{ }\mbox{ }\mbox{ }$the sequence of leading coefficients of $\mathfrak{v}$.\\[1mm]
$\bullet$ We define $v^{\pm}(t)$, where $v^{+}(t)=\sum_l v^{+}_le^{ilt}$ and $v^{-}(t)=\sum_l v^{-}_le^{ilt}$, to be the\\ $\mbox{ }\mbox{ }\mbox{ }$leading term of $\mathfrak{v}$.\\[1mm]
$\bullet$ We call $(v^+(t)\sqrt{z},v^-(t)\sqrt{\bar{z}})$ the dominant term of $\mathfrak{v}$.
\end{definition} 
 
 In the rest of this paper, we always use letters of Fraktur script, $\mathfrak{u},\mathfrak{v},\mathfrak{h},\mathfrak{c}$, etc., to denote the sections defined on $L^2(M\setminus\Sigma;\mathcal{S}\otimes\mathcal{I})$ or $L^2_1(M\setminus\Sigma;\mathcal{S}\otimes\mathcal{I})$. If they satisfy the Dirac equation on $N_R$ for some $R>0$, their corresponding sequences of ($\mathcal{K}_R$-)leading coefficients will be denoted by letters of normal script $\{u^{\pm}_l\},\{v^{\pm}_l\},\{h^{\pm}_l\},\{c^{\pm}_l\}$, etc. which are in $\ell^2\times \ell^2$. Meanwhile, the corresponding ($\mathcal{K}_R$-)leading terms will be denoted by $u^{\pm}=\sum u^{\pm}_le^{ilt},v^{\pm},h^{\pm},c^{\pm}$ which are in $L^2(S^1;\mathbb{C})$. We have the $L^2$-norm for $u^{\pm}$ will be the same as $(\|\{u^+_l\}\|_{\ell^2}^2+\|\{u^-_l\}\|_{\ell^2}^2)^{\frac{1}{2}}$.\\

By Definition 3.4, any $L^2$-harmonic spinor $\mathfrak{u}$ with $\mathcal{K}_R$-leading coefficients can be decomposed as a sum of a dominant term and a remainder term. It is also true that for any $L^2_1$-harmonic spinor, it can be decomposed as a sum of a dominant term and a reminder term. In the following proposition, we take care of the regularity estimate for these remainder terms.

\begin{pro} We have the following two properties.\\[1mm]
{\bf a.} Let $\mathfrak{u} \in \mathcal{K}_R$, then we can decompose
\begin{align*}
\mathfrak{u}=\left( \begin{array}{c}
u^+(t)\frac{1}{\sqrt{z}}\\
u^-(t)\frac{1}{\sqrt{\bar{z}}}
\end{array} \right)+\mathfrak{u}_{\mathfrak{R}}
\end{align*}
$\mbox{ }\mbox{ }\mbox{ }\mbox{ }$for some $\mathfrak{u}_{\mathfrak{R}}\in L^2_1(N_{\frac{2R}{3}};\mathbf{\mathcal{S}}\otimes \mathcal{I})$ where $u^{\pm}(t)=\sum u^{\pm}_l e^{ilt}$ and
\begin{align}
\|\mathfrak{u}_{\mathfrak{R}}\|_{L^2_1(N_{\frac{2R}{3}})}\leq CR^{-1}\|\mathfrak{u}\|_{L^2(N_{R})}\label{CC_13}
\end{align}
$\mbox{ }\mbox{ }\mbox{ }\mbox{ }$for some constant $C$. In the following paragraphs, we call $(\mathfrak{u}-\mathfrak{u}_{\mathfrak{R}})$ the $\mathcal{K}_R$-$\mbox{ }\mbox{ }\mbox{ }\mbox{ }$dominant term of $\mathfrak{u}$ and call $\mathfrak{u}_{\mathfrak{R}}$ the remainder term of $\mathfrak{u}$.\\[1mm]
{\bf b.}  Let $\mathfrak{v} \in L^2_1(N_R;\mathbf{\mathcal{S}}\otimes \mathcal{I})$ and $D\mathfrak{v}=0$, then we can decompose
\begin{align*}
\mathfrak{v}=\left( \begin{array}{c}
v^+(t)\sqrt{z}\\
v^-(t)\sqrt{\bar{z}}
\end{array} \right)+\mathfrak{v}_{\mathfrak{R}}
\end{align*}
$\mbox{ }\mbox{ }\mbox{ }\mbox{ }$for some $\mathfrak{v}_{\mathfrak{R}}\in L^2_2(N_{\frac{2R}{3}};\mathbf{\mathcal{S}}\otimes \mathcal{I})$ where $v^{\pm}(t)=\sum v^{\pm}_l e^{ilt}$ and
\begin{align}
\|\mathfrak{v}_{\mathfrak{R}}\|_{L^2_2(N_{\frac{2R}{3}})}\leq CR^{-2}\|\mathfrak{v}\|_{L^2(N_{R})}\label{CC_14}
\end{align}
$\mbox{ }\mbox{ }\mbox{ }\mbox{ }$for some constant $C$. Similarly, in the following paragraphs, we call $(\mathfrak{v}-\mathfrak{v}_{\mathfrak{R}})$ $\mbox{ }\mbox{ }\mbox{ }\mbox{ }$the dominant term of $\mathfrak{v}$ and call $\mathfrak{v}_{\mathfrak{R}}$ the remainder term of $\mathfrak{v}$.
\end{pro}
\begin{proof} {\bf(proof of part a)}.
 We claim the following two inequalities:\\
First, we have $D\left( \begin{array}{c}
u^+(t)\frac{1}{\sqrt{z}}\\
u^-(t)\frac{1}{\sqrt{\bar{z}}}
\end{array} \right)\in L^2(N_R)$ and
\begin{align}
\bigg\|D\left( \begin{array}{c}
u^+(t)\frac{1}{\sqrt{z}}\\
u^-(t)\frac{1}{\sqrt{\bar{z}}}
\end{array} \right)\bigg\|_{L^2(N_R)}^2\leq C R^{-2}\|\mathfrak{u}\|_{L^2(N_R)}^2\label{4_claim1}
\end{align}
for some $C>0$. Second,
\begin{align}
\bigg\|\left( \begin{array}{c}
u^+(t)\frac{1}{\sqrt{z}}\\
u^-(t)\frac{1}{\sqrt{\bar{z}}}
\end{array} \right)\bigg\|_{L^2(N_R)}^2\leq C\|\mathfrak{u}\|_{L^2(N_R)}^2.\label{4_claim2}
\end{align}
We will prove these inequalities in Corollary 3.8.\\

We now fix $K>0$ and define
\begin{align*}
\mathfrak{u}_{\mathfrak{R},K}=\sum_{k\neq 0}\sum_{|l|\leq K}e^{ilt}
\left( \begin{array}{c}
e^{i(k-\frac{1}{2})\theta}U^+_{k,l}\\
e^{i(k+\frac{1}{2})\theta}U^-_{k,l}
\end{array} \right)-\sum_{l\neq 0; |l|\leq K}e^{ilt}\left( \begin{array}{c}
u^+_{l}\frac{1}{\sqrt{z}}\\
u^-_l\frac{1}{\sqrt{\bar{z}}}
\end{array}
\right).
\end{align*}
We can easily see that $|\mathfrak{u}_{\mathfrak{R},K}|\leq C_{K}\sqrt{r}$ and $|\nabla \mathfrak{u}_{\mathfrak{R},K}|\leq C_K \frac{1}{\sqrt{r}}$, which means there will be no boundary term when we do the integration by part for the Lichnerowicz-Weitzenb\"ock formula. Let $\chi=1-\chi_{\frac{2}{3}R,R}$ be a cut-off function. By applying Lichnerowicz-Weitzenb\"ock formula on $\chi\mathfrak{u}_{\mathfrak{R},K}$ and using (\ref{4_claim1}), (\ref{4_claim2}) above, we have
\begin{align}
\|\mathfrak{u}_{\mathfrak{R},K}\|_{L^2_1(N_{\frac{2R}{3}})}^2&\leq \|D\mathfrak{u}_{\mathfrak{R},K}\|_{L^2(N_R)}^2 +C\frac{1}{R^2}\|\mathfrak{u}_{\mathfrak{R},K}\|_{L^2(N_R)}^2\label{4_1lema}\\
&\leq\bigg\|D\left( \begin{array}{c}
u^+(t)\frac{1}{\sqrt{z}}\nonumber\\
u^-(t)\frac{1}{\sqrt{\bar{z}}}
\end{array} \right)\bigg\|_{L^2(N_R)}^2+C\frac{1}{R^2}\|\mathfrak{u}_{\mathfrak{R},K}\|_{L^2(N_R)}^2\nonumber\\
&\leq CR^{-2}\|\mathfrak{u}\|_{L^2(N_R)}^2\nonumber
\end{align} 
for some $C>0$.\\

By taking $K\rightarrow \infty$ in (\ref{4_1lema}), we have
\begin{align*}
\|\mathfrak{u}_{\mathfrak{R}}\|_{L^2_1(N_{\frac{2R}{3}})}^2\leq CR^{-2}\|\mathfrak{u}\|_{L^2(N_R)}^2.\\
\end{align*}

{\bf (proof of part b)}. Similar to the proof of part a, we claim the following two inequalities which will be proved in Corollary 3.10. 
\begin{align}
\bigg\|D\left(\begin{array}{c}
v^+(t)\sqrt{z}\\
v^-(t)\sqrt{\bar{z}}
\end{array}\right)\bigg\|^2_{L^2(N_R)}\leq C R^{-2}\|\mathfrak{v}\|_{L^2(N_R)}^2.\label{4_clm5}
\end{align}
\begin{align}
\bigg\|\left(\begin{array}{c}
v^+(t)\sqrt{z}\\
v^-(t)\sqrt{\bar{z}}
\end{array}\right)\bigg\|^2_{L^2(N_R)}\leq C\|\mathfrak{v}\|^2_{L^2(N_R)}\label{4_clm6}.
\end{align}

Fix $K>0$, define
\begin{align*}
\mathfrak{v}_{\mathfrak{R},K}=\sum_{k\neq 0}\sum_{|l|\leq K}e^{ilt}
\left( \begin{array}{c}
e^{i(k-\frac{1}{2})\theta}V^+_{k,l}\\
e^{i(k+\frac{1}{2})\theta}V^-_{k,l}
\end{array} \right)-\sum_{l\neq 0; l\leq K}e^{ilt}\left( \begin{array}{c}
v^+_{l}\sqrt{z}\\
v^-_l\sqrt{\bar{z}}
\end{array}
\right).
\end{align*}
We have $|\mathfrak{v}_{\mathfrak{R},K}|\leq C_{K}\sqrt{r^3}$ and $|\nabla \mathfrak{v}_{\mathfrak{R},K}|\leq C_K {\sqrt{r}}$ and $|\nabla\nabla \mathfrak{v}_{\mathfrak{R},K}|\leq C_K \frac{1}{{\sqrt{r}}}$. So by applying Lichnerowicz-Weitzenb\"ock formula on $\chi\mathfrak{v}_{\mathfrak{R},K}$, we have
\begin{align}
\|\mathfrak{v}_{\mathfrak{R},K}\|_{L^2_1(N_{\frac{2R}{3}})}^2&\leq \|D\mathfrak{v}_{\mathfrak{R},K}\|_{L^2}^2 +C\frac{1}{R^2}\|\mathfrak{v}_{\mathfrak{R},K}\|_{L^2(N_R)}^2\label{CC_15}\\
&\leq\bigg\|D\left( \begin{array}{c}
v^+(t)\sqrt{z}\nonumber\\
v^-(t)\sqrt{\bar{z}}
\end{array} \right)\bigg\|_{L^2(N_R)}^2+C\frac{1}{R^2}\|\mathfrak{v}_{\mathfrak{R},K}\|_{L^2(N_R)}^2\nonumber\\
&\leq CR^{-2}\|\mathfrak{v}\|_{L^2(N_R)}^2\nonumber
\end{align} 
for some $C>0$. By taking the limit $K\rightarrow \infty$, we have
\begin{align*}
\|\mathfrak{v}_{\mathfrak{R}}\|^2_{L^2_1(N_{\frac{2R}{3}})}\leq CR^{-2}\|\mathfrak{v}\|_{L^2(N_R)}^2.
\end{align*}

Notice that $[\nabla_i, D]=0$, so we can use the same argument on $\nabla_i \mathfrak{v}$. Here we need the following inequalities which are also proved in Corollary 3.10.
\begin{align}
\bigg\|D\big(\nabla\left(\begin{array}{c}
v^+(t)\sqrt{z}\\
v^-(t)\sqrt{\bar{z}}
\end{array}\right)\big)\bigg\|^2_{L^2(N_R)}\leq C R^{-4}\|\mathfrak{v}\|_{L^2(N_R)}^2\label{4_clm7}
\end{align}
and
\begin{align}
\bigg\|\left(\begin{array}{c}
v^+(t)\sqrt{z}\\
v^-(t)\sqrt{\bar{z}}
\end{array}\right)\bigg\|^2_{L^2_1(N_R)}\leq CR^{-2}\|\mathfrak{v}\|^2_{L^2(N_R)}\label{4_clm8}.
\end{align}
So we have
\begin{align}
\|\mathfrak{v}_{\mathfrak{R},K}\|_{L^2_2(N_{\frac{2R}{3}})}^2&\leq \|D\big(\nabla\mathfrak{v}_{\mathfrak{R},K}\big)\|_{L^2(N_R)}^2 +C\frac{1}{R^2}\|\mathfrak{v}_{\mathfrak{R},K}\|_{L^2_1(N_R)}^2\label{CC_16}\\
&\leq\bigg\|D\big(\nabla\left( \begin{array}{c}
v^+(t)\sqrt{z}\nonumber\\
v^-(t)\sqrt{\bar{z}}
\end{array} \right)\big)\bigg\|_{L^2(N_R)}^2+C\frac{1}{R^2}\|\mathfrak{v}_{\mathfrak{R},K}\|_{L^2_1(N_R)}^2\nonumber\\
&\leq CR^{-4}\|\mathfrak{v}\|_{L^2(N_R)}^2\nonumber
\end{align} 
for some $C>0$. By taking the limit $K\rightarrow \infty$, we prove this proposition.
\end{proof}

\subsection{Regularity properties and the asymptotic behavior of $L^2$-harmonic sections on the tubular neighborhood} In this subsection, we will derive some regularity theorems for harmonic spinors $\mathfrak{u}\in L^2(N_R;\mathbf{\mathcal{S}}\otimes \mathcal{I})$. These estimates are similar to the doubling estimate appearing in \cite{L}. Recall that, by standard interior regularity theorem, $\mathfrak{u}$ is a smooth section on any compact subset of $N_R$. We write
\begin{align*}
\mathfrak{u}=\sum_{l,k}e^{ilt}\left( \begin{array}{c}
e^{(k-\frac{1}{2})\theta}U^+_{k,l}\\
e^{(k+\frac{1}{2})\theta}U^-_{k,l}
\end{array} \right)
\end{align*} 
where
\begin{align*}
\left( \begin{array}{c}
U^{+}_{k,l}\\
U^{-}_{k,l}
\end{array} \right)
=
\left( \begin{array}{c}
u_{k,l}^{+}\mathfrak{I}_{k-\frac{1}{2},l}(r)-u_{k,l}^{-}l\mathfrak{I}_{-k+\frac{1}{2},l}(r)\\
-u^{+}_{k,l}l\mathfrak{I}_{k+\frac{1}{2},l}(r)+u^{-}_{k,l}\mathfrak{I}_{-k-\frac{1}{2},l}(r)
\end{array} \right)
\end{align*}
for $l\neq 0$ and
\begin{align*}
\left( \begin{array}{c}
U^{+}_{k,0}\\
U^{-}_{k,0}
\end{array} \right)
=
\left( \begin{array}{c}
u_{k,0}^{+}r^{k-\frac{1}{2}}\\
u^{-}_{k,0}r^{-k-\frac{1}{2}}
\end{array} \right).
\end{align*}\\

Since $\mathfrak{u}\in L^2$, so we have
\begin{align*}
u^+_{k,l}=0 \mbox{ for }k\leq -1;\\
u^-_{k,l}=0 \mbox{ for }k\geq 1.
\end{align*}
Moreover, let us define
\begin{align*}
E_{k,l}=\Bigg\{e^{ilt}\left( \begin{array}{c}
ae^{i\frac{1}{2}\theta}\mathfrak{I}_{k-\frac{1}{2},l}(r)-be^{i\frac{1}{2}\theta}l\mathfrak{I}_{-k+\frac{1}{2},l}(r)\\
-ae^{-i\frac{1}{2}\theta}l\mathfrak{I}_{k+\frac{1}{2},l}(r)+be^{-i\frac{1}{2}\theta}\mathfrak{I}_{-k-\frac{1}{2},l}(r)
\end{array} \right)\mbox{ }\Bigg| \mbox{ }a,b\in\mathbb{R}\Bigg\},
\end{align*} then $E_{k,l}$ and $E_{k',l'}$ are $L^2$-orthogonal for any two pairs $(k,l)\neq (k',l')$ (readers can obtain this result by using the orthogonality of modified Bessel functions, see \cite{N}).\\

By using these observations, we can prove the following proposition.
\begin{pro}
Let $\mathfrak{u}\in L^2(N_R;\mathcal{S}\otimes \mathcal{I})\cap \ker(D)$ with the corresponding Fourier coefficients $\{u^{\pm}_{k,l}\}$. Then the sequence of $\mathcal{K}_R$-leading coefficients $\{u^{\pm}_l\}$
 is in $\ell^2_k$ (see $(\ref{CC_9})$ for the definition) for all $k\in \mathbb{N}$. Moreover, we have
\begin{align}
\|(l^{k} u_l^{\pm})_{l\in\mathbb{Z}}\|^2_{\ell^2}\leq 3\frac{(2k+1)!}{R^{2k+1}}\|\mathfrak{u}\|^2_{L^2}\label{4_1b}.
\end{align}
\end{pro}
\begin{proof}
First of all, let $P_{k,l}:\ker(D|_{L^2})\rightarrow E_{k,l}$ be the orthonormal projection. We have
\begin{align*}
P_{0,l}(\mathfrak{u})=
e^{ilt}\left( \begin{array}{c}
\hat{u}_{0,l}^{+}\frac{e^{2|l|r}}{\sqrt{z}}+\hat{u}_{0,l}^{-}\frac{e^{-2|l|r}}{\sqrt{z}}\\
-\mbox{sign}(l)\hat{u}^{+}_{0,l}\frac{e^{2|l|r}}{\sqrt{\bar{z}}}+\mbox{sign}(l)\hat{u}^{-}_{0,l}\frac{e^{-2|l|r}}{\sqrt{\bar{z}}}
\end{array} \right)
\end{align*}
for any $l$.\\

Recall that $(u^{+}_l,u^-_{l})=(\hat{u}^+_{0,l},-\mbox{sign}(l)\hat{u}^+_{0,l})$ for $|l|>\frac{1}{2R}$ and $(u^{+}_l,u^-_{l})=(\hat{u}^+_{0,l},-\mbox{sign}(l)\hat{u}^+_{0,l})+(\hat{u}^-_{0,l},\mbox{sign}(l)\hat{u}^-_{0,l})$ for $|l|\leq \frac{1}{2R}$. We can compute directly to get
\begin{align*}
\|\mathfrak{u}\|_{L^2(N_R)}^2&\geq \sum_l \|P_{0,l}(\mathfrak{u})\|_{L^2(N_R)}^2\\ 
&\geq \sum_l |\hat{u}^{+}_{0,l}|^2\int_0^Re^{4|l|r}dr+\sum |\hat{u}^-_{0,l}|^2\int_0^Re^{-4|l|r}dr\\
&\geq\sum_l |\hat{u}^{+}_{0,l}|^2\int_0^Re^{4|l|r}dr\\
&\geq \sum_k \sum_l |\hat{u}^{+}_{0,l}|^2\frac{(2l)^{2k}R^{2k+1}}{(2k+1)!}.
\end{align*}
Meanwhile, the second line of this inequality also tells us that 
\begin{align*}
\|\mathfrak{u}\|_{L^2(N_R)}^2 &\geq \sum |\hat{u}^-_{0,l}|^2\int_0^Re^{-4|l|r}dr\\
&\geq \sum_{|l|\leq \frac{1}{2R}}e^{-1}|\hat{u}^-_{0,l}|^2R\\
&\geq \sum_{|l|\leq \frac{1}{2R}}e^{-1}|\hat{u}^-_{0,l}|^2|l|^{2k}R^{2k+1}.
\end{align*}
 So we prove (\ref{4_1b}).

\end{proof}
 By using this proposition, we can prove (\ref{4_claim1}) and (\ref{4_claim2}) in the following way.\\

\begin{cor}
Suppose that $\left( \begin{array}{c}
u^+(t)\frac{1}{\sqrt{z}}\\
u^-(t)\frac{1}{\sqrt{\bar{z}}}
\end{array} \right)$ is the $\mathcal{K}_R$-dominant term of an $L^2$-harmonic section $\mathfrak{u}$ as we showed in Proposition 3.6, then
\begin{align*}
\bigg\|D\left( \begin{array}{c}
u^+(t)\frac{1}{\sqrt{z}}\\
u^-(t)\frac{1}{\sqrt{\bar{z}}}
\end{array} \right)\bigg\|_{L^2(N_R)}^2\leq CR^{-2}\|\mathfrak{u}\|_{L^2(N_R)}^2
\end{align*}
and
\begin{align*}
\bigg\|\left(\begin{array}{c}
u^+(t)\frac{1}{\sqrt{z}}\\
u^-(t)\frac{1}{\sqrt{\bar{z}}}
\end{array} \right)\bigg\|_{L^2(N_R)}^2\leq C\|\mathfrak{u}\|_{L^2(N_R)}^2
\end{align*}
for some constant $C>0$.
\end{cor}
\begin{proof}
We can compute directly that
\begin{align*}
D\left( \begin{array}{c}
u^+(t)\frac{1}{\sqrt{z}}\\
u^-(t)\frac{1}{\sqrt{\bar{z}}}
\end{array} \right)=\left( \begin{array}{c}
\partial_t u^+(t)\frac{1}{\sqrt{z}}\\
\partial_t u^-(t)\frac{1}{\sqrt{\bar{z}}}
\end{array} \right).
\end{align*}
Then by Proposition 3.7, we can prove this corollary immediately.
\end{proof}

\subsection{Regularity properties and the asymptotic behavior of $L^2_1$-harmonic sections on the tubular neighborhood} The main result of this section is Theorem 3.11. Our goal is to estimate the $L^2$-norms of harmonic spinor $\mathfrak{v}$ and its derivative $\partial_t\mathfrak{v}$ on a small tubular neighborhood of $\Sigma$.\\

 Suppose that $\mathfrak{v}$ is an $L^2_1$-harmonic section, then we can write
\begin{align*}
\mathfrak{v}=\sum_{l,k}e^{ilt}\left( \begin{array}{c}
e^{(k-\frac{1}{2})\theta}V^+_{k,l}\\
e^{(k+\frac{1}{2})\theta}V^-_{k,l}
\end{array} \right)
\end{align*} 
where
\begin{align*}
\left( \begin{array}{c}
V^{+}\\
V^{-}
\end{array} \right)_{k,l}
=
\left( \begin{array}{c}
v_{k,l}^{+}\mathfrak{I}_{k-\frac{1}{2},l}(r)-v_{k,l}^{-}l\mathfrak{I}_{-k+\frac{1}{2},l}(r)\\
-v^{+}_{k,l}l\mathfrak{I}_{k+\frac{1}{2},l}(r)+v^{-}_{k,l}\mathfrak{I}_{-k-\frac{1}{2},l}(r)
\end{array} \right)
\end{align*}
for $l\neq 0$ and
\begin{align*}
\left( \begin{array}{c}
V^{+}\\
V^{-}
\end{array} \right)_{k,0}
=
\left( \begin{array}{c}
v_{k,0}^{+}r^{k-\frac{1}{2}}\\
v^{-}_{k,0}r^{-k-\frac{1}{2}}
\end{array} \right).
\end{align*}\\

Since $\mathfrak{v}\in L^2_1$, so we have
\begin{align*}
v^+_{k,l}=0 \mbox{ for }k\leq 0;\\
v^-_{k,l}=0 \mbox{ for }k\geq 0.
\end{align*}

\begin{pro}
Let $\mathfrak{v}\in L_1^2(N_R;\mathbf{\mathcal{S}}\otimes \mathcal{I})\cap \ker(D)$ with the corresponding coefficients $\{v^{\pm}_{k,l}\}$. Then the sequence of leading coefficients $\{(v^{\pm}_{l})\}$ defined in Definition 3.5 is in $\ell^2_k$ for all $k\in \mathbb{N}\cup \{0\}$. Moreover, we have
\begin{align}
\|(l^kv_l^{\pm})_{l\in\mathbb{Z}}\|^2_{\ell^2}\leq \frac{(2k+3)!}{R^{2k+3}}\|\mathfrak{v}\|^2_{L^2}\label{4_2b}
\end{align}
\end{pro}
\begin{proof}
We use the same notation defined in Proposition 3.7.
\begin{align*}
P_{-1,l}(\mathfrak{v})=e^{ilt}\left( \begin{array}{c}
-v_{-1,l}^{-}l\mathfrak{I}_{\frac{3}{2},l}(r)\\
v^{-}_{-1,l}\mathfrak{I}_{\frac{1}{2},l}(r)
\end{array} \right),\mbox{ }
P_{1,l}(\mathfrak{v})=e^{ilt}\left( \begin{array}{c}
v_{1,l}^{+}\mathfrak{I}_{\frac{1}{2},l}(r)\\
-v^{+}_{1,l}l\mathfrak{I}_{\frac{3}{2},l}(r)
\end{array} \right).
\end{align*}
for $l\neq 0$ and
\begin{align*}
P_{-1,0}(\mathfrak{v})=
\left( \begin{array}{c}
0\\
v^{-}_{-1,0}r^{\frac{1}{2}}
\end{array} \right),\mbox{ }
P_{1,0}(\mathfrak{v})=
\left( \begin{array}{c}
v^{+}_{1,0}r^{\frac{1}{2}}\\
0
\end{array} \right).\\
\end{align*}

Since $\mathfrak{I}_{\frac{1}{2},l}=\frac{\sinh (2lr)}{2l\sqrt{r}}$, we have
\begin{align*}
\|\mathfrak{v}\|_{L^2}^2&\geq \sum_{l}\|P_{-1,l}(\mathfrak{v})\|_{L^2}^2+\|P_{1,l}(\mathfrak{v})\|_{L^2}^2\\
&\geq \sum_{l\neq 0}( |v^+_{1,l}|^2+|v^-_{-1,l}|^2)\int_0^R\frac{\sinh^2 (2lr)}{4l^2}dr+(|v^+_{1,0}|^2+|v^-_{-1,0}|^2)\int_0^Rr^2dr\\
&\geq \sum_{l}( |v^+_{1,l}|^2+|v^-_{-1,l}|^2)\int_0^R\sum_{k=1}^{\infty}\frac{l^{2k}r^{2k+2}}{(2k+2)!}\\
&=\sum_{l}|v_{l}^{\pm}|^2\sum_{k=1}^{\infty}\frac{l^{2k}R^{2k+3}}{(2k+3)!}.
\end{align*}
Therefore, we prove this proposition.
\end{proof}

The proof of the following corollary is similar to the proof of Corollary 3.8. So we omit the proof of this corollary. 
\begin{cor}
Suppose $\left(\begin{array}{c}
v^+(t)\sqrt{z}\\
v^-(t)\sqrt{\bar{z}}
\end{array}\right)$ is the dominant term of an $L^2_1$-harmonic section $\mathfrak{v}$ as we showed in Proposition 3.6, then we have\\[1mm]
{\bf a}.
\begin{align*}
\bigg\|D\left(\begin{array}{c}
v^+(t)\sqrt{z}\\
v^-(t)\sqrt{\bar{z}}
\end{array}\right)\bigg\|^2_{L^2(N_R)}\leq C R^{-2}\|\mathfrak{v}\|_{L^2(N_R)}^2
\end{align*}
$\mbox{ }\mbox{ }\mbox{ }\mbox{ }$and
\begin{align*}
\bigg\|\left(\begin{array}{c}
v^+(t)\sqrt{z}\\
v^-(t)\sqrt{\bar{z}}
\end{array}\right)\bigg\|^2_{L^2(N_R)}\leq C\|\mathfrak{v}\|^2_{L^2(N_R)}
\end{align*}
$\mbox{ }\mbox{ }\mbox{ }\mbox{ }$for some constant $C>0$.\\[1mm]
{\bf b}.
\begin{align*}
\bigg\|D\Big(\nabla\left(\begin{array}{c}
v^+(t)\sqrt{z}\\
v^-(t)\sqrt{\bar{z}}
\end{array}\right)\Big)\bigg\|^2_{L^2(N_R)}\leq C R^{-4}\|\mathfrak{v}\|_{L^2(N_R)}^2
\end{align*}
$\mbox{ }\mbox{ }\mbox{ }\mbox{ }$and
\begin{align*}
\bigg\|\left(\begin{array}{c}
v^+(t)\sqrt{z}\\
v^-(t)\sqrt{\bar{z}}
\end{array}\right)\bigg\|^2_{L^2_1(N_R)}\leq CR^{-1}\|\mathfrak{v}\|^2_{L^2(N_R)}
\end{align*}
$\mbox{ }\mbox{ }\mbox{ }\mbox{ }$for some constant $C>0$.
\end{cor}

Finally, we can prove the following theorem by using Proposition 3.9 now. 

\begin{theorem}
For any $\mathfrak{v}\in L^2_1(N_{R})\cap \ker(D)$, we have
\begin{align*}
\|\mathfrak{v}\|_{L^2(N_r)}^2\leq r^3\frac{C}{R^3}\|\mathfrak{v}\|_{L^2(N_R)}^2.
\end{align*}
In addition, if we let $\mathfrak{v}_t=\partial_t \mathfrak{v}$, then we can prove that
\begin{align*}
\|\mathfrak{v}_t\|_{L^2(N_r)}^2\leq r^3\frac{C}{R^5}\|\mathfrak{v}\|_{L^2(N_R)}^2.
\end{align*}
for some constant $C>0$ and all $r\leq \frac{R}{2}$.
\end{theorem}
\begin{proof}
To prove the first statement, we use Lemma 2.6 to get
\begin{align}
\|\mathfrak{v}\|_{L^2(N_r)}^2\leq Cr^2\|\nabla\mathfrak{v}\|_{L^2(N_r)}^2\label{CC_17}
\end{align}
for all $\mathfrak{v}\in L^2_1(N_R)$ and $r<R$. By Proposition 3.9 and Proposition 3.6 {\bf b}, (\ref{CC_17}) implies
\begin{align*}
\int_{N_r}|\mathfrak{v}|^2&\leq Cr^2\int_{N_r}|\nabla\mathfrak{v}|^2\leq 2Cr^2\int_{N_r}\Big|\nabla\left(\begin{array}{c}
v^+(t)\sqrt{z}\\
v^-(t)\sqrt{\bar{z}}
\end{array}\right)\Big|^2+|\nabla\mathfrak{v}_{\mathfrak{R}}|^2\\
&\leq 2C\frac{r^3}{R^3}\|\mathfrak{v}\|^2_{L^2(N_R)}+2Cr^4\|\mathfrak{v}_{\mathfrak{R}}\|_{L^2_2(N_R)}^2\\
&\leq 4C\frac{r^3}{R^3}\|\mathfrak{v}\|^2_{L^2(N_R)}
\end{align*}
for some $C>0$.\\

To prove the second statement, we notice that by applying Lemma 2.6 on $\mathfrak{v}_t$,
\begin{align*}
\int_{N_r}|\mathfrak{v}_t|^2&\leq r^2\int_{N_r}|\nabla\mathfrak{v}_t|^2\leq 2r^2\int_{N_r}\Big|\nabla\left(\begin{array}{c}
v^+_t(t)\sqrt{z}\\
v^-_t(t)\sqrt{\bar{z}}
\end{array}\right)\Big|^2+|\nabla(\mathfrak{v}_{\mathfrak{R}})_t|^2.
\end{align*}
By using Proposition 3.9, we have
\begin{align*}
r^2\int_{N_r}\Big|\nabla\left(\begin{array}{c}
v^+_t(t)\sqrt{z}\\
v^-_t(t)\sqrt{\bar{z}}
\end{array}\right)\Big|^2 \leq 2\frac{r^3}{R^5}\|\mathfrak{v}\|^2_{L^2(N_R)}.
\end{align*}
So we have
\begin{align*}
\int_{N_r}|\mathfrak{v}_t|^2&\leq 2\frac{r^3}{R^5}\|\mathfrak{v}\|^2_{L^2(N_R)}+2r^2\|\mathfrak{v}_{\mathfrak{R}}\|_{L^2_2(N_{r})}^2.
\end{align*}

Then by the first statement proved above and Proposition 3.6 {\bf b},
\begin{align*}
\|\mathfrak{v}_{\mathfrak{R}}\|_{L^2_2(N_{r})}^2 \leq \frac{C}{R^4}\|\mathfrak{v}\|_{L^2(N_{2r})}^2\leq C \frac{r^3}{R^7}\|\mathfrak{v}\|_{L^2(N_R)}^2.
\end{align*}
So we prove the second statement.
\end{proof}

\begin{remark}\ \\[1mm]
{\bf a.} By using this theorem, Proposition 2.3 and Lemma 2.6, we can prove that for $\mbox{ }\mbox{ }\mbox{ }\mbox{ }\mbox{ }$any $\mathfrak{v}\in L^2_1(N_{R})\cap \ker(D)$, we have
\begin{align*}
\|\mathfrak{v}\|_{L^2_{-1}(N_r)}^2\leq r^5\frac{C}{R^3}\|\mathfrak{v}\|_{L^2(N_R)}^2.
\end{align*}
$\mbox{ }\mbox{ }\mbox{ }\mbox{ }\mbox{ }$Moreover, we have
\begin{align*}
\|\mathfrak{v}_t\|_{L^2_{-1}(N_r)}^2\leq r^5\frac{C}{R^5}\|\mathfrak{v}\|_{L^2(N_R)}^2
\end{align*}
$\mbox{ }\mbox{ }\mbox{ }\mbox{ }\mbox{ }$for some constant $C>0$.\\
\ \\
{\bf b.} By Proposition 3.9 and the definition of modified Bessel functions, one can $\mbox{ }\mbox{ }\mbox{ }\mbox{ }\mbox{ }$prove directly that the remainder term $\mathfrak{v}_{\mathfrak{R}}$ is bounded by the order $r^{p}$ for $\mbox{ }\mbox{ }\mbox{ }\mbox{ }\mbox{ }$any $p\leq\frac{3}{2}$. Similarly, for an $L^2$-harmonic spinor $\mathfrak{u}$ with $\mathcal{K}_R$-leading coefficients, $\mbox{ }\mbox{ }\mbox{ }\mbox{ }\mbox{ }$the remainder term $\mathfrak{u}_{\mathfrak{R}}$ is bounded by the order $r^p$ for any $p\leq \frac{1}{2}$. This can be $\mbox{ }\mbox{ }\mbox{ }\mbox{ }\mbox{ }$obtained by Proposition 3.7. This result is also true for $L^2$-harmonic spinors $\mbox{ }\mbox{ }\mbox{ }\mbox{ }\mbox{ }$with its sequence of leading coefficients has $L^2_1$ bound.
\end{remark}

\section{Variational formula and perturbation of curves}
The previous Section gives us some analytic tools to handle the perturbation of $\psi$ later. Section 4 follows below will give us some important analytic tools to deal with the perturbation of the metric $g$ and $\Sigma$.

\subsection{Variational formula}
We should review the following fact about the Sobolev inequality and introduce a modified Poincare inequality first.\\

Let $\mathfrak{v}\in L^2(M\setminus\Sigma;\mathbf{\mathcal{S}}\otimes \mathcal{I} )$. We have $|\mathfrak{v}|\in L^2(M\setminus\Sigma;\mathbb{R})$. Since $\Sigma$ is a measure zero subset of $M$, $|\mathfrak{v}|$ can be extended as an $L^2$ section over $M$. Moreover, suppose $\mathfrak{v}$ is in $L^2_1(M\setminus\Sigma;\mathbf{\mathcal{S}}\otimes \mathcal{I} )$, then we will have $|\mathfrak{v}|\in L^2_1(M;\mathbb{R})$ and
\begin{align*}
\|(|\mathfrak{v}|)\|_{L^2_1(M;\mathbb{R})}\leq C\|\mathfrak{v}\|_{L^2_1(M\setminus\Sigma;\mathbf{\mathcal{S}}\otimes \mathcal{I} )}.
\end{align*}

Now, by Sobolev inequality, we have
\begin{align}
\|\mathfrak{v}\|_{L^6(M\setminus\Sigma;\mathbf{\mathcal{S}}\otimes \mathcal{I} )}= \|(|\mathfrak{v}|)\|_{L^6(M;\mathbb{R})}\leq C_0\|(|\mathfrak{v}|)\|_{L^2_1(M;\mathbb{R})}\leq C\|\mathfrak{v}\|_{L^2_1(M\setminus\Sigma;\mathbf{\mathcal{S}}\otimes \mathcal{I} )}\label{6_1}
\end{align}
for some constant $C_0,C>0$.\\

 Another important tool we need is the following modified Poincare inequality.
\begin{lemma}
Let $\mathfrak{v}\in L^2_1$ and $\mathfrak{v}\perp \ker(D)$ in $L^2-$sense, then we have
\begin{align}
\|\mathfrak{v}\|_{L^2_1}\leq C \|D \mathfrak{v}\|_{L^2}\label{6_3}
\end{align}
for some $C$ depending only on the volume of $M$.
\end{lemma}
\begin{proof}
The inequality,
\begin{align}
\|\mathfrak{v}\|_{L^2}\leq C \|D \mathfrak{v}\|_{L^2}\label{6_3a},
\end{align}
can be obtained immediately by proving Dirac operator has empty residual spectrum, empty continuous spectrum and has nonnegative first eigenvalue. See Chapter 4 in \cite{D} for the proof. Then, (\ref{6_3}) can be obtained by (\ref{6_3a}) and (\ref{1_1}).
\end{proof}

We will use $L^2_1\cap\ker(D|_{L^2_1})^{\perp}$ to denote the collection of elements in $L^2_1$ which is perpendicular to $\ker(D|_{L^2_1})$ in $L^2$-{\bf sense}. The terminology ``$\perp$'' used in this section is always in $L^2$-sense.

\begin{definition}
Let $\mathfrak{f}\in L^2_{-1}$, we define the functional
\begin{align*}
E_{\mathfrak{f}}(\mathfrak{v})=\int_{M\setminus\Sigma}|D\mathfrak{v}|^2+\langle \mathfrak{v}, \mathfrak{f}\rangle
\end{align*}
for all $\mathfrak{v}\in L^2_1$.
\end{definition}
Since $D$ is self-adjoint, the Euler-Lagrange equation of $E_{\mathfrak{f}}$ will be
\begin{align}
D^2\mathfrak{v}=\mathfrak{f}\label{6_2}.
\end{align}
However, the following proposition and its corollary (Proposition 4.3 and Corollary 4.4) tell us a solutions of (\ref{6_2}) exists only if $\mathfrak{f}\in L^2_{-1}\cap\ker(D|_{L^2_1})^{\perp}$\footnote{The space $L^2_1\cap \ker(D|_{L^2_1})^{\perp}$ is composed by elements $\mathfrak{f}\in L^2_{-1}$ where $\int\langle \mathfrak{f},\mathfrak{v}\rangle=0$ for all $\mathfrak{v}\in  \ker(D|_{L^2_1})$. By Definition 2.1, this is well-defined. As I mentioned before, the perpendicularity is in $L^2$-sense.}. We will have further discussion in Section 4.5.

\begin{pro}
Let $\mathfrak{f}\in L^2_{-1}$ be given. There exists $\alpha>0, \beta\in \mathbb{R}$ such that for any $\mathfrak{v}\in L^2_1\cap \ker(D|_{L^2_1})^{\perp}$,
\begin{align}
E_{\mathfrak{f}}(\mathfrak{v})\geq \alpha\|D\mathfrak{v}\|^2_{L^2}-\beta\label{6_2_1}
\end{align}
(This property is usually called coercivity). Moreover, if we minimize $E_{\mathfrak{f}}$ in the admissible set $L^2_1\cap \ker(D|_{L^2_1})^{\perp}$, then we have a unique minimizer for $E_{\mathfrak{f}}$.
\end{pro}

\begin{proof}
The inequality (\ref{6_2_1}) can be obtained directly from Proposition 2.3 and Lemma 4.1. So we should only prove that $E_{\mathfrak{f}}$ has a unique minimizer in $L^2_1\cap \ker(D|_{L^2_1})^{\perp}$ by using (\ref{6_2_1}). Suppose we have a sequence $\{\mathfrak{v}_n\}\subset L^2_1\cap \ker(D|_{L^2_1})^{\perp}$ such that
\begin{align*}
\lim_{n\rightarrow \infty}E_{\mathfrak{f}}(\mathfrak{v}_n)=\inf_{\mathfrak{v}\in L^2_1\cap \ker(D)^{\perp}}E_{\mathfrak{f}}(\mathfrak{v}).
\end{align*}
Let us call $\inf_{\mathfrak{v}\in L^2_1\cap \ker(D)^{\perp}}E_{\mathfrak{f}}(\mathfrak{v})=m$. Then there exists $n_0\in \mathbb{N}$ such that
\begin{align*}
E_{\mathfrak{f}}(\mathfrak{v}_n)\leq m+1
\end{align*}
for all $n>n_0$. So
\begin{align*}
\alpha\|D\mathfrak{v}_n\|^2_{L^2}-\beta\leq E_{\mathfrak{f}}(\mathfrak{v}_n)\leq m+1
\end{align*}
for all $n>n_0$. This inequality implies that the sequence $\{\|D\mathfrak{v}_n\|_{L^2}\}_{n>n_0}$ is bounded. By Lemma 4.1, $\{\|\mathfrak{v}_n\|_{L^2_1}\}$ is bounded. So a subsequence of $\{\mathfrak{v}_n\}$ has a weak limit, say $\mathfrak{v}$. Because of the Lichnerowicz-Weitzenb\"ock formula
\begin{align*}
\|D\mathfrak{v}_n\|_{L^2}=\|\nabla\mathfrak{v}_n\|_{L^2}+\int\frac{\mathscr{R}}{4}|\mathfrak{v}_n|^2,
\end{align*}
one can apply Theorem 5 on p. 103 in \cite{O} to obtain the inequality $\|D\mathfrak{v}\|_{L^2}\leq \liminf_{n\rightarrow\infty} \|D\mathfrak{v}_n\|_{L^2}$. So $\mathfrak{v}$ is a minimizer of $E_{\mathfrak{f}}$.\\

Finally, we prove the uniqueness. Suppose we have $\mathfrak{v}_a,\mathfrak{v}_b$ are two minimizers in $L^2_1\cap \ker(D|_{L^2_1})^{\perp}$, then
\begin{align*}
E_{\mathfrak{f}}\Big(\frac{\mathfrak{v}_a+\mathfrak{v}_b}{2}\Big)&=\int\frac{1}{4}(|D\mathfrak{v}_a+D\mathfrak{v}_b|^2)+\frac{1}{2}\langle\mathfrak{v}_a,\mathfrak{f}\rangle+\frac{1}{2}\langle \mathfrak{v}_b,\mathfrak{f}\rangle\\
&\leq \int \frac{1}{2}|D\mathfrak{v}_a|^2+\frac{1}{2}|D\mathfrak{v}_b|^2+\frac{1}{2}\langle\mathfrak{v}_a,\mathfrak{f}\rangle+\frac{1}{2}\langle \mathfrak{v}_b,\mathfrak{f}\rangle\\
&=m
\end{align*}
by Cauchy's inequality. The equality holds if and only if $D\mathfrak{v}_a=D\mathfrak{v}_b$, which implies $\mathfrak{v}_a=\mathfrak{v}_b$ by Lemma 4.1.
\end{proof}
\begin{cor}
For any $\mathfrak{f}=\mathfrak{f}_0+\mathfrak{f}_1\in L^2_{-1}$ with $\mathfrak{f}_0\in\ker(D|_{L^2_1})$ and $\mathfrak{f}_1\in L^2_{-1}\cap\ker(D|_{L^2_1})^{\perp}$. The minimizer $\mathfrak{v}\in L^2_1\cap \ker(D|_{L^2_1})^{\perp}$ of $E_{\mathfrak{f}}$ given by Proposition 4.3 will satisfy the equation
\begin{align*}
D^2\mathfrak{v}=\mathfrak{f}_1.
\end{align*}
In particular, if $\mathfrak{f}\in L^2_{-1}\cap\ker(D|_{L^2_1})^{\perp}$, we have $D^2\mathfrak{v}=\mathfrak{f}$.
\end{cor}
\begin{proof}
For any $\mathfrak{w}\in L^2_1\cap \ker(D|_{L^2_1})^{\perp}$, we have
\begin{align}
0=\frac{d}{ds}E_{\mathfrak{f}}(\mathfrak{v}+s\mathfrak{w})\Big|_{s=0}=\int_{M\setminus\Sigma} \langle D\mathfrak{w},D\mathfrak{v}\rangle+\langle\mathfrak{w},\mathfrak{f}\rangle\label{DD_1}
\end{align}
following by Propositin 4.3 and the standard argument in Theorem 4, p. 473 in \cite{G}. Since $\mathfrak{w}\in L^2_1\cap \ker(D|_{L^2_1})^{\perp}$, we have
\begin{align*}
\int_{M\setminus\Sigma}\langle\mathfrak{w},\mathfrak{f}\rangle=\int_{M\setminus\Sigma}\langle \mathfrak{w},\mathfrak{f}_1\rangle.
\end{align*}
So (\ref{DD_1}) implies that
\begin{align}
\int_{M\setminus\Sigma}\langle D\mathfrak{w},D\mathfrak{v}\rangle+\langle \mathfrak{w},\mathfrak{f}_1\rangle
=0\label{DD_2}
\end{align}
for all $\mathfrak{w}\in L^2_1\cap \ker(D|_{L^2_1})^{\perp}$. Since $\mathfrak{f}_1\in \ker(D|_{L^2_1})^{\perp}$, (\ref{DD_2}) holds for all $\mathfrak{w}\in L^2_1$. In particular, it holds for any smooth section with compact support. So by taking integration by parts, we have
\begin{align}
\int_{M\setminus\Sigma}\langle w,(D^2\mathfrak{v}-\mathfrak{f}_1)\rangle=0\label{DD_3}
\end{align} 
for all $w\in C^{\infty}_{cpt}(M\setminus\Sigma;\mathcal{S}\otimes\mathcal{I})$. This implies Corollary 4.4 immediately.
\end{proof}

\subsection{Perturbation of $\Sigma$: Local trivialization} In this section, we define some notation and explain the local trivialization of $\mathcal{E}$ (We follow the notation in Section 1). First of all, let $N_R$ be the tubular neighborhood of $\Sigma\in\mathcal{A}$. There exists a neighborhood of $\Sigma$ in $\mathcal{A}$, say $\mathcal{V}_{\Sigma}$, such that
$\Sigma'\subset N_{\frac{R}{2}}$ for all $\Sigma'\in\mathcal{V}_{\Sigma}$. Therefore, we can parametrize elements in $\mathcal{V}_{\Sigma}$ by $\{\eta:S^1\rightarrow \mathbb{C}| \eta\in C^1\mbox{ and } \|\eta\|_{C^1}\leq C_R\}$ for some $C_R$ depending on $R$. We map $\eta$ to $\{(t,\eta(t))\}=\Sigma'\subset N_R$.\\

Here we choose a variable $\mathfrak{r}<\frac{R}{4}$. This variable will also be used in the rest of this paper. Also, we fix a $T>1$ which will be specified in the following sections. We can assume $\mathfrak{r}$, $R$ and $T$ are fixed although they will be modified finite many times in this paper (The precise value of $\mathfrak{r}$ and $R$ can be assumed to decrease between each successive appearance; $T>1$ can be assumed to increase between each successive appearance).\\

 We define
\begin{align}
\chi^{(\mathfrak{r})}=\left\{ \begin{array}{cc}
1-\chi_{\frac{\mathfrak{r}}{T},\mathfrak{r}} &\mbox{ on }N_{R}\\
0 &\mbox{ on } M\setminus N_{R}
\end{array} \right.\label{DD_4}
\end{align}
(We will omit the superscript $(\mathfrak{r})$ later, but keep in mind that this function depends on $\mathfrak{r}$).\\

For each $(\eta, \mathfrak{r})$, we now define the following map
\begin{align}
\phi^{(\mathfrak{r})}&:M\setminus\Sigma \rightarrow M\setminus\Sigma';\nonumber\\
\phi^{(\mathfrak{r})}&(p)=p\mbox{ for all }p\in M\setminus N_R,\label{DD_5}\\
&(t,z)\mapsto (t,z+\chi^{(\mathfrak{r})}(z)\eta(t))\mbox{ on }N_{R}\nonumber
\end{align}
with $\Sigma'=\{(\eta(t),t)\}$. This map is a diffeomorphism if $\|\eta\|_{C^1}\leq C_{\mathfrak{r}}$ for some constant $C_{\mathfrak{r}}$ depending on $\mathfrak{r}$.\\

We fix $g$ for a moment. Recall that the fiber of $\mathcal{E}$ over $(g,\Sigma',e)\in\mathcal{X}\times\mathcal{A}_{H}$ is the space $L^2_1(M\setminus\Sigma';\mathcal{S}_{g,\Sigma',e})$, which can be identified with $L^2_1(M\setminus\Sigma;\mathcal{S}_{\phi^{(\mathfrak{r})*}g,\Sigma,e})$. Therefore, for any element $(g,\Sigma)\in \mathcal{X}\times\mathcal{A}_{H}$, there exists $\mathcal{N}\subset  \mathcal{X}\times\mathcal{A}_{H}$, a neighborhood of $(g,\Sigma)$, such that the bundle $\mathcal{E}|_{\mathcal{N}}\simeq \pi_1(\mathcal{N})\times \mathbb{B}_{\varepsilon}\times L^2_1$ where $L^2_1\simeq L^2_1(M\setminus\Sigma;\mathcal{S}_{g,\Sigma,e})$ and $\mathbb{B}_{\varepsilon}=\{\eta:S^1\rightarrow \mathbb{C}|\eta\in C^1\mbox{ and }\|\eta\|_{C^1}\leq \varepsilon\}$ for some small $\varepsilon>0$.\\

By the same token, we have the local trivialization of $\mathcal{F}$ near $(g,\Sigma,e)$ to be $\pi_1(\mathcal{N})\times\mathbb{B}_{\varepsilon}\times L^2$. The Dirac operator $D:\mathcal{E}\rightarrow \mathcal{F}$ will be a family of first order differential operator mapping from $\mathbb{B}_{\varepsilon}\times L^2_1$ to $\mathbb{B}_{\varepsilon}\times L^2$. Therefore, the kernel of the lineariztion map of $\mathfrak{M}$ (when $g$ is fixed),$\mathbb{K}_0$, will be contained in $\mathbb{V}\times L^2_1$ where
\begin{align*}
\mathbb{V}=\{\eta:S^1\rightarrow \mathbb{C}|\eta\in C^1\}.
\end{align*}
By Proposition 2.4, we know the projection of $\mathbb{K}_0$ on the second factor ,$L^2_1$, is finite dimensional. We will prove that the projection of $\mathbb{K}_0$ on $\mathbb{V}$ is also finite dimensional in Section 6.

\subsection{Perturbation of $\Sigma$: Estimates} Recall that we assume the product metric being defined on $N_{R}$, which is $g_{N_R}=dt^2+dr^2+r^2d\theta$. In the following sections, we choose a positive constant $\mathfrak{r}<\frac{R}{4}$. The precise value of $\mathfrak{r}$ can be assumed to decrease between each successive appearance. Also, we fix a $T>1$ which will be specified in the following sections.\\

Consider a pair $(\chi,\eta)$ where $\eta\in C^{\infty}(S^1;\mathbb{C})$ (here $\chi=\chi^{(\tau)}$). 
We can define the corresponding one-parameter family of diffeomorphisms
\begin{align}
\phi_s:&M\setminus\Sigma \rightarrow M\setminus\Sigma_s;\nonumber\\
&(t, z)\mapsto (t, z+s\chi(z)\eta(t))\mbox{ on }N_R,\label{DD_6}\\
&\phi_s(p)=p \mbox{  for all }p\in M\setminus N_R\nonumber
\end{align}
with $0\leq s\leq t_0$ for some small $t_0$ and $\Sigma_s=\{(t,s\eta(t))\}$. We fix a positive $s\leq t_0$ and use $(\tau,u)\in [0,2\pi]\times\{z\in\mathbb{C}||z|< R\}$ to denote the coordinates on $\phi_s(N_R)$ in the following paragraphs. We also define the notation
\begin{align}
\eta_t&=\partial_t\eta,\\
\chi_z&=\partial_z\chi,\\
\chi_{\bar{z}}&=\partial_{\bar{z}}\chi.\label{DD_7}
\end{align}

If we write down the relationship of $\partial_t$, $\partial_z$ and $\partial_{\bar{z}}$ and the push-forward tangent vectors $(\phi_s^{-1})_*(\partial_\tau)$, $(\phi_s^{-1})_*(\partial_u)$ and $(\phi_s^{-1})_*(\partial_{\bar{u}})$,
\begin{align*}
\left\{ \begin{array}{cc}
\partial_t =(\phi_s^{-1})_*\Big(\dfrac{\partial \tau}{\partial t}\partial_{\tau}+\dfrac{\partial u}{\partial t}\partial_u+\dfrac{\partial \bar{u}}{\partial t}\partial_{\bar{u}}\Big)\\[3mm]
\partial_z =(\phi_s^{-1})_*\Big(\dfrac{\partial u}{\partial z}\partial_u+\dfrac{\partial \bar{u}}{\partial z}\partial_{\bar{u}}\Big)\mbox{ }\mbox{ }\mbox{ }\mbox{ }\mbox{ }\mbox{ }\mbox{ }\mbox{ }\mbox{ }\mbox{ }\\[3mm]
\partial_{\bar{z}} =(\phi_s^{-1})_*\Big(\dfrac{\partial u}{\partial \bar{z}}\partial_u+\dfrac{\partial \bar{u}}{\partial \bar{z}}\partial_{\bar{u}}\Big)\mbox{ }\mbox{ }\mbox{ }\mbox{ }\mbox{ }\mbox{ }\mbox{ }\mbox{ }\mbox{ }\mbox{ }
\end{array} \right.,
\end{align*}
we will have
\begin{align*}
(\phi_s^{-1})_*\left( \begin{array}{c}
\partial_{\tau}\\
\partial_{u}\\
\partial_{\bar{u}}
\end{array} \right)=
\mathcal{M}
\left( \begin{array}{c}
\partial_{t}\\
\partial_{z}\\
\partial_{\bar{z}}
\end{array} \right)
\end{align*}
where
\begin{align*}
\mathcal{M}=\frac{1}{1+s(\chi_z\eta+\chi_{\bar{z}}\bar{\eta})}\left( \begin{array}{ccc}
1+s(\chi_z\eta+\chi_{\bar{z}}\bar{\eta})&0&0\\
-s\chi\eta_t-s^2\chi\chi_{\bar{z}}(\eta_t\bar{\eta}-\bar{\eta}_t\eta)&1+s\chi_{\bar{z}}\bar{\eta}&-s\chi_{\bar{z}}\eta\\
-s\chi\bar{\eta}_t-s^2\chi\chi_z(\eta\bar{\eta}_t-\eta_t\bar{\eta})&-s\chi_z\bar{\eta}&1+s\chi_z\eta
\end{array} \right).
\end{align*}

Since the metric and spinor bundle are fixed over $M$ here, so the Clifford multiplication $\kappa:TM\rightarrow Cl(TM)$ will always send $\partial_{\tau},\partial_u,\partial_{\bar{u}}$ to $e_1=\left( \begin{array}{cc}
-i & 0\\
0 & i 
\end{array} \right),e_2=\left( \begin{array}{cc}
0 & 0\\
-1 & 0 
\end{array} \right),e_3=\left( \begin{array}{cc}
0 & 1\\
0 & 0 
\end{array} \right)$ respectively. Therefore, the Dirac operator $D_s$ defined on $\phi_s(N_R)$ will be
\begin{align*}
D_{s}=&e_1\cdot\partial_{\tau}+e_2\cdot\partial_u+e_3\cdot\partial_{\bar{u}}\\
&+\frac{1}{2}\sum_{i=1}^3e_i\sum_{k,l}\omega_{kl}(e_i)e_ke_l
\end{align*}
where $\omega_{kl}$ is the forms defining the Levi-Civita connection.\\

In the following sections, all these perturbed curves will be identified with $\Sigma$ by using the pull-back operator $(\phi_s^{-1})_*$. So we have to write down the corresponding Dirac operator explicitly
\begin{align*}
D_{s\chi\eta}=(\phi_s^{-1})_*\circ D_s=&e_1\cdot(\phi_s^{-1})_*(\partial_{\tau})+e_2\cdot(\phi_s^{-1})_*(\partial_u)+e_3\cdot(\phi_s^{-1})_*(\partial_{\bar{u}})\\
&+\frac{1}{2}\sum_{i=1}^3e_i\sum_{k,l}(\phi_s^{-1})_*(\omega_{kl}(e_i))e_ke_l.
\end{align*}

Suppose that we have the following assumptions: There exist $\kappa_0$ such that
\begin{align}
\|\eta\|_{L^2(S^1)}\leq \kappa_0 \mathfrak{r}^2,\label{6_a1}\\
\|\eta_t\|_{L^2(S^1)}\leq \kappa_0 \mathfrak{r},\label{6_a2}\\
\|{\eta}_{tt}\|_{L^2(S^1)}\leq \kappa_0.\label{6_a3}
\end{align}
We will see that these inequalities will imply that there exists $\kappa_1=O(\kappa_0)$ such that
\begin{align}
\max\{|\chi_z||\eta|,|\chi_{\bar{z}}||\eta|, |\eta_t|\}\leq \gamma_{_T}\kappa_1\mathfrak{r}^{\frac{1}{2}}\label{6_a4}\\
\|\chi_{z}\eta_t\|_{L^2}, \|\chi_{\bar{z}}\eta_t\|_{L^2}\leq \gamma_{_T}\kappa_1\label{6_a5}\\
\|\chi_{zz}\eta\|_{L^2},\|\chi_{z\bar{z}}\eta\|_{L^2}, \|\chi_{\bar{z}\bar{z}}\eta\|_{L^2}\leq \gamma_{_T}^2\kappa_1\label{6_a6}.
\end{align}
where we denote $(\frac{T}{T-1})$ by $\gamma_{_T}$.\\

For the perturbed Dirac operator $D_{s\chi\eta}$, we have the following proposition.
\begin{pro} There exists $\kappa_1=O(\kappa_0)$ depending on $\kappa_0$ with the following significance.
The perturbed Dirac operator $D_{s\chi\eta}$ with $\eta$ satisfying (\ref{6_a1}) - (\ref{6_a3}) can be written as follows:
\begin{align}
D_{s\chi\eta}=(1+\varrho_{s\chi\eta})D+s(\chi_z\eta+\chi_{\bar{z}}\bar{\eta})(e_1\partial_t)+\Theta_{s}^0
+\mathcal{R}_s^0+\mathcal{H}_s^0+\mathcal{F}_s^0\label{DD_8}
\end{align}
where\\[1mm]
$\bullet$ $\Theta_s^0=[e_1(s\chi\eta_t\partial_z+s\chi \bar{\eta}_t\partial_{\bar{z}})
+e_2(s\chi_{\bar{z}}\bar{\eta}\partial_z-s\chi_z\bar{\eta}\partial_{\bar{z}})
+e_3(-s\chi_{\bar{z}}\eta\partial_z+s\chi_z\eta\partial_{\bar{z}})]$ is a \\$\mbox{ }\mbox{ }\mbox{ }$first order differential operator.\\[1mm]
$\bullet$ $\mathcal{R}_s^0: L^2_1\rightarrow L^2$ is an $O(s^2)$-first order differential operator supported on $N_\mathfrak{r}-N_{\frac{\mathfrak{r}}{T}}$ $\mbox{ }\mbox{ }\mbox{ }$with its operator norm $\|\mathcal{R}_s\|\leq \gamma_{_T}^2\kappa_1^2s^2$.\\[1mm]
$\bullet$ $\mathcal{H}_s^0$ is an $O(s^2)$-zero order differential operator supported on $N_{\mathfrak{r}}-N_{\frac{\mathfrak{r}}{T}}$. Moreover, $\mbox{ }\mbox{ }\mbox{ }$let us denote $\partial_r$ by $\vec{n}$, the vector field defined on $N_R$, then 
\begin{align}
\int_{\{r=r_0\}}|\mathcal{H}_s^0|^2i_{\vec{n}}dVol(M)\leq \gamma_{_T}^2\kappa_1^4 \mathfrak{r} s^4\label{DD_9}
\end{align}
$\mbox{ }\mbox{ }\mbox{ }$for all $r_0\leq \mathfrak{r}$.\\[1mm]
$\bullet$ $\mathcal{F}_s^0$ is an $O(s)$-zero order differential operator where
\begin{align}
\mathcal{F}_s^0=D(s(\chi_z\eta+\chi_{\bar{z}}\bar{\eta})Id)+D\Big(\left(\begin{array}{cc}
 0&si\chi\eta_t\\
-si\chi\bar{\eta}_t&0
\end{array}\right)\Big).\label{DD_10}
\end{align}
\end{pro}

The reason to rewrite the perturbed Dirac operator in this form is because in Section 7 and Section 8, we will construct an iteration which defines the homeomorphism in Theorem 1.5. In the process of the iteration, the first two terms and $\Theta_s$ play the leading role in the iteration. The terms $\mathcal{R}_s$, $\mathcal{H}_s$ and $\mathcal{F}_s$ are relatively unimportant. They give some elements which converge to $0$.

\begin{proof}\ \\
{\bf Step 1}. We can see that, after some standard computation,
\begin{align}
(\phi_s^{-1})_*(\omega_{kl}(e_i))=\mathcal{M}(\omega(e_i))\mathcal{M}^{-1}+(d\mathcal{M})\mathcal{M}^{-1}=(d\mathcal{M})\mathcal{M}^{-1}.\label{DD_11}
\end{align}
Here we write down precisely the $O(s)$ order term of $\sum_{i=1}^3e_i\sum_{k,l}(\phi_s^{-1})_*(\omega_{kl}(e_i))e_ke_l$, which is
\begin{align*}
&-[(d\mathcal{M})_{11}(e_1)Id+(d\mathcal{M})_{11}(e_2)e_2+(d\mathcal{M})_{11}(e_3)e_3]\\
&+[-(d\mathcal{M})_{12}(e_1)e_2-(d\mathcal{M})_{13}(e_1)e_3
+(d\mathcal{M})_{23}(e_1)e_1e_2e_3+(d\mathcal{M})_{32}(e_1)e_1e_3e_2]\\
&=D(s(\chi_z\eta+\chi_{\bar{z}}\bar{\eta})Id)+D\Big(\left(\begin{array}{cc}
 0&si\chi\eta_t\\
-si\chi\bar{\eta}_t&0
\end{array}\right)\Big):=\mathcal{F}_s^0.
\end{align*}
So the term $\frac{1}{2}\sum_{i=1}^3e_i\sum_{k,l}(\phi_s^{-1})_*(\omega_{kl}(e_i))e_ke_l$ can be expressed as
\begin{align}
\frac{1}{2}\sum_{i=1}^3e_i\sum_{k,l}(\phi_s^{-1})_*(\omega_{kl}(e_i))e_ke_l=\mathcal{F}_s^0 +\mathcal{H}_s^0\label{DD_12}
\end{align}
where $\mathcal{F}_s^0 $ is the $O(s)$-zero order differential operator described as above and $\mathcal{H}_s^0 $
is an $O(s^2)$-zero order differential operator.\\

Here we prove (\ref{6_a4}), (\ref{6_a5}) and (\ref{6_a6}). First, notice that by Sobolev inequality, we have $\eta$ is continuous. So
\begin{align*}
|\eta|^2(t)&\leq\frac{1}{2\pi}\int_0^{2\pi}|\eta|^2+\int_0^{2\pi}\partial_t(|\eta|^2)\\
&\leq \frac{1}{2\pi}\|\eta\|_{L^2}^2+2\|\eta\|_{L^2}\|\eta_t\|_{L^2}\\
&\leq \frac{1}{2\pi}\kappa_0^2\mathfrak{r}^4+2 \kappa_0^2\mathfrak{r}^3\\
&\leq (\frac{1}{2\pi}+2)\kappa_0^2\mathfrak{r}^3.
\end{align*}
Meanwhile, we have $|\chi_{z}|,|\chi_{\bar{z}}|\leq C\frac{\gamma_{_T}}{\mathfrak{r}}$. Therefore,
\begin{align*}
|(\chi_i)_z||\eta_i|,|(\chi_i)_{\bar{z}}||\eta_i|\leq C\kappa_0\mathfrak{r}^{\frac{1}{2}}.
\end{align*}
This implies (\ref{6_a4}). The inequality (\ref{6_a5}) can be proved by the fact $|\chi_{z}|,|\chi_{\bar{z}}|\leq C\gamma_{_T}\frac{1}{\mathfrak{r}}$; (\ref{6_a6}) can be proved by the fact $|\chi_{zz}|,|\chi_{\bar{z}z}|, |\chi_{\bar{z}\bar{z}}|\leq C\gamma_{_T}^2\frac{1}{\mathfrak{r}^2}$ and (\ref{6_a1}).\\

Under these assumptions, for any $s$ small, we have
\begin{align*}
\Bigg|\frac{1}{1+s(\chi_z\eta+\chi_{\bar{z}}\bar{\eta})}-1\Bigg|\leq 2 s\gamma_{_T}\kappa_1\mathfrak{r}^{\frac{1}{2}}.
\end{align*}
We can write $\frac{1}{1+s(\chi_z\eta+\chi_{\bar{z}}\bar{\eta})}=1+\varrho_{s\chi\eta}$. Then
\begin{align}
|\varrho_{s\chi\eta}|\leq 2s\gamma_{_T}\kappa_1\mathfrak{r}^{\frac{1}{2}}.\label{DD_13}
\end{align}

{\bf Step 2}. Using the conventions defined above, we have
\begin{align}
\mathcal{M}=(1+\varrho_{s\chi\eta})\Bigg[
&\left( \begin{array}{ccc}
1+s(\chi_z\eta+\chi_{\bar{z}}\bar{\eta})&0&0\\
0&1&0\\
0&0&1
\end{array} \right)
+\left( \begin{array}{ccc}
0&0&0\\
-s\chi\eta_t&s\chi_{\bar{z}}\bar{\eta}&-s\chi_{\bar{z}}\eta\\
-s\chi \bar{\eta}_t&-s\chi_z\bar{\eta}&s\chi_z\eta
\end{array} \right)\label{6_1m}
\\
&\mbox{ }\mbox{ }\mbox{ }\mbox{ }\mbox{ }\mbox{ }\mbox{ }\mbox{ }\mbox{ }\mbox{ }\mbox{ }\mbox{ }\mbox{ }\mbox{ }\mbox{ }\mbox{ }\mbox{ }\mbox{ }\mbox{ }\mbox{ }\mbox{ }\mbox{ }\mbox{ }\mbox{ }\mbox{ }\mbox{ }\mbox{ }\mbox{ }\mbox{ }\mbox{ }\mbox{ }\mbox{ }\mbox{ }\mbox{ }\mbox{ }+
\left( \begin{array}{ccc}
0&0&0\\
-s^2\chi\chi_{\bar{z}}(\eta_t\bar{\eta}- \bar{\eta}_t\eta)&0&0\\
-s^2\chi\chi_z(\eta \bar{\eta}_t-\eta_t\bar{\eta})&0&0
\end{array} \right)
\Bigg].\nonumber
\end{align}

 Therefore, by (\ref{DD_12}), we can rewrite 
\begin{align}
D_{s\chi\eta}=&(1+\varrho_{s\chi\eta})(D+s(\chi_z\eta+\chi_{\bar{z}}\bar{\eta})(e_1\partial_t))\label{DD_14}\\
&+(1+\varrho_{s\chi\eta})[e_1(s\chi\eta_t\partial_z+s\chi\bar{\eta}_t\partial_{\bar{z}})\nonumber
+e_2(s\chi_{\bar{z}}\bar{\eta}\partial_z-s\chi_z\bar{\eta}\partial_{\bar{z}})\\
&\mbox{ }\mbox{ }\mbox{ }\mbox{ }\mbox{ }\mbox{ }\mbox{ }\mbox{ }\mbox{ }\mbox{ }\mbox{ }\mbox{ }\mbox{ }\mbox{ }\mbox{ }\mbox{ }\mbox{ }\mbox{ }\mbox{ }\mbox{ }\mbox{ }\mbox{ }\mbox{ }\mbox{ }\mbox{ }\mbox{ }\mbox{ }\mbox{ }\mbox{ }\mbox{ }\mbox{ }\mbox{ }\mbox{ }\mbox{ }\mbox{ }\mbox{ }\mbox{ }\mbox{ }\mbox{ }\mbox{ }\mbox{ }\mbox{ }
+e_3(-s\chi_{\bar{z}}\eta\partial_z+s\chi_z\eta\partial_{\bar{z}})]\nonumber\\
&+\hat{\mathcal{R}}_s+\mathcal{H}_s+\mathcal{F}_s\nonumber
\end{align}
where

\begin{align}
\hat{\mathcal{R}}_s:=\frac{-1}{1+s(\chi_z\eta+\chi_{\bar{z}}\bar{\eta})}[e_2(s^2\chi\chi_{\bar{z}}(\eta_t\bar{\eta}-\bar{\eta}_t\eta)\partial_t)+e_3(s^2\chi\chi_z(\eta \bar{\eta}_t-\eta_t\bar{\eta})\partial_t)]\label{DD_15}
\end{align}
is a first order differential operator satisfying $\|\hat{\mathcal{R}}_s\|\leq \gamma_{_T}^2\kappa_1^2s^2$.\\

Finally, we define the following two terms for the second term on the right-hand side of (\ref{DD_14}):
\begin{align*}
\Theta_{s}^0=&[e_1(s\chi\eta_t\partial_z+s\chi\bar{\eta}_t\partial_{\bar{z}})\nonumber\\
&+e_2(s\chi_{\bar{z}}\bar{\eta}\partial_z-s\chi_z\bar{\eta}\partial_{\bar{z}})\nonumber\\
&+e_3(-s\chi_{\bar{z}}\eta\partial_z+s\chi_z\eta\partial_{\bar{z}})]
\end{align*}
and
\begin{align*}
\delta\mathcal{R}^{(1)}_s=\varrho_{s\chi\eta}&[e_1(s\chi\eta_t\partial_z+s\chi \bar{\eta}_t\partial_{\bar{z}})\nonumber\\
&+e_2(s\chi_{\bar{z}}\bar{\eta}\partial_z-s\chi_z\bar{\eta}\partial_{\bar{z}})\nonumber\\
&+e_3(-s\chi_{\bar{z}}\eta\partial_z+s\chi_z\eta\partial_{\bar{z}})]
\end{align*}
where $\delta\mathcal{R}^{(1)}_s$ is an $O(s^2)$-first order differential operator.
We can also simplify the first term on the right-hand side of (\ref{DD_14}) by writing $(1+\varrho_{s\chi\eta})(s\chi_z\eta+\chi_{\bar{z}}\bar{\eta})(e_1\partial_t)=s(\chi_z\eta+\chi_{\bar{z}}\bar{\eta})(e_1\partial_t)+\delta\mathcal{R}^{(2)}_s$ where $\delta\mathcal{R}^{(2)}_s$ is also an $O(s^2)$-first order differential operator. So we can rewrite (\ref{DD_14}) as the following.
\begin{align*}
D_{s\chi\eta}=(1+\varrho_{s\chi\eta})D+s(\chi_z\eta+\chi_{\bar{z}}\bar{\eta})(e_1\partial_t)+\Theta_{s}^0
+\mathcal{R}_s^0+\mathcal{H}_s^0+\mathcal{F}_s^0\nonumber
\end{align*}
where $\mathcal{R}_s^0=\hat{\mathcal{R}}_s+\delta\mathcal{R}^{(1)}_s+\delta\mathcal{R}^{(2)}_s$.\\

To prove the estimate (\ref{DD_9}) for $\mathcal{H}_s$, we notice that
the term $(d\mathcal{M})\mathcal{M}^{-1}$ involves at most the second derivative of $\chi$ and $\eta$, which can be estimated by (\ref{6_a3}), (\ref{6_a5}) and (\ref{6_a6}). So we get (\ref{DD_9}).
\end{proof}

Using the same notation introduced in this proposition, we can state the following proposition. This proposition will be used in Section 7.3.

\begin{pro}
Let $\psi\in L^2_1$ be a harmonic section. Then
\begin{align*}
\|\mathcal{R}_s(\psi)^0\|_{L^2}\leq C \gamma_{_T}^\frac{3}{2}\kappa_1^2\mathfrak{r}^2s^2
\end{align*}
for some constant $C$ depending on the $\|\psi\|_{L^2_1}$. In fact, this estimate is true for any $\psi\in L^2_1$ which can be expressed as $\psi=\sqrt{r}v(t,\theta,r)$ with $v$ being a $C^1$-bounded section.
\end{pro}
\begin{proof}
By Proposition 3.6 {\bf b}, we have $\psi=\sqrt{r}v(t,\theta,r)$ where $v$ is a $C^1$-bounded section. We write down by the definition:
\begin{align*}
\mathcal{R}_s^0=&\frac{-1}{1+s(\chi_z\eta+\chi_{\bar{z}}\bar{\eta})}[e_2(s^2\chi\chi_{\bar{z}}(\eta_t\bar{\eta}- \bar{\eta}_t\eta)\partial_t)+e_3(s^2\chi\chi_z(\eta \bar{\eta}_t-\eta_t\bar{\eta})\partial_t)]\\
&+\varrho_{s\chi\eta}[e_1(s\chi\eta_t\partial_z+s\chi \bar{\eta}_t\partial_{\bar{z}})+e_2(s\chi_{\bar{z}}\bar{\eta}\partial_z-s\chi_z\bar{\eta}\partial_{\bar{z}})+e_3(-s\chi_{\bar{z}}\eta\partial_z+s\chi_z\eta\partial_{\bar{z}})]\\
&+\varrho_{s\chi\eta}(s\chi_z\eta+\chi_{\bar{z}}\bar{\eta})(e_1\partial_t).
\end{align*}
By (\ref{DD_13}), we can bound $\Big|\dfrac{-1}{1+s(\chi_z\eta+\chi_{\bar{z}}\bar{\eta})}\Big|$ by $1+2s\gamma_{_T}\kappa_1\mathfrak{r}^{\frac{1}{2}}$. Then by using (\ref{6_a1}), (\ref{6_a2}), (\ref{6_a3}), (\ref{6_a4}), (\ref{6_a5}) and (\ref{6_a6}) we notice that every term in $\mathcal{R}_s$ can be written as the type $\sum_{i=1}^3s^2\alpha_i\beta_i\partial_i$ with $(\partial_1,\partial_2,\partial_3)=(\partial_t,\partial_x,\partial_y)$, $\|\alpha_i\|_{L^{\infty}}\leq \gamma_{_T}\kappa_1 \mathfrak{r}^{\frac{1}{2}}$ and 
\begin{align*}
\int_{r=r_0}|\beta_i|^2i_{\vec{n}}dVol(M)\leq \gamma_{_T}\kappa_1^2 \mathfrak{r}^2.
\end{align*}
So we have
\begin{align*}
\|\mathcal{R}_s(\psi)^0\|_{L^2}\leq s^2\|v\|_{C^1}\gamma_{_T}^\frac{3}{2}\kappa_1^2 \mathfrak{r}^2.
\end{align*}
\end{proof}

\subsection{Series of perturbations: Estimates}
In this section, we discuss a series of perturbations and its corresponding Dirac operator. These results will be used in the Section 7 when we construct the iteration for the proof of Theorem 1.5.\\

Let $\mathfrak{r}<\frac{R}{4}$, $T>P>1$ be fixed for a moment. We consider a sequence $\{(\chi_i,\eta_i)\}$ satisfying the following conditions:\\

1. $\chi_i:=1-\chi_{\frac{\mathfrak{r}}{T^{i+1}},\frac{\mathfrak{r}}{T^i}}$ is a cut-off function (Recall the definition (\ref{ctf})).\\

2. There exists $\kappa_2>0$ such that
\begin{align}
\|\eta_i\|_{L^2(S^1)}\leq \kappa_2\frac{\mathfrak{r}^2}{T^{2i}},\label{6_2a0}\\
\|(\eta_i)_t\|_{L^2(S^1)}\leq \kappa_2\frac{\mathfrak{r}}{T^{i}},\label{6_2a1}\\
\|{(\eta_i)}_{tt}\|_{L^2(S^1)}\leq \kappa_2,\label{6_2a2}.
\end{align}
for all $i\in \mathbb{N}$.\\

 Similar to the argument of (\ref{6_a4}), (\ref{6_a5}) and (\ref{6_a6}), we have the following results
\begin{align}
\max\big\{|(\chi_i)_z||\eta_i|,|(\chi_i)_{\bar{z}}||\eta_i|,|(\eta_i)_t|\big\}\leq \gamma_{_T}\kappa_3\frac{\mathfrak{r}^{\frac{1}{2}}}{T^{\frac{i}{2}}}\label{6_2a3},\\
\|(\chi_i)_z\eta_i \|_{L^2}, \|(\chi_i)_{\bar{z}}\eta_i\|_{L^2}\leq\gamma_{_T}\kappa_3\label{6_2a4},\\
\|(\chi_i)_{zz}\eta_i\|_{L^2},\|(\chi_i)_{z\bar{z}}\eta_i\|_{L^2}, \|(\chi_i)_{\bar{z}\bar{z}}\eta_i\|_{L^2}\leq \gamma_{_T}^2\kappa_3,\label{6_2a5}
\end{align} for some $\kappa_3=O(\kappa_2)$. We define
\begin{align}
 \eta^i=\sum_{n=0}^{i}\chi_n\eta_n.\label{DD_16}
\end{align}

As we have shown in the previous section, we define the following family of diffeomorphisms
\begin{align}
\phi^{i}_s:&M\setminus\Sigma \rightarrow M\setminus\Sigma_s;\nonumber\\
&(t,z)\mapsto (t,z+s\eta^i(t)) \mbox{ on }N_R,\label{DD_17}\\
&\phi^{i}_s(p)=p \mbox{ for all }p\in M\setminus N_R\nonumber
\end{align}
with $0\leq s\leq t_0$ for some small $t_0$ and $\Sigma_s=\{(t,s(\eta^{i}(t)))\}$. Now, we fix $s$ and use $(u,\tau)$ to denote the coordinates on $\phi^{i}_s(N_R)$.\\

The Dirac operator $D_{s\eta^{i}}$ on $M\setminus\Sigma$ will be
\begin{align*}
D_{s\eta^i}=(\phi_s^i)^{-1}_*\circ D_s=&e_1\cdot(\phi_s^i)^{-1}_*(\partial_{\tau})+e_2\cdot(\phi_s^i)^{-1}_*(\partial_u)+e_3\cdot(\phi_s^i)^{-1}_*(\partial_{\bar{u}})\\
&+\frac{1}{2}\sum_{i=1}^3e_i\sum_{k<l}(\phi_s^i)^{-1}_*(\omega_{kl}(e_i))e_ke_l.
\end{align*}

\begin{pro}
There exists $\kappa_3=O(\kappa_2)$ depending on $\kappa_2$ with the following significance.
The perturbed Dirac operator $D_{s\eta^i}$ with $\eta^i$ satisfying (\ref{6_2a0}) - (\ref{6_2a2}) can be written as follows:
\begin{align}
D_{s\eta^{i+1}}=(1+\varrho^{i+1})D_{s\eta^i}+s((\chi_{i+1})_z\eta_{i+1}&+(\chi_{i+1})_{\bar{z}}\bar{\eta}_{i+1})e_1\partial_t\label{DD_18}\\
&+\Theta^{i}_s+\mathcal{R}^{i}_{s}+\hat{\mathcal{H}}^{i}_s+\mathcal{F}^{i}_s\nonumber
\end{align}
where\\[1mm]
$\bullet$ $\Theta^{i}_s$, the $(\chi,\eta)=(\chi_{i+1},\eta_{i+1})$ version of $\Theta_s^0$, is a first order differential operator\\
$\mbox{ }\mbox{ }\mbox{ }$with order $O(s)$.\\[1mm]
$\bullet$ $\mathcal{R}^{i}_s: L^2_1\rightarrow L^2$ is an $O(s^2)$-first order differential operator supported on\\ $\mbox{ }\mbox{ }\mbox{ }$ $N_{\frac{\mathfrak{r}}{T^i}}-N_{\frac{\mathfrak{r}}{T^{i}}}$ with its operator norm bounded in the following way:
\begin{align*}
\|\mathcal{R}_s^{i}\|\leq \gamma_{_T}^2\kappa_3^2s^2.
\end{align*}
$\bullet$ $\hat{\mathcal{H}}^{i}_s$ is an $O(s^2)$-zero order differential operator. Moreover, let us denote by\\$\mbox{ }\mbox{ }\mbox{ }$ $\vec{n}=\partial_r$ the vector field defined on $N_R$, then
\begin{align}
\int_{\{r=r_0\}}|\hat{\mathcal{H}}^{i}_s|^2i_{\vec{n}}dVol(M)\leq \gamma_{_T}^4\kappa_3^4(\frac{(i+1)\mathfrak{r}}{T^{i+1}}) s^4.\label{DD_19}
\end{align}
$\mbox{ }\mbox{ }\mbox{ }$for all $r_0\leq \frac{\mathfrak{r}}{T^i}$.\\[1mm]
$\bullet$ $\mathcal{F}^{i}_s$ is an $O(s)$-zero order differential operator where
\begin{align}
\mathcal{F}^{i}_s=D(s((\chi_{i+1})_z\eta_{i+1}&+(\chi_{i+1})_{\bar{z}}\bar{\eta}_{i+1})Id)\label{DD_20}\\
&+D\Big(\left(\begin{array}{cc}
 0&si\chi_{i+1}(\eta_{i+1})_t\\
-si\chi_{i+1}(\bar{\eta}_{i+1})_t&0
\end{array}\right)\Big).\nonumber
\end{align}
\end{pro}

\begin{proof}
We can define the matrix $\mathcal{M}^i$ to be
\begin{align*}
(\phi_s^i)^{-1}_*\left( \begin{array}{c}
\partial_{\tau}\\
\partial_{u}\\
\partial_{\bar{u}}
\end{array} \right)=
\mathcal{M}^i
\left( \begin{array}{c}
\partial_{t}\\
\partial_{z}\\
\partial_{\bar{z}}
\end{array} \right).
\end{align*}

Notice that the support of $(\chi_i)_z$ and $(\chi_j)_{\bar{z}}$ are disjoint for all $i\neq j$. Therefore, we can write $\mathcal{M}^{i+1}$ as follows
\begin{align}
\mathcal{M}^{i+1}=&\frac{1}{1+s((\chi_{i+1})_z\eta_{i+1}+(\chi_{i+1})_{\bar{z}}\bar{\eta}_{i+1})}\mathcal{M}^i + \mathcal{N}^{i+1}\label{matrix}
\end{align}
where $\mathcal{N}^{i+1}$ is a $(\chi_{i+1},\eta_{i+1})$ version of $\mathcal{M}$:
\begin{align*}
\mathcal{N}^{i+1}&=\frac{1}{1+s((\chi_{i+1})_z\eta_{i+1}+(\chi_{i+1})_{\bar{z}}\bar{\eta}_{i+1})}\cdot\\[3mm]
&\left( \begin{array}{ccc}
s((\chi_{i+1})_z\eta_{i+1}+(\chi_{i+1})_{\bar{z}}\bar{\eta}_{i+1})
&0&0\\
\ \\-s\chi_{i+1}(\eta_{i+1})_t\mbox{ }\mbox{ }\mbox{ }\mbox{ }\mbox{ }\mbox{ }\mbox{ }\mbox{ }\mbox{ }\mbox{ }\mbox{ }\mbox{ }\mbox{ }\mbox{ }\mbox{ }\mbox{ }\mbox{ }\mbox{ }\\
-s^2\chi_{i+1}(\chi_{i+1})_{\bar{z}}\big[(\eta_{i+1})_t\bar{\eta}_{i+1}
&s(\chi_{i+1})_{\bar{z}}\bar{\eta}_{i+1}&-s(\chi_{i+1})_{\bar{z}}\eta_{i+1}\\
\mbox{ }\mbox{ }\mbox{ }\mbox{ }\mbox{ }\mbox{ }\mbox{ }\mbox{ }\mbox{ }\mbox{ }\mbox{ }\mbox{ }\mbox{ }\mbox{ }\mbox{ }\mbox{ }\mbox{ }\mbox{ }\mbox{ }\mbox{ }\mbox{ }\mbox{ }\mbox{ }\mbox{ }\mbox{ }-(\bar{\eta}_{i+1})_t\eta_{i+1}\big]
\ \\
\ \\-s\chi_{i+1}(\bar{\eta}_{i+1})_t\mbox{ }\mbox{ }\mbox{ }\mbox{ }\mbox{ }\mbox{ }\mbox{ }\mbox{ }\mbox{ }\mbox{ }\mbox{ }\mbox{ }\mbox{ }\mbox{ }\mbox{ }\mbox{ }\mbox{ }\mbox{ }\\
-s^2\chi_{i+1}(\chi_{i+1})_z\big[\eta_{i+1}(\bar{\eta}_{i+1})_t&-s(\chi_{i+1})_z\bar{\eta}_{i+1}&s(\chi_{i+1})_z\eta_{i+1}\\
\mbox{ }\mbox{ }\mbox{ }\mbox{ }\mbox{ }\mbox{ }\mbox{ }\mbox{ }\mbox{ }\mbox{ }\mbox{ }\mbox{ }\mbox{ }\mbox{ }\mbox{ }\mbox{ }\mbox{ }\mbox{ }\mbox{ }\mbox{ }\mbox{ }\mbox{ }\mbox{ }\mbox{ }\mbox{ }-(\eta_{i+1})_t\bar{\eta}_{i+1}\big]
\end{array} \right).\\
\end{align*}

Let us define 
\begin{align*}
\frac{1}{1+s((\chi_{i})_z\eta_{i}+(\chi_{i})_{\bar{z}}\bar{\eta}_{i})}=1+\varrho^{i}.
\end{align*}
Define $\Theta^{i}_{s}$ and $\mathcal{R}^{i}_s$ to be the $(\chi_{i+1},\eta_{i+1})$ version of  $\Theta_{s}^0$ and $\mathcal{R}_s^0$. Then we have
\begin{align*}
D_{s\eta^{i+1}}=(1+\varrho^{i+1})(D_{s\eta^i}-\mathbb{A}^i_s)+s((\chi_{i+1})_z\eta_{i+1}&+(\chi_{i+1})_{\bar{z}}\bar{\eta}_{i+1})e_1\partial_t\\
&+\Theta^{i}_s+\mathcal{R}^{i}_{s}+\mathbb{A}^{i+1}_s\\
\mbox{ }=(1+\varrho^{i+1})D_{s\eta^i}+s((\chi_{i+1})_z\eta_{i+1}&+(\chi_{i+1})_{\bar{z}}\bar{\eta}_{i+1})e_1\partial_t\\
+\Theta^{i}_s+\mathcal{R}^{i}_{s}&+[\mathbb{A}^{i+1}_s-(1+\varrho^{i+1})\mathbb{A}^i_s]
\end{align*}
where $\mathbb{A}^{i+1}_s=\sum_{j=1}^3e_j\sum_{k<l}(\omega^{(i+1)}_{kl}(e_j))e_ke_l$ with $\omega^{(i+1)}$ being the pull back of Levi-Civita connection $(\phi_s^{i+1})^{-1}_*(\omega)$. Using these conventions, we have
\begin{align*}
|\mathbb{A}^{i+1}_s-(1+\varrho^{i+1})\mathbb{A}^i_s
|=|(d\mathcal{M}^{i+1})(\mathcal{M}^{i+1})^{-1}
-(d\mathcal{M}^{i})(\mathcal{M}^{i})^{-1}
-\varrho^{i+1}(d\mathcal{M}^{i})(\mathcal{M}^{i})^{-1}|.
\end{align*}

Now, by (\ref{DD_20}) and (\ref{matrix}), we can see that $\mathcal{F}^{i}_s$ is the $O(s)$ order term of $(d\mathcal{M}^{i+1})(\mathcal{M}^{i+1})^{-1}
-(d\mathcal{M}^{i})(\mathcal{M}^{i}))^{-1}$. Therefore, by using (\ref{6_2a0}), (\ref{6_2a1}) and (\ref{6_2a3}), we have
\begin{align*}
\int_{r=r_0}|(d\mathcal{M}^{i+1})(\mathcal{M}^{i+1})^{-1}
-(d\mathcal{M}^{i})(\mathcal{M}^{i})^{-1}-\mathcal{F}^{i+1}_s
|^2i_{\vec{n}}dVol(M)
\leq C\gamma_{_T}^4(\frac{\mathfrak{r}}{T^{i+1}})\kappa_0^4s^4
\end{align*}
for some universal constant $C$. Therefore, we can choose $\kappa_3=O(\kappa_2)$ large enough such that the right-hand side of this equation is smaller than $\frac{1}{4}\gamma_{_T}^4(\frac{\mathfrak{r}}{T^{i+1}})\kappa_3^4s^4$.\\

Meanwhile,
\begin{align}
\int_{r=r_0}|(d\mathcal{M}^{j+1})(\mathcal{M}^{j+1})^{-1}
-(d\mathcal{M}^{j})(\mathcal{M}^{j})^{-1}
|^2i_{\vec{n}}dVol(M)
\leq C\gamma_{_T}^2\kappa_3^2s^2\label{DD_21}
\end{align}
for all $j$, so we have
\begin{align*}
\int_{r=r_0}|(d\mathcal{M}^{j})(\mathcal{M}^{j})^{-1}|^2i_{\vec{n}}dVol(M)
&\leq \gamma_{_T}^2(i+1)\kappa_3^2s^2.
\end{align*}
Now recall that $|\varrho^{i+1}|\leq \gamma_{_T}s\kappa_3(\frac{\mathfrak{r}}{T^{i+1}})^{\frac{1}{2}}$. So we have
\begin{align}
\int_{r=r_0}|\varrho^{i+1}(d\mathcal{M}^{j})(\mathcal{M}^{j})^{-1}|^2i_{\vec{n}}dVol(M)\leq \gamma_{_T}^4(i+1)(\frac{\mathfrak{r}}{T^{i+1}})\kappa_3^4s^4.\label{DD_22}
\end{align}
Therefore, by taking
\begin{align*}
\hat{\mathcal{H}}^{i}_s=\mathbb{A}^{i+1}_s-(1+\varrho^{i+1})\mathbb{A}^i_s-\mathcal{F}^{i+1}_s,
\end{align*}
we prove this proposition.
\end{proof}

 We have a similar version of Proposition 4.6 as follows.
\begin{pro}
Let $\psi\in L^2_1$ be a harmonic section. Then
\begin{align*}
\|\mathcal{R}^{i+1}_s(\psi)\|_{L^2}\leq C \gamma_{_T}^\frac{3}{2}\kappa_3^2(\frac{\mathfrak{r}}{T^{i+1}})^2s^2. 
\end{align*}
for some constant $C$ depending on the $\|\psi\|_{L^2_1}$. In fact, this estimate is true for any $\psi\in L^2_1$ which can be expressed as $\psi=\sqrt{r}v(t,\theta,r)$ where $v$ is a $C^1$-bounded section.
\end{pro}

\subsection{Variational formula for perturbed Dirac operators} In Section 4.1, we proved that there exists a unique minimizer of $E_{\mathfrak{f}}$ in $L^2_1\cap \ker(D|_{L^2_1})^{\perp}$ when the metric $g$ is Euclidean on a tubular neighborhood of $\Sigma$. The argument in Section 4.1 works not only for $D$ but also for the perturbed Dirac operator $D_{s\chi\eta}$, $D_{s\eta^j}$ appearing in Section 4.3 and 4.4. However, using the variational method to find the solution $D_{s\eta^j}\mathfrak{u}_s=\mathfrak{f}$ wouldn't give us enough information about $\mathfrak{u}_s$ changing by varying $s$. Therefore, we need to prove the following proposition (Proposition 4.9) to clarify this part.\\

In addition, let $\mathfrak{v}$ be the minimizer of $E_{\mathfrak{f}}$ in $L^2_1 \cap \ker(D|_{L^2_1})^{\perp}$. Then $\mathfrak{u}:=D\mathfrak{v}$ will satisfy the equation $D\mathfrak{u}=\mathfrak{f}$ only if $\mathfrak{f}\in L^2_{-1}\cap \ker(D|_{L^2_1})^{\perp}$. Namely, Corollary 4.4 gives us the following statement:
For any $\mathfrak{f}\in L^2_{-1}$, there exists $\mathfrak{u}\in D(L^2_1 \cap \ker(D|_{L^2_1})^{\perp})$ such that
\begin{align*}
D \mathfrak{u}=\mathfrak{f}+\mbox{ some elements in }\ker(D|_{L^2_1}).
\end{align*}
We will use $mod(\ker(D|_{L^2_1}))$ to denote "$\mbox{some elements in }\ker(D|_{L^2_1})$" in the rest of our paper.

\begin{pro} Fix $j\in\mathbb{N}$ and $\mathfrak{f}\in L^2_{-1}$. Suppose that $\mathfrak{u}_0\in L^2$ satisfies
\begin{align*}
D\mathfrak{u}_0=\mathfrak{f}\mbox{ }mod(\ker(D|_{L^2_1})),
\end{align*}
then there exist $\mathfrak{u}=\mathfrak{u}_0+\mathfrak{u}^s$ and $t_0>0$ such that
\begin{align*}
D_{s\eta^j}\mathfrak{u}=\mathfrak{f}\mbox{ }mod(\ker(D|_{L^2_1}))
\end{align*}
and $\|\mathfrak{u}^s\|_{L^2}\leq C(\|\mathfrak{u}_0\|_{L^2}+\|\mathfrak{f}\|_{L^2_{-1}})s$ for $s\in [0,t_0]$ and $C$ being a universal constant $C$. Furthermore, the existence of $\mathfrak{u}_0$ will be given by Corollary 4.4 or Proposition 6.2 which appears later.
\end{pro}

\begin{proof} We can assume $\ker(D|_{L^2_1})=0$ for a moment. The general case follows the same argument as below.\\

Suppose $D_{s\eta^j}$ is the perturbed Dirac operator and $\mathfrak{f}\in L^2_{-1}$. We want to solve $\mathfrak{u}\in L^2$ satisfying
\begin{align*}
D_{s\eta^j}\mathfrak{u}=\mathfrak{f}.
\end{align*} 

We solve this equation iteratively. First, we know that the perturbed Dirac operator $D_{s\eta^j}$ can be written as $D+\delta^j_s$ where $\delta^j_s:L^2\rightarrow L^2_{-1}$ is a first order differential operator with its operator norm $\|\delta_s^j\|\leq Cs$ for some $C>0$. Meanwhile, by Corollary 4.4, there exists $\mathfrak{u}_0\in L^2$ such that
\begin{align*}
D\mathfrak{u}_0=\mathfrak{f}.
\end{align*}
So we have
\begin{align*}
D_{s\eta^j}\mathfrak{u}_0=\mathfrak{f}-\delta_s^j(\mathfrak{u}_0).
\end{align*}
Since $\|\mathfrak{u}_0\|_{L^2}\leq C\|\mathfrak{f}\|_{L^2_{-1}}$, we have $\|\delta^j_s(\mathfrak{u}_0)\|_{L^2_{-1}}\leq Cs\|\mathfrak{f}\|_{L^2_{-1}}$. By taking $s$ small enough, we have
$\|\delta^j_s(\mathfrak{u}_0)\|_{L^2_{-1}}\leq \frac{1}{2}\|\mathfrak{f}\|_{L^2_{-1}}$.\\

Now we solve $\mathfrak{v}_1\in L^2$ such that
\begin{align*}
D\mathfrak{v}_1=\delta^j_s(\mathfrak{u}_0)
\end{align*}
by using Corollary 4.4. Then we have $D_{s\eta^j}(\mathfrak{u}_0+\mathfrak{v}_1)=\mathfrak{f}+\delta^j_s(\mathfrak{v}_1)$ where
$\|\delta^j_s(\mathfrak{v}_1)\|_{L^2_{-1}}\leq \frac{1}{2}\|\delta^j_s(\mathfrak{u}_0)\|\leq \frac{1}{4}\|\mathfrak{f}\|_{L^2_{-1}}$.\\

We call $\delta^j_s(\mathfrak{u}_0)=\mathfrak{z}_0$, $-\delta^j_s(\mathfrak{v}_1)=\mathfrak{z}_1$ and $\mathfrak{u}_0+\mathfrak{v}_1=\mathfrak{u}_1$. Suppose that we have $(\mathfrak{u}_i,\mathfrak{z}_i)$ satisfying
\begin{align*}
D_{s\eta^j}\mathfrak{u}_i=\mathfrak{f}-\mathfrak{z}_i
\end{align*} 
with $\|\mathfrak{z}_i\|_{L^2_{-1}}\leq \frac{1}{2^{i+1}}\|\mathfrak{f}\|_{L^2_{-1}}$ for some $i\in\mathbb{N}$, then we can solve $\mathfrak{v}_{i+1}\in L^2$ which satisfies
\begin{align*}
D\mathfrak{v}_{i+1}=\mathfrak{z}_i
\end{align*}
by Corollary 4.4. So we have
\begin{align*}
D_{s\eta^j}(\mathfrak{u}_i+\mathfrak{v}_{i+1})=\mathfrak{f}+\delta^j_s(\mathfrak{v}_{i+1})
\end{align*}
where $\|\delta^j_s(\mathfrak{v}_{i+1})\|_{L^2_{-1}}\leq \frac{1}{2}\|\mathfrak{z}_{i}\|\leq \frac{1}{2^{i+2}}\|\mathfrak{f}\|_{L^2_{-1}}$. By taking $\mathfrak{u}_i+\mathfrak{v}_{i+1}=\mathfrak{u}_{i+1}$ and $-\delta^j_s(\mathfrak{v}_{i+1})=\mathfrak{z}_{i+1}$, we can repeat this argument inductively.\\

Finally, by taking the limit $i\rightarrow \infty$, then we have $\mathfrak{u}_{i+1}\rightarrow \mathfrak{u}$ in $L^2$-sense which satisfies
\begin{align*}
D_{s\eta^j}(\mathfrak{u})=\mathfrak{f}.
\end{align*}
Moreover, since $\mathfrak{u}-\mathfrak{u}_0=\sum_{i=1}^{\infty}\mathfrak{v}_i$ and $D\mathfrak{v}_i=(-1)^{(i-1)}\delta^j_s(\mathfrak{v}_{i-1})$, we have $\sum_i^{\infty}\mathfrak{v}_i$ is an $O(s)$-order $L^2$ section. We call $\sum_i^{\infty}\mathfrak{v}_i=\mathfrak{u}^s$.\\

Therefore, $\mathfrak{u}_0+\mathfrak{u}^s$ satisfies
\begin{align}
D_{s\eta^j}(\mathfrak{u}_0+\mathfrak{u}^s)=\mathfrak{f}.\label{6_4_1}
\end{align}
\end{proof}

\begin{remark}
In our proof, since we can always write $\delta^j_s=\sum_{i=1}^{\infty} s^i\delta^j_i$ where the operator norm of $\delta^j_{i}$ is bounded uniformly, $\mathfrak{u}$ can be written as $\sum_{i=0}^{\infty}s^i\mathfrak{u}^{(i)}$.\\ $\|\sum_{i=m}^{\infty}s^i\mathfrak{u}^{(i)}\|_2\rightarrow 0$ as $m\rightarrow \infty$.
\end{remark}

\section{$\Sigma$ with a non-Euclidean neighborhood} We now try to derive same results as we did in previous section without assuming that $\Sigma$ has a product type metric on the tubular neighborhood. The discussion in this section is necessary because even if we start with a Euclidean metric, it will change to a non-Euclidean one when we perturb the curve $\Sigma$ and use a diffeomorphism to identify it with the original curve.\\

There are two parts in this section. In Section 5.1, we prove an important result in Lemma 5.1. This lemma shows that the regularity properties proved in Section 3 for the leading terms of $\mathbb{Z}/2$-harmonic spinors in $L^2_1$ still holds when the metric is non-Euclidean near $\Sigma$; In Section 5.2, we formulate the way Dirac operator changes when we deform $\Sigma$. This part generalizes the results in Section 4.3 by taking away the assumption that $\Sigma$ has a Euclidean metric on the tubular neighborhood. The results in this section will be used in Section 7 (see also Section 9.1).

\subsection{Asymptotic behavior of the $L^2_1$-harmonic section} Let $g$ be a smooth metric and $\Sigma\subset M$ be a $C^1$ curve embedded in $M$. We use the exponential map to send elements in the normal bundle $\{v\in \nu_{\Sigma}||v|\leq R\}$ to the tubular neighborhood of $\Sigma$ in $M$. We can parametrize this neighborhood by a cylindrical coordinates $(t,r,\theta)$ and $g=dt^2+dr^2+r^2d\theta^2+O(r^2)$ on $N_{2R}$ for some $R>0$. Let $\mathbf{\mathcal{S}}\otimes\mathcal{I}$ be the twisted spinor bundle defined on $M\setminus\Sigma$ with respect to $g$.\\

Now, we define $g_{E}:=(\chi_{R,2R})g+(1-\chi_{R,2R})(dt^2+dr^2+r^2d\theta^2)$ and let $\mathbf{\mathcal{S}}_{E}\otimes\mathcal{I}$ be the twisted spinor bundle defined on $M\setminus\Sigma$ with respect to $g_E$. Then we can see that they are isomorphic (recall that they are classified by $H^1(M;\mathbb{Z}_2)$). So we can regard them as the same complex vector space with different Clifford multiplications and Dirac operators. Denote by $D$ the Dirac operator with respact to $g$ and $D_{E}$ the Dirac operator with respect to $g_E$.\\

The main result of this section is Lemma 5.1 below.
\begin{lemma}
Let $\mathfrak{v}\in L^2_1(N_R;\mathcal{S}\otimes\mathcal{I})$ satisfies $\|\mathfrak{v}\|_{C^{0,\frac{1}{2}}(M\setminus\Sigma)}<\infty$ and $D(\mathfrak{v})=0$. Then there exists $\mathfrak{v}^*\in L^2_1(N_R;\mathcal{S}\otimes\mathcal{I})$ and $w^{\pm}(t)\in L^2_2(S^1;\mathbb{C})$ such that $D_{E}\mathfrak{v}^*=0$ and
\begin{align}
\Bigg\|\mathfrak{v}-\mathfrak{v}^*-\left(\begin{array}{c}
w^+(t)\sqrt{z}\\
w^-(t)\sqrt{\bar{z}}
\end{array}
\right)(1-\chi_{\frac{R}{2},R})
\Bigg\|_{L^2(N_R)}\leq O(r^\frac{5}{2}).\label{EE_2}
\end{align}
and
\begin{align}
\Bigg\|\mathfrak{v}-\mathfrak{v}^*-\left(\begin{array}{c}
w^+(t)\sqrt{z}\\
w^-(t)\sqrt{\bar{z}}
\end{array}
\right)(1-\chi_{\frac{R}{2},R})
\Bigg\|_{L^2_1(N_R)}\leq O(r^\frac{3}{2}).\label{EE_2fix}
\end{align}
\end{lemma}
\begin{proof} We divide our proof into two parts.\\
{\bf Step 1}. 
Here we set up the strategy of the proof. First, it is clear that we can write $D=D_{E}+O(r^2)\mathcal{L}_1+O(r)\mathcal{L}_0$ where $\mathcal{L}_1$ is a bounded first order differential operator and $\mathcal{L}_0$ is a zero order operator, composed by Clifford multiplications.\\

Second, the argument in Lemma 2.6 still works for elements in $L^2_1(N_r;\mathcal{S}\otimes\mathcal{I})$. So the right hand side of the equation
\begin{align*}
D_{E}\mathfrak{v}=O(r^2)\mathcal{L}_1(\mathfrak{v})+O(r)\mathcal{L}_0(\mathfrak{v})
\end{align*}
will satisfy $\|O(r^2)\mathcal{L}_1(\mathfrak{v})+O(r)\mathcal{L}_0(\mathfrak{v})\|_{L^2(N_a)}\leq O(a^2)$ for all $a<R$.
Namely, we have
\begin{align}
D_{E}\mathfrak{v}=f\label{ODEs}
\end{align}
for some $f$ satisfying 
\begin{align}
\|f\|_{L^2(N_a)}\leq O(a^2)\label{est_f}
\end{align}
for all $a<R$.\\

By the standard theory of ODEs, the solution of (\ref{ODEs}) has the form
\begin{align}
\mathfrak{v}=\mathfrak{w}+\mathfrak{v}^*\label{STODEs}
\end{align}
where $D_E\mathfrak{w}=f$ is a particular solution and $\mathfrak{v}^*$ satisfies $D_E\mathfrak{v}^*=0$. In particular, if we can prove that $\mathfrak{w}$ has the leading term
\begin{align}
\left(\begin{array}{c}
w^+(t)\sqrt{z}\\
w^-(t)\sqrt{\bar{z}}
\end{array}
\right)
\end{align}
satisfying $w^{\pm}(t)\in L^2_2(S^1;\mathbb{C})$ and the remainder term
\begin{align}
\mathfrak{v}-\mathfrak{v}^*-\left(\begin{array}{c}
w^+(t)\sqrt{z}\\
w^-(t)\sqrt{\bar{z}}
\end{array}
\right)(1-\chi_{\frac{R}{2},R}):=\mathfrak{w}_\mathfrak{R}
\end{align}
satisfying (\ref{EE_2}), then we prove Lemma 5.1.\\
\ \\
{\bf Step 2}. Here we study the solution $\mathfrak{w}\in L^2_1$. We write down the Fourier expression of $\mathfrak{w}$ on $N_R$ as we have done in Section 3.
\begin{align*}
\mathfrak{w}(t,r,\theta)=\sum_{l,k}e^{ilt}\left( \begin{array}{c}
e^{i(k-\frac{1}{2})\theta}W^+_{k,l}\\
e^{i(k+\frac{1}{2})\theta}W^-_{k,l}
\end{array} \right).
\end{align*}
The equation $D_E\mathfrak{w}=f$ will give us
\begin{align*}
\frac{d}{dr}W_{k,l}^-+\frac{(k+\frac{1}{2})}{r} W_{k,l}^-+ 2lW^+_{k,l}+P^+_{k,l}(f)=0;\\
\frac{d}{dr}W_{k,l}^+-\frac{(k+\frac{1}{2})}{r} W_{k,l}^++2lW^-_{k,l}+P^-_{k,l}(f)=0
\end{align*}
where $P^+$ is the projection mapping to the Fourier modes of the first component and $P^-$ is the projection to the modes of the second component.\\

Therefore, we have
\begin{align}
&\frac{d}{dr}(r^{k+\frac{1}{2}}W^-_{k,l})=-r^{k+\frac{1}{2}}2l W^+_{k,l}-r^{k+\frac{1}{2}}P^+_{k,l}(f);\label{EE_4}\\
&\frac{d}{dr}(r^{-k-\frac{1}{2}}W^+_{k,l})=-r^{-k-\frac{1}{2}}2l W^-_{k,l}-r^{-k-\frac{1}{2}}P^-_{k,l}(f)\label{EE_5}
\end{align}
for all $k,l$.\\

The integral of (\ref{EE_4}) shows that there exists a double sequence $n_{k,l}>0$ satisfying $\sum_{k,l\in\mathbb{Z},l\neq -1}n_{k,l}^2<\infty$ such that
\begin{align}
|b^{k+\frac{1}{2}} W_{k,l}^-(b)-a^{k+\frac{1}{2}}W_{k,l}^-(a)|&\leq \int_a^br^{k+\frac{1}{2}}(|W^+_{k,l}|+|P^+_{k,l}(f)|)\label{EE_6}\\
&\leq (b^{2k+2}-a^{2k+2})^{\frac{1}{2}}(\int_a^b O(1))^{\frac{1}{2}}\nonumber\\
&\leq n_{k,l}(b^{2k+2}-a^{2k+2})^{\frac{1}{2}}(b-a)^{\frac{1}{2}}\nonumber
\end{align}
for any $b>a>0$ and $k\neq -1$. Meanwhile, since we have $|W^+_{k,l}|$ $|P^+_{k,l}(f)|$ are $o(1)$, we have
\begin{align}
|b^{-\frac{1}{2}}W_{-1,l}^-(b)-a^{-\frac{1}{2}}W_{-1,l}^-(a)|&\leq \int_a^br^{-\frac{1}{2}}(|W^+_{-1,l}|+|P^+_{-1,l}(f)|)\leq C(b-a)^{\frac{1}{2}}\label{EE_6*}
\end{align}
for some $C>0$, $b>a>0$.\\

 Suppose $k\geq 0$, (\ref{EE_6}) implies
\begin{align*}
\lim_{r\rightarrow 0}r^{k+\frac{1}{2}}W^-_{k,l}(r)=c
\end{align*}
for some $c\in \mathbb{C}$. $|W^-_{k,l}|>\frac{|c|}{2}r^{-k-\frac{1}{2}}\geq \frac{|c|}{2}r^{-\frac{1}{2}}$ which is contradictory to Lemma 2.6 if $c\neq 0$. So we have $\lim_{r\rightarrow 0}r^{k+\frac{1}{2}} W^-_{k,l}(r)=0$. By taking $a\rightarrow 0$ in (\ref{EE_6}), we have
\begin{align*}
C_1b^{k+\frac{1}{2}}|W^-_{k,l}|(b)\leq b^{k+\frac{3}{2}}. 
\end{align*}
So we have
\begin{align}
|W^-_{k,l}|(r)\leq n_{k,l}r\label{EE_7}
\end{align}
for all $k\geq 0$ with $\sum_{k,l}|n_{k,l}|^2<\infty$. Similarly, by using the same argument, we can also prove that
\begin{align}
|W^+_{k,l}|(r)\leq n_{k,l}r.\label{EE_8}
\end{align}
for all $k\leq 0$.\\

For the case $k=-1$, by (\ref{EE_6*}) we have $\lim_{r\rightarrow 0}r^{-\frac{1}{2}}W^-_{k,l}=c$ for some $c\in\mathbb{C}$. So we have
\begin{align}
W^-_{-1,l}(r)= w_{-1,l}^-r^{\frac{1}{2}}+o(r^{\frac{1}{2}}).\label{EE_2a}
\end{align}
Similarly, we have
\begin{align}
W^+_{1,l}(r)= w_{1,l}^+r^{\frac{1}{2}}+o(r^{\frac{1}{2}}).\label{EE_3a}
\end{align}

For the case $k<-1$, if we have
\begin{align*}
\limsup_{r\rightarrow 0}|r^{k+\frac{1}{2}}W_{k,l}^-|(r)=c<\infty
\end{align*}
then $|W^-_{k,l}|(r)\leq cr^{-k-\frac{1}{2}}\leq cr^{\frac{3}{2}}$. On the other hand, if we have
\begin{align*}
\limsup_{r\rightarrow 0}|r^{k+\frac{1}{2}}W_{k,l}^-|(r)=\infty,
\end{align*}
$k<-2$ by (\ref{EE_6}). Moreover, (\ref{EE_6}) implies that
\begin{align*}
|b^{k+\frac{1}{2}}W^-_{k,l}(b)-a^{k+\frac{1}{2}}W^-_{k,l}(a)|\leq C a^{k+\frac{5}{2}}.
\end{align*}
So
\begin{align*}
\Bigg|\frac{b^{k+\frac{1}{2}}}{a^{k+\frac{5}{2}}}W^-_{k,l}(b)-a^{-2}W^-_{k,l}(a)\Bigg|\leq n_{k,l}O(1).
\end{align*}
Therefore, we have
\begin{align*}
\limsup_{a\rightarrow 0}|a^{-2}W^-_{k,l}(a)|\leq n_{k,l}O(1)
\end{align*}
which implies
\begin{align*}
|W^-_{k,l}|(r)\leq n_{k,l}r^2.
\end{align*}
So we can conclude that
\begin{align}
|W^-_{k,l}|(r)\leq n_{k,l}r^{\frac{3}{2}}\label{EE_4a}
\end{align}
for all $k<-1$.\\

We summarize (\ref{EE_7}), (\ref{EE_2a}), (\ref{EE_3a}) and (\ref{EE_8}): There exists a double sequence $n_{k,l}>0$ with $\sum_{k,l\in\mathbb{Z}}n_{k,l}^2<\infty$ such that 
\begin{align}
W^-_{k,l}(r)=\left\{\begin{array}{cc}
n_{k,l}r&\mbox{ when }k\neq -1,\\
w^-_{-1,l}r^{\frac{1}{2}}+o(r^{\frac{1}{2}})&\mbox{ when }k=-1
\end{array}\right.\label{EEE_x1}
\end{align}
and
\begin{align}
W^+_{k,l}(r)=\left\{\begin{array}{cc}
n_{k,l}r&\mbox{ when }k\neq 1,\\
w^+_{1,l}r^{\frac{1}{2}}+o(r^{\frac{1}{2}})&\mbox{ when }k=1
\end{array}\right..\label{EEE_x2}
\end{align}
Now, by using (\ref{EE_6}) and (\ref{EE_7}) again, we have
\begin{align}
W^-_{k,l}(r)=\left\{\begin{array}{cc}
n_{k,l}r^{\frac{3}{2}}&\mbox{ when }k\neq -1,\\
w^-_{-1,l}r^{\frac{1}{2}}+n_{1,l}(r^{\frac{3}{2}})&\mbox{ when }k=-1
\end{array}\right.\label{EEE_y1}
\end{align}
and
\begin{align}
W^+_{k,l}(r)=\left\{\begin{array}{cc}
n_{k,l}r^{\frac{3}{2}}&\mbox{ when }k\neq 1,\\
w^+_{1,l}r^{\frac{1}{2}}+n_{-1,l}(r^{\frac{3}{2}})&\mbox{ when }k=1
\end{array}\right..\label{EEE_y2}
\end{align}
\ \\
{\bf Step 3}. Here we prove $\{w^+_{-1,l}\}$, $\{w^-_{1,l}\}\in \ell^2_2$. By (\ref{EEE_y1}) and (\ref{EEE_y2}), this implies (5.1).\\

First of all, by Definition 1.4, we have
\begin{align}
[\partial_t, D]=O(r^2)\mathcal{L}^*_1+O(r)\mathcal{L}^*_0
\end{align}
for some continuous first order differential operator $\mathcal{L}^*_1$ and zero order differential operator $\mathcal{L}^*_0$. So we have
\begin{align}
D\mathfrak{v}_t=O(r^2)\mathcal{L}^*_1(\mathfrak{v})+O(r)\mathcal{L}^*_0(\mathfrak{v})\in L^2\label{eq_fix}
\end{align}
By (\ref{EEE_y1}), (\ref{EEE_y2}), (\ref{eq_fix}) and Lichnerowicz-Weitzenb\"ock formula, $\mathfrak{v}_t$ is in $L^2_1$. Meanwhile, since $\mathfrak{v}^*\in L^2_1$ satisfies $D_E\mathfrak{v}^*=0$, we also have $\mathfrak{v}^*_t\in L^2_1$. Therefore, by (\ref{STODEs}), we have $\mathfrak{w}_t$ is in $L^2_1$ and
\begin{align}
D_E\mathfrak{w}_t=&\big(O(r^2)\mathcal{L}_1(\mathfrak{v}_t)+O(r)\mathcal{L}_0(\mathfrak{v}_t))+\big(O(r^2)\mathcal{L}_1^*(\mathfrak{v})+O(r)\mathcal{L}_0^*(\mathfrak{v}))\\
:=&f^*.\nonumber
\end{align}
One can check directly that
\begin{align}
\|f^*\|_{L^2(N_r\setminus N_{\frac{r}{2}})}\leq C r^2(\|\mathfrak{v}_t\|_{L^2_1(N_{R})}+\|\mathfrak{v}\|_{L^2_1(N_R)})\label{fix1}
\end{align}
for all $r<R$. Consider the following partition of unity for $\mathbb{R}^+$:
\begin{align}
\chi_n:=\chi_{\frac{R}{2^{2n+3}},\frac{R}{2^{2n+2}}}(1-\chi_{\frac{R}{2^{2n+1}},\frac{R}{2^{2n}}})
\end{align}
which satisfies $\sum_{n\in\mathbb{N}}\chi_n=1$. By taking $f^*_n=\chi_n f^*$, a particular solution of the equation $D_E \mathfrak{q}_n=f^*_n$ can be solve by using the method of variation of constants in section 2.2.2 of \cite{F} and it will satisfies
\begin{align}
\|\partial_t\mathfrak{q}_n\|_{L^2}\leq C\|f^*_n\|_{L^2}\label{fix2}
\end{align}
by using the argument in Lemma 2.2.3 of \cite{F}. Now, since $f_n^*=0$ on $N_{\frac{R}{2^{2n+3}}}$, by using Proposition 3.9, (\ref{fix1}) and (\ref{fix2}), we have
\begin{align}
\Big(\frac{R}{2^{2n+3}}\Big)^{\frac{3}{2}}\|\partial_t q^{\pm}_n(t)\|_{L^2(S^1;\mathbb{C})}&\leq C\|\partial_t\mathfrak{q}_n\|_{L^2}
\leq C\|f^*_n\|_{L^2}\label{fix_estimate}\\
&\leq C \Big(\frac{R}{2^{2n}}\Big)^{2}(\|\mathfrak{v}_t\|_{L^2_1(N_R)}+\|\mathfrak{v}\|_{L^2_1(N_R)})\nonumber
\end{align}
where $q^{\pm}_n(t)$ is the leading term of $\mathfrak{q}_n$. (\ref{fix_estimate}) implies that
\begin{align}
\|q^{\pm}_n(t)\|_{L^2_1(S^1;\mathbb{C})}\leq C\sqrt{R}\frac{1}{2^n}(\|\mathfrak{v}_t\|_{L^2_1(N_R)}+\|\mathfrak{v}\|_{L^2_1(N_R)})
\end{align}
So we have $\sum_{n\in\mathbb{N}} \mathfrak{q}_n$ satisfies the equation
\begin{align}
D_E(\sum_{n\in\mathbb{N}} \mathfrak{q}_n)=\sum_{n\in\mathbb{N}} f^*_n=f^*
\end{align}
with its leading terms in $L^2_1(S^1;\mathbb{C})$. Since $\mathfrak{w}_t$ and $\sum_{n\in\mathbb{N}} \mathfrak{q}_n$ satisfy the same equation, we have
\begin{align}
\mathfrak{w}_t=\sum_{n\in\mathbb{N}} \mathfrak{q}_n+\mathfrak{p}
\end{align}
for some $\mathfrak{p}\in L^2_1$ satisfies $D_E\mathfrak{p}=0$. Proposition 3.9 implies the leading terms of $\mathfrak{p}$ is smooth. So the leading terms of $\mathfrak{w}_t$ is in $L^2_1(S^1;\mathbb{C})$, which implies the leading terms of $\mathfrak{w}$ is in $L^2_2(S^1;\mathbb{C})$, i.e., $\{w^+_{-1,l}\}$, $\{w^-_{1,l}\}\in \ell^2_2$.\\
\ \\
{\bf Step 4}. We should prove (\ref{EE_2fix}) now. Let $\chi=(1-\chi_{\frac{R}{2},R})$. We have
\begin{align}
\|f\|_{L^2}^2=\|D_E\mathfrak{w}\|_{L^2}^2=\Bigg\| D_E \Bigg[\left(\begin{array}{c}
w^+(t)\sqrt{z}\\
w^-(t)\sqrt{\bar{z}}
\end{array}\right)
\chi+\mathfrak{w}_{\mathfrak{R}}\Bigg]  \Bigg\|_{L^2}^2.
\end{align}
By (\ref{est_f}) and the fact that $w^{\pm}(t)\in L^2_2$, we have
\begin{align}
\|D_E\mathfrak{w}_{\mathfrak{R}}\|_{L^2}^2&\leq \|f\|_{L^2}^2+\Bigg\| D_E\Bigg[\left(\begin{array}{c}
w^+(t)\sqrt{z}\\
w^-(t)\sqrt{\bar{z}}
\end{array}\right)\chi\Bigg] \Bigg\|_{L^2}^2\nonumber\\
&\leq\|f\|_{L^2}^2+ \Bigg\| \left(\begin{array}{c}
\frac{d}{dt}w^+(t)\sqrt{z}\\
\frac{d}{dt}w^-(t)\sqrt{\bar{z}}
\end{array}\right)\chi \Bigg\|_{L^2}^2+C\Bigg\| \left(\begin{array}{c}
w^+(t)\sqrt{z}\\
w^-(t)\sqrt{\bar{z}}
\end{array}\right) \Bigg\|_{L^2}^2=O(r^{\frac{3}{2}}).
\end{align}
By Lichnerowicz-Weitzenb\"ock formula and (\ref{EE_2}), we have
\begin{align}
\|\mathfrak{w}_{\mathfrak{R}}\|_{L^2_1}^2\leq \|D_E\mathfrak{w}_{\mathfrak{R}}\|_{L^2}^2+C\|\mathfrak{w}_{\mathfrak{R}}\|_{L^2}^2
\leq O(r^{\frac{3}{2}}).
\end{align}
This implies (\ref{EE_2fix}).
\end{proof}
\begin{remark}
By the same token, we can also show that elements in $\ker(D|_{L^2})$ have a similar decomposition. To be more precisely, for any $\mathfrak{u}\in \ker(D|_{L^2})^0$, there is a decomposition 
\begin{align}
\mathfrak{u}=\left(\begin{array}{c}
u^+(t)\frac{1}{\sqrt{z}}\\
u^-(t)\frac{1}{\sqrt{\bar{z}}}
\end{array}\right)+\mathfrak{u}_{\mathfrak{R}}
\end{align}
such that $u^{\pm}(t)\in L^2(S^1;\mathbb{C})$ and $\|\mathfrak{u}_{\mathfrak{R}}\|_{L^2(N_r)}=O(r^\frac{3}{2})$.
\end{remark}

\subsection{Properties on a non-Euclidean tubular neighborhood} Now we modify results in Section 4 without assuming a Euclidean metric on the tubular neighborhood.\\

First of all, we should set up some notation. Let $N_R$ to be the tubular neighborhood of $\Sigma$,  and $D_{E}$ to be the Dirac operator with respect to Euclidean metric on $N_R$. We define
$D^{(n)}=\chi_n D_{E}+(1-\chi_n)D$, where $\chi_n=1-\chi_{\frac{\mathfrak{r}}{T^{n+1}},\frac{\mathfrak{r}}{T^{n}}}$ is defined in Section 4.4. So we have
\begin{align*}
D^{(n)}=D_{E}\mbox{ on }N_{\frac{\mathfrak{r}}{T^{n+1}}}.
\end{align*}
Moreover, we have the following proposition (Here we take $\partial_1=\partial_r$, $\partial_2=\partial_{\theta}$ and $\partial_3=\partial_t$).
\begin{pro}
Let $(D^{(n)}-D)=\delta^{(n)}$, we have
\begin{align*}
\delta^{(n)}=\delta_1^{(n)}+\delta_0^{(n)} 
\end{align*}
where\\[1mm]
$\bullet$ $\delta_1^{(n)}$ is a first order differential operator supported on $N_{\frac{\mathfrak{r}}{T^n}}$ such that 
\begin{align*}
\delta_1^{(n)}=\sum_{i=1}^3 a_i\partial_i\mbox{ with }|a_1|\leq O(r^2)\mbox{ and }|a_2|, |a_3|\leq O(r).
\end{align*}
$\bullet$ $\delta_0^{(n)}$ is a zero order differential operator supported on $N_{\frac{\mathfrak{r}}{T^n}}$ such that 
\begin{align*}
|\delta_0^{(n)}|=O(r).
\end{align*}
\end{pro}

We follow the setting in Section 4. Suppose $(\eta_1,\chi_1)$ satisfies (\ref{6_a1}), (\ref{6_a2}), (\ref{6_a3}). We also define
\begin{align*}
\phi_s(t,z)=(t,z+s\eta_1(t))\mbox{ on }N_R,\\
\phi_s(p)=p\mbox{ on }M\setminus N_R
\end{align*}
and
\begin{align*}
D_{s\eta^i}=\sum_{i=1}^3e_i\cdot (\phi_s^{-1})_*(e_i)+\sum_{i=1}^3e_i\sum_{j,k=1}^3(\phi_s^{-1})_*(w_{jk})e_je_k.
\end{align*}

 Then we have the following proposition.
\begin{pro}
The perturbed Dirac operator can be written as
\begin{align*}
D_{s\eta_1}=(1+\rho^1)D^{(1)}+s((\chi_1)_z\eta_1+(\chi_1)_{\bar{z}}\bar{\eta}_1)(e_1\partial_t)+\Theta_s^0+\mathcal{R}_s^0+\mathcal{H}_s^0+\mathcal{F}_s^0+\delta^{(1)}.
\end{align*}
where\\[1mm]
$\bullet$ $\Theta_s^0=[e_1(s\chi \eta_t\partial_z+s\chi \bar{\eta}_t\partial_{\bar{z}})
+e_2(s\chi_{\bar{z}}\bar{\eta}\partial_z-s\chi_z\bar{\eta}\partial_{\bar{z}})
+e_3(-s\chi_{\bar{z}}\eta\partial_z+s\chi_z\eta\partial_{\bar{z}})]$ is a $\mbox{ }\mbox{ }\mbox{ }$first order differential operator.\\[1mm]
$\bullet$ $\mathcal{R}_s^0: L^2_1\rightarrow L^2$ is an $O(s^2)$-first order differential operator supported on $N_\mathfrak{r}-N_{\frac{\mathfrak{r}}{T}}$ $\mbox{ }\mbox{ }\mbox{ }$with its operator norm $\|\mathcal{R}_s^0\|\leq \gamma_{_T}^2\kappa_1^2s^2$. Moreover, for any $\psi\in L^2_1\cap \ker(D)$, we\\$\mbox{ }\mbox{ }\mbox{ }$have
\begin{align*}
\|\mathcal{R}_s(\psi)^0\|_{L^2}\leq C\gamma_{_T}^{\frac{3}{2}}\kappa_1^2\mathfrak{r}^2s^2
\end{align*}
$\mbox{ }\mbox{ }\mbox{ }$for some constant $C$ depending on $\|\psi\|_{L^2_1}$.\\[1mm]
$\bullet$ $\mathcal{H}_s^0$ is an $O(s^2)$-zero order differential operator supported on $N_{\mathfrak{r}}-N_{\frac{\mathfrak{r}}{T}}$. Moreover, $\mbox{ }\mbox{ }\mbox{ }$let us denote $\partial_r$ by $\vec{n}$, the vector field defined on $N_R$, then 
\begin{align}
\int_{\{r=r_0\}}|\mathcal{H}_s^0|^2i_{\vec{n}}dVol(M)\leq \gamma_{_T}^4\kappa_1^4 \mathfrak{r} s^4\label{EE_9}
\end{align}
$\mbox{ }\mbox{ }\mbox{ }$for all $r_0\leq \mathfrak{r}$.\\[1mm]
$\bullet$ $\mathcal{F}_s^0$ is an $O(s)$-zero order differential operator where
\begin{align}
\mathcal{F}^0_s=D\Big(s(\chi_z\eta+\chi_{\bar{z}}\bar{\eta})Id)+D(\left(\begin{array}{cc}
 0&si\chi \eta_t\\
-si\chi \bar{\eta}_t&0
\end{array}\right)\Big).\label{EE_10}
\end{align}
$\bullet$ $\delta^{(1)}$ can be written as $\delta^{(1)}=\delta^{(1)}_0+\delta^{(1)}_1$ where $\delta_1^{(1)}$ is a first order operator with
\begin{align*}
\delta_1^{(1)}=\sum a_i\partial_i\mbox{ with }|a_i|\leq O(r^2)\mbox{ and }|a_2|, |a_3|\leq O(r)
\end{align*}
$\mbox{ }\mbox{ }\mbox{ }$and $\delta_0^{(1)}$ is a zero order operator with
\begin{align*}
|\delta_0^{(1)}|=O(r).
\end{align*}
$\mbox{ }\mbox{ }\mbox{ }$Moreover, $\delta^{(1)}$ is supported on $N_R$.
\end{pro}

Similarly, we have a new version of Proposition 4.7. Suppose that we have a sequence of pairs, $\{(\chi_i,\eta_i)\}$, which is defined in Section 4.4. Moreover, we suppose that $\eta_i$ satisfies (\ref{6_2a0}), (\ref{6_2a1}), (\ref{6_2a2}) and we write $\eta^i=\sum_{j=0}^i\chi_j\eta_j$. Then we have
\begin{pro}
There exists $\kappa_3=O(\kappa_2)$ depending on $\kappa_2$ with the following significance.
The perturbed Dirac operator $D_{s\eta^i}$ which satisfies (\ref{6_2a0}) - (\ref{6_2a2}) can be written as follows:
\begin{align}
D_{s\eta^{i+1}}=(1+\varrho^{i+1})D^{(i+1)}_{s\eta^i}&+s((\chi_{i+1})_z\eta_{i+1}+(\chi_{i+1})_{\bar{z}}\bar{\eta}_{i+1})e_1\partial_t\label{EE_11}\\
&+\Theta^{i}_s+\mathcal{R}^{i}_{s}+\hat{\mathcal{H}}^{i}_s+\mathcal{F}^{i}_s+\delta^{(i+1)}\nonumber
\end{align}
where\\[1mm]
$\bullet$ $\Theta^{i}_s$, the $(\chi,\eta)=(\chi_{i+1},\eta_{i+1})$ version of $\Theta_s^0$, is a first order differential operator\\
$\mbox{ }\mbox{ }\mbox{ }$with order $O(s)$.\\[1mm]
$\bullet$ $\mathcal{R}^{i}_s: L^2_1\rightarrow L^2$ is an $O(s^2)$-first order differential operator supported on\\ $\mbox{ }\mbox{ }\mbox{ } N_{\frac{\mathfrak{r}}{T^i}}-N_{\frac{\mathfrak{r}}{T^{i+1}}}$ with its operator norm
\begin{align}
\|\mathcal{R}_s^{i}\|\leq \gamma_{_T}^2\kappa_3^2s^2.\label{EE_12}
\end{align}
$\bullet$ $\hat{\mathcal{H}}^{i}_s$ is an $O(s^2)$-zero order differential operator. Moreover, let us denote $\vec{n}=\partial_r$\\$\mbox{ }\mbox{ }\mbox{ }$ be the vector field defined on $N_R$, then
\begin{align}
\int_{\{r=r_0\}}|\hat{\mathcal{H}}^{i}_s|^2i_{\vec{n}}dVol(M)\leq \gamma_{_T}^4\kappa_3^4(\frac{(i+1)\mathfrak{r}}{T^{i+1}}) s^4.\label{EE_13}
\end{align}
$\mbox{ }\mbox{ }\mbox{ }$for all $r_0\leq \frac{\mathfrak{r}}{T^i}$.\\[1mm]
$\bullet$ $\mathcal{F}^{i}_s$ is an $O(s)$-zero order differential operator where
\begin{align}
\mathcal{F}^{i}_s=D(s((\chi_{i+1})_z\eta_{i+1}
&+(\chi_{i+1})_{\bar{z}}\bar{\eta}_{i+1})Id)\label{EE_14}\\&+D\Big(\left(\begin{array}{cc}
 0&si\chi_{i+1} (\eta_{i+1})_t\\
-si\chi_{i+1} (\bar{\eta}_{i+1})_t&0
\end{array}\right)\Big).\nonumber
\end{align}
$\bullet$ $\delta^{(i+1)}$ can be written as $\delta^{(i+1)}=\delta^{(i+1)}_0+\delta^{(i+1)}_1$ where $\delta_1^{(i+1)}$ is a first order\\
$\mbox{ }\mbox{ }\mbox{ }$operator with
\begin{align}
\delta_1^{(i+1)}=\sum a_i\partial_i\mbox{ with }|a_i|\leq O(r^2)\mbox{ and }|a_2|, |a_3|\leq O(r)\label{EE_15}
\end{align}
$\mbox{ }\mbox{ }\mbox{ }$and $\delta_0^{(i+1)}$ is a zero order operator with
\begin{align}
|\delta_0^{(i+1)}|=O(r).\label{EE_16}
\end{align}
$\mbox{ }\mbox{ }\mbox{ }$Moreover, $\delta^{(i+1)}$ is supported on $N_{\frac{\mathfrak{r}}{T^i}}$.
\end{pro}

\section{Fredholm property}

\subsection{Basic setting}
 In this section, we develop an important theorem which indicates that the the perturbation along $\mathbb{V}$ is finite dimensional as I mentioned in Section 4.2. The operator $\mathcal{T}_{a^+,a^-}$, which we construct in this section, is an important part in the linear approximation of the moduli space $\mathfrak{M}$ we defined in our main theorem. We explain the idea of construction this operator in the following sections first.\\

The idea comes from \cite{H}. Let $N$ be a tubular neighborhood of $\Sigma$ equipped with the Euclidean metric. By the computation in Section 3.1, we know that for any $\mathfrak{u}$ in the $\ker(D|_{L^2(N;\mathcal{S}\otimes \mathcal{I})})$ can be written as
\begin{align*}
\mathfrak{u}=
\sum_{l\in\mathbb{Z}}e^{ilt}\Bigg[\hat{u}^{+}_{0,l}\left( \begin{array}{c}
\frac{e^{2|l|r}}{\sqrt{z}}\\
-\mbox{sign}(l)\frac{e^{2|l|r}}{\sqrt{\bar{z}}}
\end{array} \right)+
\hat{u}^{-}_{0,l}\left( \begin{array}{c}
\frac{e^{-2|l|r}}{\sqrt{z}}\\
\mbox{sign}(l)\frac{e^{-2|l|r}}{\sqrt{\bar{z}}}
\end{array} \right)\Bigg]\\
+\mbox{ higher order term}
\end{align*}
where the higher order term is $O(r^{p})$ for some $p>-\frac{1}{2}$.\\

Recall the space $\ker(D|_{L^2(N;\mathcal{S}\otimes \mathcal{I})})^0$ shown in Definition 3.3, we firstly define
\begin{align}
&B:\ker(D|_{L^2(N;\mathcal{S}\otimes \mathcal{I})})^0\rightarrow L^2(S^1;\mathbb{C}^2);\label{F_1}\\
&\mathfrak{u}\mapsto(\sum_l \hat{u}^+_{0,l}e^{ilt}+\sum_l \hat{u}^-_{0,l}e^{ilt},-\sum_l \mbox{sign}(l)\hat{u}^+_{0,l}e^{ilt}+\sum_l \mbox{sign}(l)\hat{u}^-_{0,l}e^{ilt}).\nonumber
\end{align}

Secondly, we define the following spaces
\begin{align*}
&Exp^+=\Big\{(\sum_l u_le^{ilt}, \sum_l -\mbox{sign}(l)u_le^{ilt})|\{u_l\}_{l\in\mathbb{Z}}\in \ell^2\Big\}\mbox{ and}\\
&Exp^-=\Big\{(\sum_l u_le^{ilt}, \sum_l \mbox{sign}(l)u_le^{ilt})|\{u_l\}_{l\in\mathbb{Z}}\in \ell^2\Big\}.
\end{align*}
Then we have the corresponding projections $\pi^{\pm}:L^2(S^1;\mathbb{C}^2)\rightarrow Exp^{\pm}$. For any $u=\big(\sum_{l}a_le^{ilt},\sum_l b_le^{ilt}\big)\in L^2(S^1;\mathbb{C}^2)$, it can be written as $u=u^++u^-$ where
\begin{align*}
u^+&=\Big(\sum_l\frac{(a_l-sign(l)b_l)}{2}e^{ilt},\sum_l\frac{(b_l-sign(l)a_l)}{2}e^{ilt}\Big)\in Exp^+;\\
u^-&=\Big(\sum_l\frac{(a_l+sign(l)b_l)}{2}e^{ilt},\sum_l\frac{(b_l+sign(l)a_l)}{2}e^{ilt}\Big)\in Exp^-.
\end{align*}
So we have $L^2(S^1;\mathbb{C}^2)=Exp^+\oplus Exp^-$.\\

Finally, we define the space
\begin{align*}
\ker(D|_{L^2(M\setminus\Sigma;\mathcal{S}\otimes \mathcal{I})})^0:=\big\{\mathfrak{u}\in&\ker(D|_{L^2(M\setminus\Sigma;\mathcal{S}\otimes\mathcal{I})})\\
&\mbox{ with its leading coefficients in }\ell^2\times\ell^2\big\}.\nonumber
\end{align*}
It is a Banach space with norm $\|\mathfrak{u}\|_0:=\|\mathfrak{u}\|_{L^2(M\setminus\Sigma;\mathcal{S}\otimes\mathcal{I})}+\|B(\mathfrak{u})\|_{L^2(S^1;\mathbb{C}^2)}$. Clearly, $B:\ker(D|_{L^2(M\setminus\Sigma;\mathcal{S}\otimes \mathcal{I})})^0\rightarrow L^2(S^1;\mathbb{C}^2)$ is a bounded linear operator.\\

Then we have the following proposition.
\begin{pro}
Define the maps $p^{\pm}=\pi^{\pm}\circ B$ in the following diagram.
\begin{diagram}
      &    &     Exp^+   \\
      &\ruTo^{p^+}    &    \uTo \mbox{ }\pi^+    \\
\ker(D|_{L^2(M\setminus\Sigma;\mathcal{S}\otimes \mathcal{I})})^0&    \rTo^{\mbox{ } \mbox{ } \mbox{ }B\mbox{ } \mbox{ } \mbox{ }}        &L^2(S^1;\mathbb{C}^2) \\
     &\rdTo^{p^-}    &   \dTo \mbox{ }\pi^- \\
     &    &    Exp^-    
\end{diagram}
We have\\[1mm]
{\bf a}. $p^-$ is a Fredholm operator.\\[1mm]
{\bf b}. $p^+$ is a compact operator.
\end{pro}
\begin{proof}
{\bf (proof of part a)}. Let $\{\mathfrak{u}_n\}_{n\in\mathbb{N}}$ be a bounded sequence in $\ker(p^-)$. Assuming that $\|\mathfrak{u}_n\|_{0}\leq 1$. If we can show that there exists a convergent subsequence, then we have $\ker(p^-)$ is finite dimensional.\\

Since $\mathfrak{u}_n\in\ker(p^-)$, we have $B(\mathfrak{u}_n)\in Exp^+$. So $\mathfrak{u}_n\in\mathcal{K}_R$. By Proposition 3.6, we have
\begin{align}
\mathfrak{u}_n=\left(\begin{array}{c}
u^+_n(t)\frac{1}{\sqrt{z}}\\
u^-_n(t)\frac{1}{\sqrt{\bar{z}}}
\end{array}\right)+\mathfrak{u}_{n,\mathfrak{R}}
\end{align}
on the tubular neighborhood of $N_R$ of $\Sigma$. So we have
\begin{align}
\mathfrak{u}_n=\left(\begin{array}{c}
u^+_n(t)\frac{1}{\sqrt{z}}\\
u^-_n(t)\frac{1}{\sqrt{\bar{z}}}
\end{array}\right)(1-\chi_{\frac{R}{2},R})+\mathfrak{u}_{n,\mathfrak{R}}^*
\end{align}
with
\begin{align}
\|\mathfrak{u}_{n,\mathfrak{R}}^*\|_{L^2_1(M\setminus\Sigma)}\leq C\|\mathfrak{u}_n\|_{L^2(M\setminus\Sigma)}\leq C\|\mathfrak{u}_n\|_0=C
\end{align}
for some $C>0$ by Proposition 3.6 and Lichnerowicz-Weitzenb\"ock formula (\ref{1_4}). Therefore, there exists a subsequence of $\{\mathfrak{u}_{n,\mathfrak{R}}^*\}_{n\in\mathbb{N}}$ convergences weakly in $L^2_1$-sense to $\mathfrak{u}_{\mathfrak{R}}^*$, so it converges strongly in $L^2$.\\

Meanwhile, $\mathfrak{u}_n\in\mathcal{K}_R$ also implies that
\begin{align}
\|u^{\pm}_n(t)\|_{L^2_2(S^1;\mathbb{C})}\leq C_R\|\mathfrak{u}_n\|_{L^2(M\setminus\Sigma)}\leq C\|\mathfrak{u}_n\|_0=C.
\end{align}
By Proposition 3.7. This means that there exists a subsequence of $\{u^{\pm}_n\}_{n\in\mathbb{N}}$ converging strongly in $L^2(S^1;\mathbb{C})$. So
\begin{align}
\left(\begin{array}{c}
u^+_n(t)\frac{1}{\sqrt{z}}\\
u^-_n(t)\frac{1}{\sqrt{\bar{z}}}
\end{array}\right)(1-\chi_{\frac{R}{2},R})
\end{align}
also converges in $L^2(M\setminus\Sigma;\mathcal{S}\otimes\mathcal{I})$ as $n\rightarrow\infty$. Therefore, we have a subsequence of $\{\mathfrak{u}_n\}_{n\in\mathbb{N}}$ converges to $\mathfrak{u}$ in $L^2$-sense. It is easy to check that $\mathfrak{u}$ satisfies the Dirac equation. So $\ker(p^-)$ is finite dimensional.\\

To prove that $p^-$ has finite dimensional cokernel, we need several steps. Firstly we consider the extension
\begin{align}
\overline{p^-}:\ker(D|_{L^2(M\setminus\Sigma;\mathcal{S}\otimes \mathcal{I})})\rightarrow \overline{Exp^-}.\label{FF_3a}
\end{align}
where $\overline{Exp^-}$ is the completion of $Exp^-$ with its Fourier coefficients in $\ell^2_{-1}$. $\overline{Exp^-}$ has the norm $\|\cdot\|_{L^2_{-1}(S^1)}$ induced by the norm defined on $\ell^2_{-1}\times \ell^2_{-1}$:
\begin{align*}
\Big\langle (\sum_{l\in\mathbb{Z}} a^+_le^{ilt}, \sum_{l\in\mathbb{Z}} a^-_le^{ilt}),(\sum_{l\in\mathbb{Z}} b^+_le^{ilt},\sum_{l\in\mathbb{Z}} b^-_le^{ilt}) \Big\rangle=\sum_{l\in\mathbb{Z}}(1+|l|)^{-2}\big(a_l^+\bar{b}_l^++a^-_l\bar{b}^-_l\big).
\end{align*}
Since $p^-=\overline{p^-}\big|_{\ker(D|_{L^2(M\setminus\Sigma;\mathcal{S}\otimes \mathcal{I})})^0}$, we can prove $\overline{p^-}$ has finite dimensional cokernel instead.\\
\ \\
{\bf Claim :} There exists $n>0$ with the following significance: For any 
\begin{align*}
u=\big(\sum _{|l|\leq n}\hat{u}_{0,l}^-e^{ilt}, \sum_{|l|\leq n}{\rm sign}(l)\hat{u}_{0,l}^-e^{ilt} \big)\in Exp^-,
\end{align*}
there exists $\mathfrak{u}\in \ker(p^+)$ such that 
\begin{align*}
\|B(\mathfrak{u})-u\|_{L^2_{-1}(S^1)}\leq \frac{1}{2}\|u\|_{L^2_{-1}(S^1)}.
\end{align*}
\ \\

Suppose this claim is true. Let 
\begin{align*}
\mathbb{W}:=\Big\{\sum_le^{ilt}\hat{u}_{0,l}^-\big| \hat{u}_{0,l}^-=0\mbox{ for all }|l|> n\Big\}.
\end{align*}
We prove that $\text{range}(p^-)+\mathbb{W}=Exp^-$ as follows. Suppose not; there exists $v\in L^2(S^1;\mathbb{C})$ such that $v\notin \text{range}(p^-)+\mathbb{W}$. Then we can assume that $v\perp (\text{range}(p^-)+\mathbb{W})$. So by using the claim in previous paragraph, for any 
\begin{align*}
\big(\sum _l\hat{u}_{0,l}^-e^{ilt}, \sum _l{\rm sign}(l)\hat{u}_{0,l}^-e^{ilt} \big)\in Exp^-
\end{align*}
with $\|\sum_l e^{ilt}\hat{u}_{0,l}^-\|_{L^2_{-1}(S^1)}=1$, we have
\begin{align*}
&\Big\langle v, \big(\sum _l\hat{u}_{0,l}^-e^{ilt}, \sum _l{\rm sign}(l)\hat{u}_{0,l}^-e^{ilt} \big)\Big\rangle\\ = &\Big\langle v, \big(\sum_{|l|\leq n}\hat{u}_{0,l}^-e^{ilt}, \sum_{|l|\leq n}{\rm sign}(l)\hat{u}_{0,l}^-e^{ilt} \big)\Big\rangle+\Big\langle v, \big(\sum_{|l|> n}\hat{u}_{0,l}^-e^{ilt}, \sum_{|l|> n}{\rm sign}(l)\hat{u}_{0,l}^-e^{ilt} \big)\Big\rangle\\ 
=&\Big\langle v, \big(\sum_{|l|\leq n}\hat{u}_{0,l}^-e^{ilt}, \sum_{|l|\leq n}{\rm sign}(l)\hat{u}_{0,l}^-e^{ilt} \big)\Big\rangle+\Big\langle v, B(\mathfrak{u}) \Big\rangle + X \\
=&X
\end{align*}
for some $|X|\leq \frac{1}{2}\|v\|_{L^2_{-1}(S^1)}$, which is a contradiction. Therefore, we have 
\begin{align*}
dim(\text{coker}(p^-))\leq 2n+1.
\end{align*}

 To prove the claim, we can consider the following section
\begin{align*}
\mathfrak{u}_0&= \chi\sum_{|l|\geq n}e^{ilt}u^-_{0,l}\Bigg[\left( \begin{array}{c}
e^{-i\frac{1}{2}\theta}\mathfrak{I}_{-\frac{1}{2},l}(r)\\
-le^{i\frac{1}{2}\theta}\mathfrak{I}_{\frac{1}{2},l}(r)
\end{array} \right)+
\mbox{sign}(l)\left( \begin{array}{c}
-le^{-i\frac{1}{2}\theta}\mathfrak{I}_{\frac{1}{2},l}(r)\\
e^{i\frac{1}{2}\theta}\mathfrak{I}_{-\frac{1}{2},l}(r)
\end{array} \right)\Bigg]\\
&=
\chi\sum_{|l|\geq n}e^{ilt}
\hat{u}^{-}_{0,l}\left( \begin{array}{c}
\frac{e^{-2|l|r}}{\sqrt{z}}\\
\mbox{sign}(l)\frac{e^{-2|l|r}}{\sqrt{\bar{z}}}
\end{array} \right)
\end{align*}
with $\chi=1-\chi_{\frac{2R}{3},R}$. So by this setting, we have
\begin{align*}
\|D(\mathfrak{u}_0)\|_{L^2}\leq C\frac{e^{-nR}}{R}.
\end{align*}
By using the arguemnt in Corollary 4.4, we minimize the functional $E_{D(\mathfrak{u}_0)}$ among $L^2_1\cap \ker(D|_{L^2_1})^{\perp}$. We can find $\mathfrak{u}^*$ such that $D(\mathfrak{u}^*)=D(\mathfrak{u}_0)$. Moreover, we have $\|B(\mathfrak{u}^*)\|_{L^2_{-1}(S^1)}\leq C\frac{e^{-nR}}{R}$. So by taking $\mathfrak{u}=\mathfrak{u}_0-\mathfrak{u}^*$, we finish the proof of this claim.\\

{\bf (proof of part b)}. Notice that the coefficients of $\mathfrak{u}$ in $Exp^+$ are corresponding to exponential increasing Fourier modes. Therefore, we have
\begin{align*}
 \sum_l|l||\hat{u}_{0,l}^+|^2\leq C\|\mathfrak{u}\|^2_{L^2(M\setminus\Sigma)}.
\end{align*}
So any bounded sequence $\{\mathfrak{u}^{(n)}\}$ such that $\{p^+(\mathfrak{u}^{(n)})=(\hat{u}_{0,l}^{(n)+})\}$ converges, we have 
\begin{align*}
 \sum_l|\hat{u}_{0,l}^+|^2+\sum_l|l||\hat{u}_{0,l}^+|^2\leq C.
\end{align*}
This implies that there exists a convergent subsequence of $\{\mathfrak{u}^{(n)}\}$ which converges to some $\mathfrak{u}$ and $\lim_{n\rightarrow \infty} p^+(\mathfrak{u}^{(n)})=p^+(\mathfrak{u})$. Therefore, $p^+$ is compact.
\end{proof}
We should remember that under a small perturbation of the metric and $\Sigma$, the dimension of cokernel of $p^+$ will be an upper semi-continuous function. I will leave this proof in Appendix 9.2.\\

Corollary 4.4 shows that the equation
\begin{align*}
D\mathfrak{h}=\mathfrak{f}\mbox{ }mod(\ker(D|_{L^2_1}))
\end{align*}
is solvable for any $\mathfrak{f}\in L^2_{-1}$. By Proposition 6.1, we have the following enhanced result. It tells us that we can find the solution $\mathfrak{h}$ such that $B(\mathfrak{h})\in L^2_2(S^1;\mathbb{C}^2)$ when $\mathfrak{f}$ is identically zero near $\Sigma$.
\begin{pro}
Suppose that $\mathfrak{f}\in L^2_{-1}(M\setminus\Sigma;\mathcal{S}\otimes \mathcal{I})$ and $\mathfrak{f}|_{N_{r_0}}=0$ for some $r\leq \mathfrak{r}$. Then there exists $\mathfrak{h}\in L^2(M\setminus\Sigma;\mathcal{S}\otimes \mathcal{I})$ such that
$
D\mathfrak{h}=\mathfrak{f}\mbox{ }mod(ker(D|_{L^2_1}))
$
and\\[1mm]
{\bf a}. $\|\mathfrak{h}\|_{L^2}\leq C\|\mathfrak{f}\|_{L^2_{-1}}$ for some universal constant $C>0$.\\[1mm]
{\bf b}. The leading term of $\mathfrak{h}$, $h^{\pm}$, will satisfy
\begin{align*}
{r_0}\|h^{\pm}\|^2_{L^2},{r_0}^3\|(h^{\pm})_t\|^2_{L^2},{r_0}^5\|(h^{\pm})_{tt}\|^2_{L^2}\leq C\|\mathfrak{\mathfrak{f}}\|^2_{L^2_{-1}}
\end{align*}
$\mbox{ }\mbox{ }\mbox{ }\mbox{ }$for some universal constant $C>0$.
\end{pro}
\begin{proof}
First of all, we claim that, for any $l>0$, there exists $\mathfrak{u}_l\in L^2(M\setminus\Sigma;\mathcal{S}\otimes \mathcal{I})$ with
\begin{align*}
\mathfrak{u}_l=e^{ilt}\left(\begin{array}{c}
\frac{e^{-2|l|r}}{\sqrt{z}}\\
\mbox{sign}(l)\frac{e^{-2|l|r}}{\sqrt{\bar{z}}}
\end{array}
\right)
\end{align*}
on $N_R$ such that $D\mathfrak{u}_l=0$ on $M\setminus\Sigma$. Since Proposition 6.1 tells us that $p^+$ is a compact operator, limit of finite dimensional operator, and $p^-$ is a Fredholm operator, this claim can be regarded as a special case. We will modify the proof of this claim and get the proof for general cases later.\\

We have
\begin{align}
\|\mathfrak{u}_l\|_{L^2}\leq \frac{2C}{|l|^{\frac{1}{2}}}.\label{FF_4}
\end{align}

Meanwhile, by using Corollary 4.4, there exists $\mathfrak{v}\in L^2_1(M\setminus\Sigma;\mathcal{S}\otimes\mathcal{I})$ such that 
\begin{align*}
D^2\mathfrak{v}=\mathfrak{f}.
\end{align*}
Taking $\tilde{\mathfrak{h}}=D\mathfrak{v}$, we have 
\begin{align*}
D\tilde{\mathfrak{h}}=\mathfrak{f}.
\end{align*}
Now, since $\tilde{\mathfrak{h}}\in \text{range}(D|_{L^2_1})$, it is perpendicular to $\ker(D|_{L^2})$ by Proposition 2.4. So it is perpendicular to $\mathfrak{u}_l$. Suppose that the Fourier coefficients of $\tilde{\mathfrak{h}}$ are $h^{\pm}_{k,l}$. We define
\begin{align*}
\hat{\mathfrak{u}}_l:=\frac{\hat{h}^-_{0,l}}{|\hat{h}^-_{0,l}|}\mathfrak{u}_l
\end{align*}
where $\hat{h}^+_{0,l}=(h^+_{0,l}-\mbox{sign}(l)h^-_{0,l})$ and $\hat{h}^-_{0,l}=(h^+_{0,l}+\mbox{sign}(l)h^-_{0,l})$. Then we have 
\begin{align}
\|\hat{\mathfrak{u}}_l\|_{L^2}\leq \frac{2C}{|l|^{\frac{1}{2}}}.\label{FF_5a}
\end{align}
Now, by a straightforward computation, we have
\begin{align*}
\int_{M\setminus\Sigma}\langle \tilde{\mathfrak{h}},\hat{\mathfrak{u}}_l \rangle=0=|\hat{h}^-_{0,l}|\int_0^{r_0}e^{-4|l|r}dr+\int_{M\setminus N_r}\langle \tilde{\mathfrak{h}},\hat{\mathfrak{u}}_l \rangle.
\end{align*}
This implies that
\begin{align}
|\hat{h}^-_{0,l}|\leq \frac{4C|l|^{\frac{1}{2}}}{1-e^{-2|l|{r_0}}}\|P_l(\tilde{\mathfrak{h}})\|_{L^2}\label{FF_6a}
\end{align}
where $P_l$ is the orthogonal projection from $L^2(M\setminus N_{r_0};\mathcal{S}\otimes\mathcal{I})$ to $\mbox{span}\{\mathfrak{u}_l\}$. Now define
\begin{align*}
\mathfrak{y}=\sum_{|l|>\frac{1}{2{r_0}}} \hat{h}^-_{0,l}\mathfrak{u}_{l}.
\end{align*}
Then, by (\ref{FF_6a}), we have
\begin{align*}
\|\mathfrak{y}\|_{L^2}\leq C\|{r_0}^{\frac{1}{2}}\mathfrak{y}|_{\partial N_{r_0}}\|_{L^2_{-1/2}}&=\sum_{|l|>\frac{1}{2{r_0}}}\frac{|\hat{h}^-_{0,l}|^2}{|l|}\leq \sum_{|l|>\frac{1}{2{r_0}}}\frac{4C}{(1-e^{-2|l|{r_0}})^2}\|P_l(\tilde{\mathfrak{h}})\|_{L^2}^2
\\
&\leq \frac{4C}{(1-e^{-2})^2}\sum_{l}\|P_l(\tilde{\mathfrak{h}})\|^2_{L^2}\leq C\|\tilde{\mathfrak{h}}\|^2_{L^2}.
\end{align*}

Let $\mathfrak{h}=\tilde{\mathfrak{h}}-\mathfrak{y}$, which satisfies $D\mathfrak{h}=0$ and $\mathfrak{h}\in \mathcal{K}_{r_0}$. We have
\begin{align*}
\|\mathfrak{h}\|_{L^2}\leq C\|\tilde{\mathfrak{h}}\|_{L^2}.
\end{align*}
Notice that by Lemma 4.1, we have by integration by parts and Cauchy inequality
\begin{align*}
\|\mathfrak{v}\|^2_{L^2_1}\leq C\|\tilde{\mathfrak{h}}\|^2_{L^2}\leq C\|\mathfrak{v}\|_{L^2_1}\|\mathfrak{f}\|_{L^2_{-1}}\leq \varepsilon \|\mathfrak{v}\|_{L^2_1}+\frac{C}{4\varepsilon}\|\mathfrak{f}\|_{L^2_{-1}}.
\end{align*}
So by choosing $\varepsilon$ small enough, we have
\begin{align*}
\|\tilde{\mathfrak{h}}\|_{L^2}\leq \|\mathfrak{v}\|_{L^2_1}\leq C\|\mathfrak{f}\|_{L^2_{-1}}.
\end{align*}
Therefore, we prove {\bf a}. For {\bf b}, we can get it immediately by using Proposition 3.7.\\

For the general case ($p^+$ is nonzero and $p^-$ is Fredholm), we have similar argument by modifying $\mathfrak{u}_l$ to be 
\begin{align*}
\mathfrak{u}_l=e^{ilt}\left(\begin{array}{c}
\frac{e^{-2|l|r}}{\sqrt{z}}\\
\mbox{sign}(l)\frac{e^{-2|l|r}}{\sqrt{\bar{z}}}
\end{array}
\right)+O_l
\end{align*}
where $\|O_l\|_{L^2}\leq Ce^{-|l|\frac{R}{2}}$ (We can choose $\mathfrak{r}$ small such that $|l|>\frac{1}{2\mathfrak{r}}$ is very large). The existence of these $\mathfrak{u}_l$ can be proved by using Corollary 4.4. The term 
\begin{align*}
e^{ilt}\left(\begin{array}{c}
\frac{e^{-2|l|r}}{\sqrt{z}}\\
\mbox{sign}(l)\frac{e^{-2|l|r}}{\sqrt{\bar{z}}}
\end{array}
\right)
\end{align*}
has $L^2$-norm $O(\frac{1}{|l|^{\frac{1}{2}}})$, which will dominate $O_l$. So we can check that the argument above works for these $\mathfrak{u}_l$ and the argument in Proposition 3.7 works for them, too. Therefore, we prove this proposition.
\end{proof}

\subsection{Linearization: The crucial observation}
Here we derive the linearization of $\mathfrak{M}$. Suppose that $\psi$ is an $L^2_1$-harmonic spinor with respect to metric $g$, which is locally written as
\begin{align}
\psi=\left(\begin{array}{c}
a^+(t)\sqrt{z}\\
a^-(t)\sqrt{\bar{z}}
\end{array}\right)+\mbox{higher order term}.\label{FF_7}
\end{align}
$\Sigma$ is its zero locus. Denote by $p$ the triple $(g, \Sigma, \psi)\in \mathfrak{M}$;
\begin{align*}
\mathcal{B}&=\{C^{\infty}-\mbox{real valued }(2,0)\mbox{-tensor }\delta\mbox{ with }supp(\delta)\cap\Sigma=\emptyset\};\\
\mathbb{V}&=\{\eta:S^1\rightarrow \mathbb{C}\mbox{ }|\mbox{ }\eta\in C^1\}.
\end{align*}

 Now suppose that we have a differentiable one-parameter perturbation $(g_s,\Sigma_s, \psi_s)$ with $(g_0,\Sigma_0,\psi_0)=(g, \Sigma, \psi)$ which can be written as
\begin{align}
g_s&=g_0+s\delta+O(s^2),\label{F_2}\\
\Sigma_s&=\{(t,s\eta(t)+O(s^2))\},\label{F_3}\\
\psi_s(t,z)&=\psi(t,z-s\eta+O(s^2))+s\phi_s=\left(\begin{array}{c}
a^+(t)\sqrt{z-s\eta}\\
a^-(t)\sqrt{\bar{z}-s\bar{\eta}}
\end{array}\right)+O_{L^2_1}(s)+s\phi_s\label{F_4}
\end{align}
for some $\delta \in \mathcal{B}$, $\eta\in\mathbb{V}$ and 
\begin{align*}
\phi_s:=\phi+O_{L^2_1}(s)\in L^2_1(M\setminus\Sigma_s;\mathcal{S}_{g_s,\Sigma_s})\cong L^2_1(M\setminus\Sigma;\mathcal{S}_{g,\Sigma})
\end{align*}
Here we use $O_{L^2_1}(s)$ to denote a one-parameter family of sections $f_s$ satisfying $\|f_s\|_{L^2_1}\leq Cs$ for some constant $C$. Let $D^{(s)}$ be the Dirac operator with respect to $g_s$, then we have
\begin{align}
D^{(s)}=D+sT+O(s^2)\label{F_5}
\end{align}
for some first order differential operator $T$. Notice that the support of $T$ is disjoint from a tubular neighbourhood of $\Sigma$.\\

 Therefore, the linearization of $\mathfrak{M}$ at $p$ can be written as
\begin{align}
\mathfrak{L}_p(\delta,\eta,\phi):&=\frac{\partial}{\partial s}(D^{(s)}\psi_s)\Big|_{s=0}\label{F_6}\\
&=T(\psi)+D\big(\frac{\partial}{\partial s}\Big(\psi(t,z-s\eta+O(s^2))\Big)\Big|_{s=0}+O_{L^2_1}(1)+\phi\big).\nonumber
\end{align}
$\mathfrak{L}_p$ is a map from $\mathcal{B}\times\mathbb{V}\times L^2_1(M\setminus\Sigma;\mathcal{S}_{g,\Sigma})$ to $L^2(M\setminus\Sigma;\mathcal{S}_{g,\Sigma})$.\\

Notice that $T(\psi)\in L^2$ is compactly supported away from $\Sigma$. By Proposition 6.2, there exists $\mathfrak{h}\in L^2$ such that $D\mathfrak{h}=-T(\psi)\mbox{ }mod(\ker(D|_{L^2_1}))$ with
\begin{align}
\mathfrak{h}=\left(
\begin{array}{c}
\frac{h^+}{\sqrt{z}}\\
\frac{h^-}{\sqrt{\bar{z}}}
\end{array}
\right)+\mbox{higher order term}\label{F_7}.
\end{align}
Therefore, the right-hand side of (\ref{F_2}) can be rewritten as
\begin{align*}
D\Big(\left(
\begin{array}{c}
\frac{a^+\eta}{2\sqrt{z}}\\
\frac{a^-\bar{\eta}}{2\sqrt{\bar{z}}}
\end{array}
\right)
+
\left(
\begin{array}{c}
\frac{h^+}{\sqrt{z}}\\
\frac{h^-}{\sqrt{\bar{z}}}
\end{array}
\right)+O_{L^2_1}(1)+\phi\Big).
\end{align*}
This implies that if $(\delta,\eta,\phi)\in \ker(\mathfrak{L}_p)$, the element 
\begin{align}
\Big(\left(
\begin{array}{c}
\frac{a^+\eta}{2\sqrt{z}}\\
\frac{a^-\bar{\eta}}{2\sqrt{\bar{z}}}
\end{array}
\right)
+
\left(
\begin{array}{c}
\frac{h^+}{\sqrt{z}}\\
\frac{h^-}{\sqrt{\bar{z}}}
\end{array}
\right)+O_{L^2_1}(1)+\phi\Big)\label{F_8}
\end{align}
 is an $L^2$-harmonic spinor. Using the notation of Proposition 6.1, we can rewrite this condition as follows:
\begin{align*}
(a^+\eta+2h^+,a^-\bar{\eta}+2h^-)=B(\mathfrak{u})
\end{align*}
for some $\mathfrak{u}\in \ker(D|_{L^2})^0$. In particular, this defines a map $\Psi:\mathcal{B}\times\mathbb{V}\times L^2_1\rightarrow \ker(D|_{L^2})^0$ with $\Psi(\delta,\eta,\phi)=\mathfrak{u}$. Our goal is to prove that for any $h^{\pm}$ given, there are only finite dimensional solutions $\eta$ satisfying $(a^+\eta+2h^+,a^-\bar{\eta}+2h^-)\in B(\ker(D|_{L^2})^0)$. Namely, we have to show that the equations
\begin{align*}
a^+\eta+c^+=-2h^+,\\
a^-\bar{\eta}+c^-=-2h^-.
\end{align*}
for $(c^{\pm})\in B(\ker(D|_{L^2})^0)$ have a finite dimensional solution space. These equations have the following constraint:
\begin{align}
|a^+|^2+|a^-|^2> 0.\label{F_9}
\end{align}
which comes from the assumption that $\frac{|\psi|(p)}{\mbox{dist}(p,\Sigma)^{\frac{1}{2}}}>0$ for all $p$. By some basic computation, these equations imply
\begin{align}
\bar{a}^-c^+-a^+\bar{c}^-=-2\bar{a}^-h^++2a^+\bar{h}^-.\label{F_10}
\end{align}
Therefore, we can define the following operator
\begin{align}
\mathcal{T}_{a^+,a^-}:L^2(S^1;\mathbb{C}^2)\rightarrow L^2(S^1;\mathbb{C});\label{F_11}\\
\mathcal{T}_{a^+,a^-}(c^{\pm})=\bar{a}^-c^+-a^+\bar{c}^-.\nonumber
\end{align}

One can check that 
\begin{align*}
\ker(\mathfrak{L}_p|_{\delta=0})= \ker(\mathcal{T}_{a^+,a^-}\circ B)
\end{align*}
and
\begin{align*}
\text{coker}(\mathfrak{L}_p|_{\delta=0})\subset \text{coker}(\mathcal{T}_{a^+,a^-}\circ B)\oplus \ker(D|_{L^2_1}),
\end{align*}
we leave the proof of this part in Appendix 9.3. Therefore, we define
\begin{align}
&\mathbb{K}_0=\ker(\mathcal{T}_{a^+,a^-}\circ B);\label{F_12}\\
&\mathbb{K}_1=\text{coker}(\mathcal{T}_{a^+,a^-}\circ B)\oplus \ker(D|_{L^2_1}).\label{F_13}
\end{align}
So our goal in this section is to show the following Proposition
\begin{pro}
$\mathcal{T}_{a^+,a^-}\circ B$ is Fredholm.
\end{pro}
It implies that $\mathbb{K}_1$ and $\mathbb{K}_0$ are finite dimensional.\\

To prove Proposition 6.3, recall that $p^+$ is a compact operator. So 
\begin{align*}
\mathcal{T}_{a^+,a^-}\circ B=\mathcal{T}_{a^+,a^-}|_{Exp^-}\circ p^-+\mbox{ a compact operator}.
\end{align*}
We have that $\mathcal{T}_{a^+,a^-}\circ B$ is Fredholm if and only if $\mathcal{T}_{a^+,a^-}|_{Exp^-}$ is Fredholm (because $p^-$ is a Fredholm operator). This will be proved in Theorem 6.14.

\begin{remark}
By (\ref{F_5}), we can see that the operator $T$ depends only on $\delta$. So the operator $\Phi$ in (\ref{AA_10}) is defined by $\Phi(\delta)=\mathfrak{h}$ where $\mathfrak{h}$ is defined in (\ref{F_7}).
\end{remark}

\subsection{Fredholmness of finite Fourier mode case}
The main result of Section 6.3 are Theorem 6.6 and Lemma 6.12. Theorem 6.6 shows the Fredholmness for $\mathcal{T}_{a^+,a^-}$ when $a^+$ and $a^-$ have only finite Fourier modes. The reason we have to deal with this special case first is because $\mathcal{T}_{a^+,a^-}$ is in general not Fredholm when the condition (\ref{F_9}) fails. So we have to find a sequence of Fredholm operators converging to $\mathcal{T}_{a^+,a^-}$ and prove the Fredholmness holds under this limit.\\

The proof of Theorem 6.6 is quite long. In fact, there should be a easier way to prove it. However, by following our proof, it is more easier to see how the Fredholmness holds when $a^{\pm}$ have infinite Fourier modes. Lemma 6.12 plays the crucial rule to prove this statement in Section 6.4.\\

Consider the equation
\begin{align}
\left\{ \begin{array}{c}
a^+\eta+c^+=-2h^+,\\
a^-\bar{\eta}+c^-=-2h^-.
\end{array} \right.\label{F_EQ}
\end{align}
with constraint (\ref{F_4}) and $c^{\pm}\in Exp^-$. So there is the following relationship between $c^+$ and $c^-$: if we write $c^+=\sum p_le^{ilt}$, we have
$c^-=\sum \mbox{sign}(l)p_le^{ilt}$. Namely, $c^-$ is determined by $c^+$.\\

In this section, we will use the following notation.
\begin{definition}
Let $g=\sum_{l\in\mathbb{Z}}g_le^{ilt}\in L^2(S^1;\mathbb{C})$. We define $g^{aps}=\sum_{l\in\mathbb{Z}}\mbox{sign}(l)g_le^{ilt}$.
\end{definition}

So we can rewrite operator $\mathcal{T}_{a^+,a^-}$ in the following way:
\begin{align*}
\mathcal{T}_{a^+,a^-}(c):=\bar{a}^-c-a^+\overline{c^{aps}}
\end{align*} 
with $\mathcal{T}_{a^+,a^-}:L^{2}(S^1;\mathbb{C}) \rightarrow L^{2}(S^1;\mathbb{C})$. Here we should explain the meaning of this $L^2(S^1;\mathbb{C})$ space: We can easily see that, $\mathcal{T}_{a^+,a^-}$ is not a $\mathbb{C}-$linear operator, since the conjugate term $\overline{c^{aps}}$ involved. However, it is still an $\mathbb{R}-$linear operator. Therefore, we define our index under the real vector spaces. To simplify the notation, we sometime use $L^2(S^1)$ to denote $L^2(S^1;\mathbb{C})$ in the rest of this paper.\\
 
So in our case, we define the inner product to be
\begin{align*}
(f,g):=Re(\int_{\mathbb{S}^1}f\cdot \bar{g} dt)
\end{align*}
for all $f,g\in C^{\infty}(S^1;\mathbb{C})$. We can see that, under this definition, the $L^2-$bounded space will be coincident with the one equipped with the usual inner product over $\mathbb{C}$.\\

The following is the main Theorem of Section 6.3:

\begin{theorem}
 $\mathcal{T}_{a^+,a^-}$ is a Fredholm operator and $index(\mathcal{T}_{a^+,a^-})=0$ when both $a^+$ and $a^-$ have only finite many Fourier modes:
\begin{align*}
a^+=\sum_{-M}^{M}a^+_{l}e^{ilt}\mbox{; }
a^-=\sum_{-M}^{M}a^-_le^{ilt}
\end{align*}
 for some $M\in \mathbb{N}$. Namely, we have\\[1mm]
{\bf a.} $\ker(\mathcal{T}_{a^+,a^-})$ and $\ker(\mathcal{T}^*_{a^+,a^-})$ are finite dimensional.\\[1mm]
{\bf b.} ${\rm range}(\mathcal{T}_{a^+,a^-})$ and ${\rm range}(\mathcal{T}^*_{a^+,a^-})$ are closed.\\[1mm]
Here $\mathcal{T}^*_{a^+,a^-}$ is the dual operator of $\mathcal{T}_{a^+,a^-}$.
\end{theorem}

To obtain the proof, we need few notation.

\begin{definition}
Let $a=(x,y)\in \mathbb{C}\times \mathbb{C}$, we define the {\bf spouse} of $a$, denoted by $\hat{a}$, to be $(\bar{y},-\bar{x})\in \mathbb{C}\times \mathbb{C}$. We can easily see that $\hat{\hat{a}}=-a$.
\end{definition}
Similarly, for any $p$-tuple of complex pairs, we have the following definition.

\begin{definition}
Let $A=(a_1, a_{2},...,a_{p-1}, a_p)\in (\mathbb{C}\times \mathbb{C})^{p}$ for some $p\in\mathbb{N}$. We define the {\bf spouse} of $A$, denoted by $\hat{A}$, to be $(\hat{a}_{p}, \hat{a}_{p-1},...,\hat{a}_{2}, \hat{a}_1)\in (\mathbb{C}\times \mathbb{C})^{p}$.
\end{definition}

In the following 7 Steps , we will prove part {\bf a} of Theorem 6.6 from Step 1 to Step 3 and prove part {\bf b} of Theorem 6.6 from Step 4 to Step 7.\\

{\bf Step 1}. In Step 1 and Step 2, we prove that $\ker(\mathcal{T}_{a^+,a^-})$ is finite dimensional. Let $c=\sum_{l\in\mathbb{Z}}p_le^{ilt}$. First, we notice that the $n$-th Fourier coefficient of $(\bar{a}^-c-a^+\overline{c^{aps}})$ can be written as
\begin{align*}
(\bar{a}^-c-a^+\overline{c^{aps}})_n= \sum_{l=-M}^{M}\bar{a}^-_{-l}p_{n-l}+\mbox{sign}(l-n)a^+_l\bar{p}_{l-n}.
\end{align*}
When $n>M$, $\mbox{sign}(l-n)=-1$ for all $l=-M,...,M$. So we have
\begin{align}
(\bar{a}^-c-a^+\overline{c^{aps}})_n= \sum_{l=-M}^{M}\bar{a}^-_{-l}p_{n-l}-a^+_l\bar{p}_{l-n}\label{F_14}
\end{align}
for $n>M$.\\

Similarly
\begin{align}
(\bar{a}^-c-a^+\overline{c^{aps}})_{n}= \sum_{l=-M}^{M}\bar{a}^-_{-l}p_{n-l}+a^+_l\bar{p}_{l-n}\label{FF_8}
\end{align}
for $n<-M$.\\

If we take $n=-n'$ and then take the conjugation on both side of the equation (\ref{FF_8}) above, we will have the following equation:
\begin{align}
(a^-\bar{c}-\bar{a}^+c^{aps})_{n'}=\sum_{l=-M}^{M}\bar{a}^+_{-l}p_{n'-l}+a^-_{l}\bar{p}_{l-n'}.\label{F_15}
\end{align}
for $n'>M$.\\

{\bf Step 2}. To show that the kernel of $\mathcal{T}_{a^+,a^-}$ is finite dimensional, here is the idea: We claim that every element in $\ker(\mathcal{T}_{a^+,a^-})$ can be determined by their Fourier coefficients from $-2M$ to $2M$. Therefore, the dimension of $\ker(\mathcal{T}_{a^+,a^-})$ cannot exceed $4M+2$. To prove this claim, we have to show that $c_1-c_2=0$ for any $c_1$ and $c_2$ in $\ker(\mathcal{T}_{a^+,a^-})$ which have the same Fourier coefficients from $-2M$ to $2M$. Therefore, our claim is true iff any $c\in \ker(\mathcal{T}_{a^+,a^-})$ which has zero Fourier coefficients from $-2M$ to $2M$ is identically zero.\\

Now we prove this claim. Suppose that $c\in \ker(\mathcal{T}_{a^+,a^-})$ has zero Fourier coefficients from $-2M$ to $2M$. Because $c\in \ker(\mathcal{T}_{a^+,a^-})$, we have
\begin{align}
\left\{
\begin{array}{c}
\sum_{l=-M}^{M}\bar{a}^-_{-l}p_{n-l}-a^+_l\bar{p}_{l-n}=0\\
\sum_{l=-M}^{M}\bar{a}^+_{-l}p_{n-l}+a^-_{l}\bar{p}_{l-n}=0
\end{array}
\right.\label{FF_1}
\end{align}
for $n>M$. We can rewrite (\ref{FF_1}) by pairing $(p_{j},\bar{p}_{-j}):=v_{j}$ and $(\bar{a}^-_{-j}, -a^+_j):=d_j$ for all $j\in\mathbb{Z}$:
\begin{align}
\left\{
\begin{array}{c}
\sum_{l=-M}^{M}\langle d_l, \bar{v}_{n-l}\rangle=0\\
\sum_{l=-M}^{M}\langle \hat{d}_{-l}, \bar{v}_{n-l}\rangle=0
\end{array}
\right.\label{FF_2}
\end{align}
with the bracket $\langle\cdot,\cdot\rangle$ denoting the usual inner product over $\mathbb{C}$. Here we can use the following convention: Let $U=(u_i),W=(w_i)\in (\mathbb{C}\times \mathbb{C})^{\mathbb{Z}}$. Define a new bracket $\langle\langle\cdot,\cdot\rangle\rangle_n$ to be
\begin{align*}
\langle\langle U,W\rangle\rangle_{n}= \sum_{i\in \mathbb{Z}} \langle u_i,w_{n-i}\rangle.
\end{align*}
So (\ref{FF_2}) can be written as
\begin{align}
\left\{
\begin{array}{c}
\langle\langle D, \bar{V} \rangle\rangle_n=0\\
\langle\langle \hat{D}, \bar{V}  \rangle\rangle_{n}=0
\end{array}
\right.\label{FF_3}
\end{align}
where $D=(d_l)$, $V=(v_{l})$ and $n>M$.\\

Now we apply the following lemma.
\begin{lemma} Given $A=(a_j)_{j=1,2,...,p}\in(\mathbb{C}\times \mathbb{C})^{p}$. 
If $V = (v_j)_{j\in\mathbb{Z}}\in (\mathbb{C}\times \mathbb{C})^{\mathbb{Z}}$ satisfies
\begin{align}
\langle\langle A, \bar{V} \rangle\rangle_m=0; \mbox{ }
\langle\langle \hat{A}, \bar{V}  \rangle\rangle_{m}=0\label{FF_4a}
\end{align}
for all $m>0$, then there is $B=(0,...,0,b_1,...b_q)\in (\mathbb{C}\times \mathbb{C})^{p}$ with $\det\left( \begin{array}{c}b_q\\
\hat{b}_{1}
\end{array} \right)\neq 0$ such that
\begin{align}
\langle\langle B, \bar{V} \rangle\rangle_m=0;\mbox{ }
\langle\langle B^*, \bar{V}  \rangle\rangle_{m}=0,\label{FF_5}
\end{align}
where $B^*=(0,...,0,\hat{b}_q,...,\hat{b}_1)$, for all $m>0$.
\end{lemma}

\begin{proof}
If $\det\left( \begin{array}{c}
a_p\\
\hat{a}_{1}
\end{array} \right)\neq 0$, then we can just take $A=B$. The lemma holds trivially.\\

Suppose now $\det\left( \begin{array}{c}
a_p\\
\hat{a}_{1}
\end{array} \right)= 0$. Then we have $\alpha a_p=\hat{a}_1$ for some $\alpha\in \mathbb{C}-\{0\}$. So (\ref{FF_4a}) implies
\begin{align}
\langle\langle \hat{A}, \bar{V}  \rangle\rangle_{m} - \alpha \langle\langle A, \bar{V} \rangle\rangle_m
=\langle\langle \hat{A}-\alpha  A , \bar{V} \rangle\rangle_m
=0\label{FF_9}
\end{align}
and
\begin{align}
\langle\langle A, \bar{V} \rangle\rangle_m+\bar{\alpha}
\langle\langle \hat{A}, \bar{V}  \rangle\rangle_m
=\langle\langle A+\bar{\alpha}
\hat{A}, \bar{V}  \rangle\rangle_m
=0.\label{FF_10}
\end{align}
We define
\begin{align*}
B'_1=(\hat{A}-\alpha  A)=(\hat{a}_p-\alpha a_1, \hat{a}_{p-1}-\alpha a_2,....,\hat{a}_2-\alpha a_{p-1}, 0).
\end{align*}
Notice that: Since $\alpha a_p=\hat{a}_1$ implies $\bar{\alpha} \hat{a}_p=-a_1$, we have $\hat{a}_p-\alpha a_1=\hat{a}_p+|\alpha|^2\hat{a}_p=(1+|\alpha|^2)\hat{a}_p\neq 0$. This implies $B'_1\neq 0$.\\

Now let $B_1=(0, \hat{a}_p-\alpha a_1, \hat{a}_{p-1}-\alpha a_2,....,\hat{a}_2-\alpha a_{p-1})$. We can easily verify that
\begin{align*}
\langle\langle \hat{A}-\alpha  A , \bar{V} \rangle\rangle_{m+1}=\langle\langle B_1 , \bar{V} \rangle\rangle_m=0
\end{align*}
for all $m>0$ by (\ref{FF_9}).\\

Since $(\hat{A}-\alpha A)=(A+\bar{\alpha} \hat{A})^{\wedge}$, (\ref{FF_10}) gives us
\begin{align*}
\langle\langle A-\bar{\alpha}
\hat{A}, \bar{V}  \rangle\rangle_m=\langle\langle B^*_1 , \bar{V} \rangle\rangle_m=0
\end{align*}
for all $m>0$.\\

By repeating this process inductively, we prove this lemma.
\end{proof}

Back to our problem, we have (\ref{FF_3}):
\begin{align*}
\left\{\begin{array}{c}
\langle\langle D, \bar{V} \rangle\rangle_n=0\\
\langle\langle \hat{D}, \bar{V}  \rangle\rangle_{n}=0
\end{array}
\right.
\end{align*}
for $n>M$. We can apply Lemma 6.9 with $A=(d_{-M},d_{-(M-1)},..., d_M)$, $m=n-M$. So there exists $B\in(\mathbb{C}\times \mathbb{C})^{p}$ such that $\det\left( \begin{array}{c}b_q\\
\hat{b}_{1}
\end{array} \right)\neq 0$ and
\begin{align*}
\langle\langle B, \bar{V} \rangle\rangle_n=0\\
\langle\langle B^*, \bar{V}  \rangle\rangle_{n}=0
\end{align*}
for all $n>M$. Combining with the condition $v_l=0$ for $l=0,1,...,2M$, we have
\begin{align*}
\langle\langle B, \bar{V} \rangle\rangle_{M+1}=\langle b_q, v_{2M+1}\rangle= b^+_q p_{(2M+1)}+b^-_q \bar{p}_{-(2M+1)}= 0\\
\langle\langle B^*, \bar{V}  \rangle\rangle_{M+1}=\langle \hat{b}_1, v_{2M+1}\rangle= \hat{b}^-_1 p_{(2M+1)}+\bar{b}^+_1 \bar{p}_{-(2M+1)}= 0,
\end{align*}
which implies $v_{2M+1}=0$ because $\det\left( \begin{array}{c}b_q\\
\hat{b}_{1}
\end{array} \right)\neq 0$. Now we can solve $v_k$ inductively: Suppose $v_{1},v_{2},...,v_{M+k}$ are all zero for some $k>M+1$. Then the equation tells us that
\begin{align*}
\langle\langle B, \bar{V} \rangle\rangle_{k+1}=\langle b_q, v_{M+k+1}\rangle= b^+_q p_{(M+k+1)}+b^-_q \bar{p}_{-(M+k+1)}= 0\\
\langle\langle B^*, \bar{V}  \rangle\rangle_{k+1}=\langle \hat{b}_1, v_{M+k+1}\rangle= \hat{b}^-_1 p_{(M+k+1)}+\bar{b}^+_1 \bar{p}_{-(M+k+1)}= 0.
\end{align*}
So we have $v_{M+k+1}=0$. Therefore, we have $v_{l}=0$ for all $l$, which implies $c\equiv 0$.\\

{\bf Step 3}. Here we prove that $\ker(\mathcal{T}^*_{a^+,a^-})$ is finite dimensional. We can get the following computation by the definition: Let $c=\sum p_le^{ilt}$, $k=\sum q_le^{ilt}\in L^2(S^1)$
\begin{align*}
(\mathcal{T}_{a^+,a^-}(c),k)&= Re(\int_{\mathbb{S}^1}\mathcal{T}_{a^+,a^-}(c)\cdot\bar{k} dt). \\ 
&= \frac{1}{2}(\int_{\mathbb{S}^1}\mathcal{T}_{a^+,a^-}(c)\cdot\bar{k} dt + \int_{\mathbb{S}^1}k\cdot\overline{\mathcal{T}_{a^+,a^-}(c)} dt)\\
&=\sum_{n\in \mathbb{Z}}\sum_{l=-M}^{M}(\bar{a}_{-l}p_{n-l}+\mbox{sign}(l-n)a^+_l\bar{p}_{l-n})\bar{q}_{n}\\
&\mbox{ }\mbox{ }+\sum_{n\in \mathbb{Z}}\sum_{l=-M}^{M}q_n(a^-_{-l}\bar{p}_{n-l}-\mbox{sign}(l-n)\bar{a}^+_l p_{l-n})\\
&= \sum_{n\in \mathbb{Z}}(\sum_{l=-M}^{M}a^-_{-l}q_{n+l}+\mbox{sign}(n)a_l^+\bar{q}_{-n+l})\bar{p}_{n}\\
&\mbox{ }\mbox{ }+\sum_{n\in \mathbb{Z}}p_{n}(\sum_{l=-M}^{M}\overline{(a^-_{-l}q_{n+l}+\mbox{sign}(n)a_l^+\bar{q}_{-n+l}})\\
&=(c,\mathcal{T}^*_{a^+,a^-}(k)).
\end{align*}
We get the last equality by taking
\begin{align*}
\mathcal{T}^*_{a^+,a^-}(k)=\sum_{n\in \mathbb{Z}}(\sum_{l=-M}^Ma^-_{-l}q_{n+l}+\mbox{sign}(n)a^+_l\bar{q}_{-n+l})e^{int}.
\end{align*}\\

Now we can repeat the argument in Step 1 and 2 on $\mathcal{T}^*_{a^+,a^-}$, then we will get $dim(\ker(\mathcal{T}^*_{a^+,a^-}))<\infty$.\\

{\bf Step 4}. The proof of closeness for ${\rm range}(\mathcal{T}_{a^+,a^-})$ and ${\rm range}(\mathcal{T}^*_{a^+,a^-})$ are similar. Here we only prove that ${\rm range}(\mathcal{T}_{a^+,a^-})$ is closed from Step 4 to Step 7. Readers can prove the other result by applying the same argument\footnote{ In fact, the Fredholmness holds when property {\bf a}. and closeness of ${\rm range}(\mathcal{T}_{a^+,a^-})$ hold. It is redundant to check the closeness of ${\rm range}(\mathcal{T}^*_{a^+,a^-})$ because it will automatically hold (although this statement is non-trivial).}.\\

\begin{lemma}
Let $P_k:L^2(S^1)\rightarrow L^2(S^1)$ is the projection defined by
\begin{align*}
P_k:\sum f_n e^{int}\longmapsto \sum_{|n|\leq k}f_n e^{int}.
\end{align*}
Then we have
\begin{align*}
\mathcal{T}_{a^+,a^-}|_{{\rm range}(I-P_{2M}) }:(I-P_{2M})(L^2(S^1)) \rightarrow (I-P_M)(L^2(S^1))
\end{align*}
is injective.
\end{lemma}

\begin{proof}
Let $f\in (I-P_{2M})(L^{2}(S^1))$. Clearly $\mathcal{T}_{a^+,a^-}(f)\in (I-P_M)(L^2(S^1))$, so we should prove this map is one to one. We prove this by induction.\\
 
Suppose $f=\sum f_k e^{ikt} \in (I-P_M)(L^2(S^1))$, by the equation given by Lemma 6.9,
\begin{align*}
\left\{
\begin{array}{c}
\langle\langle B, \bar{V} \rangle\rangle_{M+1}=\langle b_q, v_{2M+1}\rangle= b^+_q p_{(2M+1)}+b^-_q \bar{p}_{-(2M+1)}= f_{M+1},\\
\langle\langle B^*, \bar{V}  \rangle\rangle_{M+1}=\langle \hat{b}_1, v_{2M+1}\rangle= \hat{b}^-_1 p_{(2M+1)}+\bar{b}^+_1 \bar{p}_{-(2M+1)}= \bar{f}_{-(M+1)}
\end{array}
\right.
\end{align*}
 we can solve $v_{(2M+1)}=(p_{2M+1},\bar{p}_{-(2M+1)})$, which is unique.\\

Now suppose $v_{(2M+1)},...,v_{M+k}$ are uniquely determined (where $k>M+1$) , we consider the equation
\begin{align*}
\langle\langle B, \bar{V} \rangle\rangle_{k+1}&=\langle b_q, v_{M+k+1}\rangle\\&= b^+_q p_{(M+k+1)}+b^-_q \bar{p}_{-(M+k+1)}+F_k(v_{(2M+1)},...,v_{M+k})= f_{k+1}\\
\langle\langle B^*, \bar{V}  \rangle\rangle_{k+1}&=\langle \hat{b}_1, v_{M+k+1}\rangle\\&= \hat{b}^-_1 p_{(M+k+1)}+\bar{b}^+_1 \bar{p}_{-(M+k+1)}+G_k(v_{(2M+1)},...,v_{M+k})= \bar{f}_{-(k+1)}.
\end{align*}
where $F_k(v_{(2M+1)},...,v_{M+k})= f_{k+1}$ and $G_k(v_{(2M+1)},...,v_{M+k})$ are determined by $\{v_{(2M+1)},...,v_{M+k}\}$. So we can solve $v_{(M+k+1)}$ uniquely.\\

Therefore, $\mathcal{T}_{a^+,a^-}|_{(I-P_{2M})(L^2(S^1)) }$ is an injective map from $(I-P_{2M})(L^2(S^1))$ to $(I-P_M)(L^2(S^1))$.
\end{proof}

Notice that, by using the notations of this lemma, we have
\begin{align*}
\mathcal{T}_{a^+,a^-}(P_{2M}(L^2(S^1)))\subset P_{3M}(L^2(S^1))
\end{align*}
and
\begin{align*}
\mathcal{T}_{a^+,a^-}((I-P_{2M})(L^2(S^1)))\subset (I-P_M)(L^2(S^1)).
\end{align*}
\ \\

{\bf Step 5}. Suppose we have $c^k\in L^2(S^1)$, $k\in\mathbb{N}$ such that 
\begin{align*}
\lim_{k\rightarrow\infty}\mathcal{T}_{a^+,a^-}(c^k)=f
\end{align*}
in $L^2$-sense for some $f\in L^2(S^1)$. Here we can assume that $c^k\perp \ker(\mathcal{T}_{a^+,a^-})$ without loss of generality. In this step, we will show that there exists $c$ such that $\mathcal{T}_{a^+,a^-}(c)=f$ when $\{c^k\}_{k\in\mathbb{N}}$ is bounded.\\

First of all, suppose that $\{c^k\}_{k\in\mathbb{N}}$ is bounded by some constant $K$. Denote by $\{v_p^k\}$ the corresponding pairing $\ell^2$-coefficients of $c^k$. We choose a subsequence, which is denoted by $c^k$ again, such that $\{v_p^k\}_{k\in \mathbb{N}}$ converges for all $p\leq 3J$ with some $J>M$. Let us say
\begin{align*}
v_p^k\rightarrow v_p
\end{align*}
for $p\leq 3J$ and choose $J$ large enough such that $v_p\neq 0$. Now, by Lemma 6.9, there is a unique solution $c$ such that
\begin{align*}
\mathcal{T}_{a^+,a^-}(c)=f
\end{align*}
where the corresponding $\ell^2$-coefficients of $c$ are $v_p$ when $p\leq 3M$. We shall show that $c$ is in $L^2(S^1)$.\\

Now, for any $r\in \mathbb{N}$, we have
\begin{align*}
\sum_{i\leq r}\|v_i\|^2_{\ell^2}\leq \sum_{i\leq r}\|v_i^k-v_i\|^2_{\ell^2}+\sum_{i\leq r}\|v^k_i\|^2_{\ell^2}\\
\leq \sum_{i\leq r}\|v_i^k-v_i\|^2_{\ell^2}+K.
\end{align*}
We notice that the first term converges to 0 as $k\rightarrow \infty$. Therefore, we have
\begin{align*}
\sum_{i\leq r}\|v_i\|^2_{\ell^2} \leq 1+K
\end{align*}
for any $r>0$. So $c\in L^2(S^1)$.\\

{\bf Step 6}. Suppose that $c^k$ is unbounded, say $\|c^k\|_{L^2(S^1)}\rightarrow \infty$ (by taking subsequence if this is necessary). we can take $\dot{c}^k=\frac{c^k}{\|c^k\|_{L^2(S^1)}}$ which satisfies $\mathcal{T}_{a^+,a^-}(\dot{c}^k)\rightarrow 0$. We will prove that this case will lead a contradiction in Step 6 and Step 7. This is the part that condition (\ref{F_4}) involved.\\

To begin with, we should define the following notation.
\begin{definition}
Let $\varepsilon>0$. We define the number $\tau=\inf\{ \sqrt{|a^+|^2+|a^-|^2}\}$ and the following sets:\\[1mm]
$\bullet$ $\Omega_1=\big\{|a^+|=|a^-|\big\}\subset S^1$, \\[1mm]
$\bullet$  $\Omega_{1,\varepsilon}=\big\{\big||a^-|-|a^+|\big|\leq\varepsilon\tau\big\}$,\\[1mm]
$\bullet$  $\Omega^+_{\varepsilon}=\big\{|a^+|>|a^-|+\varepsilon\tau\big\}$,\\[1mm]
$\bullet$  $\Omega^-_{\varepsilon}=\big\{|a^-|>|a^+|+\varepsilon\tau\big\}$.\\[1mm]
Clearly, we have $S^1=\Omega_{1,\varepsilon}\cup\Omega^+_{\varepsilon}\cup\Omega^-_{\varepsilon}$.
\end{definition}

Now we fix a $\varepsilon\leq \frac{1}{6}$ which will be specified later. We define 
 $\chi_{1,\varepsilon}$ to be a nonnegative real valued function defined on $S^1$ which has value $1$ in $\Omega_{1,\frac{\varepsilon}{2}}$ and $0$ in $\Omega^+_{\varepsilon}\cup\Omega^-_{\varepsilon}$. Also, define $\chi_{2,\varepsilon}$ to be $1$ in $\Omega^+_{\varepsilon}$ and $0$ on $\{|a^+|\leq |a^-|+\frac{\varepsilon}{2}\tau\}$. Define
 $\chi_{3,\varepsilon}$ to be $1$ in $\Omega^-_{\varepsilon}$ and $0$ on $\{|a^+|\geq |a^-|+\frac{\varepsilon}{2}\tau\}$. Moreover, suppose that $\{\chi_{i,\varepsilon}\}_{i=1,2,3}$ is a partition of unity, i.e.,
\begin{align*}
\chi_{1,\varepsilon}+\chi_{2,\varepsilon}+\chi_{3,\varepsilon}\equiv 1.
\end{align*}

In the the rest of Section 6, we will use $\|\cdot\|_{L^2}$ to denote $\|\cdot\|_{L^2(S^1)}$ (Similarly, we will use $\|\cdot\|_{L^{\infty}}$ to denote $\|\cdot\|_{L^{\infty}(S^1)}$). If we consider the $L^2$-norm of some functions defined on a set $\Omega\subset S^1$, we will use $\|\cdot\|_{L^2(\Omega)}$ to denote it.\\

{\bf Step 7}. In this Step, we will rephrase the statement in the first paragraph of Step 6 by some observation. Then we will prove that the new statement leads to a contradiction.\\

 First, by Step 6, we have a bounded sequence $\{\dot{c}^k\}$ with their $L^2$ norms equalling 1 and
\begin{align*}
\lim_{k\rightarrow\infty}\mathcal{T}_{a^+,a^-}(\dot{c}^k)=0
\end{align*}
in $L^2$ sense. Denote by $\{v_p^k\}$ the corresponding pairing $\ell^2$-coefficients of $\dot{c}^k$. For any $i\in\mathbb{Z}$ fixed, suppose that 
\begin{align*}
\limsup_{k\rightarrow\infty} |v^k_i|\neq 0,
\end{align*}
than we can use the argument in Step 5 by taking $J>i$ to achieve a contradiction. So we must have
\begin{align}
\lim_{k\rightarrow\infty} |v^k_i|= 0\label{FF_11}
\end{align}
for all $i\in\mathbb{Z}$.\\

Now, for any $L\in\mathbb{N}$ let $P_L:L^2(S^1)\rightarrow L^2(S^1)$ be the projection which maps $\sum_{l\in \mathbb{Z}}q_le^{ilt}$ to $\sum_{|l|\leq L}q_le^{ilt}$. By using (\ref{FF_11}), for any $L\in\mathbb{N}$ given, we can add the additional assumption into our statement: $P_L(\dot{c}^k)=0$ for some $k$. This number $L$, which will be specified later, is determined by $\varepsilon$ and $\chi_{i,\varepsilon}$. Now, we can rephrase the statement in the first paragraph of Step 6 as follows.
\begin{lemma}
There exists $L\in \mathbb{N}$ depending only on $a^{\pm}$, such that for any sequence $\{c^k\}_{k\in \mathbb{N}}\subset L^2(S^1)$ satisfying
\begin{align*}
\|c^k\|_{L^2}=1, \mbox{ }P_L(c^k)=0 \mbox{ for all } k\in\mathbb{N},
\end{align*}
we have $\inf_{k\in\mathbb{N}}\big\{\|\mathcal{T}_{a^+,a^-}(c^k)\|_{L^2}\big\}>C_0$, where $C_0$ depending only on the $C^1$-norm of $a^{\pm}$ and $\tau$.
\end{lemma}

\begin{proof}
We consider the function 
\begin{align*}
Q=\frac{a^+}{\bar{a}^-}
\end{align*}
defined on $\Omega_{1,\varepsilon}$. Extend this function as a $C^1$ function defined on $S^1$. Then we can approximate it by its first $N_2$ Fourier modes, $S$, such that the $L^2$-norm and $L^{\infty}$-norm of $|Q-S|$ are $O(\varepsilon)$. Notice that by taking $\varepsilon$ small, we can assume
\begin{align}
\|S\|_{L^{\infty}}\leq 2.\label{FF_12}
\end{align}
\ \\
Since $\chi_{1,\varepsilon}+\chi_{2,\varepsilon}+\chi_{3,\varepsilon}\equiv 1$, we have
\begin{align*}
1=\|c^k\|_{L^2}\leq \|\chi_{1,\varepsilon}c^k\|_{L^2}+\|\chi_{2,\varepsilon}c^k\|_{L^2}+\|\chi_{3,\varepsilon}c^k\|_{L^2}.
\end{align*}
Therefore, there exists $i\in\{1,2,3\}$ such that $\|\chi_{i,\varepsilon}c^k\|_{L^2}\geq \frac{1}{3}$ infinite many times. We take this subsequence and renumber them consecutively from 1. Since $\chi_{i,\varepsilon}$ is a smooth function, we approximate $\chi_{i,\varepsilon}$ by its first $N_1$ Fourier mode, denoted by $\zeta_{i,\varepsilon}$, such that $\|\chi_{i,\varepsilon}-\zeta_{i,\varepsilon}\|_{L^2}\leq \varepsilon<\frac{1}{6}$ and $\sup|\chi_{i,\varepsilon}-\zeta_{i,\varepsilon}|\leq \varepsilon$, so by Cauchy's inequality, we have
$\|\zeta_{i,\varepsilon}c^k\|_{L^2}\geq \frac{1}{6}$. Notice that by taking $\varepsilon<\frac{1}{2}$, we can assume
\begin{align}
\|\zeta_{i,\varepsilon}\|_{L^{\infty}}\leq 2.\label{FF_13}
\end{align}

 Now we shall prove Lemma 6.12 case by case.\\
\ \\
{\bf Case 1}. If $i=1$,
\begin{align}
\|\zeta_{1,\varepsilon}c^k\|_{L^2}\geq \frac{1}{6}.\label{FF_14}
\end{align}
Meanwhile, we have
\begin{align}
\zeta_{1,\varepsilon}\mathcal{T}_{a^+,a^-}(c^k)=\zeta_{1,\varepsilon}f^k\label{FF_15}
\end{align}
for some $f^k\in L^2(S^1)$. The left hand side of (\ref{FF_15}) can be written as
\begin{align*}
\zeta_{1,\varepsilon}\mathcal{T}_{a^+,a^-}(c^k)=\mathcal{T}_{a^+,a^-}(\zeta_{1,\varepsilon}c^k)+(\zeta_{1,\varepsilon}\mathcal{T}_{a^+,a^-}(c^k)-\mathcal{T}_{a^+,a^-}(\zeta_{1,\varepsilon}c^k)).
\end{align*}
The second term on the right can be written as $[\zeta_{1,\varepsilon}, \mathcal{T}_{a^+,a^-}](c^k)$. Let 
\begin{align*}
\zeta_{1,\varepsilon}=\sum_{|l|\leq N_1}\varsigma_le^{ilt}.
\end{align*}
Then we get
\begin{align*}
[\zeta_{1,\varepsilon}, \mathcal{T}_{a^+,a^-}](c^k)&=\zeta_{1,\varepsilon}((c^k)^{aps})-(\zeta_{1,\varepsilon}c^k)^{aps}\\
&=\sum_{n\in \mathbb{Z}}\Big[(\sum_{|j|\leq N_1}\varsigma_j\mbox{sign}(n-j)p^k_{n-j})-(\sum_{|j|\leq N_1}\mbox{sign}(n)\varsigma_jp_{n-j})\Big]e^{int}\\
&=\sum_{|n|\leq N_1}\pm 2(\sum_{|j|\leq N_1}\varsigma_j p_{n-j})e^{int}.
\end{align*}
So this term will be $0$ by taking $L>2N_1$.\\

Therefore, we have
\begin{align}
\mathcal{T}_{a^+,a^-}(\zeta_{1,\varepsilon}c^k)&=\zeta_{1,\varepsilon}f^k\label{FF_16}\\
&=\bar{a}^-\zeta_{1,\varepsilon}c^k-a^+\overline{(\zeta_{1,\varepsilon}c^k)^{aps}}.\nonumber
\end{align}
Dived both side of (\ref{FF_16}) by $\bar{a}^-$, we have
\begin{align*}
\zeta_{1,\varepsilon}c^k-\frac{a^+}{\bar{a}^-}\overline{(\zeta_{1,\varepsilon}c^k)^{aps}}=\frac{\zeta_{1,\varepsilon}f^k}{\bar{a}^-}
\end{align*}
on $\Omega_{1,\varepsilon}$. Notice that $|\bar{a}^-|\geq \tau(1-\varepsilon)$ on $\Omega_{1,\varepsilon}$. Now, because $\frac{a^+}{\bar{a}^-}=Q$ on $\Omega_{1,\varepsilon}$ and $|\zeta_{1,\varepsilon}|\leq \varepsilon$ outside $\Omega_{1,\varepsilon}$, so we have
\begin{align*}
\zeta_{1,\varepsilon}c^k-Q\overline{(\zeta_{1,\varepsilon}c^k)^{aps}}=\frac{\zeta_{1,\varepsilon}f^k}{a^{*}}+O_{L^2}(\varepsilon)
\end{align*}
for some $a^*\in C^1(S^1;\mathbb{C})$ with $a^*=\bar{a}^-$ on $\Omega_{1,\varepsilon}$ and
\begin{align}
|a^*|\geq \frac{1}{2}\tau(1-\varepsilon)\label{FF_17}
\end{align}
every where. Here $O_{L^2}(\varepsilon)$ term has its $L^2$-norm with order $O(\varepsilon)$. Write $Q=S+(Q-S)$ where the $L^2$-norm and $L^{\infty}$-norm of $Q-S$ are $O(\varepsilon)$. So we have
\begin{align}
\zeta_{1,\varepsilon}c^k-S\overline{(\zeta_{1,\varepsilon}c^k)^{aps}}=\frac{\zeta_{1,\varepsilon}f^k}{a^*}+O_{L^2}(\varepsilon).\label{FF_18}
\end{align}
\\

Finally, we define the projections $P^{\pm}:L^2(S^1)\rightarrow L^2(S^1)$ by
\begin{align*}
P^+(\sum_{l\in\mathbb{Z}}p_le^{ilt})=\sum_{l>0}p_le^{ilt},\\
P^-(\sum_{l\in\mathbb{Z}}p_le^{ilt})=\sum_{l<0}p_le^{ilt}.
\end{align*}
and define $A^k:=\zeta_{1,\varepsilon}c^k$, $B^k:=\overline{(\zeta_{1,\varepsilon}c^k)^{aps}}$. Then (\ref{FF_18}) implies
\begin{align}
P^{\pm}A^k+P^{\pm}SB^k=P^{\pm}\Big(\frac{\zeta_{1,\varepsilon}f^k}{a^*}\Big)+O_{L^2}(\varepsilon).\label{FF_19}
\end{align}
Notice that
\begin{align}
P^{\pm}(A^k)=\overline{P^{\mp}(B^k)}.\label{FF_20}
\end{align}
Let $S=\sum_{|n|\leq N_2}S_ne^{int}$. We shall also notice that
\begin{align}
[P^+,S]B^k&=(P^+SB^k-SP^+B^k)\label{FF_21}\\
&=\sum_{n>0}(\sum_{|j|\leq N_2}S_jB_{n-j}e^{int})-\sum_{n\geq j}(\sum_{|j|\leq N_2}S_jB_{n-j}e^{int})\nonumber\\
&=\pm\sum_{ |n| < N_2}\sum_{|j|\leq N_2}S_jB_{n-j}e^{int}=0\nonumber
\end{align}
when we take $L>2N_1+2N_2$.\\

Therefore, by (\ref{FF_20}) and (\ref{FF_21}), the equation (\ref{FF_19}) implies
\begin{align*}
&P^+A^k+SP^+B^k=O_{L^2}(\varepsilon)+P^+\Big(\frac{\zeta_{1,\varepsilon}f^k}{a^*}\Big),\\
&\overline{P^+B^k}-S\overline{P^+A^k}=O_{L^2}(\varepsilon)+\overline{P^-\Big(\frac{\zeta_{1,\varepsilon}f^k}{a^*}\Big)}.
\end{align*}
Using these two equations, we have
\begin{align}
&P^+A^k+SP^+B^k-S(P^+B^k-\bar{S}{P^+A^k})\label{FF_22}\\=&(1+|S|^2)P^+A^k=O_{L^2}(\varepsilon)+P^+\Big(\frac{\zeta_{1,\varepsilon}f^k}{a^*}\Big)+\bar{S}P^-\Big(\frac{\zeta_{1,\varepsilon}f^k}{a^*}\Big).\nonumber
\end{align}
\ \\

By (\ref{FF_14}), $\|A^k\|_{L^2}\geq \frac{1}{6}$. We have either $\|P^+A^k\|_{L^2}>\frac{1}{12}$ or $\|P^-A^k\|_{L^2}>\frac{1}{12}$. Suppose that $\|P^+A^k\|_{L^2}>\frac{1}{12}$. By (\ref{FF_12}), (\ref{FF_13}), (\ref{FF_17}) and (\ref{FF_22}), we have
\begin{align*}
\frac{1}{12}\leq \|P^+A^k\|_{L^2}&\leq \|(1+|S|^2)P^+A^k\|_{L^2}\\
&\leq O(\varepsilon)+\Big\|P^+\Big(\frac{\zeta_{1,\varepsilon}f^k}{a^*}\Big)\Big\|_{L^2}+\Big\|\bar{S}P^-\Big(\frac{\zeta_{1,\varepsilon}f^k}{a^*}\Big)\Big\|_{L^2}\\
&\leq O(\varepsilon)+\frac{1}{\tau}16\|f^k\|_{L^2}
\end{align*}
for $\varepsilon$ arbitrary. So we have
\begin{align*}
\|f^k\|_{L^2}\geq \frac{\tau}{193}
\end{align*}
for all $k$ by taking $\varepsilon$ sufficiently small. The same result holds when $\|P^-A^k\|_{L^2}>\frac{1}{12}$. By taking $C_0=\frac{\tau}{193}$, we prove Lemma 6.12.\\
\ \\
{\bf Case 2}. If $i=2$ (Or $i=3$. These two cases are similar), we have
\begin{align*}
\zeta_{2,\varepsilon}\mathcal{T}_{a^+,a^-}(c^k)&=\mathcal{T}_{a^+,a^-}(\zeta_{2,\varepsilon}c^k)=\zeta_{2,\varepsilon}f^k\\
&= \bar{a}^-(\zeta_{2,\varepsilon}c^k)-a^+\overline{(\zeta_{2,\varepsilon}c^k)^{aps}}=\zeta_{2,\varepsilon}f^k
\end{align*}
Dividing both side by $a^+$ and notice that $|\frac{a^-}{a^+}|\leq 1-\frac{\tau}{2}\varepsilon$ on $\Omega_{\varepsilon}^+$, we have
\begin{align*}
\Big\|\frac{\zeta_{2,\varepsilon}f^k}{a^+}\Big\|_{L^2(\Omega_{\varepsilon}^+)}&=\Big\|\frac{\bar{a}^-}{a^+}(\zeta_{2,\varepsilon}c^k)-\overline{(\zeta_{2,\varepsilon}c^k)^{aps}}\Big\|_{L^2(\Omega_{\varepsilon}^+)}\\
&\geq \|\overline{(\zeta_{2,\varepsilon}c^k)^{aps}}\|_{L^2(\Omega_{\varepsilon}^+)}-(1-\frac{\tau}{2}\varepsilon)\|\zeta_{2,\varepsilon}c^k\|_{L^2(\Omega_{\varepsilon}^+)}\\
&= \frac{\tau}{2}\varepsilon\|\zeta_{2,\varepsilon}c^k\|_{L^2(\Omega_{\varepsilon}^+)}.
\end{align*}
Therefore, we have
\begin{align*}
\frac{\tau}{2}\varepsilon(\frac{1}{3}-O(\varepsilon))&\leq \frac{\tau}{2}\varepsilon\|\zeta_{2,\varepsilon}c^k\|_{L^2(\Omega_{\varepsilon}^+)}\\
&\leq \Big\|\frac{\zeta_{2,\varepsilon}f^k}{a^+}\Big\|_{L^2(\Omega_{\varepsilon}^+)} \leq \frac{2}{\tau}\|f^k\|_{L^2}.
\end{align*}
Then fix a small $\varepsilon$ such that the left end is a positive constant. We get
\begin{align*}
\|f^k\|_{L^2}\geq C\tau^2
\end{align*}
where $C$ is a constant depending only on $C^1$-norm of $a^{\pm}$.
\end{proof}

\begin{remark}
Notice that the lower bounded $C_0$ is a constant depending on  $\tau$, $\|a^+\|_{C^1}$ and $\|a^-\|_{C^1}$. We will write $C_0=C_0(\tau,\|a^+\|_{C^1},\|a^-\|_{C^1})$. Moreover, if we have a sequence of $\{a^{\pm,(k)}\}_{k\in\mathbb{N}}$ such that the corresponding $\tau^{(k)}$, $\|a^{\pm,(k)}\|_{C^1}$ are bounded and do not accumulate at 0, then \begin{align*}
\inf_{k\in\mathbb{N}}\big\{C_0(\tau^{(k)},\|a^{+,(k)}\|_{C^1},\|a^{-,(k)}\|_{C^1})\big\}>0.
\end{align*}
\end{remark}

\subsection{The general case}
Now we turn to the proof of the general case: $a^{\pm}$ have infinite many Fourier modes. We will prove the following theorem.

\begin{theorem}
Let
\begin{align*}
\mathcal{T}_{a^+,a^-}(c)=\bar{a}^-c-a^+\overline{c^{aps}}
\end{align*}
be the operator from $L^2(S^1;\mathbb{C})$ to $L^2(S^1;\mathbb{C})$,
with the following constraint:
\begin{align}
|a^+|^2+|a^-|^2> 0\label{F_16}.
\end{align}
Suppose that
\begin{align*}
\|a^+\|_{C^1}, \|a^-\|_{C^1}<\infty.
\end{align*}
 Then we have $\mathcal{T}_{a^+,a^-}$ is a Fredholm operator with the index 0.
\end{theorem}
Recall that we have the following well-known equivalent statement for the Fredholm operators \cite{I}. This will be the key lemma for proving Theorem 6.14.
\begin{lemma}
Let $X$ be a Hilbert space and $F\in Hom(X)$. Then $F$ is a Fredholm operator iff there is $S\in Hom(X)$ such that
\begin{align*}
SF=FS=I \mbox{ } mod(Com(X))
\end{align*}
where $Com(X)$ is the subring(ideal) consisted by all compact operators mapping from $X$ to itself.
\end{lemma}

\begin{proof}[Proof of Theorem 6.14]\ \\
\emph{\bf Step 1}. To prove this theorem, notice that we can approximate the operator $\mathcal{T}_{a^+,a^-}$ by a sequence of Fredholm operators
$\{\mathcal{T}_{a^{+,(k)},a^{-,(k)}}\}_{k\in\mathbb{N}}$, where $a^{\pm,(k)}$ are summations of the first $k$ Fourier modes of $a^{\pm}$. Since that the Fredholm operators form an open set in $Hom(L^2(S^1;\mathbb{C}))$, this is insufficient to say that $\mathcal{T}_{a^+,a^-}$ itself is a Fredholm operator. However, by using Lemma 6.15, we have the following argument:\\

 Since $\mathcal{T}_{a^{+,(k)},a^{-,(k)}}$ is a Fredholm operator for all $k\in\mathbb{N}$ by Theorem 6.6, by Lemma 6.15, there exists a sequence of right inverse $\{S^{k}\}_{k\in\mathbb{N}}$ such that
\begin{align*}
\mathcal{T}_{a^{+,(k)},a^{-,(k)}}S^{k}=I \mbox{ } mod(Com(L^2(S^1;\mathbb{C}))).
\end{align*} 
Suppose that $\{\|S^k\|\}_{k\in\mathbb{N}}$ is bounded uniformly by a number $K$. For any $\varepsilon>0$, there exists a constant
$N>0$ such that $\|\mathcal{T}_{a^+,a^-}-T_{a^{+,(k)},a^{-,(k)}}\|\leq \varepsilon$ for all $k\geq N$. So we have
\begin{align*}
\mathcal{T}_{a^+,a^-}S^N=T_{a^{+,(N)},a^{-,(N)}}S^N+O(\varepsilon)S^N=I+O(\varepsilon)S^N \mbox{ } mod(Com(L^2(S^1;\mathbb{C}))).
\end{align*}
Since $\|O(\varepsilon)S^N\|\leq O(\varepsilon)K$, we can choose $\varepsilon$ small enough such that $\|O(\varepsilon)S^N\|\leq \frac{1}{2}$. Therefore, we have $I+O(\varepsilon)S^N$ invertible. Let $V$ be the inverse of $I+O(\varepsilon)S^N$, we have
\begin{align*}
\mathcal{T}_{a^+,a^-}S^NV=I \mbox{ } mod(Com(L^2(S^1;\mathbb{C}))).
\end{align*}
So $\mathcal{T}_{a^+,a^-}$ has the right inverse $S^NV$ modulo the ideal of compact operators. Similar result is true for the existence of the left inverse. Therefore, it is a Fredholm operator.\\

{\bf Step 2}. In Step 1, we prove that if there is a uniform bound for $\{\|S^k\|\}_{k\in\mathbb{N}}$, then Theorem 6.14 will be immediately true. To prove this claim, we should know how to construct inverse $S^k$ for each $k$. In the following paragraphs, we use $\mathcal{T}^k$ to denote the operator $\mathcal{T}_{a^{+,(k)},a^{-,(k)}}$ and $\mathcal{T}$ to denote $\mathcal{T}_{a^+,a^-}$.\\

 A standard way to construct $S^k$ is to use the decomposition 
\begin{align*}
 L^2(S^1)=N(\mathcal{T}^k)\oplus N(\mathcal{T}^k)^{\perp}=R(\mathcal{T}^k)\oplus N(\mathcal{T}^{k*}).
\end{align*} 
By standard Fredholm alternative, we have 
\begin{align*}
\mathcal{T}: N(\mathcal{T}^k)^{\perp}\rightarrow R(\mathcal{T}^k)
\end{align*}
is a bijection. Therefore, by open mapping theorem (see p. 168 in \cite{O}), there is a bounded inverse map $\hat{S}^k: R(\mathcal{T}^k)\rightarrow N(\mathcal{T}^k)^{\perp} $. Now, we define $S^k$ to be $\hat{S}^k\circ P_{R(\mathcal{T}^k)}$.\\
 
 Here we should imitate this idea to construct $S^k$. Here we know that 
 \begin{align*}
 {T}^k:(I-P_{L})(L^2(S^1))\rightarrow \mathcal{T}^k((I-P_{L})(L^2(S^1)))\subset (I-P_{L-k})(L^2(S^1))
 \end{align*}
is a bijection, where $L$ is the number given by Lemma 6.12. Moreover, we can prove that $\mathcal{T}^k((I-P_{2k})(L^2(S^1)))$ is a closed subspace by using the argument in Step 5,6,7 in Section 6.3. Therefore, we have a bounded inverse $\hat{R}^k:\mathcal{T}^k((I-P_{L})(L^2(S^1)))\rightarrow (I-P_{L})(L^2(S^1))$. Meanwhile, Remark 6.13 tells us that $\hat{R}^k$ have a uniform bounded norm. Now we set our $S^k$ to be $\hat{R}^k\circ P_{\mathcal{T}^k((I-P_{L})(L^2(S^1)))}$. So $\{\|S^k\|\}_{k\in\mathbb{N}}$ has a uniform bound.\\
 
 {\bf Step 3}. Finally, we should prove that $S^k$ is actually an inverse of $\mathcal{T}^k$, modulo the ideal of compact operators. To prove this, just recall that both $(I-P_L)(L^2(S^1))$ and $\mathcal{T}^k((I-P_L)(L^2(S^1)))$ are finite codimensional. We define 
\begin{align*}
A=(I-P_L)(L^2(S^1)),\\
B=\mathcal{T}^k((I-P_L)(L^2(S^1)))
\end{align*} 
for a while, then $L^2(S^1)=A\oplus A^{\perp}=B\oplus B^{\perp}$ and
 \begin{align*}
 (\mathcal{T}^kS^k-I)(v)=0 \mbox{ for any } v\in B.
 \end{align*}
So for any bounded sequence $\{v^k=(v_1^k,v_2^k)\in B\oplus B^{\perp}=L^2(S^1)\}_{k\in\mathbb{N}}$, we have
\begin{align*}
 (\mathcal{T}^kS^k-I)(v^k)= (\mathcal{T}^kS^k-I)(v_2^k)
\end{align*}
where $\{v_2^k\}$ lies in a finite dimensional space $B^{\perp}$. We can get a convergence subsequence of $\{v_2^k\}$ easily. This implies
\begin{align*}
 (\mathcal{T}^kS^k-I)=0 \mbox{ } mod(Com(L^2(S^1;\mathbb{C}))).
\end{align*}
Similarly, we have $ (S^k\mathcal{T}^k-I)=0 \mbox{ } mod(Com(L^2(S^1;\mathbb{C})))$, too. Therefore, we finish our proof for the Fredholmness.\\

The computation of the index is simple: One can choose $a^+$ be nonzero everywhere and $a^-=0$. In this case, $\mathcal{T}_{a^+,a^-}$ is invertible. So $index(\mathcal{T}_{a^+,a^-})=0$. 
\end{proof}

\begin{remark}
Remember that $(a^{+},a^-)$ is the leading term of an $L^2_1(M\setminus\Sigma;\mathcal{S}\otimes\mathcal{I})$ harmonic section, so by Proposition 3.9, it is smooth. Meanwhile, notice that $\mathcal{T}_{a^+,a^-}$ maps from $L^2_k(S^1)$ to $L^2_k(S^1)$ for any $k\in\mathbb{N}$. We can show that all these maps are Fredholm by using the same argument. In addition, since $C^{\infty}(S^1;\mathbb{C})$ is dense in $L^2_k(S^1)$ for all $k\geq 0$, $\mathcal{T}_{a^+,a^-}:L^2_k(S^1)\rightarrow L^2_k(S^1)$ share the same the kernel and cokernel in $C^{\infty}(S^1;\mathbb{C})$ for all $k\geq 0$. When the metric is not Euclidean, by Lemma 5.1, we have $a^{\pm}\in L^2_2(S^1)$. So $\mathcal{T}_{a^+,a^-}$ is a Fredholm operator form $L^2_k(S^1)$ to $L^2_k(S^1)$ when $k\leq 2$.
\end{remark}

\subsection{Relations between $\mathcal{T}$ and the original equation} Recall that by the argument in Section 6.2, we want to solve the equation
\begin{align*}
a^+\eta+c=-2h^+,\\
a^-\bar{\eta}+c^{aps}=-2h^-
\end{align*}
which will give us the equation $\mathcal{T}_{a^+,a^-}(c)=-2(\bar{a}^-h^+-a^+\bar{h}^-)$. Here we define the map 
\begin{align*}
\mathcal{J}:L^2(S^1;\mathbb{C}^2)\rightarrow L^2(S^1;\mathbb{C})
\end{align*}
by $\mathcal{J}(h^+,h^-)=-2(\bar{a}^-h^+-a^+\bar{h}^-)$ and the map 
\begin{align*}
\mathcal{O}:\ker(\mathcal{T})\rightarrow L^2(S^1;\mathbb{C})
\end{align*}
by
\begin{align}
\mathcal{O}(c)=-\frac{\bar{a}^+c+a^-\overline{c^{aps}}}{|a^+|^2+|a^-|^2}.\label{FF_23}
\end{align} 
This map will give us $\eta$ when $h^{\pm}=0$.\\

Now, using the notation in Section 6.1, we can always be decomposed the pair $(u^+,u^-)\in L^2(S^1;\mathbb{C})\times L^2(S^1;\mathbb{C})$ as $\pi^+(u^+,u^-)+\pi^-(u^+,u^-)$. By using this proposition and the Fredholmness of $\mathcal{T}_{a^+,a^-}$, we can find a finite dimensional vector space $\mathbb{U}\subset Exp^+$ such that $\text{range}(\mathcal{T}_{a^+,a^-})\oplus \mathcal{J}(\mathbb{U})=L^2(S^1;\mathbb{C})$.

\section{Proof of the main theorem: Part I}

So far, we have proved the Fredholmness for the linearization of $\mathfrak{M}$. To obtain the proof of Theorem 1.5, we need a version of implicit function theorem. Unfortunately, these types of theorems are usually established based on a suitable Sobolev norm or a suitable H\"older norm (see ,\cite{R} for example). In our case, however, since Theorem 1.5 involves the $L^2_1$-Sobolev norm for $\mathbb{Z}/2$-harmonic spinors and $C^1$-norm for the embedding circles $\Sigma$, we should build an ``hybrid'' type of our own.\\

In Section 7, we prove Theorem 1.5 of the version without showing $f$ is $C^1$. In the Section 8, we will prove that $f$ is a $C^1$ map.\\

The argument in this section assumes that the metric $g$ defined on a tubular neighborhood is Euclidean. The case with a general metric is more complicated but follows the similar argument. To be precise, in the general case, we need apply Lemma 5.1 and remark 5.2 whenever we apply Proposition 3.6 and replace Proposition 4.4 and Proposition 4.6 by Proposition 5.4 and Proposition 5.5 in our argument. See Appendix 9.1 for details.

\subsection{Reformulate $\mathbb{K}_0$ and $\mathbb{K}_1$} The definition of $\mathbb{K}_1$ and $\mathbb{K}_0$ are given by (\ref{F_12}) and (\ref{F_13}). We also notice that $\ker(D|_{L^2})^0=B(\ker(D|_{L^2})^0)\oplus \ker(D|_{L^2_1})$. Use $\mathbb{H}_0$ to denote the space $\mathcal{O}[ B(\ker(\mathcal{T}_{a^+,a^-}\circ B))]$ (see (\ref{FF_23}) for the definition of $\mathcal{O}$). In addition, we define $\mathbb{H}_1=\text{coker}(\mathcal{T}_{a^+,a^-}\circ B)$. Then $\mathbb{K}_0$ and $\mathbb{K}_1$ can be written as follows ($\mathbb{K}_1$ remains the same as (\ref{F_13})):
\begin{align*}
&\mathbb{K}_0\cong\mathbb{H}_0\oplus \ker(D|_{L^2_1});\\
&\mathbb{K}_1=\mathbb{H}_1\oplus \ker(D|_{L^2_1}).
\end{align*}
To prove this, we notice that the map $\mathcal{O}$ is injective on $\ker(\mathcal{T}_{a^+,a^-})$ since
the equation
\begin{align*}
\left\{
\begin{array}{cc}
\bar{a}^+c+a^-\overline{c^{aps}}&=0,\\
\bar{a}^-c-a^+\overline{c^{aps}}&=0
\end{array}
\right.
\end{align*}
implies $c=0$. So 
\begin{align*}
\mathbb{H}_0\oplus \ker(D|_{L^2_1})\cong B(\ker(\mathcal{T}_{a^+,a^-}\circ B))\oplus \ker(D|_{L^2_1})\cong \ker(\mathcal{T}_{a^+,a^-}\circ B).
\end{align*}

\subsection{Basic setting}
Before we start our argument, we define some notation. \\

First, in the following paragraphs, we fix two constants $\mathfrak{r}<\frac{R}{4}$, $T>1$ in the beginning. The precise values of $\mathfrak{r}$ and $T$ will be specified later (again, $\mathfrak{r}$ can be assumed to decrease between each successive appearance and $T$ can be assumed to increase between each successive appearance in this section). Moreover, let us assume that $\|\mathcal{T}^{-1}_{a^+,a^-}|_{\text{range}(\mathcal{T}^{-1}_{a^+,a^-})}\|\leq 1$.\\

Second, we suppose that there exists $t_0>0$ which is the upper bound for $s$. The precise value of $t_0$ can be assumed to decrease between each successive appearance.
\begin{definition}
For any $A\subset M$, we say that a section $\mathfrak{u}:[0,t_0]\times A\rightarrow \mathcal{S}\otimes\mathcal{I}$ is in $C^{\omega}([0,t_0];L^2_i(A;\mathcal{S}\otimes\mathcal{I}))$ if and only if $\|\mathfrak{u}(s,\cdot)\|_{L^2_i(A;\mathcal{S}\otimes\mathcal{I})}<\infty$ for all $s\in [0,t_0]$ and the function $f(s):=\|\mathfrak{u}(s,\cdot)\|_{L^2_i(A;\mathcal{S}\otimes\mathcal{I})}$ varies analytically on $[0,t_0]$. (The remainder of Taylor series will converge to zero in $L^2_i$-sense).
\end{definition}
\begin{definition}
For any $i\in \mathbb{N}$ $\kappa>0$, we define
\begin{align}
&\mathfrak{A}_{i+1}^{\kappa}=\big\{\mathfrak{f}\in C^{\omega}([0,t_0];L^2_{-1}(M\setminus N_{R};\mathbf{\mathcal{S}}\otimes \mathcal{I}))\big|\|\mathfrak{f}(s,\cdot)\|_{L^2_{-1}}\leq \frac{\kappa}{T^\frac{5i}{2}}\big\};\label{GGG_1}\\
&\mathfrak{B}_{i+1}^{\kappa}=\big\{\mathfrak{f}\in C^{\omega}([0,t_0];L^2(N_{\frac{\mathfrak{r}}{T^i}}-N_{\frac{\mathfrak{r}}{T^{i+1}}};\mathbf{\mathcal{S}}\otimes \mathcal{I}))\big|\|\mathfrak{f}(s,\cdot)\|_{L^2_{-1}}\leq \frac{\kappa}{T^\frac{5i}{2}}\big\};\label{GGG_2}\\
&\mathfrak{C}_{i+1}^{\kappa}=\big\{\mathfrak{f}\in C^{\omega}([0,t_0];L^2(N_{\frac{\mathfrak{r}}{T^{i}}};\mathbf{\mathcal{S}}\otimes \mathcal{I}))\big|\label{GGG_3}\\
&\mbox{ }\mbox{ }\mbox{ }\mbox{ }\mbox{ }\mbox{ }\mbox{ }\mbox{ }\mbox{ }\mbox{ }\mbox{ }\mbox{ }\mbox{ }\mbox{ }\|\mathfrak{f}(s,\cdot)\|^2_{L^2(N_{r_1}-N_{r_2})}\leq \kappa({r_1}^{3}-{r_2}^3)\big(\frac{\mathfrak{r}}{T^{i}}\big)^{\frac{1}{4}} \mbox{ for all } r_2<r_1\leq \frac{\mathfrak{r}}{T^{i}}\big\}.\nonumber
\end{align}
\end{definition}
Since these sets are subsets of $C^{\omega}([0,t_0];L^2_{-1}(M\setminus \Sigma;\mathbf{\mathcal{S}}\otimes \mathcal{I}))$, we use $s\mathfrak{C}_{i+1}^{\kappa}$ to denote the collection of $s\mathcal{C}$ for 
all $\mathcal{C}\in\mathfrak{C}_{i+1}^{\kappa}$,  $s^2\mathfrak{B}_{i+1}^{\kappa}+s\mathfrak{C}_{i+1}^{\kappa}$ to denote the collection of $s^2\mathcal{B}+s\mathcal{C}$ where $\mathcal{B}\in\mathfrak{B}_{i+1}^{\kappa}$ and $\mathcal{C}\in\mathfrak{C}_{i+1}^{\kappa}$, etc.\\

 Third, suppose that we perturb the metrics $g$ on the region $M\setminus N_R$ analytically with the parameter $s$. Let us call this family of perturbed metric $g_s$. We use the notation $D_{pert}=D+T^s$ to denote the Dirac operator perturb by metric. The operator $T^s:L^2(M\setminus\Sigma;\mathcal{S}\otimes\mathcal{I})\rightarrow L^2_{-1}(M\setminus\Sigma;\mathcal{S}\otimes\mathcal{I})$ will be a first order differential operator with its operator norm $\|T^s\|\leq Cs$.\\

Therefore, given $((g,\Sigma,e),\psi)\in\mathfrak{M}$, we have
\begin{align}
D_{pert}\psi=s\mathfrak{f}_0\label{GGG_4}
\end{align}
for some $\mathfrak{f}_0=T^s(\psi)\in C^{\omega}( [0,t_0]; L^2(M\setminus\Sigma;\mathcal{S}\otimes\mathcal{I}))$.\\

To prove Theorem 1.5, we should prove the following Proposition:
\begin{pro}
There exists $\varepsilon>0$ with the following significance: For any $\xi\in\mathbb{H}_0$ with $\|\xi\|_{L^2_2}= \varepsilon$ and $\hat{\psi}\in \ker(D|_{L^2_1})$ with $\|\hat{\psi}\|_{L^2_1}<\varepsilon$, there exist\\
a unique 
\begin{align*}
\eta_s=\eta_s(\xi,\hat{\psi})\in C^{\omega}([0,t_0];C^1(S^1;\mathbb{C}))
\end{align*}
and a unique 
\begin{align*}
\mathfrak{k}_s=\mathfrak{k}_s(\xi,\hat{\psi})\in \big\{\mathfrak{u}\in L^2(M\setminus\Sigma;\mathcal{S}\otimes\mathcal{I})\big|&B(\mathfrak{u})\in L^2(S^1;\mathbb{C}^2)\\&\mbox{ and }B(\mathfrak{u})\perp \ker(\mathcal{T}_{a^+,a^-})\big\}
\end{align*}
such that
\begin{align}
D_{pert,\eta_s}(\psi+\hat{\psi}+s\mathfrak{k}_s)=0\mbox{ }mod(\ker(D|_{L^2_1}))\label{r4_7_2}
\end{align}
for all $s\in [0,t_0]$ with the constraint $\eta_s=s\xi+\eta_s^{\perp}$ for some $\eta_s^{\perp}\perp \mathbb{H}_0$.
\end{pro}

 By using this Proposition, we can define the map $f$ by 
\begin{align}
 f(g_s, s\xi, \hat{\psi})=\big(\mathcal{T}_{a^+,a^-}\circ B(s\mathfrak{k}_{t_0}), P_{\ker(D|_{L^2_1})}(D_{pert,\eta_s}(\psi+\hat{\psi}+s\mathfrak{k}_s)\big)\label{GGG_5}
 \end{align} 
 for any $\hat{\psi}\in \ker(D|_{L^2_1})$ with $\|\hat{\psi}\|_{L^2_1}<\varepsilon$. Here $P_{\ker(D|_{L^2_1})}$ is the $L^2$-orthogonal projection from $L^2$ to $\ker(D|_{L^2_1})$. 
Therefore, once we prove Proposition 7.3, we will obtain the main theorem if we prove $f$ is $C^1$ and it induces a homeomorphism $\Upsilon:f^{-1}(0)\rightarrow\mathcal{N}\cap\mathfrak{M}$ for some $\mathcal{N}$, neighborhood of $(g,\Sigma,\psi)$ in $\mathcal{E}$.\\

We will complete the proof of Proposition 7.3 in  Section 7. In the Section 8, we will prove $f$ is $C^1$ and it induces the homeomorphism $\Upsilon$.
\begin{remark}
In fact, the $\xi$ we choose in our claim can be replaced by a smooth family $\xi(s):[0,t_0]\rightarrow \mathbb{H}_0$ with $\|\xi(s)\|_{L^2_2}=\varepsilon$ and $\psi$ can be chosen as a smooth family $\psi(s)\in \ker(D|_{L^2_1})$. The argument in the rest of this section will still hold under this setting.
\end{remark}

\subsection{Proof of Proposition 7.3: First order approximation of $\eta_s$ and $\mathfrak{k}_s$} We start our proof of Proposition 7.3. In Section 7.3, we will introduce two constant $\kappa_0$ and $\kappa_1$. The precise value of these two constants can be assumed to increase between each successive appearance. Then, after Section 7.3, we will fixed these two constants in the rest of this paper.\\

{\bf Step 1}. Let $\mathfrak{f}_0$ be the section defined in (\ref{GGG_4}). By using Proposition 6.2, there exists $\mathfrak{h}_0\in L^2$ such that
\begin{align}
D\mathfrak{h}_0=\mathfrak{f}_0 \mbox{ } mod(\ker(D|_{L^2_1})).\label{GGG_6}
\end{align}
Combine (\ref{GGG_4}) (\ref{GGG_6}) and the fact $D_{pert}=D+T^s$, we have 
\begin{align}
D_{pert}(\psi-s\mathfrak{h}_0)=-sT^s(\mathfrak{h}_0)\mbox{ }mod(\ker(D|_{L^2_1})).\label{GGG_7}
\end{align}
Since $T^s$ is a first order differential operator, we have
\begin{align*}
\|T^s(\mathfrak{h}_0)\|_{L^2_{-1}}\leq Cs\|\mathfrak{h}_0\|_{L^2}
\leq Cs\|\mathfrak{f}_0\|_{L^2_{-1}}.
\end{align*}
This implies
\begin{align}
sT^s(\mathfrak{h}_0)\in s^2\mathfrak{A}_1^{\frac{\kappa_0}{2}}\label{GG_1}
\end{align}
by taking $\kappa_0\geq 2C\|\mathfrak{f}_0\|_{L^2_{-1}}$.\\

{\bf Step 2}. In this step we construct $\eta_0$ and prove that $\eta_0$ will satisfy the condition (\ref{6_a1}), (\ref{6_a2}), (\ref{6_a3}).\\

 Since $\mathfrak{f}_0=0$ on $N_{\mathfrak{r}}$, we have $D\mathfrak{h}_0=0$ on $N_{\mathfrak{r}}$. So by Proposition 3.6, we can write
\begin{align*}
\mathfrak{h}_0=\left( \begin{array}{c}
\frac{h_0^+}{\sqrt{z}} \\
\frac{h^-_0}{\sqrt{\bar{z}}} 
\end{array} \right)+\mathfrak{h}_{\mathfrak{R},0}.
\end{align*}
on $N_{r}$. By Theorem 6.14, there exists $(\eta_0,c_0)$ such that
\begin{align}
\left\{ \begin{array}{ccc}
2h^+_0+a^+\eta_0-c_0=k^+_0\\
2h^-_0+a^-\bar{\eta}_0-c_0^{aps}=k^-_0
\end{array} \right.\label{GGG_8}
\end{align}
where $(k^+_0,k^-_0)\perp ker(\mathcal{T}_{a^+,a^-})$ and $(k^+_0,k^-_0)\perp (2h^+_0-k^+_0,2h^-_0-k^-_0)$ in $L^2_2$-sense (Recall that elements in $\mathbb{H}_1$ are smooth by Remark 6.16 and $(h^+_0,h^-_0)$ is in $L^2_2$ by Proposition 6.2 {\bf b}). Meanwhile, there is $\mathfrak{c}_0$ which satisfies $D\mathfrak{c}_0=0$ on $M\setminus\Sigma$ and
\begin{align*}
\mathfrak{c}_0=\left( \begin{array}{c}
\frac{c_0}{2\sqrt{z}} \\
\frac{c^{aps}_0}{2\sqrt{\bar{z}}}
\end{array} \right)+\mathfrak{c}_{\mathfrak{R},0}
\end{align*}
because $(c_0, c^{aps}_0)\in B(\ker(D|_{L^2(M\setminus\Sigma;\mathcal{S}\otimes\mathcal{I})})^0)$\footnote{It is inappropriate to denote the leading terms of a $L^2$-harmonic spinor by $(c_0,c_0^{aps})$ unless $p^-=0$ in Proposition 6.1. However, since $p^-$ is compact, we have the same argument by taking $(c_0, c^{aps}_0)$ to be an element in $B(\ker(D|_{L^2(M\setminus\Sigma;\mathcal{S}\otimes\mathcal{I})})^0)$.}.\\

 Since we have $\mathfrak{h}_0$ is given by Proposition 6.2, so
\begin{align*}
\mathfrak{r} \|h_0^{\pm}\|_{L^2}^2,
\mathfrak{r}^3\|(h_0^{\pm})_t\|^2_{L^2}, \mathfrak{r}^5\|(h_0^{\pm})_{tt}\|^2_{L^2}\leq C\|\mathfrak{h}_0\|_{L^2(N_\frac{R}{2})}^2 \leq C  \|\mathfrak{f}_{0}\|_{L_{-1}^2}^2
\end{align*}
by part {\bf c} of Proposition 6.2. By taking $\kappa_0\geq 4\frac{C}{\mathfrak{r}^5}\|\mathfrak{f}_{0}\|^2_{L^2_{-1}}$, we have
\begin{align}
\|h_0^{\pm}\|^2_{L^2}\leq \frac{\kappa_0}{2}\mathfrak{r}^2, \|(h_0^{\pm})_t\|^2_{L^2}\leq \frac{\kappa_0}{2}\mathfrak{r},\label{GGG_9}\\
 \|(h_0^{\pm})_{tt}\|^2_{L^2}\leq \frac{\kappa_0}{2}, \|\mathfrak{h}_0\|^2_{L^2}\leq \kappa_0.\nonumber
\end{align}
 Moreover, since $\mathcal{T}_{a^+,a^-}(c_0)=\bar{a}^-h^+-a^+\bar{h}^-$ $mod(\mathcal{J}(\mathbb{H}_1))$, by (\ref{GGG_9}) and Remark 6.16, we can choose $c_0$ such that
\begin{align}
\|c_0\|^2_{L^2}\leq \frac{\kappa_0}{2}\mathfrak{r}^2, \|(c_0)_t\|^2_{L^2}\leq \frac{\kappa_0}{2}\mathfrak{r},\label{GGG_10}\\ \|(c_0)_{tt}\|^2_{L^2}\leq \frac{\kappa_0}{2},  \|\mathfrak{c}_0\|^2_{L^2}\leq \kappa_0.\nonumber
\end{align}
Then
\begin{align}
\eta_0=\dfrac{\bar{a}^+}{(|a^+|^2+|a^-|^2)}(k^+_0-2h_0^+-c_0)+\dfrac{a^-}{(|a^+|^2+|a^-|^2)}(\overline{k^-_0-2h_0^--c_0^{aps}})\label{GGG_11}
\end{align}
will satisfy (\ref{6_a1}), (\ref{6_a2}) and (\ref{6_a3}). So it satisfies (\ref{6_a4}),  (\ref{6_a5}) and (\ref{6_a6}).\\

We should notice that the condition $\kappa_0\geq 2\frac{C}{\mathfrak{r}^{\frac{5}{2}}}\|\mathfrak{f}_{0}\|_{L^2_{-1}}$ will give us a constraint for $g_s$. In the following paragraphs, we should always assume $\|\mathfrak{f}_{0}\|_{L^2_{-1}}\leq \mathfrak{r}^{\frac{5}{2}}$. This assumption will give us some restriction to define $\mathcal{N}$ in Theorem 1.5. We will discuss this part in Section 7.5.\\

By (\ref{GGG_8}),(\ref{GGG_9}) and (\ref{GGG_10}), we will also have
\begin{align}
\|k_0^{\pm}\|^2_{L^2}\leq \frac{\kappa_0}{2}\mathfrak{r}^2, \|(k_0^{\pm})_t\|^2_{L^2}\leq \frac{\kappa_0}{2}\mathfrak{r},
 \|(k_0^{\pm})_{tt}\|^2_{L^2}\leq \frac{\kappa_0}{2}.\label{GGG_12}
\end{align}
Furthermore, since
\begin{align*}
\|T^s(\mathfrak{c}_0)\|_{L^2_{-1}}\leq Cs\|\mathfrak{c}_0\|_{L^2}\leq s\frac{\kappa_0}{2},
\end{align*}
we have 
\begin{align}
sT^s(\mathfrak{c}_0)\in s^2\mathfrak{A}_1^{\frac{\kappa_0}{2}}.\label{GG_2}
\end{align}

Finally, notice that we still have some freedom for the choice the $c_0$: It can be chosen differently by choosing a different leading terms or by adding a element in $\ker(D|_{L^2_1})$. By finding a suitable leading term for $c_0$, i.e., adding $c_0$ by an element in $B[\ker(\mathcal{T}\circ B)]$, we choose $c_0$ such that the corresponding $\eta_0=\xi+\eta_0^{\perp}$ with $\eta_0^{\perp}\perp \mathbb{H}_0$ and
$\xi$ satisfying (\ref{6_a1}), (\ref{6_a2}) and (\ref{6_a3}) (replacing $\eta$ by $\xi$). So (\ref{GGG_10}) still holds in this case.

\begin{remark}
 We know that $\eta_0$ satisfies (\ref{6_a1}), (\ref{6_a2}) and (\ref{6_a3}). By using the same argument in the proof of (\ref{6_a4}), we have
 \begin{align*}
\|\eta_0\|^2_{C^1}\leq C (\|\eta_0\|_{L^2_1}^2+ \|\eta_0\|_{L^2}\|(\eta_0)_t\|_{L^2}+ \|(\eta_0)_t\|_{L^2}\|(\eta_0)_{tt}\|_{L^2})\leq C\kappa_0^2\mathfrak{r}.
 \end{align*}
 Meanwhile, we can estimate the following H\"older seminorm (I follow the standard way to estimate the H\"older norms, readers can see \cite{M} for details):
 \begin{align*}
 [|\eta_t|]_{0,\frac{1}{4}}=\sup_{a\neq b} \frac{|\eta_t|(a)-|\eta_t|(b)}{|a-b|^{\frac{1}{4}}}.
 \end{align*}
 When $|a-b|\leq \mathfrak{r}$, we have
 \begin{align*}
  [|\eta_t|]_{0,\frac{1}{4}}\leq\sup_{a\neq b}\Big|\frac{1}{|a-b|^{\frac{1}{4}}}\int_a^b\partial_t|\eta_t|(s)ds\Big|\leq \sup \|\eta_{tt}\|_{L^2}|a-b|^{\frac{1}{4}}\leq C\kappa_0 \mathfrak{r}^{\frac{1}{4}}; 
 \end{align*}
when $|a-b|>\mathfrak{r}$, we have
\begin{align*}
 [|\eta_t|]_{0,\frac{1}{4}}\leq C\sup |\eta_t|\frac{1}{\mathfrak{r}^{\frac{1}{4}}}\leq C\kappa_0\mathfrak{r}^{\frac{1}{4}}.
\end{align*}
So we have the H\"older estimate
\begin{align}
\|\eta_0\|_{C^{1,\frac{1}{4}}}\leq C\kappa_0\mathfrak{r}^{\frac{1}{4}}.\label{GGG_13}
\end{align}
\end{remark}

\begin{remark}
We should also notice that the choice of $(\eta_0,k_0^{\pm})$ is unique. More precisely, for any $\xi \in\mathbb{H}_0$, the choice of $\eta_0^{\perp}$ is unique.
\end{remark}

{\bf Step 3}. First of all, since $\eta_0$ satisfies (\ref{6_a4}) , (\ref{6_a5}) and (\ref{6_a6}) (These hold because $\eta_0$ satisfies (\ref{6_a1}), (\ref{6_a2}) and (\ref{6_a3}) by Remark 7.5), we should assume that $\kappa_1$ is the constant appearing in these estimates in the beginning.\\

 On $N_R$, we can define
\begin{align*}
\mathfrak{h}_0^b=\chi_0 \left( \begin{array}{c}
\frac{h_0^+}{\sqrt{z}} \\
\frac{h^-_0}{\sqrt{\bar{z}}}
\end{array} \right);\mbox{ }\mathfrak{c}_0^b=\chi_0 \left( \begin{array}{c}
\frac{c_0}{2\sqrt{z}} \\
\frac{c^{aps}_0}{2\sqrt{\bar{z}}}
\end{array} \right);\mbox{ }\mathfrak{k}_0^b=\chi_0 \left( \begin{array}{c}
-i\frac{k_0^+}{\sqrt{z}} \\
i\frac{k^{-}_0}{\sqrt{\bar{z}}}
\end{array} \right).
\end{align*}
We also define $\mathfrak{h}^g_0=\mathfrak{h}_0-\mathfrak{h}^b_0$ and $\mathfrak{c}^g_0=\mathfrak{c}_0-\mathfrak{c}^b_0$. Here the superscript ``g'' means ``good'' because all these good terms are in $L^2_1$. On the contrary, the superscript ``b'' indicates these terms are ``bad'', which means that we should find a corresponding $\Sigma$-perturbation to eliminate them (with some finite dimensional exceptions).\\

By (\ref{GGG_7}), we have
\begin{align*}
D_{pert}(\psi+s\mathfrak{c}_0-s\mathfrak{h}_0)&= sT^s(\mathfrak{c}_0-\mathfrak{h}_0)\\
&=D_{pert}(\psi+s\mathfrak{c}^g_0-s\mathfrak{h}^g_0)+D_{pert}(s\mathfrak{c}^b_0-s\mathfrak{h}^b_0)\\
&=D_{pert}(\psi+s\mathfrak{c}^g_0-s\mathfrak{h}^g_0)+D|_{N_R}(s\mathfrak{c}^b_0-s\mathfrak{h}^b_0) \mbox{ }mod(\ker(D|_{L^2_1})),
\end{align*}
which implies 
\begin{align}
D_{pert}(\psi+s\mathfrak{c}^g_0-s\mathfrak{h}^g_0)+D|_{N_R}(s\mathfrak{c}^b_0-s\mathfrak{h}^b_0)=s^2\mathcal{A} \mbox{ }mod(\ker(D|_{L^2_1}))\label{GGG_14}
\end{align}
for some $\mathcal{A}\in\mathfrak{A}^{\kappa_0}_1$ by (\ref{GG_1}) and (\ref{GG_2}).\\

By a straightforward computation, we have
\begin{align}
D|_{N_R}(s\mathfrak{c}^b_0-s\mathfrak{h}^b_0)&=sD|_{N_R}\Bigg(\chi_0\left( \begin{array}{c}
\frac{c_0-2h_0^+}{2\sqrt{z}} \nonumber\\
\frac{c^{aps}_0-2h^-_0}{2\sqrt{\bar{z}}}
\end{array} \right)\Bigg)\label{GGG_15}\\
&=
s\chi_0\left( \begin{array}{c}
-i\frac{\partial_t a^+\eta_0+a^+ (\eta_0)_t}{\sqrt{z}} \\
i\frac{\partial_t a^-\bar{\eta}_0+a^-(\bar{\eta}_0)_t}{\sqrt{\bar{z}}}
\end{array} \right)+s\sigma(\chi_0)\mathfrak{c}^b_0-s\sigma(\chi_0)\mathfrak{h}^b_0\nonumber\\
&\mbox{ }\mbox{ }\mbox{ }\mbox{ }\mbox{ }\mbox{ }\mbox{ }\mbox{ }\mbox{ }\mbox{ }\mbox{ }\mbox{ }\mbox{ }\mbox{ }\mbox{ }\mbox{ }\mbox{ }\mbox{ }\mbox{ }\mbox{ }\mbox{ }\mbox{ }\mbox{ }\mbox{ }\mbox{ }\mbox{ }\mbox{ }\mbox{ }\mbox{ }\mbox{ }\mbox{ }\mbox{ }\mbox{ }\mbox{ }
-s\mathfrak{k}^b_0-s\sigma(\chi_0)\mathfrak{k}^b_0-sD|_{N_R}(\mathfrak{k}^b_0)\nonumber\\
&=
s\chi_0\left( \begin{array}{c}
-i\frac{\partial_t a^+\eta_0}{\sqrt{z}} \\
i\frac{\partial_t a^-\bar{\eta}_0}{\sqrt{\bar{z}}}
\end{array} \right)+
s\chi_0\left( \begin{array}{c}
-i\frac{a^+(\eta_0)_t}{\sqrt{z}} \\
i\frac{a^-(\bar{\eta}_0)_t}{\sqrt{\bar{z}}}
\end{array} \right)+s\sigma(\chi_0)\mathfrak{c}^b_0\\
&\mbox{ }\mbox{ }\mbox{ }\mbox{ }\mbox{ }\mbox{ }\mbox{ }\mbox{ }\mbox{ }\mbox{ }\mbox{ }\mbox{ }\mbox{ }\mbox{ }\mbox{ }\mbox{ }\mbox{ }\mbox{ }\mbox{ }\mbox{ }\mbox{ }\mbox{ }\mbox{ }\mbox{ }\mbox{ }\mbox{ }
-s\sigma(\chi_0)\mathfrak{h}^b_0-s\sigma(\chi_0)\mathfrak{k}^b_0-sD|_{N_R}(\mathfrak{k}^b_0).\nonumber
\end{align}
Using the operator $\Theta^0_s$ defined in Proposition 4.4, we can check that
\begin{align*}
\Theta^0_{s}(\psi)=s\chi_0\left( \begin{array}{c}
-i\frac{a^+(\eta_0)_t}{\sqrt{z}} \\
i\frac{a^-(\bar{\eta}_0)_t}{\sqrt{\bar{z}}}
\end{array} \right)+s\sigma(\chi_0)\mathfrak{c}^b_0-s\sigma(\chi_0)\mathfrak{h}^b_0-s\sigma(\chi_0)\mathfrak{k}^b_0.
\end{align*}
So (\ref{GGG_15}) implies
\begin{align}
D|_{N_R}(s\mathfrak{c}^b_0-s\mathfrak{h}^b_0)=s\chi_0\left( \begin{array}{c}
-i\frac{\partial_t a^+\eta_0}{\sqrt{z}} \\
i\frac{\partial_t a^-\bar{\eta}_0}{\sqrt{\bar{z}}}
\end{array} \right)+\Theta^0_{s}(\psi)-D_{pert}(s\mathfrak{k}^b_0).\label{6.43}
\end{align}
Now, we define
\begin{align}
\mathfrak{e_0}&=\chi_0\left( \begin{array}{c}
-i \partial_t a^-\bar{\eta}_0 \sqrt{\bar{z}}\\
-i\partial_t a^+\eta_0 \sqrt{z}
\end{array} \right)\label{GA_1}\\
\mathfrak{e}'_0&=-(\chi_z\eta+\chi_{\bar{z}}\bar{\eta})\psi-\left(\begin{array}{cc}
 0&i\chi\eta_t\\
-i\chi \bar{\eta}_t&0
\end{array}\right)\psi\label{GA_2}\\
\hat{\Theta}_s^0&=e_1(s\chi_0(\eta_0)_t\partial_z+s\chi_0 (\bar{\eta}_0)_t\partial_{\bar{z}}),\label{GA_3}\\ \mathcal{W}^0_s&=e_2(s(\chi_0)_{\bar{z}}\bar{\eta}_0\partial_z-s(\chi_0)_z\bar{\eta}_0\partial_{\bar{z}})
+e_3(-s(\chi_0)_{\bar{z}}\eta_0\partial_z+s(\chi_0)_z\eta_0\partial_{\bar{z}})\label{GA_4}
\end{align}
where $\hat{\Theta}_s^0+\mathcal{W}_s^0=\Theta_s^0$ (The RHS defined in Proposition 4.5). Then by Proposition 4.5, Proposition 4.6, (\ref{6_a4}), (\ref{6_a5}), (\ref{6_a6}) and a straightforward computation, (\ref{GGG_14}) and (\ref{6.43}) implies 
\begin{align}
D_{pert,s\chi_0\eta_0}(\psi+s\mathfrak{c}_0^g-s\mathfrak{h}^g_0+s\mathfrak{e}_0)=&s^2(\sum_{i=1}^3\mathcal{Q}_i)-D_{pert}(s\mathfrak{k}^b_0)\label{GGG_17}\\
+&s^2\mathcal{A}+s^2\mathcal{B}+s\mathcal{C}\mbox{ }mod(\ker(D|_{L^2_1}))\nonumber
\end{align}
where $\mathcal{A}\in\mathfrak{A}_1^{\kappa_0}$, $\mathcal{B}\in\mathfrak{B}_1^{\kappa_1}$, $\mathcal{C}\in\mathfrak{C}_1^{\kappa_1}$ and
\begin{align}
s^2\mathcal{Q}_1&=\hat{\Theta}_s^0(s\mathfrak{e}_0);\label{GGG_18}\\
s^2\mathcal{Q}_2&=-\hat{\Theta}_s^0(s\mathfrak{c}^g_0-s\mathfrak{h}^g_0)\label{GGG_19}\\
s^2\mathcal{Q}_3&=\hat{\Theta}_s^0(s\mathfrak{e}'_0).\label{GGG_20}
\end{align}
\ \\

{\bf Step 4}. In this step we prove that there exists $\mathfrak{e}'_i\in L^2_1$ such that $D_{s\chi_0\eta_0}(s^2\mathfrak{e}'_i)=s^2\mathcal{Q}_i+s^3\mathcal{B}+s^2\mathcal{C}$ for some $\mathcal{B}\in \mathfrak{B}_1^{\kappa_1}$ and $\mathcal{C}\in \mathfrak{C}^{\kappa_1}_1$ where $i=1,2,3$.

\begin{lemma} 
Let $Q$ be one of the following two types:
\begin{align*}
\mathcal{Q}=s^2\chi_0\left(\begin{array}{c}\frac{q^+(t)}{\sqrt{z}}\\ \frac{q^-(t)}{\sqrt{\bar{z}}}\end{array}\right)\mbox{ or }\mathcal{Q}=s^2\chi_0\left(\begin{array}{c}\frac{q^+(t)}{\sqrt{\bar{z}}}\\ \frac{q^-(t)}{\sqrt{z}}\end{array}\right)
\end{align*}
for some $q^{\pm}\in L^2_1(S^1;\mathbb{C})$ satisfying $\|q^{\pm}\|_{L^2}\leq \kappa_1 \mathfrak{r}$ and $\|(q^{\pm})_t\|_{L^2}\leq \kappa_1$. Then there exists an $L^2_1$ section $\mathfrak{e}'$ which can be written as
\begin{align*}
\mathfrak{e}'=\sum_{i\geq 2} s^i\chi_0^i\left(\begin{array}{c} e_i^+(t) \sqrt{\bar{z}} \\  e_i^-(t) \sqrt{z} \end{array}\right)
\end{align*}
for the first type and
\begin{align*}
\mathfrak{e}'=\sum_{i\geq 0} s^i\chi_0^i\left(\begin{array}{c}e_i^+(t) \sqrt{z} \\  e_i^-(t)  \sqrt{\bar{z}}\end{array}\right)
\end{align*}
for the second such that $D_{s\chi_0\eta_0}(s^2\mathfrak{e}')=s^2\mathcal{Q}+s^2\mathcal{B}+s^2\mathcal{C}$ for some $\mathcal{B}\in \mathfrak{B}_1^{\kappa_1}$ and $\mathcal{C}\in \mathfrak{C}_1^{\kappa_1}$ for all $s\leq \frac{1}{2\gamma_{_T}^2\kappa_1 \mathfrak{r}^{\frac{1}{2}}}$. Furthermore, we have $\|\mathfrak{e}'\|_{L^2_1}\leq 2\kappa_1$.
\end{lemma}

\begin{proof}
First of all, let $\mathcal{Q}$ is of the first type. We start with the element
\begin{align*}
\mathfrak{e}'_0=\chi_0\left(\begin{array}{c}
q^-\sqrt{\bar{z}}\\
q^+\sqrt{z}
\end{array}\right).
\end{align*}
Under a straightforward computation, we have
\begin{align*}
D(s^2 \mathfrak{e}'_0)= s^2\mathcal{Q}+s^2\mathcal{B}+s^2\mathcal{C}
\end{align*}
with $\mathcal{B}\in \mathfrak{B}^{s\gamma_{_T}^2\kappa_1^2}_1$ and $\mathcal{C}\in\mathfrak{C}_1^{s\kappa_1}$. Recall that by Proposition 4.4, we have
\begin{align*}
D_{s\chi_0\eta_0}=(1+\varrho^0)D+s((\chi_0)_z\eta_0+(\chi_0)_{\bar{z}}\bar{\eta}_0)e_1\partial_t+\Theta_s^0+\mathcal{H}^0_s+\mathcal{F}^0_s+\mathcal{R}_s^0.
\end{align*}
By Proposition 4.5, Proposition 4.6, (\ref{6_a4}), (\ref{6_a5}) and (\ref{6_a6}), we have
\begin{align*}
s((\chi_0)_z\eta_0+(\chi_0)_{\bar{z}}\bar{\eta}_0)e_1\partial_t(s^2\mathfrak{e}'_0)+(\mathcal{H}^0_s+\mathcal{F}^0_s)(s^2\mathfrak{e}'_0)\in s^2 \mathfrak{C}_1^{s \kappa_1^2}\\
\varrho^0 D(s^2\mathfrak{e}'_0)+\mathcal{R}^0_s(s^2\mathfrak{e}'_0)\in s^2 \mathfrak{B}_1^{s^2\gamma_{_T}^2\kappa_1^3}.
\end{align*}
Meanwhile, recall that $\Theta_s^0=[e_1(s\chi\eta_t\partial_z+s\chi\bar{\eta}_t\partial_{\bar{z}})
+e_2(s\chi_{\bar{z}}\bar{\eta}\partial_z-s\chi_z\bar{\eta}\partial_{\bar{z}})
+e_3(-s\chi_{\bar{z}}\eta\partial_z+s\chi_z\eta\partial_{\bar{z}})]$ and the decomposition
\begin{align*}
\Theta_s^0=\hat{\Theta}_s^0+\mathcal{W}^0_s
\end{align*}
given by (\ref{GA_3}) and (\ref{GA_4}). $\mathcal{W}^0_s$ is an $O(s\kappa_1)$-first order differential operator with its support on $N_{\mathfrak{r}}-N_{\frac{\mathfrak{r}}{T}}$, which implies $\mathcal{W}^0_s(s^2\mathfrak{e}'_0)\in s^2\mathfrak{B}_1^{s\gamma_{_T}^2\kappa_1^2}$. So we have
\begin{align*}
\Theta_s^0(s^2\mathfrak{e}'_0)=\hat{\Theta}_s^0(s^2\mathfrak{e}_0)
+s^2\mathcal{B}
\end{align*}
for some $\mathcal{B} \in \mathfrak{B}_1^{s\gamma_{_T}^2\kappa_1^2}$. Moreover, since
\begin{align*}
\Theta_s^0(s^2\mathfrak{e}'_0)=\chi_0 \Theta_s^0\Big(s^2\frac{\mathfrak{e}'_0}{\chi_0}\Big)+\Theta_s^0(\chi_0)s^2\frac{\mathfrak{e}'_0}{\chi_0}
\end{align*}
with the second term on the right hand side in $s^2\mathfrak{B}_1^{s\gamma_{_T}^2\kappa_1^2}$, we have
\begin{align*}
\Theta_s^0(s^2\mathfrak{e}'_0)=\chi_0 \Theta_s^0(s^2\frac{\mathfrak{e}'_0}{\chi_0})
+s^2\mathcal{B}
\end{align*}
for some $\mathcal{B} \in \mathfrak{B}_1^{s\gamma_{_T}^2\kappa_1^2}$.\\

Now we call $\mathcal{Q}^1=\chi_0 \Theta_s^0(\frac{\mathfrak{e}'_0}{\chi_0})$, which can be simplified as
\begin{align*}
\mathcal{Q}^1=s\chi_0^2\left(\begin{array}{c}\frac{q_1^+(t)}{\sqrt{\bar{z}}}\\ \frac{q_1^-(t)}{\sqrt{z}}\end{array}\right)
\end{align*}
with
\begin{align*}
q^+_1=-i(\chi_0(\bar{\eta}_0)_t)q^+\\
q^-_1=-i(\chi_0(\eta_0)_t)q^-.
\end{align*}
By using the fact $\|q^{\pm}\|_{L^2}\leq \kappa_1 \mathfrak{r}$, $\|(q^{\pm})_t\|_{L^2}\leq \kappa_1$, fundamental theorem of calculus and H\"older's inequality, we have $\|q^{\pm}\|_{L^{\infty}}\leq C\kappa_1\mathfrak{r}^{\frac{1}{2}}$. Therefore, by using (\ref{6_a3}), (\ref{6_a4}), we have $\|q^{\pm}_1\|_{L^2}\leq \kappa_1^2 \mathfrak{r}^{\frac{3}{2}}$, $\|(q^{\pm}_1)_t\|_{L^2}\leq \kappa_1^2\mathfrak{r}^{\frac{1}{2}}$. So we have
\begin{align}
D_{s\chi_0\eta_0}(s^2\mathfrak{e}'_0)=s^2\mathcal{Q}^1+s^2\mathcal{B}^0+s^2\mathcal{C}^0\label{GGG_21}
\end{align}
for some $\mathcal{B}^0 \in \mathfrak{B}_1^{s\gamma_{_T}^2\kappa_1^2 \mathfrak{r}}$ and $\mathcal{C}^0 \in \mathfrak{C}_1^{s\kappa_1 \mathfrak{r}^{\frac{1}{2}}}$.\\

Here we define an $L^{2}(S^1;\mathbb{C})$-module $\mathbb{V}$ which is generalized by
\begin{align*}
\Bigg\{\left(\begin{array}{c}
z^a\bar{z}^b\\
0
\end{array}\right),
\left(\begin{array}{c}
0\\
z^a\bar{z}^b
\end{array}\right)\Bigg|(a,b)\in (\mathbb{Z}+\frac{1}{2})\times\mathbb{Z} \mbox{ or } (a,b)\in \mathbb{Z}\times(\mathbb{Z}+\frac{1}{2}), a+b=\frac{1}{2}\Bigg\}.
\end{align*}
Now, we define a linear map $J$ by the following rule:
\begin{align*}
J\left(\begin{array}{c}
q^+z^a\bar{z}^b\\
q^-z^b\bar{z}^a
\end{array}
\right)=\left(\begin{array}{c}
-i(\eta_0)_tq^+z^b\bar{z}^a\\
i (\bar{\eta}_0)_tq^-z^a\bar{z}^b
\end{array}\right)
+
\frac{b}{a+1}\left(\begin{array}{c}
-i(\bar{\eta}_0)_tq^+z^{b-1}\bar{z}^{a+1}\\
i (\eta_0)_tq^-z^{a+1}\bar{z}^{b-1}
\end{array}\right)
\end{align*}
This map is not will defined on the entire $\mathbb{V}$ since it makes no sense when $a=-1$. However, if we start with $x=\left(\begin{array}{c}
q^+z^a\bar{z}^b\\
q^-z^b\bar{z}^a
\end{array}
\right)$ with $(a,b)=(\frac{1}{2},0)$ or $(a,b)=(0,\frac{1}{2})$, we can always define $J^n(x)$ for any $n$. Here we call that $x=\left(\begin{array}{c}
q^+z^a\bar{z}^b\\
q^-z^b\bar{z}^a
\end{array}
\right)$ is of the type $(a,b)$. To prove that $J^n(x)$ is well-defined for all $n$ when $x$ is of the type $(\frac{1}{2},0)$ or $(0,\frac{1}{2})$, we should prove that there is no term in $J^n(x)$ which is of the type $(-1,\frac{3}{2})$ or $(\frac{3}{2},-1)$. We show this fact inductively. When $n=0$, this statement is obviously true. Suppose there exists a smallest $n\in \mathbb{N}$ such that $J^n(x)$ has a component of the type $(-1,\frac{3}{2})$ or $(\frac{3}{2},-1)$. For the first case that the component appearing in $J^n(x)$ is of the type $(-1,\frac{3}{2})$, it must be generated from a component in $J^{n-1}(x)$ of the type $(\frac{3}{2},-1)$, which is a contradiction ($n$ is the smallest). For the second case that the component appearing in $J^n(x)$ is of the type $(\frac{3}{2},-1)$, either this component comes from a component in $J^{n-1}(x)$ of the type $(-1,\frac{3}{2})$, which is a contradiction again, or it comes from a component in $J^{n-1}(x)$ of the type $(\frac{5}{2},-2)$. The later case is also impossible because we start from the term of the type $(\frac{1}{2},0)$ or $(0,\frac{1}{2})$. At each time we apply $J$ on it, it will only change $(a,b)$ by adding $(\pm 1,\pm 1)$. So there must be a number $m<n-1$ such that $J^{m}(x)$ contains a component of the type $(-1,\frac{3}{2})$ or $(\frac{3}{2},-1)$, which leads a contradiction. Therefore, all the components in $J^n(x)$ are not of the type $(-1,\frac{3}{2})$, which means $J^n(x)$ is well-defined for all $n$.\\

Now we define $\mathfrak{e}'_k$ inductively by
\begin{align}
\mathfrak{e}'_{k}=s\chi_0^{k+1}J\Big(\frac{\mathfrak{e}'_{k-1}-\mathfrak{e}'_{k-2}}{\chi_0^{k}}\Big)+\mathfrak{e}'_{k-1}.\label{GGG_22}
\end{align}
By induction hypothesis, we suppose that $\mathfrak{e}_k\in L^2_1$ satisfying
\begin{align*}
D_{s\chi_0\eta_0}(s^2\mathfrak{e}'_k)=s^2\mathcal{Q}^{k+1}+s^2\mathcal{B}^k +s^2\mathcal{C}^k
\end{align*}
where $\mathcal{B}^k\in \mathfrak{B}^{\sum_{j=0}^ks^{k+1}(k+1)\gamma_{_T}^2\kappa_1^2 \mathfrak{r}}_1$,  $\mathcal{C}^k\in \mathfrak{C}^{\sum_{j=0}^ks^{k+1}(k+1)\kappa_1 \mathfrak{r}^{\frac{1}{2}}}_1$ and
\begin{align*}
\mathcal{Q}^{k+1}=\chi_0^{k+1}\hat{\Theta}^0_s\Big(\frac{\mathfrak{e}'_k-\mathfrak{e}'_{k-1}}{\chi_0^{k+1}}\Big).
\end{align*}

By taking $s<\frac{1}{\kappa_1 \mathfrak{r}^{\frac{1}{2}}}$, we can see that the sequence $\{\mathfrak{e}_{k}\}$ will converge in $L^2_1$ sense to some $\mathfrak{e}'$. Meanwhile, we can see that 
\begin{align*}
D_{s\chi_0\eta_0}(s^2\mathfrak{e}'_{k+1})=\chi_0^{k+2}\hat{\Theta}_s^0\Big(s^2\frac{(\mathfrak{e}'_{k+1}-\mathfrak{e}'_{k})}{\chi_0^{k+2}}\Big)+s^2\delta\mathcal{B}^{k+1}+s^2\delta\mathcal{C}^{k+1}+s^2\mathcal{B}^k+s^2\mathcal{C}^k.
\end{align*}
where $\delta\mathcal{B}^{k+1}\in \mathfrak{B}^{s^{k+2}(k+2)\gamma_{_T}^2\kappa_1^2}_1$ and $\delta\mathcal{C}^{k+1}\in \mathfrak{C}^{s^{k+2}(k+2)\kappa_1}_1$. We define inductively that $\mathcal{B}^{k+1}=\delta\mathcal{B}^{k+1}+\mathcal{B}^{k}$, $\mathcal{C}^{k+1}=\delta\mathcal{C}^{k+1}+\mathcal{C}^{k}$ and
\begin{align*}
\chi_0^{k+2}\hat{\Theta}_s^0\Big(\frac{\mathfrak{e}'_{k+1}-\mathfrak{e}'_{k}}{\chi_0^{k+2}}\Big)=\mathcal{Q}^{k+2}.
\end{align*}\\
Furthermore, if we take $s$ small enough such that $\sum_{j=0}^{\infty}s^{k+1}(k+1)=\frac{s}{(1-s)^2}\leq \frac{1}{\gamma_{_T}^2\kappa_1}$, e.g., $s\leq \frac{3-\sqrt{5}}{2}\frac{1}{\gamma_{_T}^2\kappa_1}$, then we have $\mathcal{B}^{k}\rightarrow \mathcal{B}\in \mathfrak{B}^{\kappa_1}_1$ and $\mathcal{C}^k\rightarrow \mathcal{C}\in \mathfrak{C}^{\kappa_1}_1$.\\

Therefore, by taking $k\rightarrow \infty$, we finish our proof by induction.\\

To get the $L^2_1$-estimate of $\mathfrak{e}'$, we have
\begin{align*}
\|\mathfrak{e}'_{k+1}-\mathfrak{e}'_{k}\|_{L^2_1}=\Big\|s^{k+1}\chi_0^{k+2}J^k\Big(\frac{\mathfrak{e}'_0}{\chi_0}\Big)\Big\|_{L^2_1}\leq \frac{1}{2^{k}}\kappa_1
\end{align*}
by using (\ref{GGG_22}), $\|q^{\pm}_{k+1}\|_{L^2}\leq \frac{\kappa_1^{k+1}\mathfrak{r}}{4}$ and $\|(q^{\pm}_{k+1})_t\|\leq \kappa_1^{k+2}$. So we have $\|\mathfrak{e}'\|_{L^2_1}\leq 2\kappa_1$
\end{proof}

Now we apply this lemma to the $\mathcal{Q}_1$, $\mathcal{Q}_2$ and $\mathcal{Q}_3$ in (\ref{GGG_18}), (\ref{GGG_19}) and (\ref{GGG_20}). Notice that for every $i\in\{1,2,3\}$, we have
\begin{align*}
s^2\mathcal{Q}_i=s^2\mathcal{Q}+s^2\mathcal{B}+s\mathcal{C}
\end{align*}
for some $\mathcal{B}\in\mathfrak{B}_1^{\kappa_1}$ and $\mathcal{C}\in\mathfrak{C}_1^{\kappa_1}$ with $\mathcal{Q}$ being one of the type in Lemma 7.7. So we can find $\mathfrak{e}'_i\in L^2_1$ such that
\begin{align}
D_{s\chi_0\eta_0}(s^2\mathfrak{e}'_i)=s^2\mathcal{Q}_i+s^2\mathcal{B}+s\mathcal{C} \mbox{ for } i=1,2,3.\label{GGG_23}
\end{align}
\\

Finally, by Proposition 4.8, we can find $\hat{\mathfrak{k}}^b_{0,s}\in L^2$ such that 
\begin{align}
D_{pert, s\chi_0\eta_0}(s\hat{\mathfrak{k}}^b_{0,s})=D_{pert}(s\mathfrak{k}^b_0).\label{GGGG_1}
\end{align}
We decompose $\hat{\mathfrak{k}}^b_{0,s}=\mathfrak{k}^b_{0,s}+s\mathfrak{k}^{\perp}_{0,s}$ where $\mathcal{T}_{a^+,a^-}\circ B(\mathfrak{k}^b_{0,s})\in\mathbb{H}_1$ and $\mathcal{T}_{a^+,a^-}\circ B(\mathfrak{k}^{\perp}_{0,s})\in \mathbb{H}_1^{\perp}$. Again, by Proposition 6.2, we have the following estimates for $B(\mathfrak{k}^{\perp}_{0,s})$:
\begin{align}
\|B(\mathfrak{k}^{\perp}_{0,s})\|^2_{L^2}\leq \frac{\kappa_0}{2}\mathfrak{r}^2, \|(B(\mathfrak{k}^{\perp}_{0,s}))_t\|^2_{L^2}\leq \frac{\kappa_0}{2}\mathfrak{r},
 \|(B(\mathfrak{k}^{\perp}_{0,s}))_{tt}\|^2_{L^2}\leq \frac{\kappa_0}{2}.\label{GGG_24}
\end{align}
 Therefore, by (\ref{GGG_23}) and (\ref{GGGG_1}), we can rewrite (\ref{GGG_17}) as
\begin{align}
D_{pert,s\chi_0\eta_0}(\psi+s\mathfrak{c}_0^g-s\mathfrak{h}^g_0+s\mathfrak{e}^g_0&+s\mathfrak{k}^b_{0,s}+s\mathfrak{k}^{\perp}_{0,s})\label{GGG_25}\\
&=s^2\mathcal{A}+s^2\mathcal{B}+s\mathcal{C} \mbox{ }mod(\ker(D|_{L^2_1}))\nonumber
\end{align}
where $\mathfrak{e}^g_0=\mathfrak{e}_0+\mathfrak{e}'_0+
s\sum_{i=1}^3\mathfrak{e}'_i$, $\mathcal{A}\in \mathfrak{A}^{\kappa_0}_1$, $\mathcal{B}\in \mathfrak{B}^{\kappa_1}_1$ and $\mathcal{C}\in \mathfrak{C}^{\kappa_1}_1$. We define
\begin{align*}
\mathfrak{k}^0:=s\mathfrak{c}_0^g-s\mathfrak{h}^g_0+s\mathfrak{e}^g_0+s\mathfrak{k}^b_{0,s}.
\end{align*}
Notice that, the argument in Section 7.3 still holds if we replace $\mathfrak{c}_0$ by any element in $\mathfrak{c}_0+ \hat{\psi}$ for some $\hat{\psi}\in\ker(D|_{L^2_1})$. So we can choose a suitable $\mathfrak{c}_0$ such that
\begin{align}
\mathfrak{k}^0\perp\ker(D|_{L^2_1}).\label{GGG_26}
\end{align}

Now we fix $\kappa_0$ and $\kappa_1$ in the rest of this paper.

\subsection{Proof of Proposition 7.3: Iteration of $(\eta_{i}, (\mathfrak{c}^g_{i}, \mathfrak{h}^g_{i}, \mathfrak{e}^g_{i},\mathfrak{k}^b_{i,s},\mathfrak{k}^{\perp}_{i,s}),\mathfrak{f}_{i},)$}
 In this section we will construct an iterative process by determining the following three constants $\mathfrak{r}<1$ and $P\in (T^{\frac{1}{8}}+1,T^{\frac{1}{5}})$ and $T>512$. We will use another constant $\varepsilon>0$ whose upper bound depends only on these constants. In addition, the upper bound for $t_0$ will be determined by these constants, too. We divide our argument into the following 4 steps.\\

{\bf Step 1}. Let $\mbox{ }\mbox{ }\mbox{ }(\eta_{j}, (\mathfrak{c}^g_{j}, \mathfrak{h}^g_{j},\mathfrak{e}^g_{j},\mathfrak{k}_{j,s}^b, \mathfrak{k}_{j,s}^{\perp}), \mathfrak{f}_{j})$ be given in the space\\
$
L^2(S^1;\mathbb{C})\times L^2_1(M\setminus\Sigma;\mathbf{\mathcal{S}}\otimes \mathcal{I})^3\times L^2(M\setminus\Sigma;\mathbf{\mathcal{S}}\otimes \mathcal{I})^2 \times (s^2\mathfrak{A}^{P^i\kappa_0}_{i+1}+s^2\mathfrak{B}^{P^i\kappa_1}_{i+1}+s\mathfrak{C}^{P^i\kappa_1}_{i+1})
$
 for all $j\leq i$ which satisfy the following condition: First, by taking
\begin{align*}
\mathfrak{k}^i:=\sum_{j=0}^i(s\mathfrak{c}_j^g-s\mathfrak{h}_j^g+s\mathfrak{e}_j^g+\mathfrak{k}_{j,s}^b)
\end{align*}
and
\begin{align*}
\psi_i:=\psi+s\mathfrak{k}^i\in L^2
\end{align*}
we have
\begin{align}
D_{pert,s\eta^i}(\psi_i+s^2\mathfrak{k}^{\perp}_{i,s})=s\mathfrak{f}_i\mbox{ }mod(\ker(D|_{L^2_1}))\label{6_5i}
\end{align}
where $\eta^i=\sum_{j=0}^i\chi_j\eta_j$.\\

 Moreover, we assume the following conditions:\\
\begin{align}
\mbox{{\bf Inductive Assumptions:}}\label{IA}
\end{align}
\ \\[-4mm]
{\bf 1.} $s\mathfrak{f}_i$ can be decomposed as
\begin{align*}
s\mathfrak{f}_i=s^{2}\mathcal{A}+s^2\mathcal{B}+s\mathcal{C}
\end{align*}
$\mbox{ }\mbox{ }\mbox{ }\mbox{ }\mbox{ }\mbox{ }\mbox{ }$where
$\mathcal{A}\in \mathfrak{A}^{P^i\kappa_0}_i$, $\mathcal{B}\in \mathfrak{B}^{P^i\kappa_1}_{i}$ and  $\mathcal{C}\in\mathfrak{C}_{i}^{P^i\kappa_1}$.\\
\ \\
{\bf 2.} The sequence $\{(\chi_j,\eta_j)\}_{1<j\leq i}$ satisfies (\ref{6_2a0}), (\ref{6_2a1}), (\ref{6_2a2}) with $\kappa_2=\varepsilon P^j \kappa_0$.\\
\ \\
{\bf 3.} We have $\mathfrak{k}^i\perp \ker(D|_{L^2_1})$ and $\{\mathfrak{k}^i \}$ converges in $L^2$ sense; $\sum_{j=0}^i(s\mathfrak{c}_j^g-s\mathfrak{h}_j^g+s\mathfrak{e}_j^g)$\\$\mbox{ }\mbox{ }\mbox{ }$  converges in $L^2_1$-sense.\\

To obtain the iteration, we need to construct the following data\\ \\
$\mbox{ }\mbox{ }\mbox{ }(\eta_{i+1}, (\mathfrak{c}^g_{i+1}, \mathfrak{h}^g_{i+1},\mathfrak{e}^g_{i+1},\mathfrak{k}_{i+1,s}^b, \mathfrak{k}_{i+1,s}^{\perp}), \mathfrak{f}_{i+1}) $\\$\in L^2(S^1;\mathbb{C})\times L^2_1(M\setminus\Sigma;\mathbf{\mathcal{S}}\otimes \mathcal{I})^3\times L^2(M\setminus\Sigma;\mathbf{\mathcal{S}}\otimes \mathcal{I})^2 \times (s^2\mathfrak{A}^{P^i\kappa_0}_{i+1}+s^2\mathfrak{B}^{P^i\kappa_0}_{i+1}+s\mathfrak{C}^{P^i\kappa_0}_{i+1})$ \\ \\form all previous data $\{(\eta_{j}, (\mathfrak{c}^g_{j}, \mathfrak{h}^g_{j}, \mathfrak{e}^g_{j},,\mathfrak{k}_{j,s}^b, \mathfrak{k}_{j,s}^{\perp}), \mathfrak{f}_{j})\}_{j\leq i}$. We will show that all conditions in (\ref{IA}) will hold inductively.\\

{\bf Step 2}. In this step, we will construct $\mathfrak{h}_{i+1}$ and determine the constant $t_0$ in terms of $\varepsilon$, $\mathfrak{r}$ and $T$. First of all, since $\mathcal{C}\in \mathfrak{C}_{i}^{P^i\kappa_1}$ so we have
\begin{align}
\chi_{i+1}\mathcal{C}\in \mathfrak{C}_{i+1}^{T^{\frac{1}{8}}P^i\kappa_1}\label{GGG_27}
\end{align}
and
\begin{align*}
(1-\chi_{i+1})\mathcal{C}\in \mathfrak{B}_{i}^{P^i\kappa_1}.
\end{align*}
Now we can rewrite
\begin{align}
s\mathfrak{f}_i=s^{2}\mathfrak{f}_{i,A}+s\mathfrak{f}_{i,B}+s\epsilon_i \label{6_5ic}
\end{align}
where
$\mathfrak{f}_{i,A}=\mathcal{A}\in \mathfrak{A}^{P^i\kappa_0}_i$, $\mathfrak{f}_{i,B}=s\mathcal{B}+(1-\chi_{i+1})\mathcal{C}$ and $\epsilon_i=\chi_{i+1}\mathcal{C}\in\mathfrak{C}_{i+1}^{T^{\frac{1}{8}}P^i\kappa_1}$.\\

 Before we start to solve $\mathfrak{h}_{i+1}$, we show that 
\begin{align}
 \mathfrak{f}_{i,B}\in \frac{\varepsilon\mathfrak{r}^{\frac{5}{2}}}{4T^{\frac{5}{2}}}\mathfrak{B}^{ P^i\kappa_1}_{i}.\label{GGG_28}
\end{align} 
 
 First, by taking $s$ small enough, we will have $s\mathcal{B}\in \frac{\varepsilon\mathfrak{r}^{\frac{5}{2}}}{8T^{\frac{5}{2}}}\mathfrak{B}^{P^i
 \kappa_1}_{i}$. This fact can be achieved if we assume $t_0\leq\frac{\varepsilon\mathfrak{r}^{\frac{5}{2}}}{8T^{\frac{5}{2}}}$.\\
 
 Second, by Lemma 2.6, for any $\zeta\in L^2_1$ and $\|\zeta\|_{L^2_1}=1$, we have
 \begin{align*}
\Big| \int\langle\zeta, (1-\chi_{i+1})\mathcal{C}\rangle\Big|&=\Big| \int\langle(1-\chi_{i+1})\zeta, \mathcal{C}\rangle\Big|\\
&\leq C\frac{\mathfrak{r}}{T^i}\|\mathcal{C}\|_{L^2}\\
&\leq  C(\frac{\mathfrak{r}}{T^i})^{\frac{5}{2}}P^i\kappa_1(\frac{\mathfrak{r}}{T^i})^{\frac{1}{8}}\\
&\leq  \frac{\varepsilon}{8}P^i\kappa_1(\frac{\mathfrak{r}}{T^{i+1}})^{\frac{5}{2}}
 \end{align*}
by taking $\mathfrak{r}$ small enough. Therefore, we have 
\begin{align*}
\|(1-\chi_{i+1})\mathcal{C}\|_{L^2_{-1}}\leq \frac{\varepsilon \mathfrak{r}^{\frac{5}{2}} P^i\kappa_1}{8T^{\frac{5(i+1)}{2}}},
\end{align*}
which implies (\ref{GGG_28}).\\

Because (\ref{6_5i}) and (\ref{6_5ic}) are true for $i$, we can solve
\begin{align*}
D_{s\eta^i}\mathfrak{h}_{i+1,A}=s\mathfrak{f}_{i,A}\mbox{ }mod(\ker(D|_{L^2_1}))\\
D_{s\eta^i}\mathfrak{h}_{i+1,B}=\mathfrak{f}_{i,B}\mbox{ }mod(\ker(D|_{L^2_1}))
\end{align*}
by using Proposition 4.9. Since $\mathfrak{f}_{i,A}|_{N_{\frac{R}{T^{i+1}}}}=\mathfrak{f}_{i,B}|_{N_{\frac{R}{T^{i+1}}}}=0$, we have
\begin{align*}
(\frac{\mathfrak{r}}{T^{i+1}})\|h^{\pm}_{i+1,A}\|_{L^2}^2,\mbox{ }  
(\frac{\mathfrak{r}}{T^{i+1}})^3&\|(h^{\pm}_{i+1,A})_t\|^2_{L^2},\mbox{ }(\frac{\mathfrak{r}}{T^{i+1}})^5\|(h^{\pm}_{i+1,A})_{tt}\|^2_{L^2}\\ &\leq \|\mathfrak{h}_{i+1,A}\|_{L^2}^2\leq s^2\|\mathfrak{f}_{i,A}\|_{L^2_{-1}}^2\leq s^2\frac{P^{2i}\kappa^2_0}{T^{5i}} .
\end{align*}
This implies that
\begin{align}
&\|h^{\pm}_{i+1,A}\|_{L^2}\leq \frac{\varepsilon P^i\kappa_0}{4T^{2(i+1)}},\mbox{ }
\|(h^{\pm}_{i+1,A})_t\|_{L^2}\leq \frac{\varepsilon P^i\kappa_0}{4T^{i+1}},\label{6_hest1}\\
&\|(h^{\pm}_{i+1,A})_{tt}\|_{L^2}\leq \frac{\varepsilon P^i\kappa_0}{4},\mbox{ }
\|\mathfrak{h}_{i+1,A}\|_{L^2}\leq \frac{\varepsilon P^i\kappa_0}{4T^{\frac{5(i+1)}{2}}}\nonumber
\end{align}
by taking $t_0\leq \frac{\varepsilon}{4}(\frac{\mathfrak{r}}{T})^\frac{5}{2}$.\\

Meanwhile, we have
\begin{align}
&\|h^{\pm}_{i+1,B}\|_{L^2}\leq \frac{\varepsilon P^i\kappa_0}{4T^{2(i+1)}},\mbox{ }
\|(h^{\pm}_{i+1,B})_t\|_{L^2}\leq \frac{\varepsilon P^i\kappa_0}{4T^{(i+1)}},\label{6_hest2}\\
&\|(h^{\pm}_{i+1,B})_{tt}\|_{L^2}\leq \frac{\varepsilon P^i\kappa_0}{4},\mbox{ }
\|\mathfrak{h}_{i+1,B}\|_{L^2}\leq \frac{\varepsilon P^i\kappa_0}{4T^{\frac{5(i+1)}{2}}}.\nonumber
\end{align}
\\

So we put these data together. Define
\begin{align*}
\mathfrak{h}_{i+1}=\mathfrak{h}_{i+1,A}+\mathfrak{h}_{i+1,B}-s\mathfrak{k}_{i,s}^{\perp}.
\end{align*}
We have
\begin{align*}
D_{pert, s\eta^i}(\psi_i-s\mathfrak{h}_{i+1})=sT^s(\mathfrak{h}_{i+1})\mbox{ }mod(\ker(D|_{L^2_1}))
\end{align*}
with
\begin{align}
sT^s(\mathfrak{h}_{i+1})\in \mathfrak{A}_{i+1}^{P^{i}\frac{\kappa_0}{2}}.\label{GGG_29}
\end{align}

{\bf Step 3}. By Theorem 6.14, we can find $(\eta_{i+1},c_{i+1})$ such that
\begin{align}
\left\{ \begin{array}{ccc}
2h^+_{i+1}+a^+\eta_{i+1}-c_{i+1}=k^+_{i+1}\\
2h^-_{i+1}+a^-\bar{\eta}_{i+1}-c_{i+1}^{aps}=k^-_{i+1}
\end{array} \right.\label{GB_1}
\end{align}
for some $(k^+_{i+1},k^-_{i+1})\perp \ker(\mathcal{T}_{a^+,a^-})$ where $(k^+_{i+1},k^-_{i+1})\perp (2h^+_{i+1}-k^+_{i+1},2h^-_{i+1}-k^-_{i+1})$ in $L^2_2$-sense. So
\begin{align}
&\|k^{\pm}_{i+1}\|^2_{L^2}\leq \frac{\varepsilon P^i\kappa_0}{2T^{2(i+1)}},\mbox{ }
\|(k^{\pm}_{i+1})_t\|^2_{L^2}\leq \frac{\varepsilon P^i\kappa_0}{2T^{i+1}},
\|(k^{\pm}_{i+1})_{tt}\|^2_{L^2}\leq \frac{\varepsilon P^i\kappa_0}{2}. \label{GGG_30}
\end{align}
 By Proposition 4.8, there exists $\mathfrak{c}_{i+1}$ where $D_{s\eta^{i}}\mathfrak{c}_{i+1}=0$ and
\begin{align*}
\mathfrak{c}_{i+1}=\left( \begin{array}{c}
\frac{c_{i+1}}{2\sqrt{z}} \\
\frac{c^{aps}_{i+1}}{2\sqrt{\bar{z}}}
\end{array} \right)+\mathfrak{c}_{\mathfrak{R},i+1}+\mathfrak{c}^s_{i+1}.
\end{align*}
Since $c_{i+1}$ satisfies $\mathcal{T}_{a^+,a^-}(c_{i+1})=\bar{a}^-(k^+_{i+1}-2h_{i+1}^+)-a^+(\overline{k^-_{i+1}-2h_{i+1}^-})$, we have
\begin{align}
&\|c_{i+1}\|^2_{L^2}\leq \frac{\varepsilon P^i\kappa_0}{2T^{2(i+1)}},\mbox{ }
\|(c_{i+1})_t\|^2_{L^2}\leq \frac{\varepsilon P^i\kappa_0}{2T^{i+1}},\label{6_cest}\\
&\|(c_{i+1})_{tt}\|^2_{L^2}\leq \frac{\varepsilon P^i\kappa_0}{2},\mbox{ }\|\mathfrak{c}_{i+1}\|^2_{L^2}\leq \varepsilon P^i\kappa_0. \nonumber
\end{align}
According to these estimates, we can show that 
\begin{align}
sT^s(\mathfrak{c}_{i+1})\in \mathfrak{A}_{i+1}^{P^{i}\frac{\kappa_0}{2}}.\label{GGG_31}
\end{align}
Moreover, we can choose $\mathfrak{c}_{i+1}$ such that $\eta_{i+1}$ in (\ref{GB_1}) satisfies
\begin{align}
\eta_{i+1}\perp \mathbb{H}_0.\label{GGG_32}
\end{align}

Meanwhile, by (\ref{6_hest1}), (\ref{6_hest2}), (\ref{GGG_30}) and (\ref{6_cest}), we can check that $\eta_{i+1}$ satisfies $i+1$-th version of (\ref{6_2a0}), (\ref{6_2a1}) and (\ref{6_2a2}) with $(\kappa_2,\kappa_3)=(\varepsilon P^i\kappa_0, \varepsilon P^i \kappa_1)$ and so it satisfies the condition (\ref{6_2a3}), (\ref{6_2a4}) and (\ref{6_2a5}). Therefore, the inductive assumption 2 in (\ref{IA}) holds. Also, we have the $\kappa_3=\varepsilon P^i \kappa_1$ version of Proposition 4.6 and Proposition 4.7. Therefore,
\begin{align}
\int_{\{r=r_0\}}|\hat{\mathcal{H}}^{i+1}_s|^2i_{\vec{n}}dVol(M)\leq \gamma_{_T}^4\varepsilon^4 P^{4i}\kappa_1^4 (\frac{\mathfrak{r}}{T^{i+1}})^{\frac{15}{16}} s^4\leq \varepsilon^2 P^{2i}\kappa_1^2 (\frac{\mathfrak{r}}{T^{i+1}})^{\frac{1}{2}} s^2\label{GGG_33}
\end{align}
by taking $P\leq T^{\frac{1}{5}}$ and $s$ small enough.\\

\begin{remark}
Here we show the estimate of the H\"older norm of $\eta_i\in \mathbb{H}_0^{\perp}$. By the argument similar to Remark 7.5, we have the following H\"older estimate
\begin{align}
\|\eta_i\|_{C^{1,\frac{1}{4}}}\leq C\kappa_0P^i\Big(\frac{\mathfrak{r}}{T^i}\Big)^{\frac{1}{4}}\leq C\kappa_0 T^{\frac{i}{5}}\Big(\frac{\mathfrak{r}}{T^i}\Big)^{\frac{1}{4}}\leq C\kappa_0\frac{\mathfrak{r}^{\frac{1}{4}}}{T^\frac{i}{20}}\label{GGG_34}
\end{align}
for all $i$.
\end{remark}

{\bf Step 4}. In this step and the next step, we construct $\mathfrak{f}_{i+1}$ and prove the inductive assumption 1 in (\ref{IA}). We follow the argument in Step 3 of Section 7.3.\\

 We can write $\mathfrak{h}_{i+1}=\mathfrak{h}_{i+1}^g+\mathfrak{h}_{i+1}^b+\mathfrak{h}_{i+1}^s$ and $\mathfrak{c}_{i+1}=\mathfrak{c}_{i+1}^g+\mathfrak{c}_{i+1}^b+\mathfrak{c}_{i+1}^s$ as follows: By Proposition 4.9, we have
$\mathfrak{h}_{i+1}=\mathfrak{h}^0_{i+1}+\mathfrak{h}^s_{i+1}$ and $\mathfrak{c}_{i+1}=\mathfrak{c}^0_{i+1}+\mathfrak{c}^s_{i+1}$ such that
\begin{align*}
D\mathfrak{h}^0_{i+1}=s\mathfrak{f}_{i,A}+\mathfrak{f}_{i,B}\mbox{ }mod(\ker(D|_{L^2_1}));\\
D\mathfrak{c}^0_{i+1}=0\mbox{ }mod(\ker(D|_{L^2_1})).
\end{align*}
Since $s\mathfrak{f}_{i,A}+\mathfrak{f}_{i,B}=0$ on $N_{\frac{r}{T^{i+1}}}$, we have
\begin{align*}
\mathfrak{h}^0_{i+1}= \left( \begin{array}{c}
\frac{h_{i+1}^+}{\sqrt{z}} \\
\frac{h^-_{i+1}}{\sqrt{\bar{z}}}
\end{array} \right)+\mathfrak{h}_{\mathfrak{R},i+1};\mbox{ }\mathfrak{c}^0_{i+1}= \left( \begin{array}{c}
\frac{c_{i+1}}{2\sqrt{z}} \\
\frac{c^{aps}_{i+1}}{2\sqrt{\bar{z}}}
\end{array} \right)+\mathfrak{c}_{\mathfrak{R},i+1}.
\end{align*}
So we define 
\begin{align*}
\mathfrak{h}^b_{i+1}=\chi_{i+1} \left( \begin{array}{c}
\frac{h_{i+1}^+}{\sqrt{z}} \\
\frac{h^-_{i+1}}{\sqrt{\bar{z}}}
\end{array} \right);\mbox{ }\mathfrak{c}^b_{i+1}= \chi_{i+1}\left( \begin{array}{c}
\frac{c_{i+1}}{2\sqrt{z}} \\
\frac{c^{aps}_{i+1}}{2\sqrt{\bar{z}}}\end{array} \right);\mbox{ }\mathfrak{k}^b_{i+1}= \chi_{i+1}\left( \begin{array}{c}
\frac{k^+_{i+1}}{\sqrt{z}} \\
\frac{k^-_{i+1}}{\sqrt{\bar{z}}}
\end{array} \right).
\end{align*}

Since $D_{s\eta^{i}}\mathfrak{c}_{i+1}=0$, (3.31) and (3.35) imply
\begin{align*}
D_{pert, s\eta^i}(\psi+s\mathfrak{c}_{i+1}-s\mathfrak{h}_{i+1})&=sT^s_0(\mathfrak{c}_{i+1}-\mathfrak{h}_{i+1})\mbox{ }mod(\ker(D|_{L^2_1}))\\
&=s^2\mathcal{A}\mbox{ }mod(\ker(D|_{L^2_1}))
\end{align*}
for some $\mathcal{A}\in\mathfrak{A}_{i+1}^{P^{i}\kappa_0}$. So we have
\begin{align}
D_{pert, s\eta^i}&(\psi+s\mathfrak{c}^g_{i+1}-s\mathfrak{h}^g_{i+1})\label{GGG_35}\\
&\mbox{ }\mbox{ }\mbox{ }\mbox{ }\mbox{ }\mbox{ }+D_{pert, s\eta^i}|_{N_{\frac{\mathfrak{r}}{T^{i+1}}}}(s\mathfrak{c}^b_{i+1}+s\mathfrak{c}^s_{i+1}-s\mathfrak{h}^b_{i+1}-s\mathfrak{h}^s_{i+1})\nonumber\\
&=s^2\mathcal{A}\mbox{ }mod(\ker(D|_{L^2_1}))\nonumber
\end{align}
for some $\mathcal{A}\in\mathfrak{A}_{i+1}^{P^{i}\kappa_0}$.\\

 We define 
\begin{align}
\mathfrak{e}_{i+1}=&\chi_{i+1}\left(\begin{array}{c}-i\partial_t a^-\bar{\eta}_{i+1}\sqrt{\bar{z}}\\-i\partial_t a^+\eta_{i+1}\sqrt{z}\end{array}\right),\label{GGG_36}\\
\hat{\Theta}_s^{i+1}=&e_1(s\chi_{i+1}(\eta_{i+1})_t\partial_z+s\chi_{i+1}(\bar{\eta}_{i+1})_t\partial_{\bar{z}}),\label{GGG_37}\\
\mathcal{W}^{i+1}_s=&e_2(s(\chi_{i+1})_{\bar{z}}\bar{\eta}_{i+1}\partial_z-s(\chi_{i+1})_z\bar{\eta}_{i+1}\partial_{\bar{z}})\label{GGG_38}\\
&\mbox{ }\mbox{ }\mbox{ }+e_3(-s(\chi_{i+1})_{\bar{z}}\eta_{i+1}\partial_z+s(\chi_{i+1})_z\eta_{i+1}\partial_{\bar{z}}).\label{GGG_39}
\end{align}

Then by $\kappa_3=\varepsilon P^i \kappa_1$ version of (\ref{DD_19}), $(\kappa_2,\kappa_3)=(\varepsilon P^i\kappa_0,\varepsilon P^i \kappa_1)$ version of (\ref{6_2a0}) - (\ref{6_2a5}), Proposition 4.7 and Proposition 4.8, (\ref{GGG_35}) implies the following equation:
\begin{align}
D_{pert, s\eta^i}(\psi_i+s\mathfrak{c}_{i+1}-\mathfrak{h}_{i+1})=s^2\mathcal{A}+s^2\mathcal{B}&+s^2\mathcal{C}+s\sum_{j=1}^2\mathcal{Q}_j\label{GGG_40}\\
&-D_{pert}(s\mathfrak{k}^b_{i+1})\mbox{ }mod(\ker(D|_{L^2_1}))\nonumber
\end{align}
where $\mathcal{A}\in  s\mathfrak{A}^{((1+C\varepsilon)P^i)\kappa_0}_{i+1}$, $\mathcal{B}\in \mathfrak{B}^{(C\varepsilon P^i)\kappa_1}_{i+1}$, $\mathcal{C}\in \mathfrak{C}^{((T^{\frac{1}{8}}+C\varepsilon)P^i)\kappa_1}_{i+1}$ for some universal constant $C>0$ and
\begin{align}
s^2\mathcal{Q}_0&=\hat{\Theta}_s^{i+1}(\psi_i-\psi)+\sum_{j=0}^i\Theta_s^j(s\mathfrak{c}_{i+1}^g-s\mathfrak{h}_{i+1}^g),\label{GGG_41}\\
s^2\mathcal{Q}_1&=\hat{\Theta}_s^{i+1}(s\mathfrak{e}_{i+1})+\sum_{j=0}^i\Theta_s^j(s\mathfrak{e}_{i+1}),\label{GGG_42}\\
s^2\mathcal{Q}_2&=-\hat{\Theta}_s^{i+1}(s\mathfrak{c}^g_{i+1}-s\mathfrak{h}^g_{i+1}).\label{GGG_43}
\end{align}
\ \\

{\bf Step 5}. In this step, we state the following lemma which is the $i+1$-th version of Lemma 7.7. The proof of this lemma can follow from the argument of Lemma 7.7 directly. So we omit the proof.
\begin{lemma}
Suppose $\mathcal{Q}$ be either the following 4 types:
\begin{align*}
s^2\chi_{i+1}\left(\begin{array}{c}\frac{q^+(t)}{\sqrt{z}}\\ \frac{q^-(t)}{\sqrt{\bar{z}}}\end{array}\right),s^2\chi_{i+1}\left(\begin{array}{c}\frac{q^+(t)}{\sqrt{\bar{z}}}\\ \frac{q^-(t)}{\sqrt{z}}\end{array}\right), s^2\left(\begin{array}{c}\frac{q^+(t)}{\sqrt{z}}\\ \frac{q^-(t)}{\sqrt{\bar{z}}}\end{array}\right)\mbox{ or }s^2\left(\begin{array}{c}\frac{q^+(t)}{\sqrt{\bar{z}}}\\ \frac{q^-(t)}{\sqrt{z}}\end{array}\right)
\end{align*}
 where $\|q^{\pm}\|_{L^2}\leq \kappa_3\frac{\mathfrak{r}}{T^{i+1}}$, $\|(q^{\pm})_t\|_{L^2}\leq \kappa_3$. Then there exists an $L^2_1$ section $\mathfrak{e}'$ which can be written as
\begin{align*}
\mathfrak{e}'=&\sum_{j\geq 0} s^j\chi_{i+1}^{j+1}\left(\begin{array}{c} e_j^+(t) \sqrt{\bar{z}} \\  e_j^-(t) \sqrt{z} \end{array}\right)
,
\sum_{j\geq 0} s^j\chi_{i+1}^{j+1}\left(\begin{array}{c}e_j^+(t) \sqrt{z} \\  e_j^-(t)  \sqrt{\bar{z}}\end{array}\right)\\
&,
\sum_{j\geq 0} s^j\chi_{i+1}^j\left(\begin{array}{c} e_j^+(t) \sqrt{\bar{z}} \\  e_j^-(t) \sqrt{z} \end{array}\right)
\mbox{ or }
\sum_{j\geq 0} s^j\chi_{i+1}^j\left(\begin{array}{c}e_j^+(t) \sqrt{z} \\  e_j^-(t)  \sqrt{\bar{z}}\end{array}\right)
\end{align*}
for the each type respectively such that $D_{s\eta^{i+1}}(s^2\mathfrak{e}')=s^2\mathcal{Q}+s^2\mathcal{B}+s\mathcal{C}$ where $\mathcal{B}\in \mathfrak{B}_{i+1}^{\kappa_3}$ and $\mathcal{C}\in \mathfrak{C}_{i+1}^{\kappa_3}$ for all $s\leq \frac{T^{\frac{i+1}{2}}}{2\gamma_{_T}^2\kappa_3 \mathfrak{r}^{\frac{1}{2}}}$. Furthermore, we have $\|\mathfrak{e}'\|_{L^2_1}\leq 2\kappa_3$.
\end{lemma}

By using this lemma, we can show that there exist $\mathfrak{e}'_{i+1,j}$, $j=0,1,2,$ such that
\begin{align*}
D_{pert, s\eta^{i+1}}(s^2\mathfrak{e}'_{i+1,j})=s^2\mathcal{Q}_j+s^2\mathcal{B}_j+s\mathcal{C}_j
\end{align*}
(with $\kappa_3=\varepsilon P^i\kappa_1$). Meanwhile, by Proposition 4.8 and Proposition 6.2, we can show that there exist $\mathfrak{k}^b_{i+1,s}$ and $\mathfrak{k}^{\perp}_{i+1,s}$ satisfying
\begin{align*}
D_{pert, s\eta^{i+1}}(\mathfrak{k}^b_{i+1,s}+s\mathfrak{k}^{\perp}_{i+1,s})=D_{pert}\mathfrak{k}^b_{i+1},
\end{align*}
$\mathcal{T}_{a^+,a^-}\circ B(\mathfrak{k}^b_{i+1,s})\in\mathbb{H}_1$, $\mathcal{T}_{a^+,a^-}\circ B(\mathfrak{k}^{\perp}_{i+1,s})\in\mathbb{H}_1^{\perp}$ and 
\begin{align}
\|B(\mathfrak{k}^{\perp}_{i+1,s})\|^2_{L^2}&\leq \frac{\varepsilon P^i\kappa_0}{2T^{2(i+1)}},\mbox{ }
\|(B(\mathfrak{k}^{\perp}_{i+1,s}))_t\|^2_{L^2}\leq \frac{\varepsilon P^i\kappa_0}{2T^{i+1}},\label{GGG_44}\\
&\|(B(\mathfrak{k}^{\perp}_{i+1,s}))_{tt}\|^2_{L^2}\leq \frac{\varepsilon P^i\kappa_0}{2}.\nonumber 
\end{align}
Therefore, we can rewrite (\ref{GGG_40}) as the following:
\begin{align}
D_{pert, s\eta^{i+1}}(\psi_i+s\mathfrak{c}^g_{i+1}-s\mathfrak{h}^g_{i+1}&+s\mathfrak{e}^g_{i+1}+s\mathfrak{k}^b_{i+1,s}+s^2\mathfrak{k}^{\perp}_{i+1,s})\label{GGG_45}\\
&=s^2\mathcal{A}+s^2\mathcal{B}+s\mathcal{C}:= \mathfrak{f}_{i+1}\mbox{ }mod(\ker(D|_{L^2_1}))\nonumber
\end{align}
with $\mathfrak{e}^g_{i+1}=\mathfrak{e}_{i+1}+\sum_{j}\mathfrak{e}'_{i+1,j}$, $\mathcal{A}\in \mathfrak{A}_{i+1}^{(1+C\varepsilon)P^i \kappa_0}$, $\mathcal{B}\in \mathfrak{B}_{i+1}^{C\varepsilon P^i\kappa_1}$ and $\mathcal{C}\in \mathfrak{C}_{i+1}^{(T^{\frac{1}{8}}+C\varepsilon) P^i \kappa_1}$. So by taking $\varepsilon \leq \frac{P-1}{2C}$ and $P>\frac{1}{T^{\frac{1}{7}}}$, we prove the inductive assumption 1 in (\ref{IA}).\\

{\bf Step 6}. Finally, we should prove the inductive assumption 3 in (\ref{IA}). To prove this part, we notice that
both $\mathfrak{h}_{i+1}^g$ and $\mathfrak{c}_{i+1}^g$ vanish on $\Sigma$. Therefore, we can do the integration by parts to get
\begin{align*}
\|\mathfrak{h}^g_{i+1}\|_{L^2_1}^2\leq \|D_{s\eta^i}\mathfrak{h}^g_{i+1}\|^2_{L^2}+C\|\mathfrak{h}^g_{i+1}\|^2_{L^2}
\end{align*}
for some constant $C$ depending on the curvature of $M$. Now, by the fact $D_{s\eta^i}\mathfrak{h}_{i+1}=0$ on $N_{\frac{\mathfrak{r}}{T^{i+1}}}$ and Corollary 3.8, we have
\begin{align*}
\|D_{s\eta^i}\mathfrak{h}_{i+1}^g\|_{L^2}\leq |\sigma(\chi_{i+1})|\|\mathfrak{h}_{i+1}\|_{L^2}+\bigg\|D_{s\eta^i}\left(\begin{array}{c}
h^+_{i+1}\sqrt{z}\\
h^-_{i+1}\sqrt{\bar{z}}
\end{array}\right)\bigg\|_{L^2(N_{\frac{\mathfrak{r}}{T^{i}}})}
\leq C\frac{P^{i+1}\kappa_1}{T^{4(i+1)}}
\end{align*}
and by (\ref{6_hest1}) and (\ref{6_hest2}) and Corollary 3.8, we have
\begin{align*}
\|\mathfrak{h}^g_{i+1}\|_{L^2}\leq C\|\mathfrak{h}_{i+1}\|_{L^2}\leq C\frac{P^{i+1}}{T^{i+1}}\kappa_1.
\end{align*}
So we have
\begin{align*}
\|\mathfrak{h}^g_{i+1}\|_{L^2_1}\leq C\frac{P^{i+1}\kappa_1}{T^{(i+1)}}.\\
\end{align*}

Similarly, we have
\begin{align*}
\|\mathfrak{c}^g_{i+1}\|_{L^2_1}\leq C\frac{P^{i+1}\kappa_1}{T^{i+1}}.
\end{align*}

For $L^2$-bounds, we have
\begin{align*}
\|\mathfrak{k}^b_{i+1,s}\|_{L^2}\leq C\|\mathfrak{k}^b_{i+1}\|_{L^2}\leq C\frac{P^i\kappa_0}{T^{2(i+1)}};\\
\|\mathfrak{k}^{\perp}_{i+1,s}\|_{L^2}\leq C\|\mathfrak{k}^{b}_{i+1}\|_{L^2}\leq C\frac{P^i\kappa_0}{T^{2(i+1)}}.
\end{align*}
So $\mathfrak{k}^{\perp}_{i+1,s}\rightarrow 0$ in $L^2$-sense. Therefore, we finish the proof of the inductive assumption 3 in (\ref{IA}).\\

By induction, we get a sequence $\psi_i\in L^2$ and a family of perturbations $\eta^i=\sum_{j=0}^i\chi_j\eta_j$ such that
\begin{align*}
D_{pert,s\eta^i}(\psi_i+s^2\mathfrak{k}^{\perp}_{i+1,s})\rightarrow 0
\end{align*}
$\mbox{ }mod(\ker(D|_{L^2_1}))$ as $i\rightarrow \infty$ in $L^2_{-1}$ sense. Moreover, since $\|\psi_{i+1}-\psi_{i}\|_{L^2}\leq C\kappa_3(\frac{P}{T})^i$ for some $C>0$, so we have $\psi_i\rightarrow \psi_s$ in $L^2$ sense. Meanwhile, since $\|\eta_i\|_{L^2_1}\leq C\kappa_3(\frac{P}{T})^i$ for some $C>0$, we have $\sum\eta_i\rightarrow \eta_s$ in $L^2_1$ sense. In addition, we choose $\mathfrak{c}_{i+1}$ (by adding an element in $\ker(D|_{L^2_1})$) such that $(\psi_{i+1}-\psi_i)\perp\ker(D|_{L^2_1})$ for all $i$. So
\begin{align}
(\psi_s-\psi_0)\perp \ker(D|_{L^2_1}).\label{GGG_46}
\end{align}

To prove that $\eta^i$ converges to a $C^1-$function, we only need to use the H\"older estimates in Remarks 7.5 and 7.8. We have
\begin{align*}
\|\eta_i\|_{C^{1,\frac{1}{4}}}\leq C\kappa_0\Big(\frac{\mathfrak{r}^{\frac{1}{4}}}{T^\frac{i}{20} }\Big).
\end{align*}
for all $i$. Therefore, by Arzela-Ascoli theorem, there is a subsequence of the partial sum $\{\eta^i\}$ converging in $C^1-$sense. So the limit, $\eta$, will be a $C^1$ circle.\\

Because $B(\psi_s)=0$, $\psi_s$ will vanish on $\Sigma$ and $D_{pert,s\eta}(\psi_s)=0\mbox{ }mod(\ker(D|_{L^2_1}))$, we have $\psi_s\in L^2_1$.

\begin{remark}
Suppose we consider a smaller neighborhood of $((g,\Sigma,e),\psi)$. This means we can take $\mathfrak{r}$, $t_0$ smaller. In this case, the constant $\varepsilon$ can be chosen smaller, too. We can see that
\begin{align*}
\frac{1}{\mathfrak{r}^{\frac{1}{4}}}\Big\|\sum_{j=1}^\infty \eta_j\Big\|_{C^1(S^1;\mathbb{C})}\rightarrow 0
\end{align*} 
as $\mathfrak{r}$ goes to 0. Similarly, we have $\mathfrak{k}_s-\mathfrak{k}^0$ is $O(\varepsilon)$. So all these terms we derived in this iteration process have order $o(s)$. 
\end{remark}

\subsection{Proof of Proposition 7.3: uniqueness of $(\eta_s,\psi_s)$}
To complete the proof of Proposition 7.3, we have to show that the solution $(\eta_s,\psi_s)$ we found in Section 7.4 is unique. The uniqueness can be obtained immediately by the following proposition.
\begin{pro}
 For any two solutions $(\eta_s,\psi_s)$ and $(\eta^*_s,\psi^*_s)$ satisfying
\begin{align}
 \eta_s-\eta_s^*=o(s)\label{G4_0}
\end{align} 
and
\begin{align}
D_{pert,s\eta_s}(\psi_s)=0,\label{G4_1a}\\
D_{pert,s\eta_s^*}(\psi^*_s)=0,\label{G4_1b}\\
(\psi_s-\psi_0)\perp\ker(D|_{L^2_1}),\label{G4_2}\\
(\psi^*_s-\psi^*_0)\perp\ker(D|_{L^2_1})\label{G4_3}
\end{align} 
 for all $s\in [0,t_0]$ and
\begin{align*}
\psi_0=\psi_0^*,
\end{align*} 
then we have $\psi_s-\psi^*_s\equiv 0$ and $\eta_s=\eta_s^*$.
\end{pro}
\begin{proof}
We can write $D_{pert,s\eta_s}=D+P^s$ where $P^s$ is a first order differential operator with the operator norm $O(s)$ and is analytic with respect to $s$. Meanwhile, since we have $\psi_s\in C^{\omega}([0,t_0];L^2_1)$, $\psi^*_s\in C^{\omega}([0,t_0];L^2_1)$ (with different zero locus), so by (\ref{G4_0}), (\ref{G4_1a}) and (\ref{G4_1b}), we have
\begin{align}
D_{pert,s\eta_s}(\psi_s-\psi^*_s)=D(\psi_s-\psi^*_s)-P^s(\psi_s-\psi^*_s)=o(s)\label{G4_4}
\end{align}
on the complement of a small neighborhood $N_R$ of $\Sigma$. So inductively, since $P^s=O(s)$, by Lemma 4.1, (\ref{G4_2}) and (\ref{G4_3}), (\ref{G4_4}) implies $(\psi_s-\psi^*_s)=O(s^k)$ for all $k$. This implies $(\psi_s-\psi^*_s)\equiv 0$ on a $M\setminus N_R$ for all $s$ small. By unique continuation property of Dirac equation (see p. 43, \cite{K}), $\psi_0=\psi_0^*$ everywhere. This implies $\eta_s=\eta_s^*$.
\end{proof}

By this proposition, we complete the proof of Proposition 7.3.

\subsection{The set $p_1(\mathcal{N})$} Here we should say more about the neighborhood $\mathcal{N}$. We define the topology on $\mathcal{Y}$ as follows. Let $((g,\Sigma,e),\psi)\in\mathfrak{M}$, we recall the definitions of $\mathcal{V}_{\Sigma,r,C}$ and $\mathscr{V}_{g, r,C'}$ in (\ref{AA_2}) and (\ref{AA_3}).We can generate the topology on $\mathcal{Y}$ by the family of open sets $\{\mathscr{V}_{g,r,C'}\times\mathcal{V}_{\Sigma,r,C}\}$ for $r<R$, $C,C'\in \mathbb{R}^+$.\\

Now we define our $\mathcal{N}=\bigcup_{r>\mathfrak{r}} \mathscr{V}_{g,r,Cr^{5/2}}\times\mathcal{V}_{\Sigma,r,C}$ for some $C$ small enough. Reader can double check the argument in Step 2 of Section 7.3: By taking $\mathcal{N}$ in this way, we have all elements in $p_1(\mathcal{N})$ will follow the argument in Section 7.

\begin{remark}
It seems to be impossible to take $\mathcal{N}$ to be $\bigcup_{r>0} \mathscr{V}_{g,r,Cr^{5/2}}\times\mathcal{V}_{\Sigma,r,C}$ because the map $f$ is not differentiable on this set. However, the choice of $\mathfrak{r}$ can be arbitrarily small.
\end{remark}

\section{Proof of the main theorem: Part II} In this section, we prove two statements. First, we have to show that $f$ we defined in (\ref{GGG_5}) is $C^1$. Second, we have to find the homeomorphism $\Upsilon:f^{-1}(0)\rightarrow\mathcal{N}\cap\mathfrak{M}$.

\subsection{$C^1$ regularity of $f$} Since the function $f$ is defined on an infinity dimensional space, so the definition of $C^1$ will be in the sense of Fr\'echet $C^1$. Here we recall the definition of Fr\'echet $C^1$.
\begin{definition}
Let $B_1$, $B_2$ are two Banach spaces. $\mathscr{F}:\mathscr{B}_1\rightarrow \mathscr{B}_2$ be a bounded operator. Then $\mathscr{F}$ is differentiable at $p$ if and only if there exists a bounded linear operator $d_p\mathscr{F}:\mathscr{B}_1\rightarrow \mathscr{B}_2$ such that
\begin{align*}
\|\mathscr{F}(p+x)-d_p\mathscr{F}(x)-\mathscr{F}(p)\|_{\mathscr{B}_2}=o( \|x\|_{\mathscr{B}_1}).
\end{align*}
In addition, if $\mathscr{F}$ is differentiable everywhere and $d_p\mathscr{F}$ vary continuously. Then we call $\mathscr{F}$ a $C^1$ map.
\end{definition}

Now let $\mathscr{F}$ maps from $\mathbb{R}^n\times \mathscr{B}$ to $\mathbb{R}^m$. Suppose we have
\begin{align}
&\frac{\partial}{\partial x_i}\mathscr{F}(p):=h_i(p)\mbox{ is continuous near }0.\label{H_1}\\
&\mbox{ The family of directional derivatives }\{D_{v}\mathscr{F}:=j_v(p)| v\in\mathscr{B}, \|v\|=1\}\label{H_2}\\&\mbox{ is equicontinuous near }0,\nonumber\\
& \{D_{v}\mathscr{F}=k_p(v)| p \in \mathbb{R}^n\times \mathscr{B}\}\mbox{ is equicontinuous on }\{v\in \mathscr{B}|\|v\|=1\}.\label{H_3}
\end{align}
Then we can define the linear operator as follows:
\begin{align}
\mathscr{L}_{p}(x,v)=\sum_{i=1}^n\frac{\partial}{\partial x_i}\mathscr{F}(p)x_i+D_{\frac{v}{\|v\|}}\mathscr{F}(p)v\label{H_4}
\end{align}
To prove this is the linear approximation, we need to check some other conditions. However, this is the only possible linear operator tangential to $\mathscr{F}$ at $0$.\\

 Now, suppose we already show that these linear operators are the differential of $\mathscr{F}$. To show $\mathscr{F}$ is $C^1$, it is sufficient to show that $\mathscr{L}_{p}$ varies continuously. So the condition (\ref{H_1}) and (\ref{H_2}) are exactly what we need to show.\\

Here I divide my proof into two parts. In first part, I will assume that $f$ is differentiable at every point and then showing that $f$ is $C^1$. In the second part, I will prove that $f$ is differentiable.\\

{\bf Step 1}. Since $\mathfrak{k}_s^b$ is analytic (w.r.t $s$), the family of directional derivatives of $f$ is actually equicontinuous at any point except $p=0$. Therefore, we only need to show conditions (\ref{H_1}) and (\ref{H_2}) hold near 0.\\

Since 
\begin{align}
s\mathfrak{k}^b_s=\sum_{i=0}^{\infty} s\mathfrak{k}^b_{i,s}=\sum_{i=0}^{\infty} s\mathfrak{k}^b_i+O(s^2),\label{H_5}
\end{align}
we can further simplify this equation by using the conclusion in Remark 7.10:
\begin{align}
s\mathfrak{k}^b_s=s\mathfrak{k}^b_0+o(s).\label{H_6}
\end{align}
Now, recall the way we construct $k^{\pm}_0$ from Step 1 and Step 2 in Section 7.3. In the case that we have no perturbation for $g$, $k^{\pm}_0=0$. That is to say, $s\mathfrak{k}^b_s=o(s)$. Therefore, the directional derivatives of $f$ along $\mathbb{H}_0$ will be 0. Meanwhile, it is obvious that they are continuous by using (\ref{H_6}).\\

To prove (\ref{H_2}), we use (\ref{H_6}) again. Here we can check that if we perturb the metric along the opposite direction, then the corresponding $\mathfrak{k}^b_0$ will only change the sign. So the directional derivatives along $p_1(\mathcal{N})$ also exist and are continuous at 0. Furthermore, since the estimates shown in Section 8 are  independent of the choice of $g_s$, so it doesn't depend on $v$. Therefore, $\{j_v(p)\}$ is equicontinuous at $0$.\\

{\bf Step 2}. In this step, we need to show that $f$ is differentiable. By Definition 8.1, we need to show that for any $p=(y,w)\in \mathbb{R}^n\times \mathscr{B}$,
\begin{align}
\|f(y+x,w+v)-\mathscr{L}_{(y,w)}(x,v)-f(y,w)\|\leq o(\sqrt{\|x\|^2+\|v\|^2_{C^2}})\label{H_7}
\end{align}  
where $x,y \in\mathbb{K}_0$ and $w,v+w\in p_1(\mathcal{N})$. All we need to show is the "small o" in (\ref{H_7}) will converge to zero uniformly. Namely, we are going to prove (\ref{H_3}) here. Now, since we already prove that the directional derivatives of $f$ are all continuous, so we can obtain (\ref{H_7}) by showing that $\{k_p(v)\}$ is equicontinuous.\\

By using the conclusion in 7.5, we suppose that $\|\partial_s g_s\|_{C^2}=  C\mathfrak{r}^{\frac{5}{2}}$, then the directional derivative of $f$ along $v=\frac{\partial_s g_s}{\|\partial_s g_s\|}$ at $g^{s_0}$ will be $\frac{1}{C\mathfrak{r}^{\frac{5}{2}}}\frac{\partial}{\partial s}(B(s\mathfrak{k}_s^b))|_{s=s_0}$. Now we can prove (\ref{H_3}) by using the fact that $\mathfrak{k}^b_s$ is analytic and the estimates (\ref{GGG_12}) and (\ref{GGG_30}).\\
 
 Therefore, we complete the proof of this part.

 \subsection{Homeomorphism $\Upsilon$} Let me summarize what we have proved in Section 7 and Section 8: For any $((g,\Sigma,e),\psi)$, there exist a neighborhood of $y=(g,\Sigma,e)$, $\mathcal{N}\subset\mathcal{Y}$, finite dimensional ball $\mathbb{B}=B_{\varepsilon}(0)\subset\mathbb{K}_0$ with $\varepsilon$ providing by Proposition 7.3 and finite dimensional vector space $\mathbb{K}_1$ all defined as above such that $f$ can be defined as
\begin{align*}
f(g_s,s\xi,\hat{\psi})=\big(\mathcal{T}_{a^+,a^-}\circ B(s\mathfrak{k}_s),P_{\ker(D|_{L^2_1})}(D_{pert,\eta_s}(\psi_s))\big)\in \mathbb{H}_1\times\ker(D|_{L^2_1})\cong\mathbb{K}_1
\end{align*}
for all $g_s\in p_1(\mathcal{N})$, $(s\xi,\hat{\psi})\in \mathbb{B}$. Here $P_{\ker(D|_{L^2_1})}$ is the orthogonal projection from $L^2$ to $\ker(D|_{L^2_1})$. Since $B(\mathfrak{k}_s)\perp\ker(\mathcal{T}_{a^+,a^-})$, $\mathcal{T}_{a^+,a^-}\circ B(s\mathfrak{k}_s)=0$ implies $B(\mathfrak{k}_s)=0$. So we can define the following map from $f^{-1}(0)$ to $\mathcal{N}$:
\begin{align}
\Upsilon:(g_s,s\xi,\hat{\psi})\mapsto ((g_s,\Sigma=h(\{(t,\eta_s(t))|t\in S^1\}),e),\psi+\hat{\psi}+s\mathfrak{k}_s)\label{H_8}
\end{align}
(Recall (\ref{AA_1}) for the definition of $h$). To complete the proof of Theorem 1.5, we have to show $\Upsilon$ is a homeomorphism.\\

 We can check that this map is injective (see Remark 7.6 and Proposition 7.11). It is also straightforward to see the inverse (from its image) of this map is continuous. Therefore, in order to prove (\ref{H_8}) defines a homeomorphism, we have to show the map $\Upsilon$ is surjective to $\mathcal{N}\cap\mathfrak{M}$ when $\mathcal{N}$, $\varepsilon$ are sufficiently small.
\begin{pro}
$\Upsilon(f^{-1}(0)\cap p_1(\mathcal{N})\times \mathbb{B})=\mathcal{N}\cap \mathfrak{M}$ for a small neightborhood $\mathcal{N}$ of $y$ and a small $\varepsilon>0$.
\end{pro}
\begin{proof}\ \\
{\bf Step 1}. Let $\mathcal{N}$ be a small neighborhood of $y$. Then we can assume that all $p\in \mathcal{N}$ is in the fiber $\mathcal{E}_x$ for some $x\in \mathscr{V}_{\Sigma,r,C}\times\mathcal{V}_{g,r,C'}$ ($\mathcal{E}$ is defined in Definition 1.3) with some small $r,C,C'\in\mathbb{R}$.\\
\ \\
{\bf Claim}. We claim the following two facts: Let $g_s,\xi,\hat{\psi}$ be given and $\mathfrak{k}_s$ be the corresponding element provided by Proposition 7.3. We have\\[1mm]
{\bf 1}. Suppose $(\mathcal{T}_{a^+,a^-}\circ B(\mathfrak{k}_s),D_{pert,\eta_s}(\psi_s))\neq 0$, then 
\begin{align}
\mathcal{N}\cap\mathcal{E}_{(g_s,s\xi,\hat{\psi})}\cap \mathfrak{M}=\emptyset.\label{H_12}
\end{align}
\ \\[-4mm]
{\bf 2}. Suppose $(\mathcal{T}_{a^+,a^-}\circ B(\mathfrak{k}_s),D_{pert,\eta_s}(\psi_s))=0$, then 
\begin{align}
\mathcal{N}\cap\mathcal{E}_{(g_s,s\xi,\hat{\psi})}\cap \mathfrak{M}\label{H_13}
\end{align}
$\mbox{ }\mbox{ }\mbox{ }$contains only one point.\\[1mm]
Suppose these two properties are true, Proposition 8.2 will be obtained directly.\\

{\bf Step 2}. To prove these two claims, we need the following fact: For any $\varepsilon_1>0$, there exists a small neighborhood $\mathcal{N}_1$ of $y$ with the following significance:\\
 For any $y'\in\mathfrak{M}\cap\mathcal{N}$, denote by $\mathbb{K}_0'$, $\mathbb{K}_1'$ the corresponding $\mathbb{K}_0$, $\mathbb{K}_1$ defined in (\ref{F_12}) and (\ref{F_13}) with respect to $y'$, we have
\begin{align}
dist(v,\mathbb{K}_0)\leq \varepsilon_1\|v\|\mbox{ for all }v\in \mathbb{K}'_0;\label{H_14}\\
dist(w,\mathbb{K}_1)\leq \varepsilon_1\|v\|\mbox{ for all }w\in \mathbb{K}'_1.\label{H_15}
\end{align}
These two inequalities can be obtained from the Fredholmness of $\mathcal{T}_{a^+,a^-}$ and the argument in Section 9.2. So we omit the proof. In the following paragraphs, we assume $\mathcal{N}=\mathcal{N}_1$ with a small $\varepsilon_1$ which will be determined later.\\

{\bf Step 3}. In this step, we prove (\ref{H_12}) in the claim. Suppose $\mathcal{N}\cap\mathcal{E}_{(g_s,s\xi,\hat{\psi})}\cap \mathfrak{M}\neq \emptyset$. Then there exists $y'=(g',\Sigma',\psi')\in \mathcal{N}\cap\mathcal{E}_{(g_s,s\xi,\hat{\psi})}\cap \mathfrak{M}$. We have
\begin{align*}
\mathbb{K}'_1=\text{coker}(\mathcal{T}_{a_1^+,a_1^-}\circ B)\oplus \ker(D|_{L^2_1}).
\end{align*}
where $(a_1^+,a_1^-)$ is the leading term of $\psi'$. Define
\begin{align}
\mathbb{H}'_1:=\text{coker}(\mathcal{T}_{a_1^+,a_1^-}\circ B).\label{H_16}
\end{align}
By (\ref{H_15}), we have
\begin{align}
dist(w,\mathbb{K}_1)\leq \varepsilon_1\|v\|\mbox{ for all }w\in \mathbb{H}'_1.\label{H_17}
\end{align}

Suppose $D_{pert,\eta_s}(\psi_s)=0$. Since $\psi_s=\psi+\hat{\psi}+s\mathfrak{k}_s$ satisfies (\ref{r4_7_2}), we have
\begin{align}
\mathcal{T}_{a_1^+,a_1^-}\circ B(\psi_s)\in \text{range}(\mathcal{T}_{a_1^+,a_1^-}\circ B)=\text{coker}(\mathcal{T}_{a_1^+,a_1^-}\circ B)^{\perp}=\mathbb{H}_1'^{\perp}\label{H_18}
\end{align}
By taking $\varepsilon_1<\frac{1}{2}$, this leads to a contradiction because 
\begin{align*}
\mathcal{T}_{a_1^+,a_1^-}\circ B(\psi_s)=\mathcal{T}_{a_1^+,a_1^-}\circ B(s\mathfrak{k}_s)\in\mathbb{H}_1
\end{align*}
and (\ref{H_17}), unless $\mathcal{T}_{a^+,a^-}\circ B(s\mathfrak{k}_s)=0$.\\

Suppose $\mathcal{T}_{a^+,a^-}\circ B(s\mathfrak{k}_s)=0$. We have $D_{pert,\eta_s}(\psi_s)\in\ker(D|_{L^2_1})\cap\text{range}(D_{pert,\eta_s}|_{L^2_1})$, which implies $D_{pert,\eta_s}(\psi_s)=0$ when $\varepsilon_1$ small (recall 3. in Proposition 2.4). So this case also leads to a contradiction.\\

{\bf Step 4}. In this step, we prove (\ref{H_13}) in the claim. Clearly, because $\psi_s$ satisfies (\ref{r4_7_2}), the set $\mathcal{N}\cap\mathcal{E}_{(g_s,s\xi,\hat{\psi})}\cap \mathfrak{M}$ is non empty. Now, suppose that we have two elements in this set, say $y_1=(g_1,\Sigma_1,\psi_1)$ and $y_2=(g_2,\Sigma_2,\psi_2)$. We can define
\begin{align}
g(s)&:=sg_1+(1-s)g_2,\label{H_19}\\
\Sigma(s)&:=h\{(s\eta(t),t)|t\in S^1\}\mbox{ such that }\Sigma(0)=\Sigma_1\mbox{ and }\Sigma(1)=\Sigma_2,\label{H_20}\\
\psi(s)&:=s\psi_1+(1-s)\psi_2\label{H_21}
\end{align}
by using $h$ as we defined in (\ref{AA_1}) to parametrize a small tubular neighborhood of $\Sigma_1$. Then, by mean value theorem, there exists $t\in(0,1)$ such that
\begin{align}
\frac{\partial}{\partial s}D^{(s)}\psi(s)\Big|_{s=t}=0\label{H_22}
\end{align}
where $D^{(s)}$ is the Dirac operator with respect to $g(s)$ and $\Sigma(s)$. By the same argument in Section 6.2, (\ref{H_22}) gives us an element $\eta$ in $\mathbb{H}'_0$ such that
\begin{align*}
a^+\eta+c&=O(\varepsilon_1);\\
a^-\bar{\eta}+c^{aps}&=O(\varepsilon_1)
\end{align*}
for some $(c,c^{aps})\in B(\ker(D|_{L^2}))$. However, since both $y_1$, $y_2$ are in $\mathcal{E}_{(g_s,s\xi,\hat{\psi})}$, we will have $\eta \perp \mathbb{H}_0$, which is also contradict to (\ref{H_14}). Therefore, we prove this proposition. 
\end{proof}

With this proposition, we have $f^{-1}(0)$ and $\mathcal{N}\cap\mathfrak{M}$ are homeomorphic.\\

 Here I should make one more remark. Recall that we define $\mathcal{N}=\bigcup_{r>\mathfrak{r}} \mathscr{V}_{g,r,Cr^{5/2}}\times\mathcal{V}_{\Sigma,r,C}$ in Section 7.5. The choice of this open set depends on $\mathfrak{r}$, so we can call it $\mathcal{N}^{(\mathfrak{r})}$ for a while. Now, what we proved in the previous section show us that there exists $C_{\mathfrak{r}}>0$, which goes to infinity as $\mathfrak{r}\rightarrow 0$,
such that $\|df|_{\mathcal{N}^{(\mathfrak{r})}}\|\leq C_{\mathfrak{r}}$ for any $\mathfrak{r}>0$. Because of this, there is no uniform control for $\|df\|$ when $\mathfrak{r}\rightarrow 0$. So we can not choose $\mathcal{N}$ to be $\bigcup_{r>0} \mathscr{V}_{g,r,Cr^{5/2}}\times\mathcal{V}_{\Sigma,r,C}$.\\

\noindent {\bf Acknowledgement:} This paper is the part of author's PhD thesis. The author wants to thank his advisor, Clifford Taubes, for helping and encouraging him in so many aspects. He also wants to thank Professor Shing-Tung Yau and Peter Kronheimer for their help. He also wants to thank all his friends at Harvard and Professor Chang-Shou Lin and Chiun-Chuan Chen for their encouragement. Finally, he want to thank his friend Chen-Yu Chi and several anonymous referees for helping him to make this article better.\\ 
 
\section{Appendix}
\subsection{Remark of the proof when the metric is not Euclidean around $\Sigma$} Here I will sketch the proof of the general case that the metric is not Euclidean near $\Sigma$. The idea is to replace Proposition 4.4 and 4.6 by Proposition 5.4 and 5.6 in the argument contained in Section 7.\\

First of all, let me summarize what we have done in Section 7. We start with a perturbation $g_s$ which gives us an extra term $\mathfrak{f}_0$ such that $D_{pert}\psi=\mathfrak{f}_0$. Then in the next step, we construct a triple $(\mathfrak{h}_0,\mathfrak{c}_0,\eta_0)$ such that $D\mathfrak{h}_0=\mathfrak{f}_0$ (mod a finite dimensional space), $D\mathfrak{c}_0=0$ and  "eliminate" the $\frac{1}{\sqrt{r}}$ part in $\mathfrak{h}_0$ by $(\mathfrak{c}_0,\eta_0)$. Then we repeat this process. Each time we will produce a new $\mathfrak{f}$ which can be decomposed into 3 parts, which belongs to $\mathfrak{A}$, $\mathfrak{B}$ and $\mathfrak{C}$ defined in Definition 7.2 (We omit all subscripts here).\\

Now, we restart the process of producing $(\mathfrak{h}_0,\mathfrak{c}_0,\eta_0)$ for the general case, but this time we replace the Dirac operator $D$ by $D^{(1)}$ defined in Section 5.2. So $D^{(1)}\mathfrak{h}_0=\mathfrak{f}_0$(mod a finite dimensional space) and $D^{(1)}\mathfrak{c}_0=0$. By using the same argument, we will still generate $\mathfrak{f}_1$. The only difference will be an extra term in $\mathfrak{C}$, which is something we can deal with. This part is generated by the operator $\delta^{(1)}$ defined in Proposition 5.3.\\

Now we do this process step by step. We replace $D$ by $D^{(i)}$ in $i$-th step, then we will get the same result. So the whole argument works for the general case. 

\subsection{Upper semi-continuity of $dim(\text{coker}(p^-))$} In this final part, I will answer the question about the upper semi-continuity of $dim(\text{coker}(p^-))$.\\

Since $p^-$ is a Fredholm operator, we can decompose $Exp^-=\text{range}(p^+)\oplus \mathbb{W}$ where $\mathbb{W}$ is finite dimensional. Now, for any $c^{\pm}\in \text{range}(p^-)$, there exists $\mathfrak{c}\in \ker(D|_{L^2})^0$ such that $B(\mathfrak{c})=c^{\pm}$. Suppose we have a perturbed Dirac operator $D_{pert}$. We can follow the argument in the proof of Proposition 4.8 to get a $\mathfrak{c}'$ such that $D_{pert}(\mathfrak{c}')=0$ and $\|B(\mathfrak{c}-\mathfrak{c}')\|\leq \varepsilon \|B(\mathfrak{c})\|$.\\

To prove $\text{coker}(p^-)$ is upper semi-continuous, we need to show that the dimension of cokernel under a small perturbation will be less or equal than the dimension of $\mathbb{W}$. We can prove this fact by showing that $\text{range}(p^-_{pert})+\mathbb{W}=Exp^-$.\\

Suppose this is not the case, then we can find $v\in Exp^-$, $\|v\|=1$ such that $v\perp \mathbb{W}$ and $v\perp \text{range}(p^-_{pert})$. So we have
\begin{align*}
\langle v, B(\mathfrak{c}')\rangle =0=\langle v, B(\mathfrak{c}')\rangle + O(\varepsilon).
\end{align*}
This means that, if we decompose $v=v_0+v_1$ where $v_0\in \text{range}(p^-)$ and $v_1=\mathbb{W}$, then we have $\|v_0\|\leq O(\varepsilon)$ and $v_1=0$. Therefore, we have $\|v\|=O(\varepsilon)$, which is a contradiction. Therefore, we prove the upper semi-continuity of $dim(\text{coker}(p^-))$.

\subsection{The bijection between $\ker(\mathfrak{L}_p|_{\delta=0})$ and $\mathbb{K}_0$ and the injection between $\text{coker}(\mathfrak{L}_p|_{\delta=0})$ and $\mathbb{K}_1$}
First of all, we prove that $\ker(\mathfrak{L}_p|_{\delta=0})$ is isomorphic to $\mathbb{K}_0=\ker(\mathcal{T}_{a^+,a^-}\circ B)$. Recall the notation in Section 6.2, we have the following map:
\begin{align*}
\mathbf{J}:\ker(\mathfrak{L}_p|_{\delta=0})&\rightarrow \ker(\mathcal{T}_{a^+,a^-}\circ B);\\
(\eta,\phi)&\mapsto \Big(\left(
\begin{array}{c}
\frac{a^+\eta}{2\sqrt{z}}\\
\frac{a^-\bar{\eta}}{2\sqrt{\bar{z}}}
\end{array}
\right)
+O_{L^2_1}(1)+\phi\Big)
\end{align*}
where the right-hand side is the element defined in (\ref{F_8}) (The element $\mathfrak{h}$ defined in (\ref{F_7}) is zero since $\delta=0$). Notice that the $O_{L^2_1}(1)$ term on the right is determined by $\eta$ and $\psi_0$. To prove $\mathbf{J}$ is a bijection, we need to find the inverse. Suppose we have $\mathfrak{u}\in \ker(\mathcal{T}_{a^+,a^-}\circ B)$, $B(\mathfrak{u})=(u^+,u^-)$. Then we can solve $\eta=\frac{2u^+}{a^+}=\frac{2u^-}{\bar{a}^-}$. Once we solve $\eta$, $\phi$ will be solved immediately. So we can construct the inverse map. Therefore, $\mathbf{J}$ is a bijection.\\

Next, we prove that there exists an injection from $\text{coker}(\mathfrak{L}_p|_{\delta=0})$ to 
\[
\mathbb{K}_1=\text{coker}(\mathcal{T}_{d^+,a^-}\circ B)\oplus \ker(D|_{L^2_1}).
\]
Notice that
\begin{align*}
\text{coker}(\mathfrak{L}_p|_{\delta=0})\subset \text{range}(D|_{L^2_1})^{\perp}= \ker(D|_{L^2})
\end{align*}
by (\ref{F_2}) and Proposition 2.4. Since we have $\ker(D|_{L^2})^0\subset \ker(D|_{L^2})$ is a dense subspace, there exists a dense subspace of $\text{coker}(\mathfrak{L}_p|_{\delta=0})$:
\begin{align*}
{\rm coker}(\mathfrak{L}_p|_{\delta=0})^0\subset \ker(D|_{L^2})^0=B(\ker(D|_{L^2})^0)\oplus  \ker(D|_{L^2_1}).
\end{align*}
 So for any $\mathfrak{u}\in \text{coker}(\mathfrak{L}_p|_{\delta=0})^0$, there is a unique corresponding pair $(B(\mathfrak{u}),\mathfrak{v})=((u^+,u^-),\mathfrak{v})\in B(\ker(D|_{L^2})^0)\oplus  \ker(D|_{L^2_1})$. Since $\mathfrak{u}\perp \text{range}(\mathfrak{L}_p|_{\delta=0})$, by integration by parts, we have $a^-\bar{u}^+=\bar{a}^+u^-$. Therefore, we can define the following map
\begin{align*}
\mathbf{L}:\text{coker}(\mathfrak{L}_p|_{\delta=0})^0&\rightarrow \text{coker}(\mathcal{T}_{a^+,a^-}\circ B)\oplus \ker(D|_{L^2_1});\\
\mathfrak{u}&\mapsto (c,\mathfrak{v})\mbox{ where }c=\frac{\bar{u}^+}{\bar{a}^+}=\frac{u^-}{a^-}.
\end{align*}
To prove this element $c$ is in $\text{coker}(\mathcal{T}_{a^+,a^-}\circ B)$, we use the inner product defined in the Section 6.3 and integration by parts:
\begin{align*}
\langle \mathcal{T}_{a^+,a^-}\circ B(\mathfrak{w}),c\rangle=Re(\int_{\mathbb{S}^1}\bar{u}^-w^+- u^+\bar{w}^-)=Re(\int_{M\setminus\Sigma}\langle D\mathfrak{u},\mathfrak{w}\rangle+\langle \mathfrak{u},D\mathfrak{w}\rangle)=0
\end{align*}
for any $\mathfrak{w}\in \ker(D|_{L^2})^0$ with $B(\mathfrak{w})=(w^+,w^-)$. Finally, it is easy to see from the definition that $\mathbf{L}$ is injective. Since $\text{coker}(\mathcal{T}_{a^+,a^-}\circ B)\oplus \ker(D|_{L^2_1})$ is finite dimensional, we have $\text{coker}(\mathfrak{L}_p|_{\delta=0})^0$ is finite dimensional. So $\text{coker}(\mathfrak{L}_p|_{\delta=0})^0=\text{coker}(\mathfrak{L}_p|_{\delta=0})$ and $L$ is defined on $\text{coker}(\mathfrak{L}_p|_{\delta=0})$ as an injective map. So we finish this proof.

\subsection{Proof of Remark 2.2}
For each $n\in\mathbb{N}$ given, let us denote the cut-off function $\chi_{\frac{1}{2n},\frac{1}{n}}$ by $\chi_n$ (see (\ref{ctf}) for the definition of $\chi_{\frac{1}{2n},\frac{1}{n}}$). So we have
\begin{align}
|\nabla\chi_n|\leq Cn
\end{align}
for some universal constant $C>0$.

By the definition of $L^2_1(M\setminus\Sigma;\mathbf{\mathcal{S}}\otimes\mathcal{I})$ and  $L^2_{1,cpt}(M\setminus\Sigma;\mathbf{\mathcal{S}}\otimes\mathcal{I})$, we have
\begin{align}
L^2_{1,cpt}(M\setminus\Sigma;\mathbf{\mathcal{S}}\otimes\mathcal{I})\subseteq L^2_1(M\setminus\Sigma;\mathbf{\mathcal{S}}\otimes\mathcal{I}).
\end{align}
To prove that they are equal, we have to show that for any $\mathfrak{u}\in L^2_1(M\setminus\Sigma;\mathbf{\mathcal{S}}\otimes\mathcal{I})$, there exists a sequence $\{\mathfrak{u}_n\}\subset L^2_{1,cpt}(M\setminus\Sigma;\mathbf{\mathcal{S}}\otimes\mathcal{I})$ such that $\mathfrak{u}_n\rightarrow\mathfrak{u}$ in $L^2_1$-norm.\\

We define $\mathfrak{u}_n=(1-\chi_n)\mathfrak{u}$. Clearly, we have
$\mathfrak{u}_n\in L^2_{1,cpt}(M\setminus\Sigma;\mathbf{\mathcal{S}}\otimes\mathcal{I})$ and
\begin{align}
\int_{M\setminus\Sigma}|\mathfrak{u}_n-\mathfrak{u}|^2\rightarrow 0
\end{align}
as $n\rightarrow \infty$. Meanwhile, we have
\begin{align}
\int_{M\setminus\Sigma}|\nabla(\mathfrak{u}_n-\mathfrak{u})|^2&=\int_{M\setminus\Sigma}|(\nabla\chi_n)\mathfrak{u}+\chi_n\nabla\mathfrak{u}|^2\nonumber\\
&\leq 2\int_{M\setminus\Sigma}|(\nabla\chi_n)\mathfrak{u}|^2+|\chi_n\nabla\mathfrak{u}|^2.
\end{align}
By (9.1) and Lemma 2.6, we have
\begin{align}
\int_{M\setminus\Sigma}|(\nabla\chi_n)\mathfrak{u}|^2\leq C^2n^2\Big(64\pi^2\frac{1}{n^2}\Big)\int_{N_{\frac{1}{n}}}|\nabla\mathfrak{u}|^2= 64C^2\pi^2\int_{N_{\frac{1}{n}}}|\nabla\mathfrak{u}|^2.
\end{align}
The left hand side of (9.5) converges to $0$ as $n\rightarrow\infty$. Meanwhile, we have
\begin{align}
\int_{M\setminus\Sigma}|\chi_n\nabla\mathfrak{u}|^2\leq \int_{N_{\frac{1}{n}}}|\nabla\mathfrak{u}|^2.
\end{align}
So the left hand side of (9.6) also converges to $0$. By (9.3), (9.4), (9.5) and (9.6), we have $\mathfrak{u}_n\rightarrow\mathfrak{u}$ in $L^2_1$-norm. We complete our proof.

\bibliographystyle{amsplain}

\end{document}